\documentclass[14pt,oneside,american]{extbook}
\usepackage[T1]{fontenc}
\usepackage[latin9]{inputenc}
\usepackage{geometry}
\usepackage{float}
\usepackage{mathrsfs}
\usepackage{amsmath}
\usepackage{amsthm}
\usepackage{amssymb}
\usepackage{setspace}
\PassOptionsToPackage{normalem}{ulem}
\usepackage{ulem}
\makeatletter
\theoremstyle{definition}
\newtheorem{example}{\protect\examplename}[section]
\theoremstyle{remark}
\newtheorem{rem}{\protect\remarkname}[section]
\theoremstyle{plain}
\newtheorem{prop}{\protect\propositionname}[section]
\theoremstyle{plain}
\newtheorem{lem}{\protect\lemmaname}[section]

\usepackage[all]{xy}

\makeatother

\usepackage{babel}
\providecommand{\examplename}{Example}
\providecommand{\lemmaname}{Lemma}
\providecommand{\propositionname}{Proposition}
\providecommand{\remarkname}{Remark}

\begin{document}
\title{Introdução à Teoria da Homotopia Abstrata}
\author{Yuri Ximenes Martins}
\maketitle

\chapter*{Prefácio}

\thispagestyle{empty}
\addcontentsline{toc}{chapter}{Introdução}

$\quad\;\,$Existe um certo consenso de que tudo aquilo que pertence
à matemática pode ser classificado dentro de três disciplinas: a \emph{álgebra},
a \emph{análise} e a \emph{topologia}. 

Na álgebra, introduz-se operações e relações num determinado conjunto,
de modo a satisfazerem uma lista de propriedades. Diz-se, então, que
em mãos se têm uma estrutura algébrica. Fala-se que duas estruturas
algébricas de mesma natureza são equivalentes quando existe uma aplicação
bijetiva entre elas, chamada de isomorfismo, cuja ação preserva as
operações e as relações envolvidas.

Na análise, são estudadas classes de funções. Por exemplo, a classe
das funções (ditas diferenciáveis) que, nas vizinhanças de cada ponto
de seu domínio, admitem uma aproximação linear. Ou então das funções
chamadas integráveis, que podem ser utilizadas como via de medida
às dimensões de alguma região.

Na topologia (e aqui estamos considerando a geometria como parte integrante
desta disciplina), tem-se os espaços topológicos como entidades fundamentais.
É nela que o conceito de continuidade vem a público em seu sentido
mais profundo. É a partir dela que se pode diferir aquilo que é localmente
válido daquilo que é globalmente verdadeiro. A noção de equivalência
ali empregada é a de homeomorfismo: compreende-se que dois espaços
topológicos são homeomorfos quando existe uma aplicação contínua e
bijetiva entre eles, cuja inversa também é contínua. Preocupa-se,
dentre outras coisas, em determinar se dois espaços topológicos dados
são ou não equivalentes entre si.

Em que pese a enorme abrangência da topologia, na ausência de imposições
adicionais e/ou elaboração de novos métodos de trabalho, não se deve
estranhar a predominância de resultados pouco profundos. Por outro
lado, tendo conhecimento da organização da álgebra e das poderosas
ferramentas da análise, é quase que impensável não utilizar destas
como meio de desenvolvimento da topologia. É neste espírito que se
veem surgir a \emph{topologia algébrica} e a \emph{topologia diferencial}.

Para afirmar que duas estruturas algébricas são isomorfas, deve-se
obter uma correspondência entre elas (digamos $f$) que seja bijetiva
e que preserve operações e relações. A princípio, em contrapartida
ao conceito de homeomorfismo (no qual é a exigida a continuidade da
inversa), pode-se ter estruturas isomorfas sem nada ser dito acerca
de $f^{-1}$. Isto reforça uma máxima: \emph{a tarefa de decidir a
equivalência entre duas estruturas algébricas é geralmente mais simples
do que o problema de encontrar um homeomorfismo entre dois espaços
topológicos}. Com isto em mente, a topologia algébrica se ocupa de
associar a cada espaço topológico uma gama de estruturas algébricas
(geralmente grupos e módulos), de modo que dois deles serão homeomorfos
somente se as respectivas estruturas associadas forem isomorfas.

A topologia diferencial se caracteriza não só pelas ferramentas que
utiliza, mas também pela classe de espaços topológicos que abrange.
Com efeito, ela se restringe às variedades diferenciáveis: entidades
localmente semelhantes aos espaços euclidianos, mas que globalmente
podem ser diferentes destes. Esta semelhança permite trazer para cá
alguns conceitos da análise que sejam de caráter local. Por exemplo,
tem sentido falar de aplicações entre variedades que, ponto a ponto,
podem ser linearmente aproximadas. A integração de funções também
tem seu lugar: é realizada via formas diferenciais. 

No espírito das disciplinas acima descritas, pode-se dizer que o trabalho
aqui apresentado se encontra na intersecção entre a topologia algébrica
e a topologia diferencial. Há um enfoque especial em aspectos abstratos
de uma área da topologia algébrica conhecida como \emph{teoria da
homotopia}. 

Inicia-se o texto desenvolvendo os pré-requisitos básicos em álgebra,
topologia e análise que se farão necessários para a leitura do texto.
Tenta-se, ali, preparar o leitor para vários conceitos que serão introduzidos
em outras situações. Para tanto, a álgebra clássica, a topologia e
a análise em variedades são apresentadas já no contexto da teoria
das categorias. 

No segundo capítulo, introduz-se efetivamente os conceitos de categoria,
de functores e de transformações naturais, os quais constituem a linguagem
moderna utilizada para associar invariantes a espaços topológicos.
Ali também se apresenta algumas situações nas quais os functores se
mostram úteis, como nos problemas de classificação de categorias e
de levantamento/extensão de morfismos. 

O terceiro capítulo é marcado pelas extensões de Kan e pelos limites
e colimites. Mostra-se, por exemplo, que produtos e equalizadores
determinam todos os limites e que limites determinam todas as extensões
de Kan. O quarto capítulo (intitulado Álgebra Abstrata) trata das
categorias monoidais, as quais são as ``categorificações'' dos monoides
usuais.

Já no quinto e no sexto capítulo, a homotopia é estudada em seu contexto
mais abstrato: em categorias com equivalências fracas e em categorias
modelo. O sétimo capítulo se inicia com uma extensão do conceito de
categoria para as $n$-categorias estritas: se uma categoria possui
objetos e morfismos, uma $2$-categoria possui também ``morfismos
entre morfismos'' (chamados de $2$-morfismos), os quais haverão
de corresponder às homotopias. Em seguida, trabalha-se no contexto
dos objetos simpliciais e damos uma estratégia de definição de categorias
``não-estritas'' em altas dimensões. Isto significa que certas propriedades
de $n$-morfismos só são válidas módulo $(n+1)$-morfismos. No fim,
mostra-se que as categorias próprias para o estuda da homotopia são,
precisamente, aquelas que possuem uma noção coerente de homotopia
entre seus morfismos.

Por fim, o oitavo capítulo trata da teoria da homotopia clássica.
Ali, um modelo na categoria dos espaços topológicos ``bem comportados''
é introduzido e, com ele, demostra-se o teorema de classificação de
fibrados principais, estuda-se os grupos de homotopia e apresenta-se
estratégias que nos permitem calculá-los em algumas situações. Assim,
é neste capítulo que os primeiros invariantes topológicos via métodos
puramente algébricos são construídos.

Deve-se ressaltar que, durante a escrita do texto, tentou-se torná-lo
tão acessível quanto foi possível. Isto levou à introdução do primeiro
capítulo, bem como à elaboração de diversos exemplos, os quais foram
cuidadosamente escolhidos. Com efeito, eles foram introduzidos ou
porque seriam úteis em outros momentos, ou por fazerem um \emph{link}
entre novos conceitos e o conhecimento prévio do leitor, ou mesmo
por realçarem diferentes aspectos da teoria que não foram expostos
no texto.

O incentivo de Mario Henrique Andrade Claudio, Rodney Josué Biezuner,
Fábio Dadam e Helvecio Geovani Fargnoli Filho foi, certamente, um
ponto crucial durante a escrita deste texto. Romero Solha leu parte
da versão preliminar e fez diversas sugestões. Por sua vez, tive a
grande sorte de contar (não só durante os momentos de escrita) com
o apoio, o carinho e a paciência da minha família, da Livian e da
Pp. A todos vocês, meus mais sinceros agradecimentos.$\underset{\underset{\underset{\underset{\;}{\;}}{\;}}{\;}}{\;}$

Yuri Ximenes Martins,

Belo Horizonte, 2018.

\chapter*{Convenções}

\addcontentsline{toc}{chapter}{Convenções}
\begin{itemize}
\item Denotaremos os conjuntos numéricos utilizando a convenção de Bourbaki:
na respectiva ordem, $\mathbb{N}$, $\mathbb{Z}$, $\mathbb{Q}$,
$\mathbb{R}$, $\mathbb{C}$ e $\mathbb{H}$ representam os números
naturais (que supomos conter o zero), os inteiros, os racionais, os
reais, os complexos e os quatérnios;
\item trabalhamos, desde o começo, com teoria ingênua de conjuntos. Assim,
pressupomos que o conceito de conjunto e a relação de pertinência
são entes primitivos dentro de nosso arcabouço lógico;
\item a intersecção, a reunião, a reunião disjunta e o produto cartesiano
entre conjuntos $X$ e $Y$ são respectivamente denotados por $X\cap Y$,
$X\cup Y$, $X\sqcup Y$ e $X\times Y$. Escrevemos $X\subset Y$
para indicar que $X$ é subconjunto de $Y$ e utilizamos de $Y-X$
para representar a coleção de todo elemento de $Y$ que não pertence
à $X$;
\item os conjuntos que possuem somente um elemento são confundidos com o
próprio elemento. Assim, por exemplo, escreve-se $x$ ao invés de
$\{x\}$. Desta forma, uma função definida em $\{x\}$ e assumindo
valores em um outro conjunto $Y$ se escreve $f:x\rightarrow Y$;
\item sempre que não especificado o domínio de uma função, este é assumido
como sendo o maior possível (isto é, as funções serão sempre tomadas
em seu domínio natural). Por exemplo, a notação $f(x)=\sqrt{x}$ indica
que $f$ tem $x\geq0$ como domínio.
\item assumimos sempre o produto interno canônico em $\mathbb{R}^{n}$.
Com respeito a ele, $\mathbb{S}^{n}$ e $\mathbb{D}^{n}$ denotam,
respectivamente, a esfera e o disco unitário fechado em dimensão $n$;
\item o bordo, o fecho e o interior de um espaço $X$ são denotados por
$\partial X$, $\overline{X}$ e $\mathrm{int}X$. 
\end{itemize}

\chapter{Preliminares}

$\quad\;\,$Neste capítulo relembramos conceitos fundamentais da álgebra,
da topologia e da análise, os quais serão fortemente utilizados ao
longo do texto. Sobre estes assuntos, pouco se exige além daquilo
que aqui é apresentado. Entretanto, não acreditamos que os singelos
esboços cunhados abaixo substituam a leitura de textos específicos
sobre os assuntos a que se designam. Neste sentido, sugerimos ao leitor
que consulte outras referências sempre que sentir necessidade. Por
exemplo, textos clássicos sobre álgebra, topologia e análise incluem
\cite{LANG_algebra,MACLANE_algebra}, \cite{Kelley_general_topology,munkres}
e \cite{lang_variedades,Hirsch}, respectivamente. 

Na primeira secção, somos devotos à álgebra. Ali se têm como objetos
básicos as \emph{estruturas algébricas}: tratam-se simplesmente de
conjuntos nos quais se sabe operar. Apresentamos diversos exemplos
de tais estruturas e estudamos os mapeamentos naturais entre elas
(estes nada mais são que funções que preservam todas as operações
envolvidas). Também estudamos maneiras universais de se projetar e
de se incluir. Ao final, no contexto da teoria dos módulos, apresentamos
o \emph{produto tensorial}, o qual é universal com respeito à propriedade
de tornar lineares as correspondências bilineares.

A segunda secção é marcada pelo estudo dos \emph{espaços topológicos}:
são os objetos mais genéricos nos quais se sabe dizer, de modo qualitativo,
se dois pontos estão próximos ou estão distantes. É neste ambiente
que se desenvolve a topologia. Os mapeamentos naturais entre dois
destes espaços são simplesmente as funções, ditas \emph{contínuas},
que preservam todas as estruturas envolvidas. Isto é, que levam pontos
próximos em pontos próximos. Mostramos que, assim como na álgebra,
na topologia existem maneiras universais de projetar e de incluir.
Para terminar, discutimos a necessidade de (e o interesse em) regras
que transfiram problemas da topologia para a álgebra.

Na terceira secção, por sua vez, estudamos as \emph{variedades diferenciáveis}.
Estas são as entidades mais gerais nas quais o conceito de diferenciabilidade
tem lugar e, portanto, constituem o ambiente ideal para falar de análise.
Os mapeamentos entre variedades são, naturalmente, as aplicações diferenciáveis.
Finalizamos a secção (e também o capítulo) discutindo que, diferentemente
do que acontece nas categorias algébricas e na categoria topológica,
na categoria das variedades não se tem formas universais de projetar
e de incluir. Isto ressalta a complexidade (e, portanto, o poderio)
da análise diante da álgebra e da topologia.

\section{Álgebra}

$\quad\;\,$Grosso modo, a álgebra clássica tem como objetos básicos
os conjuntos, chamados de \emph{estruturas algébricas}, nos quais
se encontram definidas operações sujeitas à condições de compatibilidade
e de coerência. Para nós, uma \emph{operação} em $X$ é simplesmente
uma função $*:X\times X\rightarrow X$ que goza de associatividade
e da existência de elemento neutro. Isto significa que existe um elemento
$1\in X$ tal que, para quaisquer que sejam $x,y,z\in X$, tem-se
\[
1*x=x=x*1\quad\mbox{e}\quad(x*y)*z=x*(y*z).
\]

Uma outra maneira de descrever tais relações é através da comutatividade
dos diagramas abaixo. Neles, o símbolo $\simeq$ indica uma bijeção
natural, ao passo que o mapa constante no elemento neutro $1$ é também
denotado por $1$.$$
\xymatrix{\ar[d]_-{id \times *} X\times (X\times X) \ar[r]^{\simeq} & (X\times X)\times X \ar[r]^-{*\times id} & X\times X \ar[d]^{*} & \ar[rd]_{\simeq} 1\times X \ar[r]^-{1\times id} & X\times X \ar[d]^{*} &  X \ar[ld]^{\simeq} \times 1 \ar[l]_-{id \times 1} \\
X \times X \ar[rr]_{*} && X && X }
$$

Na construção de estruturas algébricas, costuma-se impor mais algumas
condições sobre as operações. As mais usuais são a \emph{comutatividade},
significando que $x*y=y*x$, e a \emph{existência de inversos}, traduzida
por uma aplicação $inv:X\rightarrow X$, tal que $x*inv(x)=1$. Estas
condições também podem ser descritas por meio de diagramas, os quais
apresentamos abaixo. No primeiro deles, $b(x,x')=(x',x)$, ao passo
que, no segundo, $\Delta$ é o mapa diagonal $\Delta(x)=(x,x)$. Uma
estrutura é \emph{comutativa} (ou \emph{abeliana})\emph{ }quando todas
as suas operações o são.$$
\xymatrix{X\times X \ar[rd]_{*}  \ar[rr]^b && X\times X \ar[ld]^{*} & \ar[d]_{\Delta} X \ar[r]^{1} & 1 \ar[r]^{1} & X \\
& X & & X\times X \ar[rr]_{id\times inv} && X\times X \ar[u]_{*} }
$$
\begin{example}
O exemplo mais simples de estrutura algébrica é um \emph{monoide}.
Este é consistido de um conjunto $X$ dotado de uma única operação
$*:X\times X\rightarrow X$. Em seguida, tem-se os monoides abelianos.
Um monoide abeliano com inversos é chamado de \emph{grupo}. Um grupo
abeliano com outra operação, esta compatível com a primeira, recebe
o nome de \emph{anel}. Se a nova operação é comutativa e também possui
inversos, obtém-se um \emph{corpo}. Assim, pode-se dizer que um monoide
é simplesmente um conjunto no qual se sabe somente somar, ao passo
que num grupo sabe-se somar e subtrair. Por sua vez, num anel é possível
somar, subtrair e multiplicar. Finalmente, além de tudo isto, num
corpo também se sabe dividir. Nesta perspectiva, com a adição e com
a multiplicação usuais de números, $\mathbb{N}$ é monoide, $\mathbb{Z}$
é anel, enquanto que $\mathbb{Q}$, $\mathbb{R}$ e $\mathbb{C}$
são corpos. 
\end{example}
Na álgebra clássica, além de operar num conjunto, sabe-se fazer uma
estrutura algébrica \emph{agir} em outra. Neste processo, a estrutura
que\emph{ }age é sempre mais complexa (isto é, tem sempre um número
maior ou igual de operações) que aquela na qual a ação é efetivada.
De maneira mais precisa, uma \emph{ação} de uma estrutura $X$ numa
outra estrutura $Y$ é um mapa $\alpha:X\times Y\rightarrow Y$, com
imagem denotada por $\alpha(x,y)=x\cdot y$, o qual é associativo,
preserva elemento neutro e também cada uma das operações envolvidas.
Por exemplo, se $*$ e $+$ são operações em $X$ e em $Y$, então
é preciso que 
\[
(x*x')\cdot y=x\cdot(x'\cdot y),\quad1\cdot y=y\quad\mbox{e}\quad x\cdot(y+y')=(x\cdot y)+(x'\cdot y).
\]

Todas estas condições podem ser traduzidas em diagramas comutativos.
Por exemplo, as duas primeiras delas se referem, juntas, ao diagrama
abaixo apresentado.$$
\xymatrix{\ar[d]_-{id \times *} X\times (X\times Y) \ar[r]^{\simeq} & (X\times X)\times Y \ar[r]^-{*\times id} & X\times Y \ar[d]^{\alpha} & 1\times Y \ar[l]_{1\times id} \ar[dl]^{\simeq}  \\
X \times Y \ar[rr]_{\alpha} && Y }
$$

\begin{example}
Um anel $R$ agindo num grupo abeliano é usualmente chamado de \emph{módulo}
\emph{sobre $R$}. No caso em que $R$ é um corpo $\mathbb{K}$, fala-se
que se tem um \emph{espaço vetorial sobre $\mathbb{K}$}. Um anel
$R$ agindo em outro anel $X$ recebe o nome de \emph{álgebra sobre
$R$}. Observamos que dar uma álgebra sobre um anel $R$ é o mesmo
que dar um $R$-módulo $X$ dotado de uma operação adicional $*:X\times X\rightarrow X$,
a qual é compatível tanto com a operação de grupo abeliano quanto
com a ação. 
\end{example}

\subsection*{\uline{Mapeamentos}}

$\quad\;\,$Se uma estrutura algébrica é composta de operações, um
mapeamento entre estruturas de mesma natureza há de preservar cada
uma das operações envolvidas. Tais mapeamentos são genericamente chamados
de \emph{homomorfismos}. 
\begin{example}
Um monoide é formado de uma única operação. Assim, um mapeamento (ou
homomorfismo) entre dois monoides (digamos $X$ e $Y$, com respectivas
operações $+$ e $+'$), é uma função $f:X\rightarrow Y$ tal que
$f(x+x')=f(x)+'f(x')$. A mesma definição se aplica para grupos, pois
estes têm igual número de operações que os monoides. Por sua vez,
se adicionarmos uma operação à tais estruturas, tornando-a anéis,
então a correspondente noção de homomorfismo haverá de ser um mapa
que satisfaz
\[
f(x+x')=f(x)+'f(x')\quad\mbox{e também}\quad f(x*x')=f(x)*'f(x').
\]
\end{example}
Quando uma estrutura age em outra, há, além das operações, uma nova
propriedade: a própria ação. Assim, um mapeamento entre tais entidades
deve preservar não só as operações, mas também as ações.
\begin{example}
Como vimos, os módulos sobre um anel $R$ são simplesmente os grupos
abelianos sujeitos a alguma ação de $R$. Assim, um homomorfismo entre
dois destes módulos há de ser um homomorfismo $f:X\rightarrow Y$
entre os grupos subjacentes, o qual preserva a ação de $R$. Isto
é, que satisfaz a condição $f(a\cdot x)=a\cdot f(x)$, com $a\in R$.
Também se diz que $f$ é $R$-\emph{linear}.

\begin{example}
A composição entre homomorfismo é sempre um homomorfismo, independente
de qual seja a classe de estruturas algébricas consideradas. Assim,
se compomos homomorfismos de grupos, obtemos aplicação com iguais
propriedades. Da mesma forma, se compomos mapeamentos entre módulos,
temos como resultado um homomorfismo de módulos. Algumas vezes, evidencia-se
este fato dizendo que a composição \emph{preserva linearidade}.

\begin{example}
Se $Y\in\mathbf{Alg}$, então no conjunto de todas as funções $f:X\rightarrow Y$
pode-se definir operações naturais que fazem dele um objeto de $\mathbf{Alg}$.
Isto é feito ``termo a termo''. De maneira mais precisa, para cada
operação e cada ação, basta pôr
\[
(f+g)(x)=f(x)+g(x)\quad\mbox{e}\quad(a\cdot f)(x)=a\cdot f(x).
\]
Assim, em particular, o conjunto $\mathrm{Hom}_{\mathbf{Alg}}(X;Y)$
de todos os homomorfismos entre $X$ e $Y$ admite uma estrutura natural
que o torna membro da categoria $\mathbf{Alg}$.
\end{example}
\end{example}
\end{example}
Fala-se que duas estruturas algébricas $X$ e $Y$, ambas de mesma
natureza, são \emph{isomorfas }no momento em que se consegue obter
homomorfismos $f:X\rightarrow Y$ e $g:Y\rightarrow X$, cujas composições
$g\circ f$ e $f\circ g$ coincidem com as respectivas identidades.
Se este é o caso, escreve-se $X\simeq Y$ e diz-se que o mapa $f$
é um \emph{isomorfismo} entre $X$ e $Y$. 

Observamos que, se duas entidades são isomorfas, então elas possuem
essencialmente as mesmas propriedades algébricas. Assim, em geral
não se trabalha com um único objeto $X$, mas sim, de maneira simultânea,
com todos os objetos que são equivalentes a ele. Quando uma propriedade
é válida numa entidade isomorfa a $X$, diz-se que ela vale em $X$
\emph{a menos de isomorfismos }(ou \emph{módulo isomorfismos}).

Uma propriedade importante das equivalências presentes na álgebra
clássica é a seguinte: para que um homomorfismo seja um isomorfismo
basta que ele seja bijetivo. Com efeito, se este é o caso, então\emph{
sua inversa é obrigatoriamente linear}. De um modo geral, esta caracterização
não é válida para as equivalências definidas em outros campos da matemática
como, por exemplo, na topologia e na análise. 
\begin{example}
O conjunto dos automorfismos de uma entidade algébrica $X$ (isto
é, o conjunto dos isomorfismos de $X$ nela mesma) sempre possui uma
estrutura de grupo quando dotado da operação de composição. O elemento
neutro é precisamente a identidade $id_{X}:X\rightarrow X$.
\end{example}
Uma vez fixada uma classe de estruturas algébricas, tem-se uma correspondente
noção de homomorfismo, a qual é preservada por composições e satisfeita
pelas funções identidade. Por conta disso, diz-se que estruturas de
mesma natureza são os \emph{objetos} de uma \emph{categoria}, tendo
os homomorfismos como \emph{morfismos} e a composição usual de funções
como \emph{lei de composição}. Desta forma, tem-se a categoria dos
conjuntos, a categoria grupos, a categoria dos anéis e a categoria
dos $R$-módulos, as quais são usualmente representadas por $\mathbf{Set}$,
$\mathbf{Grp}$, $\mathbf{Rng}$ e $\mathbf{Mod}_{R}$. No que segue,
$\mathbf{Alg}$ denota uma categoria algébrica genérica.

\subsection*{\uline{Operações}}

$\quad\;\,$Em geral, as categorias algébricas vêm dotadas de duas
operações internas, usualmente chamadas de \emph{produto }e de \emph{coproduto},
as quais nos permitem construir novas estruturas a partir de outras
dadas. Tais operações são interessantes por serem \emph{universais}
em um certo sentido. A motivação para a definição de universalidade
pode ser obtida olhando para os conjuntos (afinal, eles são as entidades
algébricas mais simples). 

Dados conjuntos $X$ e $Y$, sabe-se realizar o produto cartesiano
$X\times Y$ e também a reunião disjunta $X\sqcup Y$. Tem-se maneiras
naturais de \emph{projetar }o produto nos conjuntos subjacentes, e
também de \emph{incluí-los} na sua reunião. Com efeito, para $X$
(e, de maneira análoga, para $Y$) basta considerar as respectivas
correspondências 
\[
\pi_{1}:X\times Y\rightarrow X\quad\mbox{e}\quad\imath_{1}:X\rightarrow X\sqcup Y,\quad\mbox{tais que}\quad\pi_{1}(x,y)=x\quad\mbox{e}\quad\imath_{1}(x)=x.
\]

Este produto e estas projeções descrevem todas as outras. Afinal,
se $X'$ é um conjunto arbitrário dotado de funções $f:X'\rightarrow X$
e $g:X'\rightarrow Y$, as quais representam uma nova maneira de projetar,
então existe uma única aplicação $u:X'\rightarrow X\times Y$ cumprindo
$f=\pi_{1}\circ u$ e $g:\pi_{2}\circ u$. Dualmente, dadas $f$ e
$g$, agora representando nova maneira de incluir, obtém-se único
um $u:X\sqcup Y\rightarrow X'$ que verifica as relações $f=u\circ\imath_{1}$
e $g=u\circ\imath_{2}$. É neste sentido (representado pela comutatividade
dos digramas abaixo) que o produto cartesiano e a reunião disjunta
de conjuntos são operações \emph{universais}. 

\begin{equation}{
\xymatrix{& & & X & & & & X \ar@/_/[llld]_f \ar[dl]_{\imath _1}  \\
X' \ar@{-->}[rr]^{u} \ar@/^/[rrru]^f \ar@/_/[rrrd]_g & & X\times Y \ar[ur]^{\pi _1} \ar[dr]_{\pi _2} && X'   & & \ar@{-->}[ll]_{u} X\oplus Y   &\\
& & & Y & & & & Y \ar@/^/[lllu]^g \ar[ul]^{\imath _2}  }}
\end{equation}

Agora, uma vez que as propriedades universais do produto cartesiano
e da reunião disjunta foram descritas simplesmente em termos de mapeamentos
entre conjuntos e comutatividade de diagramas, podemos abstrair tal
conceito, levando-o para outras categorias algébricas. Com efeito,
diz-se que uma categoria algébrica $\mathbf{Alg}$ possui \emph{produtos}
quando há uma regra $\times$ que toma duas entidades $X,Y\in\mathbf{Alg}$
e associa uma terceira $X\times Y$, a qual vem acompanhada de homomorfismos
\[
\pi_{1}:X\times Y\rightarrow X\quad\mbox{e}\quad\pi_{2}:X\times Y\rightarrow Y,
\]
que são \emph{universais} no sentido do primeiro dos diagramas acima.
Igualmente, fala-se $\mathbf{Alg}$ possui \emph{coprodutos} quando
existe uma regra $\oplus$, que a cada par de objetos $X,Y\in\mathbf{Alg}$
devolve um outro $X\oplus Y$, assim como homomorfismos 
\[
\imath_{1}:X\rightarrow X\oplus Y\quad\mbox{e}\quad\mbox{\ensuremath{\imath}}_{2}:Y\rightarrow X\oplus Y,
\]
os quais são\emph{ }universais no contexto do segundo dos diagramas
anteriormente apresentados. 
\begin{example}
Em algumas categorias algébricas, os produtos e coprodutos são facilmente
descritos. Em outras, no entanto, eles (sobretudo os coprodutos) podem
ter uma expressão não muito simplificada. Por exemplo, na categoria
$\mathbf{Grp}$ dos grupos, o produto entre $X$ e $Y$ é simplesmente
o produto cartesiano entre eles, dotado da operação ``componente
à componente'' 
\[
(x,y)+(x',y')=(x+x,y+y'),
\]
em que estamos utilizando a mesma notação para representar as operações
em todos os conjuntos. As projeções, neste caso, são as projeções
usuais. Observamos agora que, se um anel $R$ age em $X$ e $Y$,
então ele também age no grupo $X\times Y$. Com efeito, basta tomá-la
``componente à componente'' mais uma vez. Isto significa que estruturas
de módulos em $X$ e $Y$ induzem correspondente estrutura em $X\times Y$,
tal que as projeções usuais são lineares. Desta forma, elas estendem
o produto de $\mathbf{Grp}$ à categoria $\mathbf{Mod}_{R}$.

\begin{example}
Em contrapartida ao que vimos ocorrer para produtos no exemplo anterior,
na categoria dos módulos, os coprodutos são facilmente descritos,
ao passo que possuem uma expressão complicada em $\mathbf{Grp}$.
A razão é simples: módulos são estruturas abelianas, algo que não
acontece com todos os grupos. Neste mesmo espírito, espera-se que
coprodutos em $\mathbf{Rng}$ sejam difíceis de explicitar. E, de
fato, assim o é. Em $\mathbf{Mod}_{R}$ o coproduto entre $X$ e $Y$
é a \emph{soma direta} entre eles: trata-se do módulo \emph{gerado}
por tais entidades. Isto é, trata-se do conjunto de todas as ``combinações
lineares'' $a\cdot x+b\cdot y$, com $a,b\in R$, $x\in X$ e $y\in Y$.
Para uma expressão dos coprodutos em $\mathbf{Grp}$ e em $\mathbf{Rng}$,
veja \cite{LANG_algebra,MACLANE_algebra}.
\end{example}
\end{example}
Observamos que, como obtido nos exemplos anteriores, os produtos de
$\mathbf{Grp}$ e $\mathbf{Mod}_{R}$ são simplesmente os produtos
em $\mathbf{Set}$. Isto significa que a inclusão de tais categorias
em $\mathbf{Set}$ \emph{preserva produtos}. Nem toda inclusão possui
esta propriedade. 

\subsection*{\uline{Produto Tensorial}}

$\quad\;\,$Se uma categoria algébrica $\mathbf{Alg}$ possui produtos,
os quais são preservados pela inclusão, ali podemos falar de \emph{mapas
bilineares}. Com efeito, diz-se que uma função $f:X\times Y\rightarrow Z$
é \emph{bilinear} quando é um homomorfismo em cada uma de suas entradas.
Isto significa que, para quaisquer operações nas estruturas $X$ e
$Y$, vale
\[
f(x+x',y)=f(x,y)+f(x',y)\quad\mbox{e}\quad f(x,y+y')=f(x,y)+f(x',y).
\]

Particularmente, se existem ações em $X$ e $Y$, então estas também
dever ser preservadas, em cada entrada, por $f$. Mais precisamente,
deve valer 
\[
f(a\cdot x',y)=a\cdot f(x,y)\quad\mbox{e}\quad f(x,a\cdot y')=a\cdot f(x',y).
\]

Observamos que, em geral, \emph{uma função bilinear} \emph{não é um
homomorfismo}. Em outras palavras, ser bilinear não é o mesmo que
ser linear de $X\times Y$ em $Z$. Isto se traduz no fato de que,
de um modo geral, não existem bijeções entre
\[
\mathrm{Hom}_{\mathbf{Alg}}(X\times Y;Z)\quad\mbox{e}\quad\mathrm{Bil}_{\mathbf{Alg}}(X\times Y;Z)
\]

O problema está justamente em $\times$. Uma regra $\otimes$, substituta
para $\times$, para a qual as bijeções anteriores podem ser encontradas,
chama-se \emph{produto tensorial} em $\mathbf{Alg}$. Em outras palavras,
este se trata de uma regra que toma objetos $X,Y\in\mathbf{Alg}$
e devolve um outro $X\otimes Y$, de tal modo que dar um morfismo
$X\otimes Y\rightarrow Z$ é o mesmo que dar uma correspondência bilinear
$X\times Y\rightarrow Z$. 
\begin{example}
Na categoria $\mathbf{Set}$, os objetos são conjuntos sem nenhum
estrutura adicional e, portanto, ser bilinear e não o ser é a mesma
coisa. Disto segue que, em tal categoria, o produto cartesiano é o
próprio produto tensorial. 

\begin{example}
Diferentemente do que acontece em $\mathbf{Set}$, o produto usual
não é tensorial na categoria dos módulos. No entanto, ali ainda se
sabe construir uma tal regra. A ideia é simples: procura-se pela relação
que equivalência no modulo gerado pelo conjunto $X\times Y$ tal que,
para cada mapa bilinear $f:X\times Y\rightarrow Z$, a função $\hat{f}(\pi(x,y))=f(x,y)$,
com domínio no espaço quociente, esteja bem definida e seja linear.
Uma vez obtida esta relação, pode-se tomar $X\otimes Y$ como sendo
o próprio espaço quociente.
\end{example}
\end{example}
Uma observação importante é a seguinte: se introduzimos a estrutura
algébrica natural no conjuntos dos homomorfismos, obtém-se bijeções
\[
\mathrm{Bil}_{\mathbf{Alg}}(X\times Y;Z)\simeq\mathrm{Hom}_{\mathbf{Alg}}(X;\mathrm{Hom}_{\mathbf{Alg}}(Y;Z)),
\]
de modo que $\otimes$ é produto tensorial em $\mathbf{Alg}$ quando,
e somente quando, computa-se 
\[
\mathrm{Hom}_{\mathbf{Alg}}(X\otimes Y;Z)\simeq\mathrm{Hom}_{\mathbf{Alg}}(X;\mathrm{Hom}_{\mathbf{Alg}}(Y;Z)).
\]

Assim, em tese, pode-se utilizar a caracterização anterior para \emph{definir
o que vem a ser um produto tensorial numa categoria algébrica $\mathbf{Alg}$},
\emph{mesmo que ela} \emph{não possua produtos}.

\section{Topologia}

$\quad\;\,$A topologia tem como objetos básicos os \emph{espaços
topológicos}. Estes são conjuntos $X$ nos quais cada elemento é dotado
de uma família de \emph{vizinhanças fundamentais}: tratam-se, pois,
de subconjuntos de $X$ contendo o respectivo elemento, sendo escolhidos
de modo a satisfazerem condições de compatibilidade. 

Mais precisamente, um espaço topológico é um conjunto $X$ dotado
de uma aplicação $\tau$, que a cada \emph{ponto} $x\in X$ faz corresponder
uma coleção $\tau(x)$ de subconjuntos de $X$, todos contendo $x$,
os quais cumprem a seguinte condição:
\begin{itemize}
\item se um ponto $z\in X$ está na intersecção $B_{x}\cap B_{y}$ entre
vizinhanças básicas de $x$ e de $y$, então este possui uma vizinhança
$B_{z}\in\tau(z)$ ali inteiramente contida. Isto significa, em particular,
que qualquer ponto $x\in X$ admite vizinhanças \emph{arbitrariamente
pequenas}: para quaisquer $B_{x},B_{x}'\in\tau(x)$ existe uma vizinhança
menor $B_{x}''\in\tau(x)$ contida em $B_{x}\cap B_{x}'$. 
\end{itemize}
$\quad\;\,$A classe mais importante de subconjuntos de um espaço
topológico $(X,\tau)$ são aqueles que podem ser escritos como a reunião
de vizinhanças básicas. Estes são os chamados \emph{abertos} de $X$
segundo $\tau$. Diz-se que a coleção de todos eles determinam uma
\emph{topologia} ou uma \emph{estrutura topológica} em $X$. Como
logo se convence, diferentes aplicações $\tau$ e $\tau'$ (isto é,
diferentes noções de vizinhança básica) podem produzir a mesma topologia.
\emph{Isto reforça o fato de que o conceito fundamental é o de} \emph{aberto}
e \emph{não o de vizinhança básica.}

Intuitivamente, um espaço topológico é um conjunto no qual se sabe
dizer \emph{qualitativamente} se dois pontos estão ou não próximos.
Com efeito, $x$ e $y$ o estarão se existirem vizinhanças $B_{x}$
e $B_{y}$, com $y\in B_{x}$ e $x\in B_{y}$. Observamos, no entanto,
que tal noção de proximidade não é \emph{quantitativa}, não é \emph{mensurável}.
Espaços topológicos nos quais é possível medir a distância entre dois
pontos são chamados de \emph{espaços métricos}. De forma mais precisa,
um espaço métrico é um espaço topológico $X$ cujas vizinhanças básicas
provém de uma função $d:X\times X\rightarrow\mathbb{R}$, chamada
\emph{métrica}, tal que $d(x,y)$ mensura a distância entre $x$ e
$y$. Isto significa que:
\begin{enumerate}
\item $d(x,y)$ nunca é negativo, sendo nulo se, e somente se, $x=y$;
\item a distância entre $x$ e $y$ é igual a distância entre $y$ e $x$;
\item a hipotenusa nunca excede a soma dos catetos: $d(x,y)\leq d(x,z)+d(z,y)$,
seja qual for o $z$.
\end{enumerate}

\subsection*{\uline{Mapeamentos}}

$\quad\;\,$Em topologia, os mapeamentos entre espaços topológicos
$(X,\tau)$ e $(X',\tau')$ são as correspondências $f:X\rightarrow X'$,\emph{
}ditas\emph{ contínuas}, que levam pontos próximos em pontos próximos.
Isto é, tais que, para quaisquer $x,y\in X$, é sempre possível fazer
$f(y)$ arbitrariamente próximo de $f(x)$, bastando tomar $x$ e
$y$ suficientemente próximos. Mais precisamente, para que $f(y)$
esteja numa dada vizinhança $B_{f(x)}$ de $f(x)$, basta que $y$
esteja numa vizinhança $B_{x}\in\tau(x)$.

A identidade de qualquer espaço é sempre uma função contínua. Assim
como vimos ocorrer para a linearidade, continuidade também é preservada
por composições. Por conta destes fatos, diz-se que os espaços topológicos
e as funções contínuas, juntamente com a composição usual de funções,
definem uma categoria $\mathbf{Top}$.

A noção de equivalência empregada na topologia (isto é, válida na
categoria $\mathbf{Top}$) é a de \emph{homeomorfismo}. Compreende-se
que dois espaços $X$ e $X'$ são \emph{homeomorfos} quando existem
funções contínuas $f:X\rightarrow X'$ e $g:X'\rightarrow X$ cujas
composições são identidades de $X$ e $X'$. 

Intuitivamente, dois espaços serão homeomorfos quando um puder ser
continuamente deformado no outro sem que, para isto, se tenha que
``rasgar'', ``furar'' ou ``colapsar dimensões''. Assim, por
exemplo, espera-se que a esfera $\mathbb{S}^{2}$ seja homeomorfa
ao cubo (basta amassar os cantos), mas não o seja ao toro $\mathbb{S}^{1}\times\mathbb{S}^{1}$,
pois este possui um furo no meio. Da mesma forma, o cilindro infinito
$\mathbb{S}^{1}\times\mathbb{R}$ há de ser homeomorfo a qualquer
cilindro de altura finita (retirado o seu bordo), por menor que esta
seja. No entanto, espera-se que $\mathbb{S}^{1}\times\mathbb{R}$
e $\mathbb{S}^{1}$ não sejam equivalentes: para chegar no círculo,
precisa-se colapsar uma dimensão do cilindro, algo que não é satisfeito
pelos homeomorfismos. 

Observamos haver uma diferença a topologia a geometria clássica. Nesta
última, para que dois objetos sejam equivalentes, é preciso que eles
tenham os mesmos ângulos, o mesmo ``formato'', a mesma ``curvatura''.
Lá não se deve esperar, por exemplo, que o cubo seja equivalente à
esfera, nem mesmo que um plano seja equivalente a uma versão ``em
formato de bacia''. 

Também há uma diferença importante entre a topologia e a álgebra:
como vimos anteriormente, para que um mapeamento entre estruturas
algébricas seja um isomorfismo, é necessário e suficiente que ele
seja bijetivo. Tal caracterização, no entanto, não é válida para homeomorfismos.

\subsection*{\uline{Operações}}

$\quad\;\,$Existem duas operações essenciais dentro de $\mathbf{Top}$.
A primeira delas é o \emph{produto topológico} entre espaços $X$
e $Y$. Este é consistido simplesmente do produto cartesiano $X\times Y$,
dotado da aplicação $\tau$, tal que $\tau(x,y)=\tau_{X}(x)\times\tau_{Y}(y)$.
Em outras palavras, as vizinhanças básicas de $(x,y)\in X\times Y$
são simplesmente produtos entre vizinhanças básicas de $x$ e de $y$.
As projeções na primeira e na segunda entrada são ambas contínuas
e caracterizadas pela seguinte propriedade universal: dado $X'$ e
dadas funções contínuas $f:X\times Y\rightarrow X$ e $g:X\times Y\rightarrow Y$,
há uma única aplicação $u:X\rightarrow X'$ tal que o primeiro dos
diagramas abaixo fica comutativo:\begin{equation}{
\xymatrix{& & & X & & & & X \ar@/_/[llld]_f \ar[dl]_{\imath _1}  \\
X' \ar@{-->}[rr]^{u} \ar@/^/[rrru]^f \ar@/_/[rrrd]_g & & X\times Y \ar[ur]^{\pi _1} \ar[dr]_{\pi _2} && X'   & & \ar@{-->}[ll]_{u} X\oplus Y   &\\
& & & Y & & & & Y \ar@/^/[lllu]^g \ar[ul]^{\imath _2}  }}
\end{equation}

A outra operação em $\mathbf{Top}$ é a \emph{soma topológica}: dados
espaços $X$ e $Y$, a \emph{soma} entre eles é a reunião disjunta
$X\sqcup Y$, dotada da topologia com maior número de abertos em que
as inclusões de $X$ e de $Y$ são ambas contínuas. Assim, a soma
entre $X$ e $Y$ é precisamente o espaço $X\oplus Y$ tal que, para
qualquer espaço $X'$ e quaisquer funções contínuas $f:X\rightarrow X'$
e $g:Y\rightarrow X'$, existe uma única correspondência contínua
$u:X\oplus Y\rightarrow X'$ que deixa comutativo o segundo dos diagramas
acima.$\underset{\underset{\;}{\;}}{\;}$

\noindent \textbf{Conclusão.} Por conta da existência de tais operações,
diz-se que, assim como as categorias algébricas, a categoria topológica
$\mathbf{Top}$ possui \emph{produtos e coprodutos}.

\subsection*{\uline{Invariantes}}

$\quad\;\,$Dado um espaço topológico, quer-se identificá-lo módulo
equivalências. Em outras palavras, quer-se saber quais outros espaços
são homeomorfos a ele. Tendo conseguido fazer isto, pode-se substituí-lo
por qualquer objeto em sua classe de homeomorfismo. Por exemplo, um
espaço difícil de visualizar pode, em princípio, ser equivalente a
outro cuja visualização é mais simples. Neste caso, não há motivos
para manter o complicado, deixando o mais simples de lado.

Deve-se dizer que o problema de obtenção de classificações é de extrema
dificuldade. Dentro deste contexto, o conceito de \emph{invariante
topológico,} análogo ao de invariante algébrico, assume valor. Tratam-se,
pois, de propriedades invariantes por homeomorfismos, de modo que,
se dois espaços são equivalentes, então eles devem ter todos os invariantes
em comum. 

Algumas classes de invariantes podem ser construídos impondo condições
somente sobre a topologia dos espaços em estudo. Por exemplo, pode-se
exigir que quaisquer dois pontos do espaço tenham vizinhanças disjuntas,
ou então que qualquer cobertura desse espaço por abertos possua uma
subcobertura finita. Os entes que cumprem tais condições dizem-se,
respectivamente, \emph{Hausdorff }e \emph{compactos}. Em ambos os
casos, a propriedade envolvida é a quantidade de abertos que o espaço
possui, a qual constitui um invariante topológico. Outro exemplo é
o \emph{número de componentes} \emph{conexas }(isto é, o número de
``pedaços'') que formam o espaço.

Observamos que os invariantes assim produzidos não são muito poderosos.
Com efeito, a partir deles não se consegue decidir nem mesmo se a
esfera é ou não homeomorfa ao toro, algo que já discutimos ser extremamente
intuitivo. Assim, quer-se construir novos invariantes, os quais sejam
suficientemente poderosos a ponto de nos permitirem mostrar que, realmente,
$\mathbb{S}^{2}$ não é homeomorfo a $\mathbb{S}^{1}\times\mathbb{S}^{1}$. 

A estratégia é buscar por regras $F:\mathbf{Top}\rightarrow\mathbf{Alg}$,
em que $\mathbf{Alg}$ é alguma categoria algébrica, que associem
a cada espaço topológico $X$ uma estrutura algébrica $F(X)$, e que
cada função contínua $f:X\rightarrow X'$ façam corresponder um homomorfismo
$F(f)$ entre $F(X)$ e $F(X')$, de tal modo que composições e identidades
são preservadas. Como consequência, homeomorfismos são mapeados em
isomorfismos, mostrando-nos que associado a $F$ tem-se um invariante
topológico. Tais regras são exemplos de \emph{functores}. 

Portanto, para mostrar que dois espaços não são homeomorfos, basta
obter um functor que associa estruturas não isomorfas a eles. Por
exemplo, constrói-se um functor $\pi_{1}$, denominado \emph{grupo
fundamental}, que associa um grupo $\pi_{1}(X)$ a cada espaço $X$,
o qual mensura exatamente o número de buracos (chamado de \emph{gênus})
que $X$ possui. Desta forma, tem-se $\pi_{1}(\mathbb{S}^{2})\neq\pi_{1}(\mathbb{S}^{1}\times\mathbb{S}^{1})$,
garantindo a inexistência de um homeomorfismo entre a esfera e o toro. 

A busca por functores $F:\mathbf{Top}\rightarrow\mathbf{Alg}$ é precisamente
o objetivo de uma disciplina da matemática chamada \emph{Topologia
Algébrica}. 
\begin{example}
Nem todos os functores definidos em $\mathbf{Top}$ fornecem bons
invariantes. Por exemplo, tem-se um functor natural $\mathrm{Aut}:\mathbf{Top}\rightarrow\mathbf{Grp}$,
que a cada espaço topológico faz corresponder o grupo $\mathrm{Homeo}(X)$
formado dos homeomorfismos de $X$, com a operação de composição.
O fato de $\mathrm{Aut}$ ser um functor nos diz que, se $X\simeq Y$,
então $\mathrm{Homeo}(X)\simeq\mathrm{Homeo}(Y)$. Assim, para mostrar
que dois espaços são homeomorfos, bastaria provar que seus grupos
de homeomorfismos não são isomorfos. O problema está ai: dado um espaço
topológico $X$, em geral não se conhece a estrutura de $\mathrm{Homeo}(X)$,
pois este é demasiadamente grande. Portanto, como determinar de dois
grupos são ou não isomorfos se nem mesmo os conhecemos? Isto ressalta
um ponto importante: não basta obter invariantes, tem-se que saber
\emph{calculá-los}. 
\end{example}

\section{Análise}

$\quad\;\,$As \emph{variedades} são os espaços topológicos que podem
ser globalmente complicados, mas que possuem uma estrutura local extremamente
simples. De maneira mais precisa, uma \emph{variedade de dimensão
$n$} é um espaço topológico $X$ tal que cada $x\in X$ admite uma
vizinhança homeomorfa a um aberto do semi-espaço $\mathbb{H}^{n}\subset\mathbb{R}^{n+1}$,
formado de todas as listas de $n+1$ números reais, cuja primeira
entrada nunca é negativa.

Os homeomorfismo definidos nas vizinhanças de $x$ chamam-se \emph{sistemas
de coordenadas }ou \emph{cartas locais} em $x$ e, sempre que não
há risco de confusão, são denotados pela mesma letra ``$x$''. Desta
forma, dizer que $x:U\rightarrow\mathbb{H}^{n}$ é uma carta local,
significa dizer que $U$ é vizinhança de $x\in X$, que $x(U)$ é
aberto e que a regra $x:U\rightarrow x(U)$ é um homeomorfismo.

Numa variedade, os sistemas de coordenadas assumem o mesmo papel que
as bases dos espaços vetoriais possuem no contexto da Álgebra Linear:
para demonstrar um teorema, escolhe-se o sistema de coordenadas que
lhe parece mais conveniente, resolve-se o problema fazendo uso explícito
dele e, ao final, verifica-se que o resultado obtido é independe de
tal escolha. Por sua vez, quando se consegue demonstrar um teorema
(ou mesmo definir um conceito) sem fazer uso de cartas locais, diz-se
que ele é \emph{intrínseco}. 

Toda variedade $X$ é a reunião de dois conjuntos disjuntos: seu \emph{bordo}
e seu \emph{interior}. O primeiro é formado de todos os pontos $x\in X$
para os quais existe uma carta $x:U\rightarrow\mathbb{H}^{n}$ tal
que algum ponto de $x(U)$ tem primeira coordenada nula. Isto é, tal
que $x(U)$ intersecta $\mathbb{R}^{n}\subset\mathbb{R}^{n+1}$. Por
sua vez, o interior de $X$ é formado pelos pontos que não estão no
bordo. Uma variedade que coincide com seu interior é dita \emph{não
possuir bordo}. Evidentemente, o bordo e o interior de qualquer variedade
$n$-dimensional são variedades sem bordo, de respectivas dimensões
$n-1$ e $n$.
\begin{example}
O disco $\mathbb{D}^{n}\subset\mathbb{R}^{n+1}$, formado de todo
$x\in\mathbb{R}^{n+1}$ tal que $\Vert x\Vert\leq1$, é uma variedade
de dimensão $n$. Com efeito, se um ponto $x$ tem norma unitária
(isto é, se está na esfera $\mathbb{S}^{n-1}$), consideramos como
sistema de coordenadas as projeções estereográficas. Por sua vez,
se o ponto tem norma menor que um, tomamos como carta a identidade
de $\mathbb{R}^{n+1}$ restrita a uma vizinhança qualquer de $x$.
\end{example}

\subsection*{\uline{Diferenciabilidade}}

$\quad\;\,$Diz-se que uma variedade $X$ é \emph{diferenciável} quando,
para quaisquer cartas locais $x$ e $x'$ definidas nas vizinhanças
de um mesmo ponto, os homeomorfismos $x'\circ x$ e $x'\circ x$ entre
abertos do $\mathbb{H}^{n}$ são ambos diferenciáveis. Isto significa
que cada um deles é estendível a abertos de $\mathbb{R}^{n+1}$, onde
podem ser linearmente aproximadas. 

Observamos que nem toda variedade pode ser dotada de uma estrutura
diferenciável. Este fato foi inicialmente demonstrado por M. Kervaire
em \cite{KERVAIRE}. Outros contraexemplos, referentes às esferas
exóticas (espaços continuamente equivalentes à esfera, mas não diferenciavelmente
equivalentes a elas) foram posteriormente obtidos por J. Milnor em
\cite{MILNOR1}. Em dimensões superiores, tem-se o trabalho \cite{TAMURA},
de I. Tamura.

Em contrapartida, não se deve preocupar com a \emph{classe de diferenciabilidade}
de uma dada variedade. Com efeito, por um resultado devido à Whitney,
se uma variedade admite uma estrutura de classe $C^{1}$, então também
admite uma estrutura $C^{\infty}$ que, em certo sentido, é equivalente
à inicial (veja, por exemplo, o segundo capítulo de \cite{Hirsch}).
Por conta disso, ao longo do texto trabalhamos sempre com entidades
infinitamente diferenciáveis (também ditas \emph{suaves}).

Os mapeamentos entre variedades diferenciáveis $X$ e $Y$ são as
\emph{funções diferenciáveis. }Estas nada mais são que\emph{ }aplicações
$f:X\rightarrow Y$ tais que para cada $x\in X$ é possível obter
cartas $x$ em $X$ e $y$ em $Y$, respectivamente definidos em vizinhanças
dos pontos $x$ e $f(x)$, tais que $y\circ f\circ x^{-1}$ é estendível
a um aberto de $\mathbb{R}^{n+1}$, onde pode ser linearmente aproximada.

Composição preserva diferenciabilidade e, além disso, para toda variedade
$X$, a correspondente função identidade $id_{X}$ é diferenciável.
Assim, da mesma forma que as entidades algébricas e os espaços topológicos
originam categorias $\mathbf{Alg}$ e $\mathbf{Top}$, diz-se haver
uma categoria $\mathbf{Diff}$, na qual se desenvolvem a análise e
a topologia diferencial, cujos objetos são as variedades e cujos morfismos
são as aplicações diferenciáveis.

As equivalências no contexto das variedades diferenciáveis são os
\emph{difeomorfismos}. Assim, diz-se que duas variedades $X$ e $Y$
são \emph{difeomorfas} quando existem funções diferenciáveis entre
elas, cujas composições produzem identidades. Em suma, $f:X\rightarrow Y$
é difeomorfismo quando, para quaisquer cartas $x$ e $y$, a composição
$y\circ f\circ x^{-1}$ é diferenciável, bijetiva e com inversa também
diferenciável.

Existe uma maneira intrínseca de se fazer corresponder um espaço vetorial
a cada ponto de uma variedade diferenciável. Apresentemo-la, pois.
Iniciamos observando que o conjunto $\mathcal{D}(X)$ de todas as
funções diferenciáveis de $X$ em $\mathbb{R}$ é um anel com respeito
às operações de soma e produto de aplicações com valores reais. Particularmente,
há uma ação natural de $\mathbb{R}$ em tal anel, dada por $(a\cdot f)(x)=af(x)$,
tornando-o uma álgebra real associativa. Tendo isto em mente, definamos:
o \emph{espaço} \emph{tangente a $M$ em $x$} é o espaço vetorial
$TX_{x}$, formado de todas as derivações de $\mathcal{D}(X)$ no
ponto $x$. Isto é, trata-se do conjunto dos funcionais lineares $v:\mathcal{D}(X)\rightarrow\mathbb{R}$
que satisfazem a regra de Leibniz em $x$:
\[
v(f\cdot g)=f(x)\cdot v(f)+v(f)\cdot g(x).
\]

Grosso modo, uma função é diferenciável num ponto quando em suas vizinhanças
admite uma aproximação linear dada pela derivada. No âmbito das variedades,
a noção de diferenciabilidade é definida fazendo uso explícito de
sistemas de coordenadas. Tendo o conceito de espaço tangente em mãos,
pode-se torná-la \emph{intrínseca}. Com efeito, a \emph{derivada}
de $f:X\rightarrow Y$ em um ponto $x$ é definida como sendo a transformação
linear 
\[
Df_{x}:TX_{x}\rightarrow TY_{f(x)},\quad\mbox{tal que}\quad Df_{x}(v)(g)=v(g\circ f).
\]

\begin{example}
A identidade é uma correspondência diferenciável, com $D(id_{X})_{x}=id_{TX_{x}}$
seja qual for o ponto $x\in X$ escolhido. Como logo se convence,
composição preserva diferenciabilidade. Além disso, diretamente da
definição de derivada retira-se a \emph{regra da cadeia}: 
\[
D(g\circ f)_{x}=Dg_{f(x)}\circ Df_{x}.
\]
\end{example}
Uma aplicação diferenciável $f$ cuja derivada é injetiva (resp. sobrejetiva)
em todo $x$ recebe o nome de \emph{imersão }(resp. \emph{submersão}).
Uma imersão que é também um homeomorfismo sobre sua imagem chama-se
\emph{mergulho}. Esta última condição é válida se, e somente se, $f:X\rightarrow Y$
é injetiva e o conjunto $f(X)\subset Y$ está dotado da topologia
induzida de $Y$ (isto é, da topologia com maior número de abertos
que torna a inclusão contínua). 

A vantagem desta exigência é a seguinte: se $f$ é um mergulho, então
podemos ``adaptar'' as cartas de $Y$ de modo a introduzir uma estrutura
diferenciável em $f(X)$. Tendo isto em mente, fica fácil entender
as condições requeridas: para que as cartas locais da variedade $Y$
continuem contínuas ao serem restritas à $f(X)$, algo que se utiliza
no processo de ``adaptação'', basta que a topologia de $f(X)$ seja
àquela induzida de $Y$. Por sua vez, se $f$ não fosse injetiva,
então sua imagem poderia se auto-intersectar em algum ponto, impedindo
que ali se introduza sistemas de coordenadas. Finalmente, a injetividade
da derivada em cada $x$ nos permite identificar o espaço tangente
à $f(X)$ em $y=f(x)$ como um subespaço de $TY_{y}$.

Quando $f$ é mergulho, diz-se que a estrutura diferenciável introduzida
em $f(X)\subset Y$ pelo processo de adaptação faz de tal subespaço
uma \emph{subvariedade regular} de $Y$. Quando $f$ é somente uma
imersão injetiva (isto é, quando se retira a exigência de que seja
um homeomorfismo), mas a topologia fixada em $f(X)$ ainda nos permite
ali introduzir uma estrutura diferenciável, diz-se que tal conjunto
é uma \emph{subvariedade imersa} de $Y$. Enfocamos a diferença entre
tais conceitos: se trocamos a topologia de uma subvariedade imersa
e consideramos àquela provinda de $Y$, então o conjunto $f(X)$ pode
nem mesmo admitir uma estrutura diferenciável.

A tarefa de determinar a existência (ou não) de mergulhos e imersões
de uma variedade é um problema bastante interessante e difícil dentro
da topologia diferencial. Mais uma vez, a ideia é procurar por invariantes
de $X$, os quais forneçam condições necessárias à existência de tais
classes de aplicações. Neste sentido, exemplos interessantes de invariantes
são fornecidos pela teoria das \emph{Classes Características. }Veja,
por exemplo, a parte final de \cite{Husemoller} (sobretudo o capítulo
18), e também o livro \cite{Milnor_classes_caract} (especialmente
o quarto capítulo).
\begin{example}
As superfícies $S\subset\mathbb{R}^{3}$ que são estudadas na geometria
diferencial clássica são exemplos importantes de subvariedades regulares.
Com efeito, a topologia delas é sempre suposta àquela advinda do $\mathbb{R}^{3}$,
ao passo que suas estruturas diferenciáveis são obtidas mediante a
exigência de que a inclusão seja uma imersão injetiva (e, portanto,
um mergulho). Por um resultado devido à Whitney, qualquer variedade
pode ser imersa ou mergulhada num espaço euclidiano de dimensão suficientemente
grande. Em outras palavras, não há prejuízo de generalidade em se
assumir que uma variedade é, na verdade, superfície regular de algum
$\mathbb{R}^{k}$. De fato, quando a dimensão da variedade é $n$,
esta pode ser imersa em $\mathbb{R}^{2n-1}$ e mergulhada em $\mathbb{R}^{2n}$.
Utilizando de classes características, mostra-se que esse é a melhor
cota inferior que pode ser obtida. Isto é, existem variedades $n$-dimensionais
que não podem ser imersas em $\mathbb{R}^{2n-2}$.
\end{example}

\subsection*{\uline{Operações}}

$\quad\;\,$Em categorias algébricas, tem-se o produto cartesiano
e a soma direta. Na categoria dos espaços topológicos, tem-se o produto
e a soma topológica. Por sua vez, em $\mathbf{Diff}$ tem-se o \emph{produto
de variedades} e também a \emph{reunião disjunta} destas. 

Com efeito, dadas variedades $X$ e $Y$, de respectivas dimensões
$n$ e $m$, define-se o \emph{produto} entre elas como sendo o produto
cartesiano $X\times Y$, dotado da topologia produto, e, para cada
par $(x,y)$, dos sistemas de coordenadas $x\times y:U\times U'\rightarrow\mathbb{H}^{n}\times\mathbb{H}^{m}$,
em que $x:U\rightarrow\mathbb{H}^{n}$ é carta nas vizinhanças de
$x$ e $y:U'\rightarrow\mathbb{H}^{m}$ é carta nas vizinhanças de
$y$. Há, no entanto, um problema: em geral, $\mathbb{H}^{n}\times\mathbb{H}^{m}$
\emph{não é um semi-espaço} \emph{de $\mathbb{R}^{n}\times\mathbb{R}^{m}$}.
Em particular, os bordos de $\mathbb{H}^{n}$ e $\mathbb{H}^{m}$
podem se intersectar quando tomado o produto cartesiano $\mathbb{H}^{n}\times\mathbb{H}^{m}$,
formando um ``bico''. Daí, se alguma das cartas $x\times y$ tiver
imagem que contempla tal ponto de intersecção, então a entidade $X\times Y$
também possuirá um ``bico'' e não será diferenciável. Isto acontece,
por exemplo, com o quadrado $I\times I$. Os ``bicos'', neste caso,
são os cantos do quadrado. 

Observamos, em contrapartida, que se nos restringimos às variedades
sem bordo, então as correspondências $x\times y$ assumem valores
em $\mathbb{R}^{n}\times\mathbb{R}^{m}$, não havendo problemas. Assim,
sob esta restrição, o produto $X\times Y$ está bem definido e é uma
variedade, também sem bordo, de dimensão $n+m$. Tal entidade é precisamente
aquela que, em conjunto com as projeções na primeira e na segunda
entrada, satisfaz a propriedade universal ilustrada no primeiro dos
diagramas abaixo, onde cada função é diferenciável:\begin{equation}{
\xymatrix{& & & X & & & & X \ar@/_/[llld]_f \ar[dl]_{\imath _1}  \\
X' \ar@{-->}[rr]^{u} \ar@/^/[rrru]^f \ar@/_/[rrrd]_g & & X\times Y \ar[ur]^{\pi _1} \ar[dr]_{\pi _2} && X'   & & \ar@{-->}[ll]_{u} X\oplus Y   &\\
& & & Y & & & & Y \ar@/^/[lllu]^g \ar[ul]^{\imath _2}  }}
\end{equation}

Ainda com as variedades $X$ e $Y$ em mãos, pode-se considerar a
soma topológica $X\oplus Y$ entre os espaços topológico subjacentes,
o que fornece uma questão natural: admitirá tal soma uma estrutura
diferenciável com respeito à qual as inclusões são diferenciáveis
e satisfazem a propriedade universal ilustrada no segundo dos diagramas
acima? A tentativa mais natural de introduzir esta estrutura é a seguinte:
dado um ponto em $z\in X\oplus Y$, quando $z\in X$ considera-se
os próprios sistemas de coordenadas de $X$, ao passo que, se $z\in Y$,
toma-se as próprias cartas de $Y$. Isto realmente faz de $X\oplus Y$
localmente trivial. Existe, no entanto, um novo problema: as cartas
de $X$ assumem valores em $\mathbb{H}^{n}$, enquanto que as de $Y$
tomam valores em $\mathbb{H}^{m}$. Assim, se for $m\neq n$, então
existirão pontos de $X\oplus Y$ com diferentes dimensões.

Se quisermos considerar esta estrutura diferenciável em $X\oplus Y$,
das duas uma: ou nos retemos às variedades com uma dada dimensão fixa,
ou redefinimos o que entendemos por variedades diferenciáveis, permitindo
que cada carta assuma valores num semi-espaço diferente. Com efeito,
pode-se mostrar que, nesta última situação, a noção de dimensão fica
bem definida em cada componente conexa da variedade (isto é, se dois
pontos estão na mesma componente, então quaisquer cartas definidas
em suas vizinhanças assumem valores no mesmo semi-espaço). O problema
em $X\oplus Y$ está justamente no fato de ser uma espaço desconexo.$\underset{\underset{\;}{\;}}{\;}$

\noindent \textbf{Conclusão.} Como definida inicialmente, a categoria
$\mathbf{Diff}$ é um pouco problemática: diferentemente do que acontece
com as categorias algébricas e com $\mathbf{Top}$, nela não existem
produtos e nem coprodutos. Para obter produtos, precisamos nos restringir
às variedades sem bordo. Por sua vez, se queremos coprodutos, ou nos
retemos às variedades de dimensão fixa ou redefinimos o que entendemos
por variedades. 

\chapter{Linguagem}

$\quad\;\,$Grosso modo, fornecer uma \emph{categoria} é o mesmo que
fornecer uma classe de objetos, entre os quais se tem uma noção de
mapeamento e uma lei de composição. Desta forma, pode-se falar, por
exemplo, da categoria cujos objetos permeiam alguma classe de estruturas
algébricas, cujos mapeamentos são os respectivos homomorfismos e cuja
lei de composição é a usual. Introduzir a concepção formal de categoria
e ilustrá-la por meio de exemplos constitui o principal objetivo da
secção inicial deste capítulo. 

A segunda secção, por sua vez, é voltada à introdução dos conceitos
de \emph{functor} e de \emph{transformação natural}. Os primeiros
nada mais são que mapeamentos entre categorias, os quais preservam
a lei de composição e levam identidades em identidades. Por sua vez,
as transformações naturais constituem os mapeamentos entre functores. 

Estes conceitos inserem-se na topologia sob o seguinte aspecto: há
uma categoria $\mathbf{Top}$, formada por espaços topológicos e tendo
aplicações contínuas como mapeamentos. A cada functor ali definido,
faz-se corresponder um invariante topológico. Uma transformação natural
entre dois de tais functores pode, então, ser pensada como um vínculo
entre os respectivos invariantes que a eles correspondem.

Aquém de sua aplicabilidade imediata em topologia, functores são úteis
na resolução de diversos problemas. Por exemplo, quer-se classificar
os objetos de certa categoria a menos de uma noção de equivalência.
Questiona-se, também, acerca da possibilidade de estender ou levantar
mapeamentos. Em ambos as situações, os functores dão condições necessárias
para que tais problemas admitam solução. Esta íntima relação é assunto
da secção final do capítulo. 

Ressaltamos a influência das clássicas referências \cite{MACLANE_categories,Mitchell_categories-1}.
Um texto mais moderno e que também foi utilizado é \cite{category_theory_TOM}.

\section{Categorias}

$\quad\;\,$Sob um ponto de vista ingênuo (e, portanto, não-axiomático),
uma \emph{categoria }$\mathbf{C}$ é uma entidade consistida de:
\begin{enumerate}
\item uma classe de objetos de mesma natureza;
\item para cada par de objetos $X$ e $Y$, um conjunto $\mathrm{Mor}_{\mathbf{C}}(X;Y)$,
cujos elementos são chamados de \emph{morfismos} de $X$ em $Y$,
sendo usualmente denotados por $f:X\rightarrow Y$;
\item uma lei de composição, que toma morfismos $f:X\rightarrow Y$ e $g:Y\rightarrow Z$,
e devolve um outro morfismo $g\circ f:X\rightarrow Z$, de tal forma
que:

\begin{enumerate}
\item vale associatividade: $(h\circ g)\circ f=h\circ(g\circ f)$;
\item para cada objeto $X$, existe um morfismo $id_{X}:X\rightarrow X$,
chamado \emph{identidade} de $X$, com a seguinte propriedade: para
todo objeto $Y$ e para quaisquer morfismos $f:X\rightarrow Y$ e
$g:Y\rightarrow X$, tem-se $id_{X}\circ g=g$ e $f\circ id_{X}=f$.
\end{enumerate}
\end{enumerate}
$\quad\;\,$Numa categoria $\mathbf{C}$, diz-se que dois objetos
$X$ e $Y$ são \emph{equivalentes} (e escreve-se $X\simeq Y$) quando
existem morfismos $f:X\rightarrow Y$ e $g:Y\rightarrow X$, tais
que $g\circ f=id_{X}$ e $f\circ g=id_{Y}$. Fala-se que $f$ é uma
\emph{equivalência} ou um \emph{isomorfismo} e que $g$ é a sua \emph{inversa}.
A relação $\simeq$ é de equivalência na classe dos objetos de $\mathbf{C}$.
Seu espaço quociente será denotado por $\mathrm{Iso}(\mathbf{C})$.$\underset{\underset{\;}{\;}}{\;}$

\noindent \textbf{Advertência.} A ingenuidade com a qual encararemos
a definição acima é a seguinte: ao longo do texto, \emph{classes}
e \emph{conjuntos} serão tomados sinônimos. Isto, no entanto, não
é bem verdade. Com efeito, numa formulação axiomática da teoria dos
conjuntos (devida a Von Neumann, Gödel e Bernays), o conceito de classe
é introduzido como sendo mais amplo do que o conjunto, no intuito
de contornar paradoxos ``do tipo Russel''. Nela, por exemplo, é
possível falar da classe de todos os conjuntos, ao passo que o conjuntos
de todos os conjuntos não se vê bem definido (veja \cite{Bernays_axiomatic_set}
e também o apêndice de \cite{Kelley_general_topology}). Dito isto,
advertimos: no desenvolvimento que se sucederá, a teoria ingênua de
conjuntos será praticada. Assim os referidos paradoxos ali estarão
presentes.$\underset{\underset{\;}{\;}}{\;}$ 

No restante desta secção, veremos diversos exemplos e construções
de categorias que aparecerão naturalmente ao longo de todo o texto.
\begin{example}
Típicas categorias são aquelas formadas por objetos algébricos de
mesma natureza (grupos, anéis, espaços vetoriais, módulos, etc.),
com morfismos dados pela respectiva noção de homomorfismo inerente
à cada classe de tais objetos. Assim como fizemos no primeiro capítulo,
guardaremos as notações $\mathbf{Grp}$ e $\mathbf{Mod}_{R}$ para
denotar as categorias dos grupos e dos módulos sobre $R$. Quando
$R$ for um corpo $\mathbb{K}$, escreveremos $\mathbf{Vec}_{\mathbb{K}}$.

\begin{example}
Com morfismos iguais às aplicações contínuas, a coleção dos espaços
topológicos define uma categoria $\mathbf{Top}.$ Nela, dois objetos
são equivalentes quando, e só quando, são homeomorfos.

\begin{example}
Todo conjunto $X$, parcialmente ordenado por uma certa relação $\leq$,
define uma categoria $\mathrm{Pos}(X)$: seus objetos são os próprios
elementos de $X$, ao passo que há um (único) morfismo $f:x\rightarrow y$
se, e somente se, $x\leq y$. 

\begin{example}
Uma categoria $\mathbf{D}$ diz-se uma \emph{subcategoria} de $\mathbf{C}$
(e escreve-se $\mathbf{D}\subset\mathbf{C}$) quando todo objeto de
$\mathbf{D}$ é também objeto de $\mathbf{C}$ e, além disso, quando
\[
\mathrm{Mor}_{\mathbf{D}}(X;Y)\subseteq\mathrm{Mor}_{\mathbf{C}}(X;Y)\quad\mbox{para todo}\quad X,Y\quad\mbox{em}\quad\mathbf{D}.
\]
Se vale a igualdade, a subcategoria $\mathbf{D}$ é denominada \emph{cheia
}(tradução para \emph{full})\emph{. }Espaços métricos, com morfismos
dados pelas aplicações contínuas, dão origem a uma subcategoria cheia
$\mathbf{Met}$ de $\mathbf{Top}$. Igualmente, variedades formam
subcategoria cheia $\mathbf{Mfd}$ de $\mathbf{Top}$. Em contrapartida,
variedades diferenciáveis com aplicações diferenciáveis originam uma
subcategoria $\mathbf{Diff}\subset\mathbf{Top}$ que não é cheia.
Afinal, existem funções contínuas que não são diferenciáveis. Ao longo
de todo o texto, $\mathbf{AbGrp}\subset\mathbf{Grp}$ denotará a subcategoria
cheia dos grupos abelianos.

\begin{example}
Todos os exemplos apresentados até agora são subcategorias de $\mathbf{Set}$.
Isto é, são formadas de conjuntos dotados de estruturas e de funções
cumprindo condições adicionais. Aqui apresentamos um exemplo de que
este nem sempre é o caso. Mais precisamente, elencamos uma categoria
cujos morfismos não são funções. Com efeito, para cada inteiro $n$,
tem-se uma categoria $n\mathbf{Cob}$, cujos objetos são as variedades
de dimensão $n-1$, compactas e sem bordo. Os morfismos $\Sigma:X\rightarrow Y$,
chamados de \emph{cobordismos} entre $X$ e $Y$, são as classes de
difeomorfismo de variedades compactas, de dimensão $n$, tendo bordo
igual à reunião disjunta de $X$ com $Y$. Isto é, tais que $\partial\Sigma=X\sqcup Y$.
A composição é obtida por colagem ao longo do bordo em comum. A unidade
de $X$ é a classe de difeomorfismo do produto $X\times I$.

\begin{example}
Associada a toda categoria $\mathbf{C}$ existe uma outra, denotada
por $\mathbf{C}^{op}$ e chamada de \emph{oposta }de $\mathbf{C}$,
a qual se vê caracterizada pelas seguintes propriedades: 
\end{example}
\end{example}
\end{example}
\begin{enumerate}
\item seus objetos são os próprios objetos de $\mathbf{C}$;
\item para quaisquer $X$ e $Y$, tem-se $\mathrm{Mor}_{\mathbf{C}^{op}}(X;Y)=\mathrm{Mor}_{\mathbf{C}}(Y;X)$.
Dado um morfismo $f$ de $\mathbf{C}$, utilizaremos de $f^{op}$
para representar seu correspondente em $\mathbf{C}^{op}$;
\item por definição, $g^{op}\circ f^{op}=(f\circ g)^{op}$.
\end{enumerate}
\begin{example}
Partindo de duas subcategorias $\mathbf{C}$ e $\mathbf{D}$ de $\mathbf{Set}$,
constrói-se uma terceira: o \emph{produto }entre elas, denotado por
$\mathbf{C}\times\mathbf{D}$. Seus objetos são os pares $(X,Y)$,
com $X\in\mathbf{C}$ e $Y\in\mathbf{D}$. Seus morfismos são também
pares $(f,g)$, em que $f$ é morfismo de $\mathbf{C}$ e $g$ é morfismo
de $\mathbf{D}$.

\begin{example}
Uma vez fixado um objeto $A\in\mathbf{C}$, constrói-se uma categoria
$\mathbf{C}\rightarrow A$, cujos objetos são aqueles $X\in\mathbf{C}$
para os quais há ao menos um $f:X\rightarrow A$, e cujos morfismos
são os $h:X\rightarrow Y$ tais que, dado $f$ e existe ao menos um
$g:Y\rightarrow A$ cumprindo $f=h\circ g$. De maneira análoga define-se
$A\rightarrow\mathbf{C}$.

\begin{example}
Seja $\mathbf{C}$ uma categoria e suponha a existência de uma relação
de equivalência $\sim$ em cada conjunto $\mathrm{Mor}(X;Y)$, a qual
é compatível com a composição: se $f\sim g$ e $f'\sim g'$, então
$f\circ f'\sim g\circ g'$. Sob estas condições, tem-se uma categoria
$\mathscr{H}\mathbf{C}$, cujos objetos são os mesmos que os de $\mathbf{C}$,
cujos morfismos são as classes de equivalência $[f]$ de morfismos
de $\mathbf{C}$, e cuja composição é dada pela relação $[f]\circ[g]=[f\circ g]$. 
\end{example}
\end{example}
\end{example}
\end{example}
\end{example}
\end{example}
Fala-se que um objeto $X'$ é \emph{quociente} de $X$ quando existe
um morfismo $\pi:X\rightarrow X'$, denominado \emph{projeção}, segundo
o qual para qualquer que seja $f:X\rightarrow Y$ há um único morfismo
$g:X'\rightarrow Y$ cumprindo $g\circ\pi=f$. Neste caso, como pode
ser conferido no primeiro dos diagramas abaixo, quando $Y'$ também
é objeto quociente (digamos de $Y$), então $f:X\rightarrow Y$ induz
um respectivo $f':X'\rightarrow Y'$, dito ser obtido de $f$ por
meio de \emph{passagem ao quociente.$$
\xymatrix{X \ar[r]^{f} \ar[d]_{\pi} & Y \ar[d]^{\pi} & & X \ar[r]^f & X' \\
X' \ar[r]_{f'} \ar@{-->}[ru]^{g} & Y' && A \ar[u]^{\imath} \ar@{-->}[r]_{g} \ar[ur] & A' \ar[u]_{\imath '}}
$$}
\begin{example}
Toda categoria $\mathbf{C}$ possui uma subcategoria $\mathscr{Q}\mathbf{C}$,
cujos objetos são quocientes de objetos de $\mathbf{C}$ e cujos morfismos
são aqueles obtidos por meio de passagem ao quociente. Quando a categoria
em questão é $\mathbf{Set}$ ou $\mathbf{Top}$, toda regra que associa
uma relação de equivalência $\sim$ a cada objeto $X\in\mathbf{C}$
define uma subcategoria $\mathbf{C}/\sim$ de $\mathscr{Q}\mathbf{C}$.
Seus objetos são, precisamente, os espaços quociente $X/\sim$, ao
passo que o conjunto dos morfismos entre $X/\sim$ e $Y/\sim$ está
em bijeção com o conjunto das $f:X\rightarrow Y$ tais que, se $x\sim x'$,
então $f(x)\sim f(x')$. 
\end{example}
Numa categoria $\mathbf{C}$, diz-se que $X'$ é \emph{subobjeto }de
$X$ quando existe um morfismo $\imath:X'\rightarrow X$, denominado
\emph{inclusão}, o qual satisfaz a seguinte propriedade: se $f,g:Y\rightarrow X'$
são tais que $\imath\circ f=\imath\circ g$, então $f=g$. Neste caso,
escreve-se $X'\subset X$. Dado um morfismo $f:X\rightarrow Y$, a
composição $f\circ\imath$ é denotada por $f\vert_{X'}$ e chamada
de \emph{restrição} de $f$ a $X'$. 
\begin{example}
Um \emph{par} de \textbf{$\mathbf{C}$ }é simplesmente uma dupla $(X,A)$,
em que $X\in\mathbf{C}$ e $A\subset X$ é subobjeto. Um \emph{morfismo}
entre dois pares $(X,A)$ e $(X',A')$ é um morfismo $f:X\rightarrow X'$
para o qual existe $g:A\rightarrow A'$ satisfazendo\footnote{Quando $\mathbf{C}\subset\mathbf{Set}$, esta condição nada mais significa
que $f(A)\subset A'$.} $f\circ\imath=\imath'\circ g$ (veja o segundo dos diagramas acima).
Junto desta noção, a coleção dos pares de \textbf{$\mathbf{C}$ }se
torna uma categoria, denotada por $\mathbf{C}_{2}$. 

\begin{example}
Dada uma categoria $\mathbf{C}$, suponhamos existir $*$ tal que
para cada $X\in\mathbf{C}$ há um único morfismo $X\rightarrow*$
(tais objetos são ditos \emph{terminais}). Neste caso, a respectiva
categoria $*\rightarrow\mathbf{C}$, aqui denotada por $\mathbf{C}_{*}$
e denominada \emph{pontuação} de $\mathbf{C}$, se identifica com
aquela formada por todos os \emph{objetos pontuados} $(X,x_{o})$,
em que o \emph{ponto base} $x_{o}$ é um morfismo $*\rightarrow X$.
Os morfismos entre $(X,x_{o})$ e $(Y,y_{o})$ são precisamente aqueles
$f:X\rightarrow Y$ que preservam o ponto base: isto é, tais que $f\circ x_{o}=y_{o}$.
Quando $\mathbf{C}$ é $\mathbf{Set}$ ou $\mathbf{Top}$, todo conjunto
formado de um só ponto é objeto terminal. Consequentemente, há uma
identificação entre os pares $(X,x_{o})$ e $(X,\{x_{o}\})$, com
$x_{o}\in X$. Por sua vez, os morfismos cumprindo $f\circ x_{o}=y_{o}$
identificam-se com as aplicações tais que $f(x_{o})=y_{o}$. Neste
contexto, considera-se $\mathbf{C}_{*}$ como subcategoria de $\mathbf{C}_{2}$.
O mesmo se aplica quando $\mathbf{C}$ é $\mathbf{Diff}$.
\end{example}
\end{example}

\section{Functores}

$\quad\;\,$Um \emph{functor} entre duas categorias $\mathbf{C}$
e $\mathbf{D}$ consiste-se de um mapeamento $F:\mathbf{C}\rightarrow\mathbf{D}$
que preserva todas as estruturas envolvidas. Isto é, que cada objeto
$X$ de $\mathbf{C}$ faz corresponder um objeto $F(X)$ de $\mathbf{D}$,
e que a cada $f\in\mathrm{Mor}_{\mathbf{C}}(X;Y)$ associa um morfismo
$F(f):F(X)\rightarrow F(Y)$, de tal forma que:
\begin{enumerate}
\item composições são preservadas. Isto é, $F(g\circ f)=F(g)\circ F(f)$;
\item identidades são levadas em identidades: $F(id_{X})=id_{F(X)}$.
\end{enumerate}
$\quad\;\,$Os functores de $\mathbf{C}^{op}$ em $\mathbf{D}$ são
denominados \emph{functores contravariantes} de $\mathbf{C}$ em $\mathbf{D}$
ou mesmo \emph{pré-feixes} de $\mathbf{C}$ com valores em $\mathbf{D}$.
Um functor de duas entradas (também chamado de \emph{bifunctor}) é
simplesmente um functor entre uma categoria produto $\mathbf{C}\times\mathbf{C}'$
e outra categoria $\mathbf{D}$. Indutivamente, define-se o que vem
a ser um functor de $n$ entradas (denominado \emph{$n$-functor}).
Como logo se convence, functores levam equivalências de uma categoria
em equivalências de outra categoria.
\begin{example}
Fixado um objeto $X$ de uma categoria $\mathbf{C}$, tem-se um functor
$h^{X}:\mathbf{C}\rightarrow\mathbf{Set}$, que a cada $Y\in\mathbf{C}$
associa o conjunto dos morfismos de $Y$ em $X$, e que a cada morfismo
$f:Y\rightarrow Z$ devolve a correspondência 
\[
h^{X}(f):\mathrm{Mor}_{\mathbf{C}}(X;Y)\rightarrow\mathrm{Mor}_{\mathbf{C}}(X;Z),\quad\mbox{tal que}\quad h^{X}(f)(g)=f\circ g.
\]
Conta-se também com um functor contravariante $h_{X}:\mathbf{C}\rightarrow\mathbf{Set}$,
caracterizado por 
\[
h_{X}(X)=\mathrm{Mor}_{\mathbf{C}}(Y;X)\quad\mbox{e}\quad\mbox{por}\quad h_{X}(f)(g)=g\circ f.
\]
\end{example}
\begin{rem}
Quando se quer evidenciar a categoria (em particular, a classe morfismos)
com os quais se está trabalhando, costuma-se escrever $\mathrm{Mor}_{\mathbf{C}}(-;X)$
ao invés de $h_{X}$, e $\mathrm{Mor}_{\mathbf{C}}(X;-)$ no lugar
de $h^{X}$. Ao longo do texto, esta prática será adotada.
\end{rem}
\begin{example}
Se $\mathbf{D}$ é subcategoria de $\mathbf{C}$, então existe um
functor de inclusão $\imath:\mathbf{D}\rightarrow\mathbf{C}$, definido
de maneira natural: $\imath(X)=X$ e $\imath(f)=f$ para quaisquer
que sejam o objeto $X$ e o morfismo $f$ de $\mathbf{D}$. Em $\mathbf{C}\times\mathbf{C}'$,
pode-se falar do functor $\mathrm{pr}_{1}$ de \emph{projeção }na
primeira entrada. Ele é definido através de $\mathrm{pr}_{1}(X,Y)=X$
e de $\mathrm{pr}_{1}(f,g)=f$. De forma semelhante, define-se a projeção
noutra entrada.

\begin{example}
A \emph{composição} entre dois functores $F:\mathbf{C}\rightarrow\mathbf{C}'$
e $F':\mathbf{C}'\rightarrow\mathbf{D}$ consiste-se de um novo functor
$F\circ F'$, de $\mathbf{C}$ em $\mathbf{D}$, definido por 
\[
(F\circ F')(X)=F(F'(X))\quad\mbox{e}\quad\mbox{por}\quad(F\circ F')(f)=F(F'(f)).
\]
Quando $\mathbf{C}$ é pequena, tal operação introduz uma estrutura
de monóide no conjunto $\mathrm{Func}(\mathbf{C})$, formado de todos
os functores de $\mathbf{C}$ em si mesma, cujo elemento neutro é
o functor identidade $id_{\mathbf{C}}$, definido de maneira óbvia.

\begin{example}
Há um functor natural $T:\mathbf{Diff}_{*}\rightarrow\mathbf{Vec}_{\mathbb{R}}$,
que a cada par $(X,x)$ associa o espaço tangente a $X$ em $x$,
e que a cada função diferenciável $f:X\rightarrow Y$, com $f(x)=y$,
associa a transformação linear $Df_{x}:TX_{x}\rightarrow TY_{y}$.

\begin{example}
Para qualquer que seja a categoria $\mathbf{C}$, tem-se functores
$\Pi:\mathbf{C}\rightarrow\mathscr{Q}\mathbf{C}$, que a cada objeto
$X$ associam um de seus quocientes, e que passam morfismos $f:X\rightarrow Y$
ao quociente. Particularmente, no caso em que $\mathbf{C}$ é $\mathbf{Set}$
ou $\mathbf{Top}$, toda regra que faz corresponder uma relação de
equivalência a cada objeto define um functor $\pi:\mathbf{C}\rightarrow\mathbf{C}/\sim$.

\begin{example}
Assim como morfismos, functores também podem ser passados ao quociente.
De maneira mais precisa, seja $F:\mathbf{C}\rightarrow\mathbf{D}$
um functor e suponhamos que nos conjuntos $\mathrm{Mor}_{\mathbf{C}}(X;Y)$
e $\mathrm{Mor}_{\mathbf{D}}(X',Y')$ estejam respectivamente definidas
relações de equivalência $\sim$ e $\approx$, cujas classes serão
ambas denotadas por $[f]$. Se $F(f)\approx F(g)$ sempre que $f\sim g$,
então a regra que a cada $X$ associa o próprio $X$, e que a cada
$[f]$ faz corresponder $[F(f)]$, está bem definida e estabelece
um functor de $\mathscr{H}\mathbf{C}$ em $\mathscr{H}\mathbf{D}$.
\end{example}
\end{example}
\end{example}
\end{example}
\end{example}
No que segue, apresentamos alguns exemplos de functores que aparecem
naturalmente no contexto da álgebra. Eles expressam um ``fenômeno''
conhecido como \emph{adjunção}. Com efeito, diz-se que dois functores
$F:\mathbf{C}\rightarrow\mathbf{D}$ e $F':\mathbf{D}\rightarrow\mathbf{C}$
são \emph{adjuntos} quando suas imagens produzem os mesmos morfismos.
Mais precisamente, quando existem bijeções 
\[
\mathrm{Mor}_{\mathbf{D}}(F(X);Y)\simeq\mathrm{Mor}_{\mathbf{C}}(X;F'(Y))\quad\mbox{para todos}\quad X\in\mathbf{C}\quad\mbox{e}\quad Y\in\mathbf{D}.
\]

\begin{example}
Todo morfismo $\imath:S\rightarrow R$ entre anéis induz um functor
$R_{S}^{\imath}:\mathbf{Mod}_{R}\rightarrow\mathbf{Mod}_{S}$, denominado
\emph{restrição por escalar,} e definido como segue: dado um $R$-módulo
$X$, olha-se para ele enquanto grupo abeliano e ali se introduz a
ação $S\times X\rightarrow X$, tal que $s\cdot x=\imath(s)x$, resultando
em $R_{S}^{\imath}(X)$. Por sua vez, $\imath$ induz functor $E_{R}^{\imath}:\mathbf{Mod}_{S}\rightarrow\mathbf{Mod}_{R}$,
usualmente chamado de \emph{extensão por escalar}. De fato, fornecido
$S$-módulo $Y$, toma-o enquanto grupo abeliano e considera-se
\[
R\times(R\otimes_{S}Y)\rightarrow R\otimes_{S}Y,\quad\mbox{tal que}\quad a\cdot(b\otimes_{R}x)=(a\cdot b)\otimes_{R}x,
\]
onde, na primeira igualdade, estamos vendo $R$ como $S$-módulo.
O resultado é $E_{R}^{\imath}(Y)$. A ação de $E_{R}^{\imath}$ em
morfismos é $E_{R}(f)(a\otimes_{L}x)=a\otimes_{L}f(x)$. Os functores
$R_{S}^{\imath}$ e $E_{R}^{\imath}$ são adjuntos. Observamos que,
para qualquer anel $R$ há um único morfismo $char:\mathbb{Z}\rightarrow R$,
de tal modo que é sempre possível restringir ou estender escalar a
partir dos inteiros. O núcleo de $char$ é da forma $n\mathbb{Z}$
para algum natural $n$. Este marca o número máximo de vezes que se
pode somar $1\in R$ sem resultar em $0\in R$, ao qual se dá o nome
de \emph{característica} de $R$. Veja \cite{MACLANE_algebra,LANG_algebra}.

\begin{example}
Diz-se que uma categoria $\mathbf{C}$ gera \emph{objetos} \emph{livres}
de uma subcategoria $\mathbf{D}\subset\mathbf{C}$ quando existe um
functor $\jmath:\mathbf{C}\rightarrow\mathbf{D}$ que é adjunto à
inclusão $\imath:\mathbf{D}\rightarrow\mathbf{C}$. Por exemplo, a
categoria $\mathbf{Mod}_{R}$ possui objetos livres e gerados por
conjuntos. Afinal, há $\jmath:\mathbf{Set}\rightarrow\mathbf{Mod}_{R}$
tal que, para cada conjunto $S$ e cada grupo $G$, existem bijeções
\[
\mathrm{Mor}_{\mathbf{Set}}(S;\imath(G))\simeq\mathrm{Mor}_{\mathbf{Mod}_{R}}(\jmath(S);G).
\]
Costuma-se dizer que $S$ é uma \emph{base} para $\jmath(S)$. A motivação
é evidente: as bijeções anteriores nos dizem que toda função $f:S\rightarrow G$
se estende, de maneira única, a um homomorfismo $\overline{f}:\jmath(S)\rightarrow G$.
Esta é precisamente a condição satisfeita pela base de um espaço vetorial.
Por conta disso, costuma-se identificar $\jmath(S)$ com o conjunto
das combinações formais $\sum a_{i}\cdot s_{i}$, com $a_{i}\in R$
e $s_{i}\in S$. Assim, a extensão de $f:S\rightarrow G$ torna-se
``linear''. Isto é, fica definida por 
\[
\overline{f}(\sum a_{i}\cdot s_{i})=\sum a_{i}\cdot f(s_{i}).
\]
Observamos que inclusões entre categorias mais complexas, como é o
caso de $\imath:\mathbf{Alg}_{R}\rightarrow\mathbf{Grp}$, em que
$\mathbf{Alg}_{R}$ é a categoria das álgebras sobre $R$, também
possuem adjuntos. Nesta particular situação, cai-se no estudo dos
\emph{anéis de grupos}. Veja \cite{Polcino}. Outro exemplo de inclusão
que possui adjunto à esquerda é $\imath:\mathbf{Top}\rightarrow\mathbf{Set}$:
seu adjunto nada mais é que o functor que introduz a topologia discreta
em cada conjunto.
\end{example}
\end{example}

\subsection*{\uline{Transformações Naturais}}

$\quad\;\,$Uma \emph{transformação natural} entre dois functores
$F,F':\mathbf{C}\rightarrow\mathbf{D}$ é uma aplicação $\xi$, que
a cada objeto $X\in\mathbf{C}$ associa um morfismo $\xi(X)$ de $\mathbf{D}$,
de tal forma que os diagramas abaixo são sempre comutativos (se cada
$\xi(X)$ é isomorfismo, fala-se que $\xi$ é \emph{isomorfismo natural}):$$
\xymatrix{
F(X) \ar[d]_{\xi (X)} \ar[r]^{F(f)} & F(Y) \ar[d]^{\xi (Y)} & &  \\
F'(X) \ar[r]_{F'(f)}  & F'(Y) && \mathbf{C} \ar@/_{0.5cm}/[rr]_{F'} \ar@/^{0.5cm}/[rr]^{F} & \;\; \Big{\Downarrow} \text{\footnotesize{$\xi $}} & \mathbf{D}  }
$$

Assim, de um ponto de vista intuitivo, uma transformação natural é
uma correspondência entre dois functores, responsável por conectar
morfismos entre objetos distintos. Por este motivo, às vezes se escreve
$\xi:F\rightarrow F'$, representando-a diagramaticamente de maneira
mais simples, assim como exposta no segundo dos diagramas acima.

Sabe-se compor transformações naturais de duas maneiras distintas.
A primeira delas é válida quando fixamos as categorias que servem
de domínio e contradomínio aos functores entre os quais as transformações
atuam. Mais precisamente, dadas transformações $\xi:F\rightarrow F'$
e $\xi':F'\rightarrow F''$, em que todos os functores saem de uma
mesma categoria $\mathbf{C}$ e chegam numa mesma categoria $\mathbf{D}$,
define-se a \emph{composição vertical} entre elas como sendo a transformação
$\xi'\bullet\xi$, de $F$ em $F''$, tal que $\xi'\bullet\xi(X)=\xi'(X)\circ\xi(X)$.
Abaixo a representamos diagramaticamente:$$
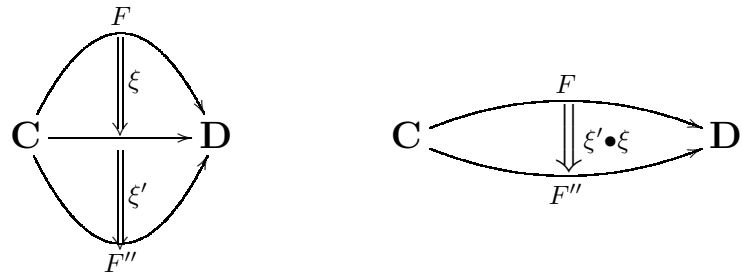
\xymatrix{& \ar@{=>}[d]^{\xi} \\
\mathbf{C} \ar@/^{1.4cm}/[rr]^{F} \ar@/_{1.4cm}/[rr]_{F''}  \ar[rr] & \ar@{=>}[d]^{\xi'} & \mathbf{D} && \mathbf{C} \ar@/_{0.5cm}/[rr]_{F''} \ar@/^{0.5cm}/[rr]^{F} & \;\;\;\;\; \Big{\Downarrow} \text{\footnotesize{$\xi'$$\bullet$$\xi$}} & \mathbf{D} \\
& &} 
$$

A segunda forma de compor transformações naturais é obtida sobre uma
nova exigência: ao invés de se fixar as categorias nos quais os functores
estão definidos, fixa-se os próprios functores. Com efeito, dadas
transformações $\xi:F\rightarrow G$ e $\xi':F'\rightarrow G'$, em
que agora $F'$ e $G'$ saem das categorias em que $F$ e $G$ chegam,
a \emph{composição horizontal} entre elas é a regra $\xi'\circ\xi:F'\circ F\rightarrow G'\circ G$,
definida por $\xi'\circ\xi(X)=\xi'(G(X))\circ F'(\xi(X))$. Sua representação
diagramática é a seguinte:$$
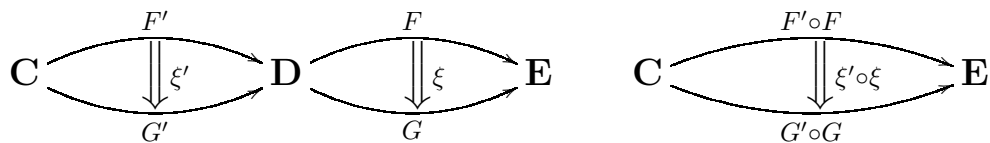
\xymatrix{\mathbf{C} \ar@/_{0.5cm}/[rr]_{G'} \ar@/^{0.5cm}/[rr]^{F'} & \;\; \Big{\Downarrow} \text{\footnotesize{$\xi '$}} & \mathbf{D} \ar@/_{0.5cm}/[rr]_{G} \ar@/^{0.5cm}/[rr]^{F} & \;\; \Big\Downarrow \text{\footnotesize{$\xi $}}  & \mathbf{E} & \mathbf{C} \ar@/_{0.5cm}/[rr]_{G' \circ G} \ar@/^{0.5cm}/[rr]^{F' \circ F} & \;\;\;\;\;\; \Big{\Downarrow} \text{\footnotesize{$\xi'$$\circ$$\xi $}} & \mathbf{E}   }
$$

Tais composições são compatíveis, no sentido de que elas se distribuem
uma com relação à outra. Isto pode ser descrito nos seguintes termos:
fixadas categorias $\mathbf{C}$ e $\mathbf{D}$ tem-se uma nova categoria
$\mathrm{Func}(\mathbf{C};\mathbf{D})$, cujos objetos são os functores
que saem de $\mathbf{C}$ e chegam em $\mathbf{D}$, cujos morfismos
são transformações naturais e cuja lei de composição é a \emph{composição
vertical.} Por sua vez, a \emph{composição horizontal} define bifunctores
\[
\circ:\mathrm{Func}(\mathbf{C};\mathbf{C}')\times\mathrm{Func}(\mathbf{C}';\mathbf{C}'')\rightarrow\mathrm{Func}(\mathbf{C};\mathbf{C}''),
\]
os quais tomam pares $(F,F')$ de functores e $(\xi,\xi')$ de transformações,
e devolvem as correspondentes $F'\circ F$ e $\xi'\circ\xi$. 
\begin{example}
Qualquer morfismo $f:X\rightarrow Y$ induz transformações naturais
$\xi^{f}:h^{Y}\rightarrow h^{X}$ e $\xi_{f}:h_{X}\rightarrow h_{Y}$,
respectivamente definidas por $\xi^{f}(Z)(g)=g\circ f$ e por $\xi_{f}(Z)(g)=f\circ g$.
Particularmente, tem-se functores 
\[
h^{-}:\mathbf{C}^{op}\rightarrow\mathrm{Func}(\mathbf{C};\mathbf{Set})\quad\mbox{e}\quad h_{-}:\mathbf{C}\rightarrow\mathrm{Func}(\mathbf{C}^{op};\mathbf{Set}),
\]
tais que $h^{-}(X)=h^{X}$ e $h^{-}(f)=\xi^{f}$, com $h_{-}$ sendo
definido de maneira análoga. Desta forma, se $X\simeq Y$, então os
functores $h_{X}$ e $h_{Y}$, assim como $h^{X}$ e $h^{Y}$, são
naturalmente isomorfos.

\begin{example}
Há uma categoria $\mathbf{Cat}$, cujos objetos são todas as categorias
e cujos morfismos são todos os functores. Para quaisquer objetos $\mathbf{C}$
e $\mathbf{D}$, a respectiva classe de morfismos $\mathrm{Mor}_{\mathbf{Cat}}(\mathbf{C};\mathbf{D})=\mathrm{Func}(\mathbf{C};\mathbf{D})$
não é um mero conjunto, mas também possui uma estrutura de categoria.
Além disso, entre elas existem bifunctores que atuam como uma outra
operação de composição. Desta forma, $\mathbf{Cat}$ é composta não
só de objetos e de morfismos, mas também de ``morfismos entre morfismos''
(papel ocupado pelas transformações naturais), os quais podem ser
compostos de duas maneiras distintas, mas compatíveis entre si. Por
conta disto, $\mathbf{Cat}$ é o que se chama de $2$-\emph{categoria}.
Tais entidades serão discutidas no sexto capítulo. Observamos haverem
subcategorias cheias importantes de $\mathbf{Cat}$. Exemplos são
$\mathbf{Pos}$ e $\mathbf{Gpd}$. A primeira delas tem como objetos
as categorias definidas por ordenamentos parciais. A segunda, por
sua vez, é formada pelos \emph{grupoides}: categoriais cujos morfismos
são todos isomorfismos.

\begin{example}
Se dois functores $F:\mathbf{C}\rightarrow\mathbf{D}$ e $F':\mathbf{D}\rightarrow\mathbf{C}$
são adjuntos, então existem transformações naturais $\xi$, de $F'\circ F$
em $id_{\mathbf{C}}$, e $\eta$, de $id_{\mathbf{D}}$ em $F\circ F'$,
as quais são respectivamente chamadas de \emph{unidade }e \emph{counidade}
da adjunção. Com efeito, da condição
\[
\mathrm{Mor}_{\mathbf{D}}(F(X);Y)\simeq\mathrm{Mor}_{\mathbf{C}}(X;F'(Y)),
\]
tomando $Y=F(X)$ e variando o $X$ obtém-se $\xi$. Por sua vez,
tomando $X=F'(Y)$ e variando $Y$, encontra-se $\eta$. 
\end{example}
\end{example}
\end{example}

\subsection*{\uline{Equivalência}}

$\quad\;\,$Procuramos por uma boa noção de equivalência entre categorias.
A opção imediata seriam os isomorfismos de $\mathbf{Cat}$. Neste
caso, duas categorias seriam equivalentes quando existissem functores
entre elas cujas respectivas composições coincidissem com identidades.
Esta, no entanto, é uma noção demasiadamente rígida para nossos propósitos:
\emph{em geral não se tem a igualdade de functores, mas apenas a existência
de transformações naturais entre eles. }

Por sua vez, observamos que a relação que identifica functores $F,F':\mathbf{C}\rightarrow\mathbf{D}$
entre os quais há uma transformação natural $\xi:F\rightarrow F$
é de equivalência em $\mathrm{Func}(\mathbf{C};\mathbf{D})$ e compatível
com a composição vertical (denota-a escrevendo $F\simeq F'$). Assim,
está definida uma nova categoria $\mathscr{H}\mathbf{Cat}$, cujos
objetos são as próprias categorias e cujos morfismos são as classes
de functores ligados por transformações naturais. Com isto em mente,
a próxima opção seria considerar como equivalências entre categorias
os isomorfismos de $\mathscr{H}\mathbf{Cat}$. Isto é, $\mathbf{C}$
e $\mathbf{D}$ seriam equivalentes quando existissem functores $F:\mathbf{C}\rightarrow\mathbf{D}$
e $G:\mathbf{D}\rightarrow\mathbf{C}$, tais que $G\circ F\simeq id_{\mathbf{C}}$
e $F\circ G\simeq id_{\mathbf{D}}$.

Espera-se que uma equivalência entre categorias mapeie objetos isomorfos
em objetos isomorfos. Observamos que isto pode não ser satisfeito
pelos isomorfismos de $\mathscr{H}\mathbf{Cat}$, de modo que a opção
por eles proporcionada é muito pouco restritiva para o que procuramos.
A ideia é então considerar uma noção intermediária. Com isto em mente,
a noção de equivalência entre categorias que empregaremos ao longo
do texto é a seguinte: fala-se que duas categorias $\mathbf{C}$ e
$\mathbf{D}$ são \emph{fracamente isomorfas} (ou simplesmente que
são \emph{equivalentes}) quando existe um functor $F:\mathbf{C}\rightarrow\mathbf{D}$
admitindo uma inversa \textquotedbl fraca\textquotedbl{} $F':\mathbf{D}\rightarrow\mathbf{C}$.
Isto significa que, ao invés de se exigir igualdades $F'\circ F=id_{\mathbf{C}}$
e $F\circ F'=id_{\mathbf{D}}$, exige-se a existência de \emph{isomorfismos
naturais} entre eles. 

Desta forma, um functor $F:\mathbf{C}\rightarrow\mathbf{D}$ admite
uma inversa fraca se, e somente se, a respectiva aplicação $X\mapsto F(X)$
é fracamente bijetiva (isto é, $F(X)\simeq F(X')$ implica $X\simeq X'$,
ao mesmo tempo que todo objeto $Y\in\mathbf{D}$ é isomorfo a $F(X)$
para algum objeto $X\in\mathbf{C}$) e existem bijeções
\[
\mathrm{Mor}_{\mathbf{C}}(X;Y)\simeq\mathrm{Mor}_{\mathbf{D}}(F(X);F(Y)).
\]

Em termos ainda mais sucintos, $\mathbf{C}$ e $\mathbf{D}$ são equivalentes
se, e só se, existe $F:\mathbf{C}\rightarrow\mathbf{D}$ que é fracamente
bijetivo em objetos e bijetivo em morfismos. Um functor que é fracamente
injetivo em objetos e injetivo em morfismos chama-se \emph{mergulho}.
Assim, um isomorfismo fraco é simplesmente um mergulho \textquotedbl sobrejetivo\textquotedbl .
Se $F:\mathbf{C}\rightarrow\mathbf{D}$ é mergulho não-sobrejetivo,
então a equivalência se dá entre $\mathbf{C}$ e $F(\mathbf{C})\subset\mathbf{D}$.
Particularmente, a subcategoria $F(\mathbf{C})$ é cheia se, e somente
se, o mergulho $F$ é sobrejetivo em morfismos.$\underset{\;}{\;}$
\begin{rem}
Utilizando da definição acima, o exemplo 2.2.11 nos leva a conclusão
de que a existência de functores adjuntos entre duas categorias também
traduz uma noção de equivalência, a qual é bem mais modesta que aquela
que adotaremos. Observamos, por sua vez, que em outros contextos geralmente
não se consegue (e nem mesmo é útil) mostrar que duas categorias são
equivalentes no sentido apresentado. Neles, a existência de ``adjunções
coerentes'' é então empregada como a noção de equivalência padrão.
Este é o caso das \emph{categorias modelo}, estudadas no sexto capítulo,
nas quais o conceito de equivalência como isomorfismos fracos dá lugar
às \emph{adjunções de Quillen}. 
\end{rem}

\subsection*{\uline{Representação}}

$\quad\;\,$Como vimos anteriormente, para todo objeto $X\in\mathbf{C}$
existe um functor $h_{X}:\mathbf{C}\rightarrow\mathbf{Set}$, bem
como um functor contravariante $h^{X}:\mathbf{C}\rightarrow\mathbf{Set}$.
Os functores covariantes (resp. contravariantes) $F:\mathbf{C}\rightarrow\mathbf{Set}$
naturalmente isomorfos a $h_{X}$ (resp. $h^{X}$) chamam-se \emph{representáveis}
por $X$. Um questionamento natural diz respeito à representabilidade
de um dado functor. Neste espírito, de grande valia é o \emph{lema
de Yoneda }(secção 3.2 de \cite{MACLANE_categories}) que estabelece
a existência de um isomorfismo natural entre 
\[
N,E:\mathrm{Func}(\mathbf{C};\mathbf{Set})\times\mathbf{C}\rightarrow\mathbf{Set},
\]
definidos da seguinte maneira: $E$ é o functor de avaliação, que
toma $(F,X)$ e devolve $F(X)$. De outro lado, $N$ é o functor que
a cada par $(F,X)$ associa o respectivo conjunto das transformações
naturais entre $h^{X}$ e $F$. Assim, por exemplo, para que $F$
seja representável por $X$, é necessário que o conjunto $F(X)$ possua
ao menos um elemento.

A demonstração do lema de Yoneda é bem simples. A ideia é a seguinte:
para qualquer que seja a transformação natural $\xi:h^{X}\rightarrow F$,
o diagrama abaixo deve ser comutativo. Particularmente, quando $Y=X$,
deve-se ter 
\[
[F(g)\circ\xi(X)](id_{X})=[\xi(Y')\circ h^{X}(g)](id_{X}).
\]

No entanto, $h^{X}(g)(id_{X})=g$, mostrando-nos que $\xi$ fica inteiramente
determinada por $\xi(X)(id_{X})$. Assim, a transformação $\alpha:N\rightarrow E$,
definida por $\alpha(F,X)(\xi)=\xi(X)(id_{X})$, fornece o isomorfismo
natural procurado.$$
\xymatrix{\mathbf{C} \ar@/_{0.5cm}/[rr]_{F'} \ar@/^{0.5cm}/[rr]^{h^{X}} & \;\; \Big{\Downarrow} \text{\footnotesize{$\xi $}} & \mathbf{Set}   }
$$

Observamos que o lema de Yoneda possui uma versão dual, demonstrada
de maneira totalmente análoga. Com efeito, se o functor $F:\mathbf{C}\rightarrow\mathbf{Set}$
é contravariante, então também existe um isomorfismo natural entre
\[
N,E:\mathrm{Func}(\mathbf{C}^{op};\mathbf{Set})\times\mathbf{C}^{op}\rightarrow\mathbf{Set},
\]
agora definidos por 
\[
E(F,X)=F(X)\quad\mbox{e}\quad N(F,X)=\mathrm{Mor}(h_{X};F).
\]

\begin{example}
Vejamos um uso do resultado anterior. Dados $X,Y\in\mathbf{C}$, pondo
$F=h_{Y}$ no lema de Yoneda, obtém-se bijeções entre $\mathrm{Mor}(h_{X};h_{X})$
e $\mathrm{Mor}(X;Y)$, mostrando-nos que o functor $h_{-}:\mathbf{C}\rightarrow\mathrm{Func}(\mathbf{C}^{op};\mathbf{Set})$
é mergulho de $\mathbf{C}$ numa categoria de functores. Assim, \emph{toda
categoria é equivalente a uma categoria de pré-feixes}. 
\end{example}

\section{Utilidade}

$\quad\;\,$Um dos problemas básicos dentro de uma categoria $\mathbf{C}$
é o de classificar os seus objetos a menos de sua natural noção de
equivalência. Em outras palavras, procura-se determinar a estrutura
da classe $\mathrm{Iso}(\mathbf{C})$. Em geral, isto é feito através
de bijeções $f:\mathrm{Iso}(\mathbf{C})\rightarrow B$, em que $B$
é alguma outra classe (diz-se que $f$ é uma \emph{classificação}
da categoria $\mathbf{C}$ em termos de $B$). 

Vejamos como functores podem ajudar no problema acima descrito: uma
vez que eles levam equivalências em equivalências, para cada $F:\mathbf{C}\rightarrow\mathbf{D}$
a aplicação $\mathrm{Iso}(F):\mathrm{Iso}(\mathbf{C})\rightarrow\mathrm{Iso}(\mathbf{D})$,
tal que $\mathrm{Iso}(F)([X])=[F(X)]$, está bem definida. Desta forma,
se $F(X)\simeq F(Y)$ implica em $X\simeq Y$, então $f$ torna-se
injetiva e, enquanto vista sobre sua imagem, nos dá uma classificação
de $\mathbf{C}$ em termos das classes de equivalência de $\mathbf{D}$.

De maneira mais sucinta, tal estratégia de classificação se baseia
em buscar functores cuja imagem por $\mathrm{Iso}:\mathbf{Cat}\rightarrow\mathbf{Set}$
seja uma função injetiva. 
\begin{example}
Para toda categoria $\mathbf{C}$ há functor $\mathrm{Aut}:\mathbf{C}\rightarrow\mathbf{Grp}$,
que a cada $X$ faz corresponder o grupo $\mathrm{Aut}(X)$ de seus
automorfismos (isto é, dos isomorfismos de $X$ nele mesmo). Não se
espera que tal functor sirva para classificar qualquer categoria.
Afinal, existem categorias arbitrariamente grandes. Um exemplo de
categoria muito grande, a qual não pode ser classificada $\mathrm{Aut}$
é $\mathbf{Top}$. Para ela, tem-se $\mathrm{Aut}(X)=\mathrm{Homeo}(X)$:
o grupo dos homeomorfismos de $X$. Observamos, no entanto, que se
nos restringirmos à subcategoria cheia das variedades topológicas
compactas, então $\mathrm{Aut}$ fornecerá a classificação procurada.
Por sua vez, se nos restringirmos ainda mais e consideramos apenas
as variedades diferenciáveis (caso em que $\mathrm{Aut}(X)$ se torna
o grupos dos difeomorfismos de $X$), tal classificação poderá ser
obtida sem a exigência de compacidade. Veja os trabalhos originais
de Whittaker e Filipkiewicz \cite{whittaker,Filipkiewicz}.

\begin{example}
Um exemplo mais simples: tem-se um functor $\mathrm{dim}:\mathbf{Vec}_{\mathbb{K}}\rightarrow\mathbb{N}_{\infty}$,
onde $\mathbb{N}_{\infty}$ é a categoria dos naturais estendidos
ao infinito, que a cada espaço vetorial associa a sua dimensão. Ele
não serve para classificar a categoria $\mathbf{Vec}_{\mathbb{K}}$,
pois existem espaços de dimensão infinita que não são isomorfos. No
entanto, ele classifica a subcategoria cheia dos espaços de dimensão
finita.

\begin{example}
Se $\mathbf{D}\subset\mathbf{C}$ é subcategoria livremente gerada
por obtidos de $\mathbf{C}$, então a inclusão possui um adjunto à
esquerda $\jmath:\mathbf{C}\rightarrow\mathbf{D}$. No contexto da
teoria dos anéis de grupos, a escolha de um anel $R$ determinada
um adjunto $\jmath_{R}:\mathbf{Grp}\rightarrow\mathbf{Rng}$. Relembramos
que, para cada grupo $G$, o correspondente $\jmath_{R}(G)$ é o anel
das combinações lineares formais de elementos de $G$ com coeficientes
em $R$. Tem-se uma conjectura famosa, denominada \emph{conjectura
do isomorfismos para anéis de grupos}, que consiste, exatamente, na
afirmação de que $\jmath_{R}$ é fracamente injetivo.
\end{example}
\end{example}
\end{example}
Outra maneira de tentar classificar uma categoria é por meio do functor
$[-]:\mathbf{C}\rightarrow\mathbf{Set}$, que faz corresponder a cada
$X\in\mathbf{C}$ a sua classe de isomorfismos. Com efeito, se este
for representável, então existirá um objeto $B\in\mathbf{C}$, denominado
\emph{espaço classificatório} de $\mathbf{C}$, de tal modo que $\mathrm{Mor}(X;B)\simeq[X]$.
Assim, para classificar $\mathbf{C}$, bastará conhecer os morfismos
que chegam em $B$.
\begin{example}
Como discutiremos ao longo do texto, esta segunda estratégia fornece
a classificação dos fibrados principais.
\end{example}
A classificação completa de uma categoria é, em geral, um problema
difícil. O conceito de propriedade \emph{invariante} assume, portanto,
o seu valor. Diz-se que uma propriedade $\mathcal{P}$ é um \emph{invariante}
de uma categoria $\mathbf{C}$ quando sua validade num objeto $X\in\mathbf{C}$
implica na sua validade em todo objeto equivalente a $X$. Assim,
se $\mathcal{P}$ vale em $X$ e encontrarmos um outro objeto $Y$
no qual $\mathcal{P}$ não é válida, então estes não pertencerão a
mesma classe de equivalência, o que nos dará alguma informação sobre
$\mathrm{Iso}(\mathbf{C})$. Observamos que os invariantes de $\mathbf{Top}$
são, precisamente, os invariantes topológicos (por exemplo, número
de componentes conexas e compacidade).

Novamente, functores vêm à nossa ajuda. Afinal, para cada $F:\mathbf{C}\rightarrow\mathbf{D}$
tem-se um invariante de $\mathbf{C}$. Este é consistido da propriedade
\emph{``ter, a menos da noção natural de equivalência, uma específica
imagem por $F$ em $\mathbf{D}$}''. De fato, como functores preservam
equivalências, se $X\in\mathbf{C}$ possui uma determinada imagem
$F(X)=\alpha$, então cada $Y$ equivalente a $X$ terá respectiva
imagem $F(Y)$ equivalente à $\alpha$. 

Com isto em mente, quando um functor $F:\mathbf{C}\rightarrow\mathbf{D}$
classifica $\mathbf{C}$ por meio da primeira das estratégias anteriormente
apresentadas (isto é, quando ele é fracamente injetivo), diz-se que
tal classificação é obtida a\emph{través do invariante definido por
$F$}. 

Os exemplos 2.3.1 e 2.3.2 ressaltam que, em geral, functores não são
fracamente injetivos. No entanto, pode-se procurar por subcategorias
restritas as quais eles o sejam. Na nova linguagem, isto se traduz
no seguinte: \emph{um invariante em geral não classifica a categoria
inteira, mas é possível que classifique alguma subcategoria}. Por
exemplo, como discutimos, a dimensão é um invariante que não classifica
$\mathbf{Vec}_{\mathbb{K}}$, mas classifica a subcategoria $\mathrm{F}\mathbf{Vec}_{\mathbb{K}}$,
formada dos espaços de dimensão finita.$\underset{\underset{\;}{\;}}{\;}$

\noindent \textbf{Conclusão. }Se queremos obter informação classificatória
sobre uma categoria $\mathbf{C}$, um procedimento canônico é buscar
primeiramente por functores nela definidos e então verificar se os
invariantes por eles assignados classificam alguma subcategoria de
$\mathbf{C}$. Assim, por exemplo, se existem functores $F:\mathbf{C}\rightarrow\mathbf{D}$
e $F':\mathbf{C}\rightarrow\mathbf{D'}$ que classificam respectivas
subcategorias cheias $\mathbf{B},\mathbf{B}'\subset\mathbf{C}$, então
tais invariantes classificam, juntos, a subcategoria maior $\mathbf{B}\cup\mathbf{B}'$.$\underset{\underset{\;}{\;}}{\;}$

No que segue, descreveremos uma situação que acreditamos exemplificar
bem a ideia exposta no parágrafo anterior. Em alguns momentos seremos
um tanto quanto imprecisos. Ressaltamos, no entanto, que os conceitos
aqui discutimos serão mais formalmente apresentados em outros momentos
do texto. 

Iniciamos dizendo que na próxima subsecção associaremos a toda categoria
$\mathbf{C}$ uma nova categoria $\mathbf{\mathscr{B}}(\mathbf{C})$,
cujos objetos (aqui chamados de \emph{fibrados}) são os morfismos
de $\mathbf{C}$. Particularmente, veremos que a regra $\mathscr{B}:\mathbf{Cat}\rightarrow\mathbf{Cat}$
é functorial, de modo que a cada functor $F:\mathbf{C}\rightarrow\mathbf{D}$
faz-se corresponder um novo functor $\mathscr{B}F$ entre $\mathbf{\mathscr{B}}(\mathbf{C})$
e $\mathbf{\mathscr{B}}(\mathbf{D})$. Portanto, todo invariante em
$\mathbf{C}$ induz outro em $\mathbf{\mathscr{B}}(\mathbf{C})$. 

Pois bem: gostaríamos de classificar $\mathscr{B}(\mathbf{Diff}_{*})$,
aqui denotada por $\mathscr{B}_{*}$. Como a dimensão é invariante
em $\mathbf{Diff}$, também é invariante da categoria $\mathscr{B}_{*}$.
Isto significa que, para que $f$ e $g$ sejam isomorfos em $\mathscr{B}_{*}$,
situação em que se diz que eles são \emph{conjugados}, é preciso que
seus domínios e que seus contradomínios tenham a mesma dimensão. 

Observamos, por outro lado, que tal invariante classifica $\mathrm{F}\mathbf{Vec}_{\mathbb{R}}$
inteiramente, de modo que produz bons invariantes em sua categoria
de fibrados, que denotaremos por $\mathscr{B}_{\mathbb{R}}$. Na verdade,
tal invariante também classifica $\mathscr{B}_{\mathbb{R}}$: duas
transformações lineares serão isomorfas se, e somente se, possuírem
o mesmo posto (isto é, se, e somente se, possuírem imagem com igual
dimensão\footnote{Veja o exemplo 2.3.6}). Assim, boa informação classificatória
de $\mathscr{B}_{*}$ (no sentido de ser facilmente calculada) é fornecida
por functores de tal categoria em $\mathscr{B}_{\mathbb{R}}$. 

Ora, há um functor natural $T:\mathbf{Diff}_{*}\rightarrow\mathrm{F}\mathbf{Vec}_{\mathbb{R}}$,
que a cada $(X,x)$ associa o espaço $TX_{x}$, e que a cada $f$
faz corresponder a derivada $Df_{x}$. A ideia é, então, buscar por
subcategorias de $\mathscr{B}_{*}$, restritas às quais $\mathscr{B}T$
seja fracamente injetivo. Se existirem, tais subcategorias também
serão classificadas pela dimensão. Pois estas existem: pode-se mostrar
que, se as derivadas $Df_{x}$ e $Dg_{y}$ são conjugadas, então existem
abertos $(U,x)$ e $(V,y)$ restritos aos quais as aplicações $f$
e $g$ também o são (e, em particular, têm posto constante em tais
vizinhanças). Assim, a\emph{ dimensão classifica a subcategoria $\mathrm{Reg}\subset\mathscr{B}_{*}$
formada dos pares $(f,x)$, em que $f$ tem posto constante nas vizinhanças
de $x$}. 

Questionemo-nos: \emph{quão boa é a classificação obtida?} Neste sentido,
cabe dizer que, para qualquer $f$, o conjunto dos pontos $x\in X$
nas vizinhanças dos quais $f$ tem posto constante é aberto e denso
em $X$. Portanto, o resultado obtido serve para classificar \emph{quase
todos }os pares $(f,x)$. Em outras palavras, quase todos objetos
de $\mathrm{Reg}$ estão em $\mathscr{B}_{*}$. Por outro lado, se
$x$ é ponto crítico isolado de $f$ (isto é, se ele o único ponto
num aberto tal que $Df_{x}=0$), então $(f,x)$ não está em $\mathrm{Reg}$
e, portanto, não é classificado pelo invariante ``dimensão''.

Questionemo-nos: \emph{será possível melhorar a classificação obtida?
}Para tanto, precisamos obter novos invariantes que classificam outras
subcategorias de $\mathscr{B}_{*}$. A obtenção de tais invariantes
em subcategorias formadas de pares $(f,x)$, com $x$ sendo ponto
crítico isolado, é objetivo da \emph{Teoria das Singulares} (veja
\cite{guillemin_teoria_singularidades}), a qual faz uso de métodos
de diversos outros campos da matemática como, por exemplo, Topologia
Diferencial e Sistemas Dinâmicos. 

Ilustramos: a toda função real $f:X\rightarrow\mathbb{R}$, de classe
$C^{2}$, e todo ponto crítico $x\in X$ de $f$, faz-se corresponder
um número inteiro. Trata-se, pois, do número de autovalores negativos
da matriz Hessiana de $f$ em $x$, formada pelas derivadas de segunda
ordem de $f$ em algum sistema de coordenadas nas vizinhanças de $x$.
Esta regra é functorial e, portanto, define um functor $\mathrm{Ind}:\mathbf{C}\rightarrow\mathbb{N}$
numa subcategoria $\mathbf{C}\subset\mathscr{B}_{*}$. O correspondente
invariante é chamado de \emph{índice}. Por um resultado conhecido
como \emph{lema de Morse} (veja \cite{morse_homology,Morse_matsumoto,Milnor_morse}),
o functor $\mathrm{Ind}$ se torna fracamente injetivo ao ser restrito
à subcategoria $\mathrm{Morse}\subset\mathbf{C}\subset\mathscr{B}_{*}$,
formada dos pares $(f,x)$ nos quais o ponto crítico $x$ é não-degenerado
(isto é, tal que o determinante da Hessiana de $f$ em $x$ é não-nulo).
Assim, \emph{o invariante} \emph{``índice'' classifica funções reais
nas vizinhanças de pontos críticos não-degenerados e, portanto, em
conjunto com o invariante ``dimensão'', fornece uma melhor informação
classificatória sobre $\mathscr{B}_{*}$}. 

Vejamos que é possível melhorar ainda mais a classificação anterior.
Para tanto, observamos existir uma estrutura diferenciável na reunião
de todos os espaços tangentes de uma variedade $X$. A entidade resultante
é denotada por $TX$. Tem-se uma aplicação natural $\pi:TX\rightarrow X$,
que toma um vetor em $TX_{x}$ e devolve $x$, a qual é diferenciável.
Assim, $\pi\in\mathscr{B}_{*}$. Os elementos $v\in\mathscr{B}_{*}$
tais que $\pi\circ v=id_{X}$ chamam-se \emph{campos de vetores }em
$X$. 

Toda função real $f:X\rightarrow\mathbb{R}$ induz um campo de vetores
$\nabla f$: uma vez escolhido um produto interno $\langle\cdot,\cdot\rangle_{x}$
em cada espaço tangente $TX_{x}$ de $X$, o qual varia diferenciavelmente
com o ponto $x$, define-se $\nabla f$ como sendo aquele que satisfaz
$df_{x}(v)=\langle\nabla f(x),v\rangle_{x}$. Assim, um ponto crítico
de $f$ é precisamente uma \emph{singularidade }de $\nabla f$ (isto
é, um elemento $x\in X$ tal que $\nabla f(x)=0$). Em particular,
a matriz Hessiana de $f$ em $x$ é a matriz Jacobiana de $\nabla f$,
de modo que o índice de $f$ em $x$ é precisamente o número de autovalores
negativos da matriz Jacobiana de $\nabla f$. Isto nos leva a estender
a regra $\mathrm{Ind}$ para a subcategoria $\mathfrak{X}_{*}\subset\mathscr{B}_{*}$,
formada de todo os pares $(v,x)$ em que $v$ é campo de vetores,
com $v(x)=0$: define-se $\mathrm{Ind}(v,x)$ como sendo o número
de autovalores negativos de $Dv_{x}$. Assim, o índice é invariante
numa subcategoria que, em certo sentido, é mais geral que $\mathbf{C}$.

O próximo passo é procurar por uma subcategoria $\mathrm{Hip}\subset\mathfrak{X}_{*}\subset\mathscr{B}_{*}$,
restrita a qual $\mathrm{Ind}$ é fracamente injetivo. Ora, sabemos
que tal functor classifica $\mathrm{Morse}$. Assim, há de se esperar
que, se $f\in\mathrm{Morse}$, então $\nabla f\in\mathrm{Hip}$. Neste
sentido, observemos que um ponto crítico $x$ de $f$ é não-degenerado
se, e só se a matriz Jacobiana de $\nabla f$ ali tem determinante
não-nulo (neste caso, diz-se que $x$ é \emph{singularidade hiperbólica}).
Assim, há de se esperar que a hiperbolicidade de $x$ implique em
$(v,x)\in\mathrm{Hip}$. Em outras palavras, espera-se que hiperbolidade
seja condição necessária para classificação pelo índice. De fato,
como consequência de um resultado conhecido como \emph{Teorema de
Hartman-Grobman }(veja segundo capítulo de \cite{Palis_sistemas_dinamicos}
ou mesmo o quinto capítulo de \cite{Irwin_sistemas dinamicos}) tal
condição não só necessária como também \emph{suficiente}.$\underset{\underset{\;}{\;}}{\;}$

\noindent \textbf{Conclusão.} Temos dois invariantes (a dimensão e
o índice), os quais nos permitem classificar a maior parte de $\mathscr{B}_{*}$.
A dimensão classifica funções diferenciáveis em regiões de posto constante.
Por sua vez, o índice classifica funções reais nas vizinhanças de
pontos críticos não-generados e, mais geralmente, campos de vetores
nas vizinhanças de singularidades hiperbólicas.$\underset{\underset{\;}{\;}}{\;}$

Finalizamos com três observações:
\begin{enumerate}
\item ressaltamos, mais uma vez, que a Topologia Algébrica (assunto dos
próximos capítulos) se ocupa, essencialmente, do estudo de functores
$F$ definidos em $\mathbf{Top}$. Desta forma, lá se obtém uma grande
quantidade de invariantes topológicos. Estes são mais poderosos na
medida em que as estruturas algébricas $F(X)$ por eles assinadas
são mais complexas (e, portanto, mais difíceis de serem calculadas);
\item qual o papel das transformações naturais em toda essa análise? como
functores associam invariantes a objetos, as transformações naturais,
sendo mapeamentos entre functores, possuem o papel de relacionar diferentes
invariantes de um mesmo objeto. Assim, \emph{se existe uma transformação
natural entre dois functores, então os correspondentes invariantes
por eles assignados não são independentes entre si, mas estão vinculados
por naturalidade;}
\item como talvez tenha ficado mais ou menos evidente ao longo da subsecção,
muitos resultados, problemas, conjecturas e questões, todos provenientes
das mais diversas áreas da matemática, se resumem, ao final do dia,
no problema de classificar uma certa categoria. Isto se evidenciará
ainda mais ao longo do texto. 
\end{enumerate}

\subsection*{\uline{Levantamentos e Extensões}}

$\quad\;\,$Além de classificar uma categoria, procura-se determinar
condições necessárias/suficientes para que seus morfismos possam ser
estendidos e/ou levantados. Mais uma vez, no que tange tal problema,
functores se mostram úteis ferramentas. É isto o que veremos nesta
subsecção.

Um \emph{fibrado} numa categoria $\mathbf{C}$ é uma terna $(X,\pi,Y)$,
em que $X,Y\in\mathbf{C}$ e $\pi:X\rightarrow Y$. Diz-se que $Y$
é a \emph{base} do fibrado, ao passo que $X$ é chamado de \emph{espaço
total}. Um \emph{morfismo} entre dois fibrados $(X,\pi,Y)$ e $(X',\pi',Y')$
é um par $(f,g)$, formado de morfismos $f:X\rightarrow X'$ e $g:Y\rightarrow Y'$
que tornam comutativo o diagrama abaixo:$$
\xymatrix{ X \ar[r]^f \ar[d]_{\pi} & X' \ar[d]^{\pi '} \\
Y \ar[r]_{g} & Y'}
$$

Tem-se uma categoria $\mathscr{B}(\mathbf{C})$, cujos objetos são
fibrados e cujos morfismos são aqueles acima definidos\footnote{Na literatura, tal categoria é algumas vezes chamada de \emph{arrow
category of $\mathbf{C}$} e denotada por $\mathrm{Arr}(\mathbf{C})$.}. Nela, a composição entre $(f,g)$ e $(f',g')$ é definida como sendo
o par $(f\circ f',g\circ g')$. Quando existe um isomorfismo entre
dois fibrados, fala-se que eles são \emph{conjugados}.

Diz-se que um fibrado $(X,\pi,Y)$ possui a propriedade de\emph{ levantamento}
\emph{de morfismos }quando, para toda terna $(X',f,Y)$ existe $(X',g,X)$
tal que $f=g\circ\pi$. Dualmente, fala-se que $(A,\imath,X)$ tem
a propriedade de \emph{extensão de morfismos} quando, para qualquer
que seja $(A,f,Y)$, é possível obter $(X,g,Y)$ cumprindo $f=g\circ\imath$.

Para mostrar que $(X,\pi,Y)$ possui a propriedade de levantamento,
é necessário e suficiente exibir $s:Y\rightarrow X$, chamado de \emph{secção}
do fibrado $(X,\pi,Y)$, tal que $\pi\circ s=id_{X}$. A suficiência
é evidente. Para concluir a necessidade, note que uma secção de $(X,\pi,Y)$
é simplesmente um levantamento da identidade de $Y$. 

De maneira semelhante, a terna $(A,\imath,X)$ possui a propriedade
de extensão se, e somente se, existe um morfismo $r:X\rightarrow A$,
chamado de \emph{retração}, satisfazendo $r\circ\imath=id_{A}$. Os
diagramas abaixo ilustram as respectivas situações.$\underset{\;}{\;}$

$\qquad\qquad\qquad\qquad\qquad\qquad$
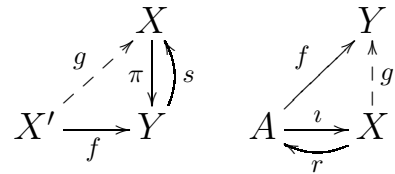
\xymatrix{& X \ar[d]_{\pi} & & Y \\
X' \ar@{-->}[ru]^{g} \ar[r]_{f} & Y \ar@/_{0.3cm}/[u]_s & A \ar[r]^{\imath} \ar[ru]^f & X \ar@/^{0.3cm}/[l]^{r} \ar@{-->}[u]_{g}  }

Todo functor $F:\mathbf{C}\rightarrow\mathbf{D}$ induz naturalmente
um outro functor $\mathscr{B}F:\mathscr{B}(\mathbf{C})\rightarrow\mathscr{B}(\mathbf{D})$,
o qual se vê caracterizado por 
\[
\mathscr{B}F(X,\pi,Y)=(F(X),F(\pi),F(Y))\quad\mbox{e}\quad\mathscr{B}F(f,g)=(F(f),F(g)).
\]

Como consequência, se $(X,\pi,Y)$ possui a propriedade de levantamento
(resp. de extensão), então $\mathscr{B}F(X,\pi,Y)$ também possui.
Em particular, se $s$ é secção e $r$ é retração, então $F(s)$ e
$F(r)$ também o são. Desta forma, \emph{functores definidos em $\mathbf{C}$
fornecem condições necessárias para que os fibrados de $\mathscr{B}(\mathbf{C})$
possuam as propriedades de extensão e levantamento, bem como para
que admitam secções e retrações}.
\begin{example}
No oitavo capítulo, obteremos uma sequência de functores $\pi_{i}:\mathbf{Top}\rightarrow\mathbf{Grp}$,
denominados \emph{grupos de homotopia, }os quais funcionam como invariantes
topológicos. Por exemplo, o grupo $\pi_{n}(\mathbb{S}^{n})$ é isomorfo
ao grupo aditivo dos inteiros, enquanto que $\pi_{i}(\mathbb{D}^{n})$
é trivial seja qual for o $i>0$. Estes fatos nos permitem mostrar
facilmente que não há função contínua $f:\mathbb{S}^{n-1}\rightarrow\mathbb{D}^{n}$
que se estende ao disco $\mathbb{D}^{n}$. Com efeito, se existisse
uma tal função, então haveria uma retração $r:\mathbb{D}^{n}\rightarrow\mathbb{S}^{n-1}$,
que nada mais seria que uma extensão da identidade de $\mathbb{S}^{n-1}$.
Consequentemente, $\pi_{n-1}(r)$ seria retração para $\pi_{n-1}(f)$,
o que não vêm ao caso (veja diagrama abaixo). O resultado assim demonstrado
é conhecido como \emph{teorema da retração de Brouwer}. Trata-se de
um clássico da Topologia Algébrica e tem como consequência imediata
um outro resultado (também devido à Brouwer) segundo o qual todo morfismo
$f:\mathbb{D}^{n}\rightarrow\mathbb{D}^{n}$ em $\mathbf{Top}$ possui
um ponto fixo.$$
\xymatrix{& \mathbb{D}^n && & 0 \\
\mathbb{S}^{n-1} \ar[ru]^f \ar[r]_{id} & \mathbb{S}^{n-1} \ar@{-->}[u]_r \ar@<0.5cm>@{=>}[rr]^-{\pi_{n-1}} &&  \mathbb{Z} \ar[ru] \ar[r]_{id} & \mathbb{Z} \ar@{-->}[u]    }
$$ 
\end{example}
Fixando um objeto $Y\in\mathbf{C}$, obtém-se uma subcategoria $\mathscr{B}_{Y}(\mathbf{C})\subset\mathscr{B}(\mathbf{C})$.
Seus objetos são os fibrados de $\mathbf{C}$ que possuem base em
$Y$. Por sua vez, os\emph{ }morfismo\emph{s} entre $(X,\pi,Y)$ e
$(X',\pi',Y)$ são morfismos $f:X\rightarrow X'$ tais que $\pi'\circ f=\pi$.
Ou seja, $(f,id_{Y})$ é morfismo de fibrados no sentido usual. Dualmente,
fixado $X\in\mathbf{C}$ também se tem uma subcategoria $\mathscr{B}^{X}(\mathbf{C})\subset\mathscr{B}(\mathbf{C})$,
cujos objetos são os fibrados que possuem $X$ como espaço total.
Tais subcategorias são úteis em diversos sentidos. Vejamos alguns:
\begin{enumerate}
\item há um functor específico $\Gamma:\mathscr{B}_{Y}(\mathbf{C})\rightarrow\mathbf{Set}$,
que a cada fibrado com base em $Y$ associa o conjunto $\Gamma(X,\pi,Y)\subset\mathrm{Mor}(Y;X)$
de suas secções, ao mesmo tempo que toma um morfismo $f:X\rightarrow X'$
e faz corresponder a aplicação $\Gamma(f)$, tal que $[\Gamma(f)](s)=f\circ s$.
Semelhantemente, existe um functor $\Omega:\mathscr{B}^{X}(\mathbf{C})\rightarrow\mathbf{Set}$,
responsável por associar a um fibrado com espaço total $X$ o seu
conjunto de retrações, sendo tal que $[\Omega(g)](r)=r\circ g$;$$
\xymatrix{X \ar[rd]^{\pi} \ar[rr]^f && X' \ar[dl]_{\pi '} & &  X \ar[ld]^{\pi} \ar[rd]_{\pi '} \\
&Y \ar@/^{0.3cm}/[lu]^s \ar@/_{0.3cm}/@{-->}[ru] & &  Y \ar@/^{0.3cm}/[ru]^{r} \ar[rr]_{g} && Y' \ar@/_{0.3cm}/@{-->}[lu] }
$$
\item a classificação de $\mathscr{B}_{Y}(\mathbf{C})$ determina a classe
de isomorfismo de $Y$. Com efeito, um morfismo $f:X\rightarrow Y$
é um isomorfismo em $\mathbf{C}$ se, e somente se, é isomorfo em
$\mathscr{B}_{Y}(\mathbf{C})$ à $id_{Y}$. Desta forma, \emph{classificar
$\mathscr{B}(\mathbf{C})$ implica classificar $\mathbf{C}$}. Por
outro lado, para que dois fibrados sejam isomorfos, é preciso que
as bases e os espaços totais também o sejam. Portanto, \emph{invariantes
de $\mathbf{C}$ determinam invariantes de $\mathscr{B}(\mathbf{C})$.}
\end{enumerate}
$\quad\;\,$Abaixo apresentamos um exemplo atípico no qual a classificação
de uma categoria é suficiente para garantir a classificação de seus
fibrados. Ele foi um dos passos utilizados na descrição apresentada
no início da secção, quando estudados aspectos classificatórios de
$\mathscr{B}_{*}$.
\begin{example}
Vimos que espaços vetoriais de dimensão finita são classificados por
sua dimensão. Mostremos, agora, que tal invariante também classifica
$\mathscr{B}_{\mathbb{K}}$. Com efeito, dado uma transformação linear
qualquer $f:X\rightarrow Y$, pelo teorema do núcleo e da imagem,
tem-se $X\simeq\ker(f)\times\mathrm{img}(f)$. Por outro lado, como
a categoria $\mathrm{F}\mathbf{Vec}_{\mathbb{K}}$ é classificada
pela dimensão, segue-se que, se os espaços $X$, $Y$ e $\mathrm{img}(f)$
tem respectivas dimensões $n$,$m$ e $k$, então $f$ é conjugado
ao mapa $g:\mathbb{K}^{n-k}\times\mathbb{K}^{k}\rightarrow\mathbb{K}^{m}$,
definido por 
\[
g(x_{1},...,x_{n-k},y_{1},...,y_{k})=(0,....0,y_{1},...,y_{k}).
\]
Particularmente, na medida em que $f$ é injetivo, sobrejetivo ou
bijetivo, então é conjugado à inclusão, à projeção ou à identidade. 
\end{example}

\chapter{Unificação}

$\quad\;\,$Neste capítulo, damos continuidade ao estudo iniciado
no anterior. Na primeira secção, estudamos o problema da extensão
de morfismos na categoria $\mathbf{Cat}$ e vemos dali surgir o conceito
de \emph{extensões de Kan}.\emph{ }No genuíno espírito da teoria das
categoriais, estas generalizam uma grande quantidade de conceitos
presentes em diversos campos da matemática e, até mesmo, da própria
teoria das categorias.

Na segunda secção, estudamos casos particulares de extensões de Kan,
denominados \emph{limites }e \emph{colimites}, os quais são suficientemente
genéricos a ponto de englobar, por exemplo, as concepções de adjunção
de functores, de produto cartesiano e de soma direta de estruturas
algébricas, de colagem de espaços topológicos, bem como um procedimento
que permite trocar a base de um fibrado sem alterar suas fibras. 

Ali também discutimos como a noção de limite indutivo nos permite
falar de ``conjugações locais'' entre fibrados. Em seguida, estudados
os fibrados que são localmente conjugados a um fibrado ``trivial''.
O espírito, aqui, é semelhante ao empregado no estudo das variedades:
conhece-se previamente a estrutural local (a qual é suposta simples
ou trivial) e tenta-se utilizar de tal conhecimento para obter informações
globais. 

Entre o final da segunda secção e o término do capítulo, mostramos
que todos os limites são determinados por produtos e equalizadores,
ao passo que os limites determinam todas as extensões de Kan. Como
consequência, conclui-se que para falar de muitas coisas dentro da
matemática, precisa-se apenas dos conceitos prévios de produto e equalizador.

Para o estudo e para a escrita, fizemos uso especial das referências
\cite{categorical_homotopy,category_theory_TOM,MACLANE_categories}.

\section{Extensões}

$\quad\;\,$Consideremos o problema de extensão na categoria $\mathbf{Cat}$,
formada por todas as categorias e todos os functores. Isto significa
que, fixado um functor $\imath:\mathbf{A}\rightarrow\mathbf{C}$ e
dado qualquer $F:\mathbf{A}\rightarrow\mathbf{D}$, procuramos por
um $F':\mathbf{C}\rightarrow\mathbf{D}$ tal que $F'\circ\imath=F$.
Para que isto ocorra, é necessário e suficiente que existam retrações. 

Tanto a formulação destes problema quanto a sua solução (em termos
de retrações) são expressas em termos de igualdades entre functores,
algo que já discutimos ser uma exigência muito forte. Em contextos
passados, isto foi resolvido substituindo a exigência de igualdade
de functores pela hipótese de existência de transformações naturais
entre eles satisfazendo certas condições. Em outras palavras, passou-se
de $\mathbf{Cat}$ para $\mathscr{H}\mathbf{Cat}$. No presente contexto,
sigamos a mesma estratégia. 

Assim, dados functores $\imath:\mathbf{A}\rightarrow\mathbf{C}$ e
$F:\mathbf{A}\rightarrow\mathbf{D}$, busquemos por ``extensões fracas''
de $F$ tais que a composição $F'\circ\imath$ está na mesma classe
que $F$. Isto é, tais que existe uma transformação $\xi:F'\circ\imath\rightarrow F$.
Tem-se particular interesse nas extensões fracas universais, as quais
são chamadas de \emph{e}xtensões de Kan\emph{.} De maneira mais precisa,
diz-se que $F$ possui \emph{extensão de Kan à esquerda} \emph{ao
longo de} $\imath$ quando existem um functor $F'$ e uma transformação
$\xi:F'\circ\imath\rightarrow F$ tal que, se $(F'',\zeta)$ é qualquer
outro par satisfazendo esta mesma condição, então há uma única transformação
natural $\varphi:F''\rightarrow F'$ cumprindo $\varphi_{\imath}\circ\xi=\xi'$.
Se este é o caso, escreve-se $\mathscr{L}F$ ao invés de $F'$. 

De maneira análoga define-se o que vem a ser uma extensão de Kan à
direita. Se $F$ as possui, utiliza-se de $\mathscr{R}F$ para denotá-las.
Como facilmente se verifica, por conta da universalidade que estão
sujeitas, tais extensões são sempre únicas a menos um único isomorfismo
natural.

\subsection*{\uline{Abstração}}

$\quad\;\,$No capítulo X de \cite{MACLANE_categories}, assim como
na secção 43 de \cite{model_categories_KAN-1} e no primeiro capítulo
de \cite{categorical_homotopy} encontram-se a expressão ``\emph{all
concepts are Kan extensions}''. No intuito de ilustrá-la, vejamos,
por exemplo, que adjunções podem ser descritas em termos de tais extensões.
Por definição, $G:\mathbf{D}\rightarrow\mathbf{C}$ possui um adjunto
$F:\mathbf{C}\rightarrow\mathbf{D}$ à esquerda se existem bijeções
\[
\mathrm{Mor}_{\mathbf{D}}(F(X);Y)\simeq\mathrm{Mor}_{\mathbf{C}}(X;G(Y))\quad\mbox{para todos}\quad X,Y.
\]

Particularmente, pondo $Y=F(X)$ e variando o $X$, obtém-se uma transformação
natural entre $F\circ G$ e a identidade de $\mathbf{C}$. Por sua
vez, colocando $X=G(Y)$ e variando o $Y$, encontra-se $F\circ G\simeq id_{\mathbf{D}}$
(relembramos se tratarem da unidade e da counidade da adjunção). Isto
nos leva a crer que, se $G$ possui um adjunto $F$ à esquerda, então
este é a extensão de Kan à direita de $id_{\mathbf{C}}$ relativamente
a $G$, ao passo que $G$ é extensão de Kan à esquerda de $id_{\mathbf{C}}$
com respeito a $F$. Tais fatos procedem, como facilmente se verifica. 

A recíproca se torna verdadeira sob a exigência de \emph{absolutez}.
Sejamos mais precisos: diz-se que uma extensão de Kan $\mathscr{R}F$
é \emph{absoluta} quando ela é \emph{preservada} por qualquer $G$.
Isto é, quando existem isomorfismos naturais entre $G\circ\mathrm{\mathscr{R}}F$
e $\mathscr{R}(G\circ F)$. Tem-se, então, o seguinte enunciado, o
qual dá a caracterização que havíamos prometido: \emph{para que $G:\mathbf{D}\rightarrow\mathbf{C}$
possua adjunto $F:\mathbf{C}\rightarrow\mathbf{D}$ à esquerda, é
necessário e suficiente que $id_{\mathbf{D}}$ admita extensão de
Kan absoluta à direita.} \emph{Neste caso, tem-se $F\simeq\mathscr{R}id_{\mathbf{D}}$
}(teorema 2 na página 248 de \cite{MACLANE_categories})\emph{. }Definições
e enunciados análogos valem no caso de adjunções à direita.

Diversos outros conceitos podem ser descritos através de extensões
de Kan: produtos, equalizadores, \emph{pullbacks}, objetos terminais,
assim como suas versões duais, estão ai inclusos. Todos eles são exemplos
de \emph{limites}, os quais discutimos na próxima secção.

\section{Limites}

$\quad\;\,$Seja $\mathbf{1}$ a categoria que possui um só objeto
e cujo único morfismo é a identidade. Evidentemente, para toda categoria
$\mathbf{C}$ há somente um functor $\imath:\mathbf{C}\rightarrow\mathbf{1}$.
A extensão de Kan à esquerda de $F:\mathbf{C}\rightarrow\mathbf{D}$
relativamente a $\imath$ chama-se \emph{limite}. Como há um único
objeto em $1\in\mathbf{1}$, quando existe, $\mathscr{L}F$ é constante
e se identifica com sua imagem por $1$. Daí, $F:\mathbf{C}\rightarrow\mathbf{D}$
possui limite se, e somente se, há um \emph{cone} $(X,\varphi)$,
onde $X\in\mathbf{D}$ e $\varphi$ é transformação natural entre
o functor constante em $X$ e $F$, no qual qualquer outro cone é
fatorado. Isto significa que, para qualquer $(X',\varphi')$ é possível
obter um único mapa $\mu:X'\rightarrow X$, chamado de \emph{fatoração},
tal que os primeiros diagramas abaixo são sempre comutativos:

\begin{equation}{
\xymatrix{& & & F(Y') & & & & F(Y') \ar@/_/[llld]_{\varphi '} \ar[ld]_{\varphi} \\ 
X' \ar@{-->}[rr]^{u} \ar@/^/[rrru]^{\varphi '} \ar@/_/[rrrd]_{\varphi '} & & X \ar[ur]^-{\varphi} \ar[dr]_-{\varphi} & & X' & & X \ar@{-->}[ll]_{u} &\\
& & & F(Y) \ar[uu]_{F(f)} & & & & F(Y) \ar[uu]_{F(f)}  \ar@/^/[lllu]^{\varphi '} \ar[ul]^{\varphi} } }
\end{equation}

Dualmente, a extensão de Kan à direita de $F:\mathbf{C}\rightarrow\mathbf{D}$
com respeito a $\imath$ é denominada \emph{colimite}. Esta se identifica
com um \emph{cocone} universal. Isto é, com um par $(X,\varphi)$,
em que $X\in\mathbf{D}$ e $\varphi$ é transformação natural entre
$F$ e o functor constante em $\mathscr{R}F(1)$, o qual fatora em
todos os outros $(X',\varphi')$, como representado no segundo dos
diagramas acima.

Toda transformação $\xi$ entre $F$ e $F'$ induz morfismos $\lim\xi$
e $\mathrm{colim}\xi$ entre os correspondentes limites colimites,
os quais são isomorfismos se, e só se, $\xi$ é um isomorfismo natural.

As distintas classes de limites de functores $F$ ficam caracterizadas
pela categoria na qual $F$ está definido. Por exemplo, como detalharemos
em seguida, produtos nada mais são que limites de functores com domínio
em categorias discretas. De outro lado, se $F$ está definido numa
categoria dirigida, então seus colimites são, precisamente, limites
indutivos.
\begin{example}
Como vimos, um functor $\imath:\mathbf{D}\rightarrow\mathbf{C}$ possui
adjunto à esquerda $\jmath:\mathbf{C}\rightarrow\mathbf{D}$ se, e
somente se, a identidade $id_{\mathbf{D}}$ admite extensão de Kan
absoluta à direita, a qual coincide com o próprio $\jmath$. Assim,
em particular, podemos usar de $\jmath$ para calcular os colimites
de $\mathbf{D}$ a partir dos colimites de $\mathbf{C}$. Por outro
lado, se $\imath$ admite $\jmath$ como adjunto à esquerda, então
$\jmath$ possui $\imath$ como adjunto à direta e, portanto, podemos
usar de $\imath$ para expressar limites de $\mathbf{D}$ em termos
de limites de $\mathbf{C}$. Relembramos que um caso particular onde
tais adjunções existem é quando $\mathbf{D}$ é uma subcategoria livremente
gerada por objetos de $\mathbf{D}$. Por exemplo, no contexto dos
módulos, a existência de uma adjunção $\jmath:\mathbf{Set}\rightarrow\mathbf{Mod}_{R}$
para $\imath:\mathbf{Mod}_{R}\rightarrow\mathbf{Set}$ nos permite
concluir que $\imath$ deve levar limites de $\mathbf{Mod}_{R}$ em
limites de $\mathbf{Set}$, ao passo que $\jmath$ deve levar colimites
de $\mathbf{Set}$ em colimites de $\mathbf{Mod}_{R}$. Isto ficará
mais claro ao longo da secção.
\end{example}

\subsection*{\uline{Produtos }}

$\quad\;\,$Uma categoria $\mathbf{J}$ é dita \emph{discreta} quando
seus únicos morfismos são as identidades. Os objetos de $\mathbf{J}$
chamam-se \emph{índices}. Os functores $F:\mathbf{J}\rightarrow\mathbf{C}$
se identificam com regras que a cada índice $j\in\mathbf{J}$ associam
um objeto $X_{j}\in\mathbf{C}$. Fala-se que a categoria $\mathbf{C}$
possui \emph{produtos} quando, independente de qual seja a categoria
discreta $\mathbf{J}$, todo $F:\mathbf{J}\rightarrow\mathbf{C}$
possui limite. 

Desta forma, $\mathbf{C}$ possuirá produtos quando, dados quaisquer
objetos $X_{j}$ (indexados da maneira como se queira), existam $X\in\mathbf{C}$
e morfismos $\pi_{j}:X\rightarrow X_{j}$, de tal forma que, se $\pi_{j}':X'\rightarrow X_{j}$
é outra família de morfismos, então $\pi_{j}'=\pi_{j}\circ f$ para
um único $f:X'\rightarrow X$. Diz-se que $X$ é o \emph{produto}
dos $X_{j}$. Os $\pi_{j}$ são chamados de \emph{projeções}. Assim,
em outras palavras, uma categoria possui produtos quando admite uma
maneira universal de projetar. 

Dualmente, fala-se que uma categoria $\mathbf{C}$ possui \emph{coprodutos
}quando cada functor $F:\mathbf{J}\rightarrow\mathbf{C}$ possui colimites.
Isto significa que, dados $X_{j}$ em $\mathbf{C}$, existem $X$
e $\imath_{j}:X_{j}\rightarrow X$ tais que, para quaisquer $\imath_{j}':X_{j}\rightarrow X'$
há um único morfismo $f:X\rightarrow X'$ através do qual $\imath_{j}=\imath_{j}'\circ f$.
O objeto $X$ é chamado de \emph{coproduto} dos $X_{j}$, ao passo
que os $\imath_{j}$ denominam-se \emph{inclusões}. Desta forma, uma
categoria haverá de ter coprodutos quando possuir uma maneira universal
de incluir.
\begin{example}
Em concordância com o que vimos no primeiro capítulo, tanto $\mathbf{Set}$
quanto cada categoria algébrica (grupos, anéis, módulos e espaços
vetoriais), possuem produtos e coprodutos. Particularmente, quando
$\mathbf{Alg}\subset\mathbf{Set}$ é categoria livremente gerada por
conjuntos (isto é, quando a inclusão $\imath:\mathbf{Alg}\rightarrow\mathbf{Set}$
possui adjunto $\jmath$ à esquerda), então o último exemplo da subsecção
anterior nos diz que produtos em\textbf{ $\mathbf{Alg}$} são levados
por $\imath$ em produtos de $\mathbf{Set}$, ao passo que coprodutos
de $\mathbf{Set}$ são levados nos respectivos coprodutos de $\mathbf{Alg}$
por $\jmath$. Assim, em tal caso, dada uma família $X_{j}$ de objetos
em $\mathbf{Alg}$, o produto entre eles é simplesmente o produto
cartesiano usual dotado (se for o caso) de operações definidas componente
a componente. Por sua vez, $\pi_{j}:X\rightarrow X_{j}$ nada mais
é que a projeção na $j$-ésima entrada. Por sua vez, o coproduto em
$\mathbf{Alg}$ deve ser tal que, quando mandado por $\jmath$ em
$\mathbf{Set}$, caia na reunião disjunta.

\begin{example}
O mesmo argumento nos permite inferir produtos e coprodutos em categorias
livremente geradas por grupos (ao invés de conjuntos, como discutido
no último exemplo). Isto se aplica, por exemplo, à teoria dos anéis
de grupos.

\begin{example}
Se $\mathbf{C}$ tem produtos, então o mesmo ocorre com $\mathscr{B}(\mathbf{C})$.
Realmente, o produto de uma família $(X_{j},f_{j},X_{j}')$ de fibrados
é $(X,f,X')$, em que $X$ e $X'$ são produtos de $X_{j}$ e $X_{j}'$,
ao passo que o morfismo $f:X\rightarrow X'$ é obtido da seguinte
maneira: se $\pi_{j}':X'\rightarrow X'_{j}$ e $\pi_{j}:X\rightarrow X_{j}$
são projeções, então cada $f_{j}\circ\pi_{j}$ é morfismo de $X$
em $X_{j}'$, donde, por universalidade, obtém-se o morfismo $f:X\rightarrow X'$
procurado. Condição análoga é válida para coprodutos.
\end{example}
\end{example}
\end{example}
Uma categoria $\mathbf{C}$ com \emph{produtos }(resp. \emph{coprodutos})\emph{
finitos} é aquela na qual o limite (resp. colimite) de cada $F:\mathbf{J}\rightarrow\mathbf{C}$,
com $\mathbf{J}$ formada apenas de um número finito de índices, sempre
existe. Ressaltamos que aqui se está incluindo a situação em que $\mathbf{J}$
é a categoria que não possui objetos nem morfismos. O limite (resp.
colimite) em tal caso é chamado de \emph{objeto terminal }(resp. \emph{inicial})
de $\mathbf{C}$. A nomenclatura se deve ao seguinte: como imediatamente
se constata, $X$ é objeto terminal (rep. inicial) de $\mathbf{C}$
se, e somente se, para qualquer outro $X'\in\mathbf{C}$ existe um
único morfismo $f:X'\rightarrow X$ (resp. $f:X\rightarrow X'$).
Um objeto que é simultaneamente inicial e final chama-se \emph{nulo}
ou \emph{trivial.}
\begin{example}
Se uma categoria $\mathbf{C}$ tem $*$ como objeto terminal, então
este é um objeto nulo para sua pontuação $\mathbf{C}_{*}$. Por exemplo,
em $\mathbf{Set}$ qualquer conjunto formado de um só elemento é objeto
terminal, de modo que este serve de objeto nulo para $\mathbf{Set}_{*}$.
Quando a inclusão $\imath:\mathbf{Alg}\rightarrow\mathbf{Set}$ preserva
limites (isto é, quando a categoria algébrica é livremente gerada
por conjuntos), então os objetos terminais de $\mathbf{Alg}$ são
levados nos objetos terminais de $\mathbf{Set}$. Assim, todo ente
com um único elemento é terminal em $\mathbf{Alg}$ e, portanto, objeto
nulo para $\mathbf{Alg}_{*}$. Ocorre que, se escolhemos o elemento
neutro como ponto base, então todo morfismo o preserva, de modo que
$\mathbf{Alg}_{*}\simeq\mathbf{Alg}$. Portanto, estruturas triviais
(com um só elemento) são objetos nulos de $\mathbf{Alg}$.

\begin{example}
Em contrapartida ao exemplo anterior, observamos que nem sempre um
objeto terminal é nulo. Com efeito, $\mathbf{Set}$ (ou mesmo $\mathbf{Top}$)
possuem como terminais qualquer objeto formado de um só elemento (digamos
$*$). Estes, no entanto, não são objetos iniciais para tais categorias.
De fato, inexiste um morfismo $*\rightarrow\varnothing$. Consequentemente,
o único objeto inicial em $\mathbf{Set}$, $\mathbf{Top}$ ou em qualquer
categoria que admita o conjunto vazio como objeto, é o próprio $\varnothing$. 
\end{example}
\end{example}
Façamos algumas observações:
\begin{enumerate}
\item por argumentos indutivos, vê-se que uma categoria possui produtos
finitos se, e somente se, possui produtos binários (isto é, limites
de functores $F:\mathbf{J}\rightarrow\mathbf{C}$, com $\mathbf{J}$
formada de só dois índices) e um objeto terminal. Dualmente, para
que uma categoria tenha coprodutos finitos é necessário e suficiente
que ela admita coprodutos binários e um objeto inicial;
\item uma categoria com produtos e coprodutos finitos isomorfos é dita ter
\emph{biprodutos}. Assim, $\mathbf{C}$ têm biprodutos se, e só se,
admite produtos e coprodutos binários isomorfos e um objeto nulo.
Aqui se enquadram $\mathbf{AbGrp}$, $\mathbf{Mod}_{R}$;
\item numa categoria com produtos binários $X\times Y$, para todo espaço
$X$ tem-se um morfismo $\Delta_{X}:X\rightarrow X\times X$, chamado
de \emph{mapa diagonal} de $X$, tal que $\pi_{i}\circ\Delta_{X}=id_{X}$.
Ele é obtido diretamente da universalidade das projeções. Da mesma
forma, se uma categoria têm coprodutos $X\oplus Y$, então, por universalidade,
existem morfismos $\nabla_{X}:X\oplus X\rightarrow X$ satisfazendo
$\nabla_{X}\circ\imath_{i}=id_{X}$.
\end{enumerate}
\noindent \textbf{Notação.} No que segue, produtos serão genericamente
denotados por $\prod$ ou $\times$, ao passo que coprodutos arbitrários
serão representados por $\bigoplus$.

\subsection*{\uline{Limites Indutivos}}

$\quad\;\,$Diz-se que uma categoria não-vazia $\mathbf{I}$ é \emph{dirigida}
(ou que está \emph{direcionada}) quando:
\begin{enumerate}
\item é pequena e em seu conjunto de objetos (aqui também chamados de índices)
está definida uma relação de ordem parcial $\leq$;
\item existe um morfismo $f:i\rightarrow j$ se, e somente se, $i\leq j$.
\end{enumerate}
$\quad\;\,$Um functor covariante (resp. contravariante) $F:\mathbf{I}\rightarrow\mathbf{C}$
chama-se \emph{sistema dirigido }(resp. \emph{codirigido}) de objetos
de $\mathbf{C}$ com coeficientes em $\mathbf{I}$. Costuma-se escrever
$X_{i}$ (resp. $X^{i}$) ao invés de $F(i)$, e $f_{ij}$ (resp.
$f^{ji}$) ao invés de $F(f)$, com $f:i\rightarrow j$. Ao longo
do texto, esta prática será adotada. Quando existe, o colimite (resp.
limite) de um sistema dirigido (resp. codirigido) é chamado \emph{limite
indutivo} (resp. de \emph{colimite indutivo}). Apesar de confusa,
tal nomenclatura é canônica.
\begin{example}
Para todo conjunto parcialmente ordenado $I$, a categoria $\mathrm{Pos}(I)$,
definida no exemplo 2.1.3, é claramente dirigida. Em particular, a
categoria definida pelo ordenamento natural de $\mathbb{N}$ é direcionada.
Um sistema dirigido em $\mathbf{C}$ com coeficientes naturais é simplesmente
uma sequência (no caso co-dirigido, as setas são invertidas)$\underset{\;}{\;}$$$
\xymatrix{ X_0 \ar@/_{0.3cm}/[rr] \ar@/_{0.4cm}/[rrr] \ar@/_{0.5cm}/[rrrr]_{\;} \ar[r] & X_1 \ar@/^{0.3cm}/[rr] \ar@/^{0.4cm}/[rrr] \ar[r] & X_2 \ar[r] & X_3 \ar[r] & \cdots }
$$

\begin{example}
Numa categoria $\mathbf{C}$, fixemos $\mathbf{J}\subset\mathbf{C}$
e $X\in\mathbf{C}$. Seja $\mathbf{J}_{X}$ a categoria formada pelos
objetos de $\mathbf{J}$ que tem $X$ como subobjeto, e com morfismos
dados por inclusões. Esta se torna dirigida quando dotada do ordenamento
parcial proporcionado pela relação usual de continência (observamos,
em particular, que se $X=\varnothing$ é objeto inicial de $\mathbf{J}$,
então $\mathbf{J}_{\varnothing}\simeq\mathbf{J}$). Aqui se enquadram
importantes situações, nas quais em geral se considera $\mathbf{C}=\mathbf{Top}$.
Por exemplo, se $\mathbf{J}$ é a subcategoria dos espaços compactos,
$\mathbf{J}_{X}$ é utilizado na definição de \emph{cohomologias com
suporte compacto} de $X$ (veja, por exemplo, os dois últimos capítulos
de \cite{Greenberg_topologia_algebrica}).

\begin{example}
Uma \emph{cobertura }de um objeto $X\in\mathbf{C}$ é uma decomposição
de $X$ em termos de um limite indutivo de subobjetos $X_{i}$ de
$X$. Outro caso enquadrado no exemplo anterior é obtido tomando $\mathbf{J}$
igual à $\mathrm{Cov}(X)\subset\mathbf{C}$, formada de todas os subobjetos
que fazem parte de coberturas convenientes de $X$. Por exemplo, quando
$\mathbf{C=}\mathbf{Top}$, pode-se considerar $\mathrm{Cov}(X)$
como sendo a categoria dos abertos de $X$. Os functores contravariantes
$\mathscr{F}:\mathrm{Cov}(X)^{op}\rightarrow\mathbf{D}$ chamam-se
\emph{pré-feixes} \emph{de $X$ com valores em $\mathbf{D}$}. Se
por um lado functores são responsáveis por associar invariantes a
espaços, por outro, os pré-feixes estão relacionados com a construção
de \emph{invariantes locais}. Quando $\mathbf{D}=\mathbf{Set}$, os
elementos $s\in\mathscr{F}(X_{i})$ são chamados de \emph{secções}
de $\mathscr{F}$ em $X_{i}$. A razão é a seguinte: como comentaremos
adiante, o estudo de alguns pré-feixes ``globalizáveis'' se resume
ao estudo dos functores $\Gamma$ de secções de certos fibrados.

\begin{example}
Pode ser que o invariante local $\mathscr{F}:\mathrm{Op}(X)^{op}\rightarrow\mathbf{Set}$,
com $\mathbf{C\subset}\mathbf{Set}$, seja ambíguo nas vizinhanças
de um determinado ponto de $X$. Mais precisamente, pode ser que,
fixado $x\in X$, diferentes elementos $\mathscr{F}(X_{i})$ e $\mathscr{F}(X_{j})$
coincidam numa vizinhança de $x$ contida em $X_{i}\cap X_{j}$. Em
muitas ocasiões, esta ambiguidade é dispensável. Isto nos leva a substituir
o pré-feixe $\mathscr{F}$ por um outro $\overline{\mathscr{F}}$,
de tal maneira que secções que coincidem nas vizinhanças de um mesmo
ponto sejam identificadas. Descrevamos esta troca com um pouco mais
de detalhes. Para cada $x\in X$, denotemos por $X(x)$ a categoria
dirigida formada de todo subobjeto conveniente de $X$ que contém
$x$. Isto é, seja $X(x)=\mathbf{J}_{x}$ quando $\mathbf{J}=\mathrm{Cov}(X)$.
Seja, também, $\mathscr{F}(x)$ a reunião dos $\mathscr{F}(X_{i})$,
com $X_{i}\in X(x)$. Define-se o \emph{germe} em $x$ de uma secção
$s\in\mathscr{F}(x)$ como sendo a sua classe de equivalência $s_{x}$
pela relação que identifica a $s$ qualquer outra secção $s'\in\mathscr{F}(x)$
tal que $s'\vert_{V}=s\vert_{V}$ em alguma vizinhança suficientemente
pequena de $x$. O espaço quociente segundo tal relação chama-se \emph{stalk}
de $\mathscr{F}$ em $x$, sendo denotado por $\mathscr{F}_{x}$.
Assim, $\mathscr{F}_{x}$ é o limite co-indutivo da restrição $\mathscr{F}\vert_{X(x)}$.
O novo pré-feixe $\overline{\mathscr{F}}:\mathrm{Cov}(X)\rightarrow\mathbf{Set}$
é aquele que a cada $X_{i}\subset X$ associa a reunião (disjunta)
dos \emph{stalks} em pontos de $X_{i}$, ao passo que, para toda inclusão
$\imath:X_{i}\rightarrow X_{j}$ e todo germe $s_{x}$, com $x\in X_{j}$,
ele satisfaz $[\overline{\mathscr{F}}(\imath)](s_{x})=(s\vert_{X_{i}})_{x}$.
Moral da história: \emph{quando somente o comportamento nas vizinhanças
de pontos é importante, para evitar ambiguidade, deve-se trocar secções
por germes e conjuntos destas por} \emph{stalks}.

\begin{example}
Dois casos particulares do que foi discutido no exemplo anterior:
para toda variedade $X$ e todo ponto $x\in X$ há um pré-feixe $\mathscr{T}:U(x)\rightarrow\mathbf{Vec}_{\mathbb{R}}$,
que a cada $U$ contendo $x$ associa $TU_{x}$, que é o conjunto
das derivações em $x$ de $\mathcal{D}(U)$. No entanto, quaisquer
duas funções que coincidam numa vizinhança de $x$ ali possuem a mesma
derivação. Portanto, ao invés do espaço tangente $TX_{x}$, seria
mais sensato considerar o \emph{stalk }de $\mathscr{T}$ em $x$.
Para naturais $m,k$ há também pré-feixes $\mathscr{J}:\mathrm{Op}(X)\rightarrow\mathbf{Set}$,
responsável por associar a cada aberto de $U\subset X$ o conjunto
dos polinômios de Taylor das funções $f:U\rightarrow\mathbb{R}^{n}$,
de classe $C^{k}$. O germe de $f$ em em $x$ é precisamente a coleção
das aplicações que têm contato de grau $k$ com $f$. Por sua vez,
a reunião $\overline{\mathscr{J}}(U)$ de todos os \emph{stalks} de
$\mathscr{J}$ em pontos de $U$ é o conjunto $J^{k}(U;\mathbb{R}^{m})$,
que aparece na definição da topologia $C^{k}$. Veja, por exemplo,
o segundo capítulo de \cite{Hirsch}.

\begin{example}
Além de pré-feixes, \emph{stalks }e cohomologias com suporte compacto,\emph{
}outra coisa que se pode fazer com limites indutivos é falar de categorias
geradas por subcategorias. Com efeito, diz-se que $\mathbf{C}$ é
\emph{gerada} por $\mathbf{D}$ quando cada $X\in\mathbf{C}$ admite
uma decomposição como um limite indutivo de objetos de $\mathbf{D}$.
Por exemplo, $\mathbf{Vec}_{\mathbb{K}}$ é gerada por $\mathrm{F}\mathbf{Vec}_{\mathbb{K}}$:
qualquer espaço vetorial é limite de seus subespaços de dimensão finita.
\end{example}
\end{example}
\end{example}
\end{example}
\end{example}
\end{example}

\subsection*{\uline{Equalizadores}}

$\quad\;\,$Seja $(\rightrightarrows)$ a categoria que só possui
dois objetos (digamos $1$ e $2$) e, além das identidades, dois únicos
morfismos $\delta,\delta':1\rightarrow2$. Para qualquer que seja
a categoria $\mathbf{C}$, existe uma bijeção natural entre o conjunto
dos functores $F:(\rightrightarrows)\rightarrow\mathbf{C}$ e o conjunto
dos pares de morfismos $\mathbf{C}$ que possuem mesmo domínio e mesmo
co-domínio. Trata-se da aplicação que toma $F$ e devolve os morfismos
$F(\delta)$ e $F(\delta')$, de $F(1)$ em $F(2)$.

Diz-se que dois morfismos $f,g:X\rightarrow X'$ de $\mathbf{C}$
podem ser \emph{equalizados} quando o respectivo functor $F:(\rightrightarrows)\rightarrow\mathbf{C}$,
correspondendo ao par $(f,g)$ pela identificação acima descrita,
possui limite. Isto acontece se, e só se, existem $E\in\mathbf{C}$
e um morfismo $eq:E\rightarrow X$, com $f\circ eq=g\circ eq$, de
tal maneira que, se $eq':E'\rightarrow X$ é outro morfismo cumprindo
$f\circ eq'=g\circ eq'$, então é possível escrever $eq'=eq\circ\imath$
para um único $\imath:E'\rightarrow E$. Na presente situação, o diagrama
(3.1) se resume a: $$ 
\xymatrix{ E' \ar@/_{0.3cm}/[rr]_{eq'} \ar@{-->}[r]^{\imath} & E \ar[r]^{eq} & X \ar@<-.5ex>[r]_g \ar@<.5ex>[r]^f & Y }
$$

Uma categoria \emph{dotada de} \emph{equalizadores} é aquela para
a qual quaisquer dois morfismos com mesmo domínio e mesmo codomínio
podem ser equalizados. Dualmente, define-se de maneira natural o que
vem a ser uma categoria \emph{dotada de} \emph{coequalizadores}.
\begin{example}
Tanto $\mathbf{Set}$ quanto $\mathbf{Top}$ possuem equalizadores.
Em ambos os casos, dadas aplicações $f,g:X\rightarrow X'$, basta
considerar o conjunto $E$, formado de todo $x\in X$ no qual $f(x)=g(x)$,
bem como a inclusão $\imath:E\rightarrow X$. Semelhantemente, as
categorias algébricas abelianas também possuem equalizadores. Com
efeitos, o \emph{kernel} da diferença $f-g$ equaliza os homomorfismos
$f,g:X\rightarrow X'$ (daí o nome \emph{difference kernel}, muitas
vezes utilizado na literatura como sinônimo de equalizador).

\begin{example}
Equalizadores também são utilizados na ``globalização'' de pré-feixes.
Com efeito, diz-se que um pré-feixe $\mathscr{F}$ é \emph{feixe}
quando os invariantes locais por ele assignados podem ser \emph{univocamente
globalizados}. Isto significa que, para toda cobertura por subobjetos
$X_{i}$, as secções globais (isto é, os elementos de $\mathscr{F}(X)$)
são descritas, de maneira única, por suas restrições aos $X_{i}$.
Desta forma, se $s,s'\in\mathscr{F}(X)$ satisfazem $s\vert_{X_{i}}=s'\vert_{X_{i}}$,
então $s=s'$. Assim, $\mathscr{F}$ é feixe quando, para qualquer
cobertura aberta de $X$, o diagrama abaixo é equalizador. Nele, $r$
associa a cada secção global $s$ a família das restrições $s\vert_{X_{i}}$.
Por sua vez, os mapas paralelos tomam uma lista de secções $s_{i}$
e devolvem as respectivas famílias de suas restrições a $X_{i}\cap X_{j}$
e a $X_{j}\cap X_{i}$. $$ 
\xymatrix{ F(X) \ar[r]^-{r} & \prod_{i}F(X_i) \ar@<-.5ex>[r] \ar@<.5ex>[r] & \prod_{i,j}F(X_{i}\cap X_{j}) }
$$
\end{example}
\end{example}
Sobre o exemplo anterior, duas observações:
\begin{enumerate}
\item a subcategoria cheia $\mathbf{Shv}(X)\subset\mathrm{Func}(\mathrm{Op}(X)^{op};\mathbf{Set})$,
formada dos feixes de conjuntos de um espaço $X$, possui uma descrição
bastante simples. Com efeito, ela é equivalente à subcategoria cheia
$\mathbf{\grave{E}tl}\subset\mathscr{B}_{X}$ dos homeomorfismos locais
$f:Y\rightarrow X$, chamados de \emph{espaços Ètale}. Uma maneira
de obter tal equivalência é através do functor $\Gamma:\mathbf{\grave{E}tl}\rightarrow\mathbf{Shv}(X)$,
que a cada fibrado associa o feixe de suas secções. A inversa fraca
é a regra que a cada pré-feixe $\mathscr{F}$ em $X$ associa o espaço
Ètale $\pi:\mathrm{\grave{E}}(\mathscr{F})\rightarrow X$, em que
$\mathrm{\grave{E}}(\mathscr{F})$ é a reunião disjunta de todos os
\emph{stalks} de $\mathscr{F}$ e $\pi(s_{x})=x$;
\item maiores detalhes sobre a Teoria de Feixes podem ser encontrados em
\cite{Tennison_sheaf,Swan_sheaves}, e também no clássico \cite{godement_feixes}.
\end{enumerate}

\subsection*{\emph{\uline{Pullbacks}}\uline{ e }\emph{\uline{Pushouts}}}

$\quad\;\,$Seja $(\rightarrow\cdot\leftarrow)$ a categoria que possui
somente três objetos (aqui denotados por $1$, 2 e $*$), e além das
identidades, dois únicos morfismos $\delta:1\rightarrow*$ e $\delta':2\rightarrow*$.
Para qualquer outra categoria $\mathbf{C}$, o conjunto dos functores
de $(\rightarrow\cdot\leftarrow)$ em $\mathbf{C}$ está em bijeção
com a coleção dos pares de morfismos de $\mathbf{C}$ que têm igual
co-domínio. 

Fala-se que dois morfismos $f:X\rightarrow Y$ e $g:X'\rightarrow Y$
possuem \emph{pullback} quando existe o limite do functor associado
a $(f,g)$ pela referida identificação. Isto significa que, dados
$f:X\rightarrow Y$ e $g:X'\rightarrow Y$, existem um objeto $\mathrm{Pb}\in\mathbf{C}$
e morfismos $\varphi:\mathrm{Pb}\rightarrow X$ e $\varphi:\mathrm{Pb}\rightarrow X'$,
os quais satisfazem a igualdade $g\circ\varphi=\varphi\circ f$. Além
disso, para quaisquer outros morfismos $\varphi':\mathrm{Pb}'\rightarrow X$
e $\varphi':\mathrm{Pb}'\rightarrow X'$ cumprindo a condição $g\circ\varphi'=\varphi'\circ f$,
existe um único $\jmath:\mathrm{Pb}'\rightarrow\mathrm{Pb}$ que torna
comutativo o primeiro dos diagramas abaixo: 

$$
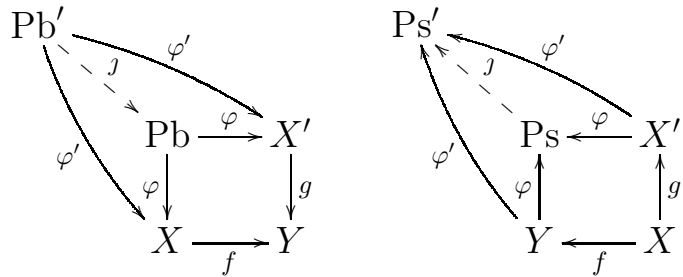
\xymatrix{ \mathrm{Pb}' \ar@{-->}[rd]^{\jmath} \ar@/^/[rrd]^{\varphi '} \ar@/_/[ddr]_{\varphi '} & & & \mathrm{Ps}' \\
& \mathrm{Pb} \ar[r]^{\varphi} \ar[d]_{\varphi} & X' \ar[d]^{g} & & \mathrm{Ps} \ar@{-->}[ul]_{\jmath}   & \ar[l]_{\varphi} \ar@/_/[ull]_{\varphi '}   X'  \\
& X \ar[r]_{f} & Y & & Y \ar[u]^{\varphi} \ar@/^/[uul]^{\varphi '}  & \ar[l]^{f}  X \ar[u]_{g} }
$$$\quad\;\,$Quando cada par $(f,g)$ de morfismos de $\mathbf{C}$
admite \emph{pullback}, fala-se que a própria categoria tem \emph{pullbacks}.
A versão dual do \emph{pullback}, totalmente caracterizada pelo segundo
dos diagramas acima, chama-se \emph{pushout}.
\begin{example}
Se uma categoria $\mathbf{C}$ possui produtos binários e equalizadores,
então ela também possui \emph{pullbacks}: o \emph{pullback} de $f:X\rightarrow Y$
e $g:X'\rightarrow Y$ é o equalizador de $f\circ\pi_{1}:X\times X'\rightarrow Y$
e $g\circ\pi_{2}:X\times X'\rightarrow Y$, em que $\pi_{i}$ são
as projeções do produto $X\times X'$. Nestas circunstâncias, é usual
escrever $X\times_{Y}X'$ ao invés de $\mathrm{Pb}$. Assim, dentro
de $\mathbf{Set}$, bem como em\textbf{ $\mathbf{Top}$,} o \emph{pullback}
de aplicações $f,g$ é o conjunto $\mathrm{X\times_{Y}X'}$ de todo
$(x,x')$ tal que $f(x)=g(x')$. 

\begin{example}
Quando $\mathbf{C}$ tem \emph{pullbacks}, então $\mathscr{B}_{Y}(\mathbf{C})$
também os tem. Em geral, estes aparecem no seguinte contexto: se $(X,\pi,Y')$
é fibrado com base $Y'$, então todo morfismo $f:Y\rightarrow Y'$
induz um fibrado $(f^{*}X,\mathrm{pr}_{1},Y')$ com base em $Y$,
definido pelo \emph{pullback} de $(f,\pi)$. Por exemplo, quando $\mathbf{C}$
é $\mathbf{Set}$ ou $\mathbf{Top}$, o objeto $f^{*}X$ nada mais
é que o conjunto de todo par $(y,x)\in Y\times X$ tal que $f(y)=\pi(x)$.
Identificando $X$ com a reunião das fibras $X_{y'}=\pi^{-1}(y')$,
vê-se que a fibra em $y$ de $f^{*}X$ é o conjunto $\mathrm{pr}_{1}^{-1}(y)=y\times X_{f(y)}$,
naturalmente isomorfo a $X_{f(y)}$. Portanto, tem-se um procedimento
que troca a base de um fibrado $(X,\pi,Y')$, deixando suas fibras
invariantes. 

\begin{example}
Dualmente, numa categoria com coprodutos binários e coequalizadores,
sempre existem \emph{pushouts}: o \emph{pushout} de $f:X\rightarrow Y$
e $g:X\rightarrow Y'$ é o coequalizador dos morfismos $\imath_{1}\circ f:X\rightarrow Y\oplus Y'$
e $\imath_{2}\circ g:X\rightarrow Y\oplus Y'$, onde $\imath_{i}$
são inclusões de $Y\oplus Y'$. Desta forma, em $\mathbf{Set}$ e
$\mathbf{Top}$, o \emph{pushout} de $f,g$ é a reunião disjunta $Y\sqcup Y'$
de seus codomínios, quocientada pela relação que identifica os elementos
$f(x)$ e $g(x)$, para todo $x\in X$. Por exemplo, se $\imath:X\rightarrow Y'$
é inclusão, então o \emph{pushout} de $f:X\rightarrow Y$ com $\imath$
é chamado de \emph{colagem }(através de $f$) de $Y$ em $Y'$ ao
longo de $X$. Esta particular situação é utilizada, por exemplo,
na definição de CW-complexos.
\end{example}
\end{example}
Uma propriedade fundamental dos \emph{pullbacks} é que, quando colados,
seja de forma horizontal ou vertical, eles produzem novo \emph{pullback}.
De maneira mais precisa, como facilmente se verifica, se os dois primeiros
diagramas são \emph{pullbacks}, então o terceiro e o quarto também
o são. Resultado análogo é válido para \emph{pushouts}.$$
\xymatrix{ &&&&&&& \mathrm{Pb} \ar@{-->}[r] \ar@{-->}[dd] & X' \ar[d] \\
\mathrm{Pb} \ar@{-->}[r] \ar@{-->}[d] & X' \ar[d] & \mathrm{Pb} \ar@{-->}[r] \ar@{-->}[d] & Y \ar[d] & \mathrm{Pb} \ar@{-->}[d] \ar@{-->}[rr] && Y \ar[d] & & Y \ar[d] \\
X \ar[r] & Y & Y \ar[r] & Z & X \ar[r] & Y \ar[r] & Z & Y \ar[r] & Z}
$$

\begin{example}
Na álgebra, o \emph{kernel} de $f:X\rightarrow Y$ é o subconjunto
$\ker(f)\subset X$ formado de todo $x\in X$ tal que $f(x)=0$. Seja
$\ker f$ a inclusão $\ker(f)\rightarrow X$. Como logo se convence,
dados um grupo abeliano $A$ e um morfismo $g:A\rightarrow X$, para
que $A=\ker(f)$ e $g=\ker f$, é necessário e suficiente que estes
constituam o \emph{pullback} do par $(0,f)$, em que $0$ é morfismo
do grupo trivial em $Y$. Por outro lado, o \emph{cokernel} de $f$
é o espaço quociente $Y/\mathrm{im}f$. Seja $\mathrm{coker}f$ a
aplicação quociente $Y\rightarrow\mathrm{coker}(f)$. Fornecidos $A\in\mathbf{AbGrp}$
e $g:Y\rightarrow A$, estes constituirão o \emph{cokernel} de $f$
se, e somente se, forem o \emph{pushout} de $(0,f)$. Tem-se, também,
as identificações $\mathrm{im}(f)=\ker(\mathrm{coker}(f))$ e $\mathrm{coim}(f)=\mathrm{coker}(\ker(f))$.
\end{example}
\end{example}
Motivados pelo exemplo anterior, em qualquer categoria com objeto
inicial $\varnothing$ (resp. final $*$), o \emph{pullback} (resp.
\emph{pushout})\emph{ }do par $(\varnothing,f)$ (resp. $(*,f$))
é chamado de \emph{kernel} (resp. \emph{cokernel}) de $f$. Por sua
vez, a \emph{imagem }e a \emph{coimagem} de $f$ serão definidas como
sendo $\ker(\mathrm{coker}(f))$ e $\mathrm{coker}(\ker(f))$. Em
tais categorias, também faz sentido falar de \emph{sequências exatas}
de morfismos: uma sequência $f_{n}$ diz-se \emph{exata} quando, para
qualquer que seja o $n$, o \emph{kernel} de $f_{n}$ coincide, a
menos de isomorfismos, com a imagem de $f_{n-1}$. Uma sequência exata
é dita \emph{curta} quando possui somente dois termos não-nulos, sendo
estes consecutivos.

A aplicabilidade das sequências exatas se deve ao seguinte: uma vez
construídos invariantes (isto é, uma vez obtidos functores que associam
a cada espaço uma estrutura), precisa-se \emph{calculá-los}. Neste
sentido, as sequências exatas constituem poderosas ferramentas. De
maneira mais precisa, procura-se, por exemplo, por functores que associem
a cada sequência exata curta de espaços uma sequência exata longa
na categoria dos correspondentes invariantes. Se tal categoria for
adequada, então é possível retirar informação dos invariantes de um
espaço da sequência a partir dos invariantes dos outros. O estudo
das sequências exatas é concernente à \emph{Álgebra Homológica} (veja,
por exemplo, \cite{Hilton_homological_algebra,Rotman_homological_algebra,Weibel_homological_algebra},
assim como os clássicos \cite{homological_algebra_CARTAN,MACLANE_homology}.
\begin{example}
Em $\mathbf{Set}$ ou em $\mathbf{Top}$, o \emph{cokernel} de uma
inclusão $\imath:A\rightarrow X$ é simplesmente o espaço $X/A$ obtido
identificando todos os elementos de $A$ num único ponto. A sequência
cujos termos não-triviais são a inclusão precedida da projeção $\pi:X\rightarrow X/A$
é exata curta. Situação análoga acontece em categorias algébricas,
em que $X/A$ agora deve ser interpretado como o quociente da estrutura
$X$ por uma subestrutura para a qual tal quociente está bem definido.
Por exemplo, em $\mathbf{Grp}$ o objeto $X$ é grupo, enquanto que
$A\subset X$ é subgrupo normal. Em $\mathbf{Rng}$, por sua vez,
$X$ é anel e $A\subset X$ deve ser um ideal.
\end{example}

\subsection*{\uline{Completude}}

$\quad\;\,$Diz-se que uma categoria é \emph{completa }(resp. \emph{cocompleta})
quando todos os functores que nela assumem valores possuem limites
(resp. colimites). Na subsecção anterior, vimos que algumas classes
de limites podem ser criados se supomos a existência de outros. Por
exemplo, mostramos que se uma categoria admite produtos binários e
equalizadores, então ela possui \emph{pullbacks}. 

A proposição abaixo nos indica que, se reforçamos a hipótese adicionando
a existência de todos os produtos (não só dos binários), ganhamos
a existência não só dos \emph{pullbacks}, mas de qualquer limite.
Ressaltamos que a prova aqui apresentada foi inteiramente baseada
em \cite{MACLANE_categories}.
\begin{prop}
Para que uma categoria $\mathbf{D}$ seja completa, é necessário e
suficiente que ela possua todos os produtos e todos os equalizadores.
\end{prop}
\begin{proof}
A necessidade é evidente. Para a suficiência, fornecido $F:\mathbf{C}\rightarrow\mathbf{D}$,
mostraremos que as hipóteses asseguram a existência de um cone $(X,\varphi)$
para $F$ tal que, se $(X',\varphi')$ é qualquer outro cone de $F$,
então há uma única fatoração $u:X'\rightarrow X$. Escolhido um morfismo
$f:Y\rightarrow Y'$ em $\mathbf{C}$, escrevamos $Y=\mathrm{dom}f$
e $Y'=\mathrm{cod}f$. Uma vez que $\mathbf{D}$ possui todos os produtos,
seja $\prod_{f}F(\mathrm{cod}f)$ o produto dos objetos de $\mathbf{D}$
indexados pelo contradomínio de morfismos de $\mathbf{C}$. Da mesma
forma, seja $\prod_{Y}F(Y)$ o produto dos objetos de $\mathbf{D}$
indexados por objetos de $\mathbf{C}$. Como todo objeto de $\mathbf{C}$
é domínio e contradomínio de ao menos um morfismo (por exemplo de
sua respectiva identidade), segue-se que $F(\alpha)\circ\pi_{\mathrm{dom}f}$
e $id\circ\pi_{\mathrm{cod}f}$ definem outras projeções para $\prod_{f}F(\mathrm{cod}f)$,
de tal modo que, por universalidade, existem as setas paralelas representadas
no diagrama abaixo:$$
\xymatrix{& & & F(\mathrm{dom} f) \ar[rr]^{F(f)} & & F(\mathrm{cod} f)\\
X' \ar@{-->}[rr]^{\mu} \ar@/^/[rrru]^{\varphi '} \ar@/_/[rrrd]_{\varphi '} & & E \ar[ur]^-{\varphi} \ar[dr]_-{\varphi} \ar[r]^-{eq} & \prod _{Y} F(Y) \ar[d]^{\pi _{\mathrm{cod} f}} \ar[u]_{\pi _{\mathrm{dom} f}} \ar@<-.5ex>[rr] \ar@<.5ex>[rr] & & \prod _{f} F(\mathrm{cod} f) \ar[u]_{\pi _{f}} \ar[d]^{\pi _{f}} \\
& & & F(\mathrm{cod} f) \ar[rr]_{id} && F(\mathrm{cod} f)}
$$
O par $(X,\varphi)$, em que $X$ é o equalizador das setas paralelas
e $\varphi(Y)=\pi_{Y}\circ eq$, formam um cone para $F$. Afirmamos
que ele é o limite procurado. Com efeito, se $(X',\varphi')$ é outro
cone, a universalidade dos produtos nos garante a existência de um
único morfismo $eq':X'\rightarrow\prod_{Y}F(Y)$ que preserva as setas
paralelas. Desta forma, a universalidade do equalizador fornece um
único $u:X'\rightarrow E$, garantindo o afirmado e concluindo a demonstração.
\end{proof}
Sobre o resultado anterior, duas observações:
\begin{enumerate}
\item ele é uma típica situação na qual a diferença entre os conceitos de
classe e de conjunto deve ser levada em consideração: durante toda
a demonstração, a classe dos objetos de $\mathbf{C}$ foi tratada
como sendo um conjunto, algo que, numa abordagem axiomática, pode
não ser verdade para uma categoria arbitrária. Mais uma vez, isto
ressalta a ingenuidade com a qual temos trabalhado;
\item tem-se uma versão dual, demonstrada de maneira estritamente análoga:
\emph{para que uma categoria seja cocompleta, é necessário e suficiente
que ela tenha coprodutos e co-equalizadores}.
\end{enumerate}
\begin{example}
Se uma categoria $\mathbf{D}$ for completa, então $\mathrm{Func}(\mathbf{C};\mathbf{D})$
também o será, independente de quem for $\mathbf{C}$. De fato, consideremos
o bifunctor de avaliação 
\[
ev:\mathrm{Func}(\mathbf{C};\mathbf{D})\times\mathbf{C}\rightarrow\mathbf{D}\quad\mbox{tal que}\quad ev(F,X)=F(X).
\]
Ele induz uma família de functores $ev_{X}:\mathrm{Func}(\mathbf{C};\mathbf{D})\rightarrow\mathbf{D}$,
definidos por $ev_{X}(F)=F(X)$. Dado qualquer $\alpha:\mathbf{C}'\rightarrow\mathrm{Func}(\mathbf{C};\mathbf{D})$,
o respectivo $ev_{X}\circ\alpha$ têm limite. Consequentemente, $\alpha$
também possui limite, garantido o que havíamos afirmado.

\begin{example}
As categorias $\mathbf{Mfd}$ e $\mathbf{Diff}$ possuem produtos
e coprodutos, mas não são completas nem cocompletas: em geral elas
não possuem \emph{pullbacks} nem \emph{pushouts}. Para \emph{pushouts},
tome, por exemplo, duas retas coladas ao longo de um único ponto.
Observamos, no entanto, que \emph{pullbacks} de mapas \emph{transversais
}existem em $\mathbf{Diff}$. Mais precisamente, se $f:X\rightarrow Y$
e $g:X'\rightarrow Y$ são diferenciáveis e transversais (isto é,
se a imagem das derivadas $Df_{x}$ e $Dg_{x'}$ geram o espaço $TY_{y}$
em cada $(x,x')$ tal que $f(x)=y=g(x')$), então o \emph{pullback}
do par $(f,g)$ é uma variedade diferenciável.
\end{example}
\end{example}

\subsection*{\uline{Espaços Pontuados}}

$\quad\;\,$Nesta subsecção, estudamos os limites e colimites em categorias
pontuadas $\mathbf{C}_{*}$. A ideia é tentar determiná-los a partir
dos limites de $\mathbf{C}$. De maneira direta, verifica-se que o
despareamento $\mathscr{D}:\mathbf{C}_{*}\rightarrow\mathbf{C}$ preserva
limites, de modo que, para todo functor $F:\mathbf{C}_{*}\rightarrow\mathbf{D}$,
se seu limite existe, então $\mathscr{D}(\lim F)\simeq\lim(\mathscr{D}\circ F)$.
Assim, por exemplo, produtos e \emph{pullbacks} em $\mathbf{C}_{*}$
podem ser diretamente calculados a partir de suas versões em $\mathbf{C}$,
bastando pontuá-los de maneira canônica. O exemplo abaixo nos mostra
que a mesma estratégia, no entanto, não se aplica para colimites.
\begin{example}
O coproduto de pares $(X_{i},x_{i})$ em $\mathbf{Set}_{*}$ é o \emph{produto
wedge} $\bigvee_{i}X_{i}$, definido da seguinte maneira: toma-se
a reunião de todos os produtos $X_{i}\times x_{i}$ e passa-se ao
quociente pela relação que identifica todos os pontos $x_{i}$ num
único, o qual há de ser o ponto base de $\bigvee_{i}X_{i}$ .
\end{example}
Observamos que o coproduto de $\mathbf{Set}_{*}$ foi obtido a partir
de colimites em $\mathbf{Set}$: primeiro tomamos a reunião disjunta
(que é o coproduto) e depois passamos ao quociente (que é um \emph{co-kernel}).
Argumentação análoga funciona para obter coprodutos em qualquer $\mathbf{C}_{*}$,
desde que $\mathbf{C}$ seja cocompleta.

\section{Reconstrução}

$\quad\;\,$Nesta secção, veremos que os limites são as extensões
de Kan fundamentais. Mais precisamente, mostraremos que sua categoria
$\mathbf{D}$ possui quantidade suficiente de coprodutos e co-equalizadores,
então todo $F:\mathbf{A}\rightarrow\mathbf{D}$ admite extensão de
Kan à esquerda relativamente a qualquer $\imath:\mathbf{A}\rightarrow\mathbf{C}$.
Em suma, mostraremos ser possível escrever 
\[
\mathscr{L}F(X)=\int^{Y}\mathrm{Mor}_{\mathbf{C}}(\imath(Y);X)\cdot F(Y),
\]
onde, para todo conjunto $S$ e todo objeto $E$, a \emph{co-potência}
$S\cdot E$ nada mais é que é o coproduto de cópias de $E$ indexadas
em $S$. Por sua vez, o símbolo de integração, usualmente denominado
\emph{coend},\emph{ }é uma espécie de soma universal sobre todos os
objetos de $\mathbf{A}$ e, como discutiremos em seguida, se vê descrita
por um co-equalizador.

Dualmente, mostraremos que se a categoria $\mathbf{D}$ admite produtos
e equalizadores suficientes, então as extensões de Kan à direita sempre
existem, com 
\[
\mathscr{R}F(X)=\int_{Y}F(Y)^{\mathrm{Mor}_{\mathbf{C}}(\imath(Y);X)},
\]
em que a \emph{potência} $E^{S}$ do objeto $E$ pelo conjunto $S$
é simplesmente o produto de cópias de $E$ indexas em $S$. Semelhantemente,
a integral representa o \emph{end}: um objeto dual ao coend e descrito
em termos de um equalizador.

Tais expressões possuem duas consequências imediatas, as quais mencionamos
desde já:
\begin{enumerate}
\item todo functor que assume valores numa categoria completa (resp. cocompleta)
possui extensões de Kan à direta (resp. à esquerda);
\item um tal functor preservará todas as extensões de Kan se, e somente
se, for \emph{contínuo} (isto é, se, e somente se, preservar todos
os limites). Dualmente, ele preservará todas as extensões de Kan à
direita se, e só se, for \emph{cocontínuo}.
\end{enumerate}

\subsection*{\emph{\uline{Ends}}}

$\quad\;\,$Na álgebra linear (isto é, na categoria $\mathbf{Vec}_{\mathbb{K}}$),
uma função do tipo $f:V\times V'\rightarrow W$ tanto pode ser uma
transformação linear do espaço produto $V\times V'$ em $W$, quanto
uma aplicação bilinear. Semelhantemente, correspondendo às transformações
lineares, tem-se as transformações naturais entre functores $F,F':\mathbf{C}^{op}\times\mathbf{C}\rightarrow\mathbf{D}$.
O análogo das aplicações bilineares são as \emph{transformações binaturais}.
Estas nada mais são que regras $\xi$, responsáveis por associar a
cada objeto $X\in\mathbf{C}$ um morfismo $\xi(X):F(X,X)\rightarrow F'(X,X)$
que deixa comutativo o diagrama abaixo para um $f:X\rightarrow X'$
qualquer. $$
\xymatrix{F(X',X) \ar[r]^{F(f,id)} \ar[d]_{F(id,f)} & F(X,X) \ar[r]^{\xi (X)} & F'(X,X) \ar[r]^{F'(id,f)} & F'(X,X') \ar[d]^{F'(f,id)} \\
F(X',X') \ar[rrr]_{\xi (X')} &&& F'(X',X') }
$$

Observamos que, se o bifunctor $F$ é constante (digamos igual a $C$),
então tal diagrama colapsa no seguinte quadrado comutativo:$$
\xymatrix{F'(X,X) \ar[r]^{F'(id,f)} & F'(X,X') \\
C \ar[u] \ar[r]_-{\xi (X')} & F'(X',X') \ar[u]_{F'(f,id) }}
$$

Vimos que o limite de um functor $F:\mathbf{C}\rightarrow\mathbf{D}$
se resume a um par $(X,\varphi)$, em que $X\in\mathbf{D}$ e $\varphi$
é transformação natural universal entre o functor constante em $X$
e $F$. Isto é, tal que qualquer outro par $(X',\varphi')$ fatora
em $(X,\varphi)$. Seguindo a mesma linha do parágrafo anterior, se
$F$ está definido numa categoria produto do tipo $\mathbf{C}^{op}\times\mathbf{C}$,
pode-se falar não só de transformações universais, mas também de transformações
binaturais universais. Realizando esta troca na caracterização do
limite, chega-se ao conceito de \emph{end }de $F$. Assim, o \emph{end}
de $F:\mathbf{C}^{op}\times\mathbf{C}\rightarrow\mathbf{D}$ é simplesmente
um \emph{cubo} $(X,\varphi)$, onde $X\in\mathbf{D}$ e $\varphi$
é uma transformação dinatural entre o functor constante em $X$ e
$F$, obtido de tal maneira que todo cubo $(X',\varphi')$ pode ser
unicamente fatorado em $(X,\varphi)$. Isto é, deve haver um único
$\mu:X'\rightarrow X$ que deixa comutativo o diagrama abaixo:\begin{equation}{
\xymatrix{& & & F(Y',Y') \ar[dr]^{F(f,id)} \\
X' \ar@{-->}[rr]^{u} \ar@/^/[rrru] \ar@/_/[rrrd] & & X \ar[ur]^{\varphi} \ar[dr]_{\varphi} && F(Y,Y') \\
& & & F(Y,Y) \ar[ur]_{F(id,f)}   }}
\end{equation}

\begin{example}
Como consequência do lema de Yoneda, o conjunto $\mathrm{Nat}(H;H')$
das transformações naturais entre dois functores $H,H':\mathbf{C}\rightarrow\mathbf{D}$
pode ser obtido em termos do \emph{end} do correspondente $F:\mathbf{C}^{op}\times\mathbf{C}\rightarrow\mathbf{Set}$,
que a cada par $(X,Y)$ associa $\mathrm{Mor}_{\mathbf{D}}(H(X);H'(Y))$.
\end{example}
De maneira dual ao \emph{end}, define-se o que vem a ser o \emph{coend}
de $F:\mathbf{C}^{op}\times\mathbf{C}\rightarrow\mathbf{D}$. Quando
existem, estes são respectivamente denotados por $\int_{\mathbf{C}}F$
e $\int^{\mathbf{C}}F$. Uma motivação para tais notações é obtida
fazendo alusão ao cálculo tensorial: lá, partindo-se de qualquer objeto
$f_{j}^{i}$ que é contravariante em um índice e covariante em outro,
para tomar seu traço (isto é, ao se contrair $i$ com $j$), deve-se
somar sobre os possíveis valores de $i$, resultando em $\mathrm{tr}(f_{j}^{i})=\sum_{i}f_{i}^{i}$.
A entidade $F:\mathbf{C}^{op}\times\mathbf{C}\rightarrow\mathbf{D}$
é contravariante em uma entrada e covariante em outra. Tomar seu traço
é ``somar'' sobre todos os possíveis valores de $F(X,X)$, com $X\in\mathbf{C}$.
Como o ``índice'' $X$ pode não ser discreto, tal ``soma'' é substituída
por uma ``integral'' $\int_{\mathbf{C}}F$. Esta pode ou não convergir,
traduzindo a existência ou inexistência do end de $F$.

Levemos a anlogia entre cálculo tensorial e ``análise categórica''
adiante. Consideremos um objeto com três índices (digamos $f_{kj}^{i}$)
sendo dois deles covariantes e um contravariante. Ao contrair $i$
com $j$, ganha-se uma soma sobre $i$ e o resultado é um novo objeto
$e_{k}$ com um único índice covariante livre e tal que, para cada
valor de $k$, sua respectiva componente é dada por $\sum_{i}f_{ki}^{i}$.
De maneira análoga, se o objeto inicial possuir dois índices contravariantes,
então o objeto final será contravariante. Assim, se consideramos um
functor do tipo $F:\mathbf{P}\times\mathbf{C}^{op}\times\mathbf{C}\rightarrow\mathbf{D}$,
em que o\emph{ end} de cada $F_{p}:\mathbf{C}^{op}\times\mathbf{C}\rightarrow\mathbf{D}$,
com $F_{p}(X,X')=F(P,X,X')$ existe, então há um $E:\mathbf{P}\rightarrow\mathbf{D}$
tal que $E(P)=\int_{\mathbf{C}}F_{p}$. Analogamente, se o functor
incial for contravariante na primeira entrada, então o respectivo
$E$ também o será, e $E(P)=\int^{\mathbf{C}}F_{p}$. Nestas condições,
diz-se que $F$ possui \emph{end} (resp. \emph{coend}) com parâmetros.
Veja as páginas 228-230 de \cite{MACLANE_categories}. Em cada caso,
costuma-se escrever 
\[
E(P)=\int_{X}F(P,X,X)\quad\mbox{e}\quad E(P)=\int^{X}F(P,X,X).
\]

Se agora consideramos um objeto com quatro índices $t_{lj}^{ki}$,
seu traço será independente da ordem que o tomamos em $i$ e em $k$.
No contexto da análise categórica isto significa que, dado um functor
do tipo $F:\mathbf{P}^{op}\times\mathbf{P}\times\mathbf{C}^{op}\times\mathbf{C}\rightarrow\mathbf{D}$
tal que o \emph{end} dos respectivos $F_{x}:\mathbf{P}^{op}\times\mathbf{P}\rightarrow\mathbf{D}$
e $F_{p}:\mathbf{C}^{op}\times\mathbf{C}\rightarrow\mathbf{D}$ existe,
então a ordem em que o calculamos é indiferente. Mais precisamente,
se $E_{x}$ e $E_{p}$ são os respectivos functores que a cada $x$
e cada $p$ associam o \emph{end} de $F_{x}$ e de $F_{p}$, então
vale $\int_{\mathbf{C}}E_{x}\simeq\int_{\mathbf{P}}E_{p}$. Sucintamente,
\[
\int_{P}\int_{X}F(P,P,X,X)\simeq\int_{X}\int_{P}F(P,P,X,X).
\]

Tal equivalência é usualmente chamada de \emph{teorema de Fubini}
por ser análogo a um resultado homônimo que aparece no cálculo e,
mais geralmente, em teoria da medida. Para uma prova formal, veja
o corolário na página 231 de \cite{MACLANE_categories}.

\subsection*{\emph{\uline{Ends}}\uline{ vs Limites}}

$\quad\;\,$Na subsecção anterior, obtivemos o conceito de \emph{end}
de um functor $F:\mathbf{C}^{op}\times\mathbf{C}\rightarrow\mathbf{D}$
trocando, na definição de limite, transformações naturais por transformações
binaturais. Nesta, exploraremos mais de perto a relação entre \emph{ends}
e limites. 

Iniciamos observando que, em geral, o \emph{end} de $F$ pode existir
sem que seu limite exista. Ainda assim, mesmo que ambos existam, eles
podem ser distintos. No entanto, se $F$ é constante na primeira entrada,
tais conceitos se equivalem. Intuitivamente, a razão é a seguinte:
no processo de categorificação, transformações naturais entre functores
então em correspondência com as transformações lineares entre espaços,
ao passo que transformações binaturais são os análogos de transformações
bilineares. Pode-se ter uma transformação bilinear $f:V\times V'\rightarrow W$
sem que esta o seja linear. No entanto, se a primeira entrada é trivial
(isto é, se vale $V\simeq0$), então toda $f:0\times V'\rightarrow W$
que seja bilinear também será linear. Daí, aplicando a analogia, se
$F$ é constante na primeira entrada, então o cubo $(X,\varphi)$
por ela definido se resume a um cone e, consequentemente, o \emph{end}
coincide com o limite. 

Sob um ponto de vista mais formal, tem-se a seguinte justificativa
para as afirmações anteriores: como vimos, limites e colimites podem
ser descritos em termos de produtos, equalizadores e suas versões
duais. Seguindo a mesma estratégia, vê-se que \emph{ends} e \emph{coends}
também são descritos por produtos e equalizadores. De maneira mais
precisa, se limites de functores $F:\mathbf{C}\rightarrow\mathbf{D}$
eram obtidos de equalizadores da forma

\begin{equation}{
\xymatrix{ \lim F \ar@{-->}[r] & \prod_X F(X) \ar@<-.5ex>[r] \ar@<.5ex>[r] &  \prod_f F(\mathrm{cod}(f)),}}
\end{equation}\emph{ends} de $F':\mathbf{C}^{op}\times\mathbf{C}\rightarrow\mathbf{D}$
passam a ser determinados por\begin{equation}{
\xymatrix{ \int_\mathbf{C} F' \ar@{-->}[r] & \prod_X F'(X,X) \ar@<-.5ex>[r] \ar@<.5ex>[r] &  \prod_f F'(\mathrm{dom} (f), \mathrm{cod}(f)).}}
\end{equation}

Daí, construindo o primeiro diagrama para $F'$, vê-se que seu limite
é dado por$$
\xymatrix{ \lim F' \ar@{-->}[r] & \prod_{(X,Y)} F'(X,Y) \ar@<-.5ex>[r] \ar@<.5ex>[r] &  \prod_{(f,g)} F'(\mathrm{cod}(f),\mathrm{cod}(g)),}
$$o qual é evidentemente distinto de (3.3), mas coincidente com (3.4)
quando $F'$ é constante na primeira entrada. Conclusões análogas
são válidas para colimites e \emph{coends}.

\subsection*{\uline{Fórmula}}

$\quad\;\,$Vejamos, finalmente, que se a categoria $\mathbf{D}$
possui coprodutos e co-equalizadores suficientes, então a extensão
de Kan à esquerda de todo $F:\mathbf{A}\rightarrow\mathbf{D}$ ao
longo de qualquer $\imath:\mathbf{A}\rightarrow\mathbf{C}$ existe
e pode ser representada em termos de um \emph{coend}. O fato fundamental
é que, se $\mathscr{L}F$ é a extensão procurada, então, para cada
$F'':\mathbf{C}\rightarrow\mathbf{D}$ há bijeções naturais 
\begin{equation}
\mathrm{Nat}(\mathscr{L}F;F'')\simeq\mathrm{Nat}(F;F''\circ\imath),\label{extensao_kan_esquerda}
\end{equation}
as quais determinam $\mathscr{L}F$: qualquer outro functor que as
induz é naturalmente isomorfo à $\mathscr{L}F$ e, portanto, uma extensão
de Kan para $F$.

Consideremos o functor $H:\mathbf{C}\times\mathbf{A}^{op}\times\mathbf{A}\rightarrow\mathbf{D}$,
definido em objetos por 
\[
H(X,Y,Z)=\bigoplus_{f\in\mathrm{Mor}_{\mathbf{C}}(\imath(Z);X)}F(Y)=\mathrm{Mor}_{\mathbf{C}}(\imath(Z);X)\cdot F(Y),
\]
e suponhamos que $\mathbf{D}$ admite coprodutos e co-equalizadores
suficientes para que o \emph{coend} de cada $H_{x}$ exista. Neste
caso, $H$ tem \emph{coend} com parâmetros, de modo que há um $E:\mathbf{C}\rightarrow\mathbf{D}$
tal que

\[
E(X)=\int^{\mathbf{C}}H_{x}=\int^{Y}\mathrm{Mor}_{\mathbf{C}}(\imath(Y);X)\cdot F(Y).
\]

Seguindo \cite{MACLANE_categories}, vamos mostrar que tal functor
satisfaz (\ref{extensao_kan_esquerda}) e, portanto, é extensão de
Kan à esquerda de $F$. Para tanto, dado um $F'':\mathbf{C}\rightarrow\mathbf{D}$
qualquer, observemos que 
\begin{eqnarray*}
\mathrm{Nat}(E;F'') & \simeq & \int_{X}\mathrm{Mor}_{\mathbf{D}}(E(X);F''(X))\\
 & \simeq & \int_{Y}\int_{X}\mathrm{Mor}_{\mathbf{Set}}(\mathrm{Mor}_{\mathbf{C}}(\imath(Y);X);\mathrm{Mor}_{\mathbf{D}}(F(Y);F''(X)))\\
 & \simeq & \int_{Y}\mathrm{Nat}(h^{\imath(Y)};\mathrm{Mor}_{\mathbf{D}}(F(Y);F''(-)))\\
 & \simeq & \mathrm{Nat}(F;F''\circ\imath),
\end{eqnarray*}
onde: na primeira passagem utilizamos do exemplo (3.3.1); na segunda
aplicamos as definições de $E$ e dos coprodutos, ao mesmo tempo que
utilizamos da continuidade do functor $h^{X}$, do teorema de Fubini
e também da identidade $\prod_{S}h^{X}(Y)\simeq h^{S}(h^{X}(Y))$;
na terceira fizemos uso do lema de Yoneda e, finalmente, aplicamos
duas vezes consecutivas o exemplo (3.3.1). 

\chapter{Álgebra Abstrata}

$\quad\;\,$Neste capítulo, estudamos as \emph{categorias monoidais}.
Elas são obtidas dos monoides da Álgebra Clássica por um processo
de ``categorificação'', no qual conjuntos são substituídos por categoriais,
elementos dão lugar à objetos, mapeamentos entre conjuntos são trocados
por functores, e relações entre mapeamentos tornam-se transformações
naturais entre os correspondentes functores. Assim, uma categoria
monoidal $\mathbf{C}$ seria aquela na qual se tem definido um bifunctor
$\otimes:\mathbf{C}\times\mathbf{C}\rightarrow\mathbf{C}$ que é associativo
e possui unidade a menos de isomorfismos naturais (o exemplo mais
simples é $\mathbf{Set}$ com $\otimes$ sendo o produto cartesiano). 

Apresentar o processo de categorificação acima descrito, bem como
dar uma definição precisa do que vem a ser uma estrutura monoidal,
ilustrando-a por meio de diversos exemplos, é o principal objetivo
da primeira secção. 

Na segunda secção, explicitamos o interesse nas categoriais monoidais:
é nelas que se pode falar de \emph{estruturas algébricas }de forma
mais genérica e abrangente. Com efeito, se na Álgebra Clássica uma
estrutura algébrica é composta de um conjunto $X$ no qual estão definidas
operações $*:X\times X\rightarrow X$ satisfazendo certas regras,
em categorias monoidais tais estruturas são objetos $X$, dotados
de morfismos $*:X\otimes X\rightarrow X$, os quais tornam comutativos
alguns diagramas. Assim, por exemplo, pode-se falar de ``monoides
em grupos abelianos'', ou mesmo de ``grupos em variedades diferenciáveis''.
Exemplos, estes, que não são meras abstrações: os monoides em grupos
nada mais são que os anéis, ao passo que os grupos em variedades nada
mais são que os grupos de Lie.

Observamos, no entanto, que existem estruturas na Álgebra Clássica
que são obtidas fazendo uma estrutura \emph{agir} em outra. Por exemplo,
módulos sobre um anel $R$ são simplesmente ações de $R$ em grupos
abelianos. A terceira secção do capítulo é marcada por uma generalização
do conceito de \emph{ação} de monóides em objetos no contexto das
categorias monoidais. Com ele em mãos, finalizamos com o estudo de
fibrados cujas fibras estão sujeitas à ação de um grupo. 

Maiores detalhes sobre o assunto podem ser encontrados em \cite{Aguiar_monoidal_categories,MACLANE_categories},
que constituíram as principais referências durante o estudo e a escrita
deste capítulo.

\section{Categorificação}

$\quad\;\,$A matemática clássica é construída sob a linguagem da
Teoria dos Conjuntos. A Teoria das Categorias, por sua vez, tem a
Teoria dos Conjuntos como um caso particular. Espera-se, portanto,
que a Teoria das Categorias forneça procedimentos que nos permita
abstrair qualquer que seja o conceito usual. Um de tais procedimentos
é a \emph{categorificação}, o qual passamos a descrever.

Relembramos que o problema de classificar uma dada categoria $\mathbf{C}$
consiste em obter uma bijeção entre o conjunto das classes de isomorfismo
$\mathrm{Iso}(\mathbf{C})$ e algum outro conjunto $S$. Dualmente,
fornecido um conjunto $S$, o problema de \emph{categorificá-lo} consiste
em obter uma categoria $\mathbf{C}$ cuja classe de isomorfismos esteja
em bijeção com $S$. Em outras palavras, categorificar um conjunto
$S$ é buscar por uma categoria $\mathbf{C}$ que pode ser classificada
em termos de $S$. Por este motivo, algumas vezes se fala que classificar
é o mesmo que \emph{descategorificar}.
\begin{example}
Sabe-se a categoria dos espaços vetoriais de dimensão finita é classificada
pelo invariante dimensão (isto é, pelo functor $\mathrm{dim}:\mathrm{F}\mathbf{Vec}_{\mathbb{K}}\rightarrow\mathbb{N}$).
Desta forma, pode-se pensar na categoria $\mathrm{F}\mathbf{Vec}_{\mathbb{K}}$
como sendo uma categorificação dos naturais.
\end{example}
Observamos que categorias são entidades mais complicadas que conjuntos:
em uma categoria tem-se duas classes (a dos objetos e a dos morfismos),
ao passo que num conjunto tem-se apenas uma. Posto isso, o problema
de categorificação tende a ser mais difícil que o de classificação:
para classificar, procura-se um conjunto (entidade simples) que represente
uma categoria (entidade complexa) previamente conhecida. Para categorificar,
precisa-se determinar uma categoria (entidade complicada) que seja
representada por um dado conjunto (entidade simples). 

Grosso modo, classificar é estudar o functor $\mathrm{Iso}:\mathbf{Cat}\rightarrow\mathbf{Set}$,
ao passo que categorificar é encontrar functores $\mathrm{Ctgz}:\mathbf{Set}\rightarrow\mathbf{Cat}$
que são ``inversas pontuais'' de $\mathrm{Iso}$. Assim, ao se classificar,
perde-se ``informação categórica'': com o objetivo de entender a
categoria, ela é substituída por um ente mais simples. Por sua vez,
ao se categorificar, ganha-se tal informação. É exatamente este o
espírito da categorificação: partir de uma entidade conhecida e substituí-la
por outra mais complexa, ganhando novas ferramentas. Por exemplo,
pode-se partir de um invariante simples e, ao final do processo, obter
outro mais poderoso. Sobre o assunto, remetemos o leitor a um artigo
bastante interessante de John Baez e James Dolan: \cite{categorification}.

\subsection*{\uline{Categorias Monoidais}}

$\quad\;\,$Para categorificar uma estrutura algébrica $S$, deve-se
obter uma categoria $\mathbf{C}$ cujo conjunto das classes de isomorfismo
admite uma estrutura isomorfa à $S$. Assim, por exemplo, se $S$
for um monóide, $\mathbf{C}$ poderá categorificá-lo somente se $\mathrm{Iso}(\mathbf{C})$
também for um monóide. Portanto, se quisermos saber quais categorias
classificam um determinada estrutura algébrica, deveremos restringir
nossa busca às categorias cujo conjunto das classes de isomorfismo
possuem correspondente estrutura. Em outras palavras, dada uma subcategoria
$\mathbf{Alg}\subset\mathbf{Set}$, devemos procurar por subcategorias
$\mathbf{Cat}_{\mathbf{Alg}}\subset\mathbf{Cat}$ restritas as quais
o functor $\mathrm{Iso}$ assume valores em $\mathbf{Alg}$.

Um \emph{insight} para determinar $\mathbf{Cat}_{\mathbf{Alg}}$ é
o seguinte: estruturas algébricas são conjuntos dotados de elementos
distinguidos (elementos neutros) e de operações binárias (que são
funções entre conjuntos), as quais satisfazem certas equações. As
entidades primárias são os elementos, os quais originam os conjuntos
e as funções, as quais se relacionam por meio das equações. Por sua
vez, em $\mathbf{Cat}$ as entidades primárias são os objetos e os
morfismos, os quais originam as categorias e os functores, os quais
se relacionam por meio das transformações naturais. 

Assim, partindo de $\mathbf{Alg}$, para obter $\mathbf{Cat}_{\mathbf{Alg}}$,
a ideia é ``elevar em um grau a informação categórica'': elementos
distinguidos tornam-se objetos distinguidos, conjuntos tornam-se categorias,
operações binárias tornam-se bifunctores e equações são reformuladas
por meio de isomorfismos naturais, chamados de \emph{condições de
coerência}.
\begin{example}
Na álgebra clássica, um monóide é conjunto $X$ dotado de uma única
operação binária $*:X\times X\rightarrow X$, e de um elemento distinguido
$1\in X$, tais que as seguintes equações são satisfeitas para cada
$x,y,z\in X$:
\begin{equation}
1*x=x=x*1\quad\mbox{e}\quad(x*y)*z=x*(y*z).\label{monoide_algebra_classica}
\end{equation}
Seguindo o \emph{insight }anterior, uma categoria em $\mathbf{Cat}_{\mathbf{Mon}}$
seria aquela $\mathbf{C}$ na qual se encontram definidos um bifunctor
$\otimes:\mathbf{C}\times\mathbf{C}\rightarrow\mathbf{C}$ e um objeto
distinguido $1\in\mathbf{C}$, para os quais se tem os seguintes isomorfismos
naturais:
\begin{equation}
1\otimes X\simeq X\simeq X\otimes1\quad\mbox{e}\quad(X\otimes Y)\otimes Z\simeq X\otimes(Y\otimes Z).\label{categoria_monoidal}
\end{equation}
Se $\mathbf{C}$ cumpre tais relações, então $\mathrm{Iso}(\mathbf{C})$
é realmente um monóide, mostrando-nos que o \emph{insight} cumpriu
com o seu papel. Com efeito, basta pôr $[X]*[Y]=[X\otimes Y]$. O
elemento neutro haverá de ser $[1]$. Observamos, no entanto, que
as relações (\ref{monoide_algebra_classica}) nos permitem retirar
os parênteses de qualquer expressão envolvendo um número arbitrário
de elementos. Por exemplo, 
\[
(x*y)*(x'*y')=x*((y*x)*y')=x*(y*(x'*y'))).
\]
Em contrapartida, as relações (\ref{categoria_monoidal}) não são
suficientes para garantir algo análogo para a categoria obtida através
do \emph{insight}. São necessárias, pois, mais algumas condições de
coerência (veja a segunda secção do capítulo VII de \cite{MACLANE_categories}).
Uma categoria que cumpre com o \emph{insight} e também com estas condições
de coerência adicionais chama-se \emph{monoidal}. Assim, pode-se dizer
que \emph{as categorias monoidais são categorificações convenientes
do conceito clássico de monóide}.

\begin{example}
Se uma categoria $\mathbf{C}$ possui produtos binários $X\times Y$
e um objeto terminal $*$, então o bifunctor $(X,Y)\mapsto X\times Y$
está bem definido e ali introduz uma estrutura monoidal, cujo objeto
neutro nada mais é que o próprio $*$. Dualmente, se $\mathbf{C}$
tem coprodutos binários $X\oplus Y$ e um objeto inicial $\varnothing$,
então $(X,Y)\mapsto X\oplus Y$ a faz monoidal, cujo objeto neutro
é $\varnothing$. Em particular, se $\mathbf{C}$ possui biprodutos,
então as estruturas monoidais induzidas por seus produtos e coprodutos
são equivalentes. Isto ocorre, por exemplo, em $\mathbf{Mod}_{R}$.
\end{example}
\end{example}
Nem sempre a estrutura monoidal provém de produtos e coprodutos. Abaixo
ilustramos este fato por meio de três exemplos. Nas próximas subsecções,
veremos que, ainda que a estrutura monoidal de $\mathbf{C}$ provenha
de produtos, sua pontuação $\mathbf{C}_{*}$ pode admitir uma estrutura
que não provém de produtos lá definidos. 
\begin{example}
Na categoria $\mathbf{Mod}_{R}$, com $R$ comutativo, consideremos
a regra $\otimes_{R}$ que a cada par de módulos $(X,Y)$ associa
o produto tensorial $X\otimes_{R}Y$ entre eles, e que a cada par
$(f,g)$ de homomorfismos faz corresponder $f\otimes_{R}g$. Ela é
associativa e possui unidade dada pelo próprio anel $R$, visto enquanto
módulo sobre si mesmo. Assim, $\otimes_{R}$ torna $\mathbf{Mod}_{R}$
uma categoria monoidal sem provir de produtos ou coprodutos.

\begin{example}
Tem-se um bifunctor $\otimes:\partial\mathbf{Diff}\times\partial\mathbf{Diff}\rightarrow\mathbf{Diff}$,
que toma variedades $X$ e $Y$ com bordo de devolve uma nova variedade
$X\otimes Y$, esta sem bordo, obtida colando $X$ e $Y$ ao longo
de seus bordos. Observamos que a subcategoria cheia $\mathbf{C}^{n}\subset\mathbf{Diff}$,
formada da variedades de dimensão $n$ que são compactas e orientáveis,
torna-se monoidal quando dotada do bifunctor $\#:\mathbf{C}^{n}\times\mathbf{C}^{n}\rightarrow\mathbf{C}^{n}$,
chamado de \emph{soma conexa}, e definido como segue: toma-se duas
variedades de mesma dimensão, retira-se um pequeno subespaço homeomorfo
ao disco de cada uma delas, e então aplica-se $\otimes$ (isto é,
cola-se os espaços resultantes ao longo de seus bordos). O objeto
neutro de tal estrutura é a esfera $\mathbb{S}^{n}$.

\begin{example}
Relembramos que, para cada inteiro $n$, a categoria $n\mathbf{Cob}$
tem como objeto as variedades de dimensão $n-1$, compactas e sem
bordo, e como morfismos os cobordismos. Como estamos trabalhando com
dimensão fixa, o coproduto em $\mathbf{Diff}$ está bem definido e
é a reunião disjunta. O bifunctor $\sqcup$ induz uma estrutura monoidal
em $n\mathbf{Cob}$.
\end{example}
\end{example}
\end{example}
De maneira estritamente análoga ao que foi feito no exemplo 4.4, poder-se-ia
aplicar o \emph{insight} de modo a obter categorificações de outras
estruturas algébricas mais complicadas. Por exemplo, a categorificação
dos monóides abelianos são as categorias monoidais $\mathbf{C}$,
com produto $\otimes:\mathbf{C}\times\mathbf{C}\rightarrow\mathbf{C}$,
para as quais existem isomorfismos naturais $X\otimes Y\simeq Y\otimes X$,
denominados \emph{braidings}, satisfazendo certas condições de coerência.
Tais categorias são ditas \emph{simétricas}. 
\begin{example}
Estruturas monoidais provenientes de produtos e coprodutos binários
são sempre simétricas. Particularmente, a soma direta em $\mathbf{Mod}_{R}$
possui inversa $X\oplus(-X)\simeq0$, em que $-X$ é o conjunto de
todo $-x$, com $x\in X$. Desta forma, $\mathbf{Mod}_{R}$ é categorização
de grupo abeliano com $\oplus$. Por sua vez, sabe-se que o produto
tensorial é comutativo a menos de isomorfismos, de modo que a estrutura
definida por $\otimes_{R}$ em $\mathbf{Mod}_{R}$ também é simétrica.
Particularmente, $\oplus$ e $\otimes_{R}$ são distributivos módulo
isomorfismos naturais. Portanto, com tais bifunctores, $\mathbf{Mod}_{R}$
é categorização de um anel comutativo com unidade.
\end{example}
Ao longo do texto, nos restringimos, ao estudo das categorias monoidais
e de suas versões simétricas. Isto porque estaremos interessados no
estudo da Álgebra Abstrata: para falar de estruturas algébricas usuais,
precisa-se falar somente de operações e objetos distinguidos. As,
operações são correspondências $X\times X\rightarrow X$, de modo
que também se faz necessária uma noção prévia de produto entre conjuntos
(trata-se do produto cartesiano). 

As categorias monoidais cumprem todas essas requisições e, portanto,
são as categorias de menor complexidade nas quais a Álgebra pode ser
desenvolvida. Mais precisamente, se $\mathbf{C}$ é monoidal (digamos
com respeito a $\otimes$), então faz sentido falar de operações em
$X\in\mathbf{C}$: tratam-se, pois, de morfismos $X\otimes X\rightarrow X$.
Assim, pode-se dizer que a Álgebra clássica é a Álgebra desenvolvida
na categoria $\mathbf{Set}$, com estrutura monoidal proveniente de
produtos binários.

\subsection*{\emph{\uline{Smash}}}

$\quad\;\,$Seja $\mathbf{C}$ uma categoria completa e cocompleta,
com objeto terminal $*$, dotada da estrutura monoidal $\otimes:\mathbf{C}\times\mathbf{C}\rightarrow\mathbf{C}$.
Esta induz um bifunctor natural $\wedge$ em sua pontuação $\mathbf{C}_{*}$,
denominado \emph{produto smash} e definido como segue: em objetos,
ele toma pares $(X,x_{o})$ e $(Y,y_{o})$ e devolve o \emph{pushout}
abaixo, pontuado pelo único $x_{o}\wedge y_{o}:*\rightarrow X\wedge Y$.
Por sua vez, a cada par de morfismos $f:X\rightarrow Y$ e $g:X'\rightarrow Y'$
ele associa o correspondente $f\wedge g$, obtido por universalidade.$$
\xymatrix{&&& X'\wedge Y' & X'\otimes Y' \ar[l] \\
X\wedge Y & X\otimes Y \ar[l] && \mathrm{*} \ar[u] & X\wedge Y \ar@{-->}[ul]^{f\wedge g} & \ar [l] X\otimes Y \ar@/_/[lu]_{f\otimes g} \\
\mathrm{*} \ar[u] & \ar[l] \ar[u] X\oplus Y &&& \mathrm{*} \ar[u] \ar@/^/[lu] & \ar[l] \ar[u] X\oplus Y}
$$
\begin{example}
Elementos de $\mathbf{Set}_{*}$ se identificam com pares $(X,x_{o})$,
em que $x_{o}\in X$, de modo que $X\times*\sqcup Y\times*\simeq X\vee Y$.
Por sua vez, o mapa $X\sqcup Y\rightarrow X\times Y$ é visto como
inclusão e o produto \emph{smash} $X\wedge Y$ é o quociente de $X\times Y$
pela relação que identifica $X\vee Y$. O mesmo se passa em $\mathbf{Top}_{*}$.
\end{example}
Observamos que, em qualquer situação, o produto $\wedge$ possui uma
unidade natural: trata-se do coproduto $*\oplus*$, denotado por $S^{0}$.
Por exemplo, em $\mathbf{Set}$ ou $\mathbf{Top}$, tem-se $S^{0}=\mathbb{S}^{0}$.
Desta forma, se $\wedge$ for associativo a menos de isomorfismos
naturais, então definirá uma estrutura monoidal na pontuação $\mathbf{C}_{*}$. 

Uma vez que o produto \emph{smash} é \emph{pushout} em $\mathbf{C}$,
pelas condições de coerência, para que valha tal associatividade,
basta que colimites finitos sejam preservados por cada $-\otimes Y$.
Isto ocorre em, particular, se $\mathbf{C}$ é\emph{ fechada}. Ou
seja, se cada um dos $-\otimes Y$ possui adjuntos à direita. Afinal,
neste caso, além de preservarem os colimites que são finitos, preservarão
qualquer outro. Uma categoria fechada com respeito à estrutura monoidal
proveniente de produtos binários costuma ser chamada de \emph{cartesianamente
fechada}.
\begin{example}
Em $\mathbf{Set}$, tem-se $h^{Y}$ como adjuntos de $-\times Y$.
Assim, tal categoria é cartesianamente fechada e, consequentemente,
o produto \emph{smash} $\wedge$ apresentado no exemplo anterior faz
de $\mathbf{Set}_{*}$ uma categoria monoidal. A mesma estratégia
não pode ser empregada em $\mathbf{Top}$: quando restrita a tal subcategoria,
o bifunctor $\wedge$ perde sua associatividade. Um contra-exemplo
é apresentado na secção 1.7 de \cite{MAY_2}. Este é um fato fundamental
no estudo da topologia e será mais detalhadamente discutido no capítulo
oito.

\begin{example}
Evidentemente, com respeito à estrutura monoidal induzida pelo produto
tensorial, a categoria dos módulos é fechada. Por conta disso, uma
categoria fechada é algumas vezes chamada de \emph{tensorial}.
\end{example}
\end{example}
Como pode ser conferido em \cite{categorical_homotopy}, quando $\mathbf{C}$
é fechada, os adjuntos de $-\otimes Y$, digamos dados por $(-)^{Y}$,
determinam adjuntos $(-)_{*}^{Y}$ em $\mathbf{C}_{*}$ para os correspondentes
$-\wedge Y$. Para cada $X\in\mathbf{C}_{*}$, estes são constituídos
dos \emph{pullbacks} abaixo apresentados.

$$
\xymatrix{\mathrm{Mor_{*}}(X;Y) \ar[d] \ar[r] & \mathrm{*} \ar[d] \\
\mathrm{Mor}(X;Y) \ar[r] & \mathrm{Mor}(*;Y). }
$$

\section{Monoides}

$\quad\;\,$Um \emph{monoide }numa categoria monoidal $\mathbf{C}$,
com produto $\otimes:\mathbf{C}\times\mathbf{C}\rightarrow\mathbf{C}$,
é a generalização do conceito usual de monoide presente na álgebra
clássica. Trata-se, pois, de um objeto $X\in\mathbf{C}$ para o qual
existem morfismos $*:X\otimes X\rightarrow X$ e $e:1\rightarrow X$,
obtidos de tal maneira que os diagramas abaixo se tornam comutativos.
O primeiro deles traduz a ``associatividade'' de $*$, enquanto
que o segundo expressa a existência de um ``elemento neutro'' em
$X$ (compare com os diagramas apresentados no início do primeiro
capítulo).$$
\xymatrix{\ar[d]_-{id \otimes *} X\otimes (X\otimes X) \ar[r] & (X\otimes X)\otimes X \ar[r]^-{*\otimes id} & X\otimes X \ar[d]^{*} & \ar[rd] 1\otimes X \ar[r]^-{e\otimes id} & X\otimes X \ar[d]^{*} &  X \ar[ld] \otimes 1 \ar[l]_-{id \otimes e} \\
X \otimes X \ar[rr]_{*} && X && X }
$$

Tem-se uma categoria $\mathrm{Mon}(\mathbf{C},\otimes)$ formada dos
monoides de $\mathbf{C}$ segundo $\otimes$ (quando não há risco
de confusão quanto ao bifunctor fixado, ele é omitido da notação).
Em tal categoria, um morfismo entre $(X,*,e)$ e $(X',*',e')$ é simplesmente
um morfismo $f:X\rightarrow X'$ que preserva $*$ e $e$. Isto é,
que satisfaz as igualdades $f\circ e=e'$ e $f\circ*=*'\circ(f\otimes f)$.
Um \emph{comonoide} em $\mathbf{C}$ é simplesmente um monoide em
$\mathbf{C}^{op}$. Mais precisamente, estes se tratam dos objetos
da categoria $\mathrm{Mon}(\mathbf{C}^{op},\otimes)^{op}$, denotada
por $\mathrm{Comon}(\mathbf{C};\otimes)$.

Numa categoria monoidal simétrica, diz-se que um monoide é \emph{comutativo}
ou \emph{abeliano} quando sua multiplicação comuta com o \emph{braiding}
$X\otimes X\simeq X\otimes X$. Invertendo setas obtém-se a correspondente
noção de comonoides cocomutativos. Tais entidades definem categorias
$\mathrm{_{\mbox{\ensuremath{c}}}Mon}(\mathbf{C},\otimes)$ e $_{c}\mathrm{Comon}(\mathbf{C};\otimes)$.
\begin{example}
Toda categoria monoidal $\mathbf{C}$, com produto $\otimes$ e objeto
neutro $1$ admite um monoide trivial: trata-se do próprio $1$, dotado
da multiplicação dada pelo isomorfismo $1\otimes1\simeq1$, com $e:1\rightarrow1$
sendo a identidade. Consequentemente, $1$ é também um comonoide em
$\mathbf{C}$: a comultiplicação é a inversa $1\simeq1\otimes1$.
Observamos que, quando $\mathbf{C}$ é simétrica, tal monoide/comonoide
é evidentemente abeliano. Por sua vez, se $\otimes$ provém de produtos
(resp. coprodutos) binários, então $1$ há de ser objeto inicial (resp.
final) em $\mathrm{Mon}(\mathbf{C},\times)$ (resp. $\mathrm{Comon}(\mathbf{C};\otimes)$).

\begin{example}
Na categoria $\times:\mathbf{Set}\times\mathbf{Set}\rightarrow\mathbf{Set}$,
os monoides coincidem com os monoides usuais da álgebra clássica.
Por sua vez, para qualquer que seja o anel comutativo $R$, os monoides
de $\mathbf{Mod}_{R}$ são, relativamente ao produto $\otimes_{R}$,
as álgebras sobre $R$. Afinal, dar um morfismo $X\otimes_{R}X\rightarrow X$
é o mesmo que dar uma aplicação bilinear $X\times X\rightarrow X$.
Assim, em particular, os monoides de $\mathbf{AbGrp}\simeq\mathbf{Mod}_{\mathbb{Z}}$
são os anéis.

\begin{example}
Vimos que a soma conexa introduz uma estrutura monoidal na categoria
das superfícies compactas e orientáveis. Assim, $\mathrm{Iso}(\mathbf{C}^{n})$
é monoide, tendo $[\mathbb{S}^{n}]$ como elemento neutro. Para cada
$g\in\mathbb{N}$, seja $\mathbb{S}_{g}^{1}=\mathbb{S}^{1}\times...\times\mathbb{S}^{1}$
com $\mathbb{S}_{0}^{1}=\mathbb{S}^{1}$. A correspondência $g\mapsto\mathbb{S}_{g}^{1}$
passa ao quociente e define um morfismo entre o monóide aditivo dos
números naturais e $\mathrm{Iso}(\mathbf{C}^{2})$. Tal morfismo é,
em verdade, um isomorfismo. Isto significa que toda superfície compacta
e orientável de $\mathbb{R}^{3}$ é classificada, módulo homeomorfismos,
pelo seu número de buracos. Tem-se um resultado análogo no caso não-orientável,
no qual o espaço projetivo $\mathbb{P}^{1}$ (obtido identificando
pares antípodas no círculo) substitui $\mathbb{S}^{1}$. Veja o capítulo
final de \cite{Hirsch}.

\begin{example}
Numa categoria monoidal $\otimes:\mathbf{C}\times\mathbf{C}\rightarrow\mathbf{C}$
definida por produtos binários, todo objeto $X$ admite uma estrutura
única de comonoide, com comultiplicação dada pelo mapa diagonal $\Delta_{X}:X\rightarrow X\otimes X$.
Por sua vez, já que o objeto neutro $1\in\mathbf{C}$ é terminal,
para cada $X$ existe um único morfismo $!:X\rightarrow1$, o qual
haverá de ser a counidade de $X$. Dualmente, se $\otimes$ provém
de coprodutos binários, então todo $X$ possui única estrutura de
monoide. Sua multiplicação é o mapa $\nabla_{X}:X\otimes X\rightarrow X$,
ao passo que sua unidade é o morfismos $!:1\rightarrow X$, obtido
do fato de $1$ ser inicial.

\begin{example}
Em contrapartida ao exemplo anterior, ainda que $\otimes:\mathbf{C}\times\mathbf{C}\rightarrow\mathbf{C}$
seja definido por coprodutos, pode ser que ali só existam comonoides
triviais. Por exemplo, o único objeto inicial de $\mathbf{Set}$ é
o conjunto vazio, de modo que existe um morfismo $1\rightarrow X$
se, e somente se, $X=\varnothing$. Assim, relativamente à estrutura
monoidal proveniente de coprodutos (diga-se reuniões disjuntas), o
único co-monóide em $\mathbf{Set}$ é $\varnothing$. Situação análoga
se passa com $\mathbf{Top}$.
\end{example}
\end{example}
\end{example}
\end{example}
\end{example}
Um functor $F:\mathbf{C}\rightarrow\mathbf{C}'$ entre categorias
monoidais pode não levar monoides em monoides e nem mesmo comonoides
em comonoides. Quando ele preserva monoides (resp. comonoides), diz-se
que ele é \emph{monoidal} (resp. \emph{comonoidal})\emph{. }Se este
é o caso, são induzidos functores 
\begin{eqnarray*}
F_{*}:\mathrm{Mon}(\mathbf{C},\otimes) & \rightarrow & \mathrm{Mon}(\mathbf{C}',\otimes')\\
F^{*}:\mathrm{CoMon}(\mathbf{C},\otimes) & \rightarrow & \mathrm{CoMon}(\mathbf{C}',\otimes').
\end{eqnarray*}

Observamos que, para $F$ ser monoidal, basta que mapeie produtos
em produtos e também objeto neutro em objeto neutro. Isto é, basta
existirem um morfismo $f:1'\rightarrow F(1)$ e transformações $\phi_{xy}:F(X)\otimes'F(Y)\rightarrow F(X\otimes Y)$
satisfazendo certas condições de coerência, as quais são descritas
em termos da comutatividade de diagramas semelhantes àqueles satisfeitos
por $*$ e $e$. Um functor monoidal para o qual $f$ e cada $\phi_{xy}$
são isomorfismos chama-se \emph{fortemente monoidal.}

Da mesma forma, para $F$ ser comonoidal, é suficiente que existam
$g:1'\rightarrow F(1)$ e também transformações $\varphi_{xy}:F(X\otimes Y)\rightarrow F(X)\otimes'F(Y)$
satisfazendo condições compatibilidade. Quando $g$ e $\varphi_{xy}$
são isomorfismos, fala-se que $F$ é \emph{fortemente comonoidal}.

Tem-se subcategorias $\mathbf{Mnd}$, $\mathrm{Co}\mathbf{Mnd}$,
$\mathrm{S}\mathbf{Mnd}$ e $\mathrm{SCo}\mathbf{Mnd}$ de $\mathbf{Cat}$,
cujos objetos são categorias monoidais, e cujos morfismos são, respectivamente,
functores monoidais, comonoidais, fortemente monoidais e fortemente
comonoidais. 
\begin{example}
Seja $\mathbf{D}\subset\mathbf{C}$ uma subcategoria tal que a inclusão
$\imath:\mathbf{D}\rightarrow\mathbf{C}$ preserva produtos e objeto
terminal. Isto ocorre, por exemplo, quando \textbf{$\mathbf{D}$ }é
livremente gerada por objetos de $\mathbf{C}$ (isto é, quando tal
que a inclusão possui um adjunto à esquerda). Se este é o caso, então,
relativamente às estruturas monoidais definidas por produtos binários
em $\mathbf{C}$ e em $\mathbf{D}$, a inclusão é fortemente monoidal.
Em particular, ela também é fortemente comonoidal, pois os únicos
comonoides existentes em ambas categorias são aqueles definidos pelo
mapa diagonal. Dualmente, se $\imath$ preserva coprodutos e objeto
inicial (por exemplo, quando admite adjunto à direita), então, relativamente
à estrutura monoidal definida por coprodutos binários, é fortemente
comonoidal e fortemente monoidal.

\begin{example}
Diz-se que um functor fortemente monoidal simétrico $F:n\mathbf{Cob}\rightarrow\mathbf{Vec}_{\mathbb{C}}$
define uma \emph{teoria quântica de campos topológica}. Assim, uma
tal teoria é uma regra simétrica, que a cada variedade compacta sem
bordo $X$ associa um espaço vetorial $F(X)$, e que a cada cobordismo
$\Sigma:X\rightarrow Y$ faz corresponder uma transformação linear
$F(\Sigma)$ entre os espaços $F(X)$ e $F(Y)$, de tal maneira que
$F(\varnothing)\simeq\mathbb{C}$, com $X\sqcup Y$ sendo mandada
em $F(X)\otimes F(Y)$. Estes functores são de grande interesse em
Física e também em Matemática.
\end{example}
\end{example}

\subsection*{\uline{Bimonóides}}

$\quad\;\,$Toda categoria monoidal simétrica $\mathbf{C}$ induz
uma estrutura monoidal em $\mathrm{Mon}(\mathbf{C},\otimes)$. Esta
é definida pelo bifunctor 
\[
\otimes_{\mathrm{M}}:\mathrm{Mon}(\mathbf{C},\otimes)\times\mathrm{Mon}(\mathbf{C},\otimes)\rightarrow\mathrm{Mon}(\mathbf{C},\otimes),
\]
que toma dois monoides $(X,*,e)$ e $(X',*',e')$ e devolve o respectivo
monoide $X\otimes X'$, tendo $e\otimes e'$ como unidade e cuja multiplicação
$\#$ é obtida compondo as setas do diagrama abaixo, em que a primeira
delas provém de associatividade de $\otimes$ e de \emph{braidings}.
$$
\xymatrix{(X\otimes X')\otimes(X\otimes X') \ar[r]^{\simeq} & (X\otimes X)\otimes(X'\otimes X') \ar[r]^-{*\otimes *'} & X\otimes X'}
$$

O objeto neutro de $\mathrm{Mon}(\mathbf{C},\otimes)$ é o monoide
formado pelo objeto neutro de $\mathbf{C}$, tendo unidade dada por
$id_{1}$ e multiplicação fornecida pelo isomorfismo $1\otimes1\simeq1$.
Por dualidade, a estrutura monoidal de $\otimes$ fixa outra em $\mathrm{Mon}(\mathbf{C}^{op},\otimes)^{op}$.
Tem-se equivalências
\begin{eqnarray*}
\mathrm{Mon}(\mathrm{Mon}(\mathbf{C},\otimes),\otimes_{\mathrm{M}}) & \simeq & \mathrm{_{\mbox{\ensuremath{c}}}Mon}(\mathbf{C},\otimes)\\
\mathrm{Comon}(\mathrm{Comon}(\mathbf{C},\otimes),\otimes_{\mathrm{M}}) & \simeq & \mathrm{_{\mbox{\ensuremath{c}}}Comon}(\mathbf{C},\otimes)
\end{eqnarray*}
as quais são usualmente chamadas de \emph{princípio de} \emph{Eckmann-Hilton}
(veja o artigo original \cite{eckmann_hilton_argument} ou a secção
1.2.7 de \cite{Aguiar_monoidal_categories}). Por sua vez, dar um
comonoide na categoria dos monoides é o mesmo que fornecer um monoide
na categoria dos comonoides. Isto reflete a existência de isomorfismos

\[
\mathrm{Comon}(\mathrm{Mon}(\mathbf{C},\otimes),\otimes_{\mathrm{M}})\simeq\mathrm{Mon}(\mathrm{Comon}(\mathbf{C},\otimes),\otimes_{\mathrm{M}}).
\]

Os objetos de ambas as categorias (aqui indistinguivelmente denotadas
por $\mathrm{Bimon}(\mathbf{C};\otimes)$), são chamados de \emph{bimonoides}
de $\mathbf{C}$. Assim, um bimonoide numa categoria monoidal $\mathbf{C}$
é um objeto que possui simultaneamente estruturas de monoide e de
comonoide, tais que sua multiplicação e sua unidade são morfismos
de comonoides, ao mesmo tempo que sua comultiplicação e counidade
são morfismos de monoides.
\begin{example}
Se $\mathbf{C}$ é definida por produtos binários, então todo objeto
possui uma estrutura natural de comonoide. Assim, os bimonoides de
$\mathbf{C}$ são seus próprios monoides. Dualmente, se a estrutura
monoidal de $\mathbf{C}$ provém de coprodutos binários, então seus
objetos são sempre monoides. Portanto, seus bimonoides nada mais são
que seus comonoides. Particularmente, se $\mathbf{C}$ têm biprodutos,
então, na estrutura monoidal por eles definida, qualquer objeto admite
uma única estrutura de bimonoide.
\end{example}

\subsection*{\uline{Grupos}}

$\quad\;\,$Um \emph{monoide de Hopf} em $\otimes:\mathbf{C}\times\mathbf{C}\rightarrow\mathbf{C}$
é um bimonoide $X$, o qual admite ``inversos''. Isto se traduz
na existência de um morfismo $inv:X\rightarrow X$, chamado de \emph{antípoda
}ou \emph{inversão}, tal que o diagrama abaixo é comutativo. Nele,
$*$ e $*'$ denotam a multiplicação e a comulitiplicação de $X$,
ao passo que $e$ e $e'$ representam a sua unidade e a sua counidade
(compare, mais uma vez, com os diagramas da secção inicial do primeiro
capítulo). $$
\xymatrix{\ar[d]_{*'} X \ar[r]^{e'} & 1 \ar[r]^{e} & X \\
X\otimes X \ar[rr]_{id\otimes inv} && X\otimes X \ar[u]_{*} }
$$

Tem-se subcategoria cheia $\mathrm{Hopf}(\mathbf{C};\otimes)\subset\mathrm{Bimon}(\mathbf{C};\otimes)$
dos monoides de Hopf.
\begin{example}
Os monoides de Hopf da categoria $\mathbf{Mod}_{R}$, com estrutura
monoidal fixada pelo produto tensorial, são usualmente chamados de
\emph{álgebras de Hopf }sobre $R$. Exemplos são os grupos de homologia
singular de um $H$-espaço e a álgebra de Steenrod. Veja, por exemplo,
\cite{whitehead_homotopy,Spanier}.
\end{example}
Observamos que, se a estrutura monoidal de $\mathbf{C}$ provém de
produtos binários, então dar um monoide de Hopf em $\mathbf{C}$ é
o mesmo que dar um monoide $X$ dotado de uma inversão (afinal, sob
tal hipótese, todo monoide é bimonoide). Em outras palavras, em tal
situação, um monoide de Hopf é simplesmente um objeto no qual se têm
uma multiplicação associativa e possuidora de unidade, ao mesmo tempo
que se sabe inverter. Por este motivo, costuma-se dizer que $X$ é
\emph{grupo} em $\mathbf{C}$.

Dualmente, quando $\otimes:\mathbf{C}\times\mathbf{C}\rightarrow\mathbf{C}$
é definido através de coprodutos, então os monoides de Hopf em $\mathbf{C}$
nada mais são que os comonoides adicionados de um mapa antípoda. Isto
nos leva a chamá-los de \emph{cogrupos} em $\mathbf{C}$. 
\begin{example}
Os grupos de $\mathbf{Set}$, com $\times$, são os grupos usuais
da álgebra clássica. Por sua vez, o único cogrupo de tal categoria
é o conjunto vazio. Afinal, este é seu único comonoide.

\begin{example}
Um grupo em $\mathbf{Top}$ é um espaço topológico no qual está definida
uma aplicação contínua $*:X\times X\rightarrow X$ que o torna um
grupo, de tal modo que a inversão $inv(x)=x^{-1}$ também é uma função
contínua. Tais entidades são conhecidas como \emph{grupos topológicos}.
Este é o caso das esferas $\mathbb{S}^{1}$ e $\mathbb{S}^{3}$: basta
considerá-las enquanto subconjuntos de $\mathbb{C}$ e de $\mathbb{H}$,
dotando-as das respectivas multiplicações de números complexos e de
números quaterniônicos.

\begin{example}
Se $X$ é grupo topológico com multiplicação $*:X\times X\rightarrow X$
e unidade $1$, então, relativamente ao produto cartesiano pontuado,
a mesma multiplicação faz de $(X,1)$ um grupo em $\mathbf{Top}_{*}$.
Por sua vez, se $(X,x_{o})$ é um grupo em $\mathbf{Top}_{*}$ com
produto $*$, então este se torna um cogrupo em $\mathbf{Top}_{*}$
quando dotado da regra $f:(X,x_{o})\rightarrow X\vee X$, tal que
$f(x)=[x*x]$.

\begin{example}
Os grupos de $\mathbf{Diff}$ (onde estamos nos resumindo às entidades
sem bordo, para que o produto entre eles esteja bem definido) são
simplesmente variedades $X$, usualmente chamadas de \emph{grupos
de Lie}, as quais se encontram dotadas de uma estrutura adicional
de grupo, tal que a multiplicação e a correspondente inversão são
ambas diferenciáveis. Escreve-se $\mathbf{GLie}$ ao invés de $\mathrm{Hopf}(\mathbf{Diff};\times)$
para denotar a categoria dos grupos de Lie. Observamos que um morfismo
em tal categoria há de ser um mapa que preserva tanto a estrutura
de grupo quando a estrutura de variedades subjacente. Assim, estes
são os \emph{homomorfismos diferenciáveis}. 
\end{example}
\end{example}
\end{example}
\end{example}
Os grupos de uma categoria monoidal admitem a seguinte caracterização
(os cogrupos possuem caracterização estritamente dual):
\begin{prop}
Um objeto $X$ é grupo em $\mathbf{C}$ se, e só se, $h_{X}$ é grupo
em $\mathrm{Func}(\mathbf{C}^{op};\mathbf{Set})$.
\end{prop}
\begin{proof}
Sendo $\mathbf{Set}$ completa, $\mathrm{Func}(\mathbf{C}^{op};\mathbf{Set})$
também o é e, portanto, possui produtos finitos. Por outro lado, como
$h_{X}$ preserva limites, vale $h_{X}\times h_{X}\simeq h_{X\times X}$.
Desta forma, todo produto $*:X\otimes X\rightarrow X$ fixa multiplicação
em $h_{X}$. Trata-se, pois, da composição entre a identificação $h_{X}\times h_{X}\simeq h_{X\times X}$
e a respectiva transformação $h_{-}(*):h_{X\times X}\rightarrow h_{X}$.
Ela faz de $h_{X}$ um grupo. Reciprocamente, pelo lema de Yoneda,
para cada transformação natural $*':h_{X\times X}\rightarrow h_{X}$
existe um único morfismo $*:X\times X\rightarrow X$. Mostra-se que,
se $*'$ torna $h_{X}$ um grupo, então $*$ faz o mesmo com $X$. 
\end{proof}
Um \emph{functor de Hopf} é aquele que preserva monoides de Hopf.\emph{
}Cada um deles induz um novo functor entre $\mathrm{Hopf}(\mathbf{C};\otimes)$
e $\mathrm{Hopf}(\mathbf{C}';\otimes')$. Existe uma subcategoria
não-cheia $\mathbf{Hopf}\subset\mathbf{Cat}$ tendo categorias monoidais
simétricas como objetos e functores de Hopf como morfismos. 

Mesmo que um functor seja monoidal e comonoidal, ele pode não ser
de Hopf. No entanto, os functores que são ao mesmo tempo (e de maneira
compatível) fortemente monoidais e fortemente comonoidais cumprem
com tal condição. Veja a proposição 3.60 de \cite{Aguiar_monoidal_categories}. 
\begin{example}
Relativamente às estruturas definidas por produtos binários, a inclusão
é um functor Hopf, desde que preserve produtos. Resultado dual é válido
para coprodutos. 
\end{example}

\section{Ações}

$\quad\;\,$Na álgebra clássica, além de operar num dado conjunto
(ato ao qual corresponde o conceito geral de monoide anteriormente
introduzido), sabe-se fazer uma estrutura algébrica agir noutra, definindo
uma terceira. Por exemplo, módulos nada mais são que anéis agindo
em grupos abelianos. Ora, anéis são monoides na categoria monoidal
dos grupos abelianos. Assim, um módulo é obtido fazendo um monoide
de $\mathbf{AbGrp}$ agir sob um objeto desta mesma categoria. 

Esta noção se estende, de maneira natural, às demais categorias monoidais.
Com efeito, numa categoria monoidal com produto $\otimes$, diz-se
que um monoide $X\in\mathbf{C}$ (com multiplicação $*:X\otimes X\rightarrow X$
e elemento neutro $e:1\rightarrow X$) \emph{age }em $Y\in\mathbf{C}$
quando existe $\alpha:X\otimes Y\rightarrow Y$ que deixa comutativo
o diagrama abaixo. Isto significa que $\alpha$ é associativo e que
preserva o elemento neutro (mantenha a comparação entre os diagramas
deste capítulo e aqueles apresentados na secção inicial do capítulo
um). $$
\xymatrix{\ar[d]_-{id \otimes *} X\otimes (X\otimes Y) \ar[r] & (X\otimes X)\otimes Y \ar[r]^-{*\otimes id} & X\otimes Y \ar[d]^{\alpha} & 1\otimes Y \ar[l]_{e\otimes id} \ar[dl]  \\
X \otimes Y \ar[rr]_{\alpha} && Y }
$$

Um objeto $Y$ no qual age um monoide $X$ chama-se \emph{módulo em
$\mathbf{C}$ sobre $X$}. Tem-se uma categoria $\mathrm{Mod}_{X}(\mathbf{C};\otimes)$
formada de todos os módulos sobre $X$, cujos morfismos são os morfismos
de $\mathbf{C}$ que preservam ações. No que segue, o conjunto de
todas as ações de um monoide $X$ num objeto $Y$ será denotado por
$\mathrm{Act}_{\mathbf{C}}(X;Y)$.
\begin{example}
Em $\mathbf{AbGrp}$, com produto $\otimes$, os monoides são os anéis.
Por sua vez, fixado um anel $R$, os módulos sobre ele coincidem com
os $R$-módulos da álgebra clássica, de tal modo que as categorias
$\mathbf{Mod}_{R}$ e $\mathrm{Mod}_{R}(\mathbf{\mathbf{AbGrp}};\otimes)$
são isomorfas.
\end{example}
Quando a estrutura monoidal de $\mathbf{C}$ advém de produtos binários,
faz sentido falar de morfismos com parâmetros. Assim, pode-se pensar
numa ação $\alpha:X\times Y\rightarrow Y$ como sendo morfismos $\alpha_{X}:Y\rightarrow Y$,
parametrizados em $X$, os quais tornam comutativos os diagramas abaixo$$
\xymatrix{X\times Y \ar[rr]^{\alpha _{X\times Y}} && X\times Y \ar[d]^{id} \ar[r]^{\pi _2} & Y \ar[ddd]^{\alpha _X} \\
X\times (X\times Y) \ar[u]_{\pi _2} \ar[d]_{id \times \alpha} \ar[r]^{\simeq } & (X\times X) \times Y \ar[r]^-{*\times id} & X\times Y \ar[d]^{\alpha} && 1\times Y \ar[rd] \ar[r]^{e\times id} & X\times Y \ar[d]^{\alpha} \ar[r]^{\pi_2} & Y \ar[ld]^{\alpha _X} \\
X\times Y \ar[d]_{\pi _2} \ar[rr]^{\alpha} && Y \ar[d]^{id} &&& Y \\
Y \ar[rr]_{\alpha _X} && Y \ar[r]_{id} & Y}
$$

Particularmente, se $X$ é um grupo, então os diagramas anteriores
podem ser combinados com aqueles que caracterizam os monoides de Hopf
para obter o seguinte, em que $\beta$ é a composição de $\pi_{2}\circ\alpha_{X\times Y}\circ\pi_{2}$
com a inversa do isomorfismo $X\times(X\times Y)\simeq(X\times X)\times Y$:$$
\xymatrix{1\times X \ar[rr] && Y \ar[r]^{id} & Y \\
X\times Y \ar[d]_{\Delta \times id} \ar[u]^-{e' \times id} \ar[r]^-{e' \times id} & 1\times Y \ar[r] & X\times Y \ar[r]^{\pi _2} \ar[u]^{\alpha} & Y \ar[u]_{\alpha _X} \\
(X\times X) \times Y \ar[rr]_-{id\times inv \times id} && (X\times X)\times Y \ar[u]^{* \times id} \ar[r]_-{\beta} & Y \ar[u]_{id} }
$$
\begin{example}
Em $\mathbf{Set}$, dotado da estrutura advinda de produtos binários,
seus monóides são usuais e, portanto, a ação de um destes, digamos
$X$ com multiplicação $*$ e unidade $e$, se traduz em $\alpha:X\times Y\rightarrow Y$
tal que $\alpha(x*x',y)=\alpha(x,\alpha(x'\cdot y))$ e $\alpha(e,y)=y$.
Se consideramos ações como morfismos parametrizados, cada $\alpha$
induz uma família de funções $\alpha_{x}:Y\rightarrow Y$, definidas
por $\alpha_{x}(y)=\alpha(x,y)$. Os diagramas anteriores significam
$\alpha_{x*x'}=\alpha_{x}\circ\alpha{}_{x'}$ e $\alpha_{e}=id$.
Particularmente, se $X$ é um grupo, com inversão $inv(x)=x^{-1}$,
então o último diagrama faz de cada $\alpha_{x}$ inversível, com
inversa $\alpha_{x^{-1}}$. Portanto, as ações de um grupo $X$ de
$(\mathbf{Set},\times)$ num certo objeto $Y\in\mathbf{Set}$ induzem
homomorfismos de $X$ em $\mathrm{Aut}(Y)$ e, consequentemente, há
uma aplicação injetiva $f:\mathrm{Act}_{\mathbf{Set}}(X;Y)\rightarrow\mathrm{Mor}_{\mathbf{Grp}}(X;\mathrm{Aut}(Y))$.

\begin{example}
Se $\mathbf{C}\subset\mathbf{Set}$ é qualquer categoria tal que a
inclusão preserva produtos (o que ocorre, por exemplo, com $\mathbf{Top}$,
$\mathbf{Mfd}$, $\mathbf{Diff}$ e com as categorias algébricas usuais),
então, para qualquer grupo $X$ em $\mathbf{C}$ e qualquer objeto
$Y\in\mathbf{C}$, tem-se $\mathrm{Act}_{\mathbf{C}}(X;Y)\subset\mathrm{Act}_{\mathbf{Set}}(X;Y)$,
onde identificamos a imagem pela inclusão com os próprios objetos.
Assim, bastando restringir a injeção $f$, a partir de qualquer ação
de $X$ em $Y$ obtém-se um homomorfismo $X\rightarrow\mathrm{Aut}(Y)$.
Em especial, se a inclusão $\imath:\mathbf{C}\rightarrow\mathbf{Set}$
possui adjunto $\jmath$ à direita, então tal restrição de $f$ cai
em $\mathrm{Mor}_{\mathrm{Hpf(\mathbf{C})}}(X;\jmath(\mathrm{Aut}(Y)))$.
Neste caso, homomorfismo induzido pela ação é entre grupos em $\mathbf{C}$
(e não somente entre grupos usuais).

\begin{example}
Seja $v$ um campo de vetores diferenciável numa variedade compacta.
Os resultados de existência, unicidade, dependência diferenciável
nas condições iniciais e extensão de soluções, usualmente válidos
em equações diferenciais ordinárias, são todos locais e, portanto,
se estendem ao contexto das variedades. Neste caso, por existência
e extensão de soluções, está definida uma aplicação $\varphi:\mathbb{R}\times X\rightarrow X$,
correspondendo ao \emph{fluxo de $v$}, tal que, para cada $x$, o
caminho $t\mapsto\varphi(t,x)$ é o único que passa em tal ponto no
instante zero, e cujo vetor velocidade $\partial\varphi(t,x)/\partial t$
é exatamente $v(\varphi(t,x))$. A diferenciabilidade com respeito
às condições iniciais faz de $\varphi$ diferenciável. Particularmente,
a unicidade de soluções garante que $\varphi(t+s,x)=\varphi(t,\varphi(s,x))$
e que $\varphi(0,x)=x$. Portanto, se vemos $\mathbb{R}$ como grupo
de Lie com a soma, então $\varphi$ é ação de tal grupo em $X$ e,
pelo exemplo anterior, a aplicação $t\mapsto\varphi_{t}$ é homomorfismo
$\mathbb{R}\rightarrow\mathrm{Diff}(X)$.
\end{example}
\end{example}
\end{example}

\subsection*{\uline{Órbitas}}

$\quad\;\,$Quando a estrutura monoidal de $\mathbf{C}$ provém de
produtos e a categoria em questão admite \emph{pushouts}, então faz
sentido falar da \emph{órbita} de uma ação. Com efeito, se um monoide
$X\in\mathbf{C}$ age num certo $Y\in\mathbf{C}$, em algumas situações
quer-se estudar $Y$ a menos das alterações produzidas pela ação (isto
acontece, por exemplo, quando $X$ descreve as simetrias de $Y$).
A órbita da ação $\alpha:X\times Y\rightarrow Y$ é e um novo espaço
$Y/X$ que contempla $Y$ módulo alterações produzidas por $\alpha$.
De maneira mais precisa, é o \emph{pushout} abaixo apresentado.$$
\xymatrix{Y/X & Y \ar@{-->}[l] \\
Y \ar@{-->}[u] & X\times Y \ar[l]^{\pi _2} \ar[u]_{\alpha} }
$$
\begin{example}
Em categorias cocompletas, \emph{pushouts} podem ser calculados em
termos de coprodutos e coequalizadores. Por exemplo, em $\mathbf{Set}$,
$\mathbf{Top}$ e nas categorias algébricas usuais, a órbita $Y/X$
coincide com o quociente de $Y$ pela relação cuja classe de um dado
$y$ é o conjunto $\alpha(X,y)$, formados de todo $\alpha(x,y)$,
com $x\in X$.

\begin{example}
Como $\mathbf{Mfd}$ e $\mathbf{Diff}$ não possuem \emph{pushouts},
ações de monoides e grupos em tais categoriais podem não admitir órbitas.
Por outro lado, se utilizamos de $\mathrm{Act}_{\mathbf{Diff}}(X;Y)\subset\mathrm{Act}_{\mathbf{Set}}(X;Y)$
e consideramos tais ações em $\mathbf{Set}$, então ali faz sentido
falar de suas órbitas. Particularmente, para cada campo de vetores
$v$, podemos tomar a órbita de seu fluxo $\varphi:\mathbb{R}\times X\rightarrow X$.
Ela é chamada de \emph{retrato de fase} de $v$. A classe de $x\in X$
é a sua \emph{trajetória}. Esta se diz \emph{singular} quando se resume
ao próprio ponto $x$, caso em que ele é dito ser uma \emph{singularidade}
de $v$. Evidentemente, $x$ é singularidade se, e só se, $v(x)=0$.
Se isto não acontece, então, por continuidade do fluxo, $v(\varphi(x,t))\neq0$
para todo $t$ e, portanto, o caminho $t\mapsto\varphi(t,x)$ é uma
imersão. Daí, pelo teorema de classificação das variedades unidimensionais,
a trajetória de $x$ é homeomorfa à reta ou ao círculo. Isto explicita
a vantagem de se considerar o retrato de fase de um campo: ele dá
uma decomposição da variedade $X$ em subvariedades unidimensionais
(trajetórias de pontos regulares) e conjuntos discretos (singularidades).
\end{example}
\end{example}

\chapter{Homotopia Abstrata}

$\quad\;\,$Os invariantes presentes na topologia algébrica podem
ser divididos essencialmente em duas classes: as teorias de homotopia
e as teorias de homologia. No presente capítulo, estudamos aspectos
abstratos e formais concernentes às teorias de homotopia. Grosso modo,
a ideia é substituir categorias cuja classificação por isomorfismos
é demasiadamente complicada por outras categorias (ditas \emph{homotópicas}
ou \emph{derivadas}) cuja classificação é mais relaxada e que, em
certo sentido, descrevam aproximadamente a categoria original. 

Iniciamos o capítulo definindo o que vem a ser uma categoria própria
para o estudo da homotopia. Estas nada mais são que categorias completas
e cocompletas nas quais se fixou uma subcategoria conveniente. Exemplos
de tais categorias englobam aquelas nas quais se tem uma noção estrita
de ``morfismos entre morfismos'', ou mesmo àquelas nas quais se
tem um conceito natural de cilindro ou de espaço de caminhos.

Na segunda secção, discutimos o procedimento de passagem de uma categoria
própria para o estudo da homotopia para a sua correspondente categoria
homotópica. A ideia é inverter formalmente os morfismos (denominados
\emph{equivalências fracas}) da subcategoria fixada. Tal passagem
possui um preço: ao fazê-la, pode-se perder completude e cocompletude.
Isto nos leva à terceira secção, onde definimos uma nova classe de
limites, chamados de \emph{limites homotópicos}, os quais são naturais
no presente contexto de trabalho e que se caracterizam por serem as
primeiras aproximações dos limites ordinários.

Durante a escrita, fomos fortemente influenciados pelas referências
\cite{categorical_homotopy,Hirschhorn_model_categories,MACLANE_categories,model_categories_KAN-1}.

\section{Estrutura}

$\quad\;\,$Numa categoria $\mathbf{C}$, diz-se que os morfismos
de uma subcategoria $\mathrm{Weak}\subset\mathbf{C}$ são \emph{equivalências
fracas }quando as seguintes condições são satisfeitas:
\begin{enumerate}
\item todo isomorfismo de $\mathbf{C}$ é um morfismo de $\mathrm{Weak}$;
\item dado um morfismo $f$, se a composição entre ele e alguma equivalência
fraca está em $\mathrm{\mathrm{Weak}}$, então o próprio $f$ é uma
equivalência fraca. Esta propriedade é chamada de \emph{2-de-3}.
\end{enumerate}
$\quad\;\,$Quando se quer evidenciar que um morfismo de $\mathbf{C}$
está em $\mathrm{Weak}$ (ou seja, que é equivalência fraca), utiliza-se
de ``$\rightsquigarrow$'' ao invés de ``$\rightarrow$'' para
denotá-lo. Se existe uma equivalência fraca $X\rightsquigarrow Y$,
fala-se que $X$ é uma \emph{resolução }de $Y$. Por sua vez, $h:X\rightarrow Y'$
é dito ser uma \emph{resolução }do morfismo $f:Y\rightarrow Y'$ quando
é a composição de $f$ por uma resolução de $X$. Isto é, quando existe
uma equivalência fraca $g:X\rightsquigarrow Y$ tal que $h=g\circ f$.

Tem-se particular interesse nas categorias com equivalências equivalências
fracas que sejam completas e cocompletas. Elas são designadas \emph{próprias
para o estudo da homotopia}.
\begin{example}
Toda categoria completa e cocompleta pode ser feita própria para homotopia
de maneira trivial: basta considerar como equivalências fracas os
isomorfismos de $\mathbf{C}$. Evidentemente, tem-se especial interesse
em categorias que se tornam próprias para homotopia de maneira não-trivial.

\begin{example}
A escolha de uma noção de equivalência fraca numa categoria $\mathbf{C}$
determina outra em sua pontuação $\mathbf{C}_{*}$: é só tomar as
equivalências fracas de $\mathbf{C}$ que preservam o ponto base dos
espaços pontuados nos quais estão definidas. 

\begin{example}
Para qualquer categoria pequena $\mathbf{D}$, se $\mathbf{C}$ é
própria para o estudo da homotopia, então a respectiva $\mathrm{Func}(\mathbf{D};\mathbf{C})$
também o é. Suas equivalências fracas são simplesmente as transformações
naturais $\xi$ tais que, para todo $X\in\mathbf{D}$, o morfismo
$\xi(X)$ é uma equivalência fraca em $\mathbf{C}$.
\end{example}
\end{example}
\end{example}
Diz-se que um functor $F:\mathbf{C}\rightarrow\mathbf{D}$ é \emph{homotópico}
quando ele mapeia equivalências fracas em equivalências fracas. Tem-se
uma subcategoria $\mathrm{W}\mathbf{Cat}\subset\mathbf{Cat}$ formada
por todas as categorias próprias para o estudo da homotopia e tendo
functores homotópicos como morfismos.

\subsection*{\uline{\mbox{$2$}-Categorias}}

$\quad\;\,$Seja $\mathbf{Cat}$ a categoria tendo categorias como
objetos e functores como morfismos. No que segue, diremos que dois
functores são \emph{equivalentes} quando existir ao menos uma transformação
natural entre eles. Relembramos que, para quaisquer objetos $\mathbf{C},\mathbf{D}\in\mathbf{Cat}$,
isto dota o respectivo $\mathrm{Mor}_{\mathbf{Cat}}(\mathbf{C};\mathbf{D})=\mathrm{Func}(\mathbf{C};\mathbf{D})$
de uma relação de equivalência. A reflexividade é evidente. Por sua
vez, sendo o conjunto do morfismos $\mathrm{Func}(\mathbf{C};\mathbf{D})$
uma categoria (o que se deve ao fato de sabermos compor verticalmente
transformações naturais), para toda transformação $\xi:F\rightarrow F'$,
a respectiva $\xi^{op}$ é morfismo de $\mathrm{Func}(\mathbf{C};\mathbf{D})^{op}$
e, portanto, transformação de $F'$ em $F$, garantindo-nos a simetria.
A transitividade é obtida tomando a composição horizontal de transformações
naturais.

Desta forma, está bem definida uma categoria $\mathscr{H}\mathbf{Cat}$
(já apresentada na segunda secção do primeiro capítulo), tendo categorias
como objetos e as classes de functores ligados por transformações
naturais como morfismos. Com isto em mente, vê-se que $\mathbf{Cat}$
serve ao estudo da homotopia: basta considerar como equivalência fraca
qualquer functor cuja classe em $\mathscr{H}\mathbf{Cat}$ é um isomorfismo.

Observamos que, na construção anterior, para definir $\mathscr{H}\mathbf{Cat}$
e, consequentemente, a noção de equivalência fraca, utilizou-se apenas
do conceito de ``morfismos entre morfismos'' proporcionados pelas
transformações naturais, os quais podem ser compostos de duas maneiras
compatíveis, formando, para cada uma delas, uma devida categoria. 

Quando numa categoria se conta com ``morfismos entre morfismos''
(chamados de \emph{$2$-morfismos} ou de \emph{homotopias}) satisfazendo
tais condições, diz que ela é uma $2$-\emph{categoria} \emph{estrita}.
De maneira mais precisa, uma $2$-categoria estrita é uma categoria
$\mathbf{C}$ tal que cada $\mathrm{Mor}_{\mathbf{C}}(X;Y)$ também
é uma categoria (com respectivas composições correspondendo às composições
verticais), e para a qual existem bifunctores 
\[
\circ:\mathrm{Mor}_{\mathbf{C}}(X;Y)\times\mathrm{Mor}_{\mathbf{C}}(Y;Z)\rightarrow\mathrm{Mor}_{\mathbf{C}}(X;Z),
\]
representando as composições horizontais, que cumprem com associatividade
e que preservam identidades. 

Posto isto, \emph{qualquer $2$-categoria estrita $\mathbf{C}$ (particularmente
$\mathbf{Cat}$) que seja completa e cocompleta é própria para estudar
homotopia}: inicia-se identificando os morfismos que estão ligados
por $2$-morfismos, depois constrói-se $\mathscr{H}\mathbf{C}$ e,
finalmente, toma-se como equivalências fracas os morfismos cujas classes
em $\mathscr{H}\mathbf{C}$ são isomorfismos. 

Um $2$-morfismo $h$ é usualmente chamado de \emph{homotopia}. Quando
há\emph{ }homotopia $h$\emph{ }entre $f$ e $g$, escreve-se $h:f\Rightarrow g$
ou $f\simeq g$ (quando não há relevância em $h$), e fala-se que
tais morfismos são \emph{homotópicos.} Sua representação diagramática,
assim como as de suas composições, imitam aquelas utilizada para as
transformações naturais, como pode ser visto abaixo. Assim, pode-se
dizer que o passo fundamental na construção de $\mathscr{H}\mathbf{C}$
(e, portanto, no fato de que $2$-categorias servem ao estudo da homotopia)
é precisamente o conceito de \emph{homotopia} \emph{entre morfismos}.$$
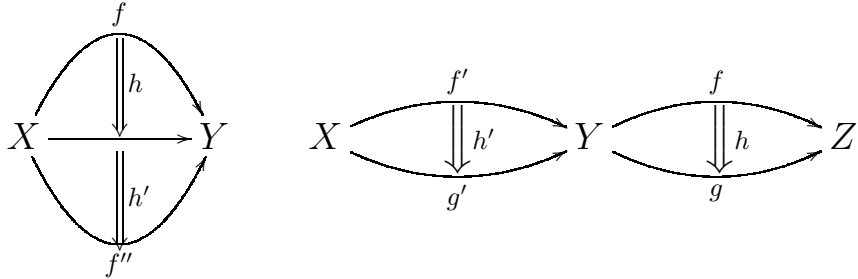
\xymatrix{& \ar@{=>}[d]^{h} &  \\
X \ar@/^{1.4cm}/[rr]^{f} \ar@/_{1.4cm}/[rr]_{f''}  \ar[rr] & \ar@{=>}[d]^{h'} & Y & X \ar@/_{0.5cm}/[rr]_{g'} \ar@/^{0.5cm}/[rr]^{f'} & \;\; \Big{\Downarrow} \text{\footnotesize{$h '$}} & Y \ar@/_{0.5cm}/[rr]_{g} \ar@/^{0.5cm}/[rr]^{f} & \;\; \Big\Downarrow \text{\footnotesize{$h$}}  & Z \\
& &} 
$$

Um functor $F:\mathbf{C}\rightarrow\mathbf{D}$ entre $2$-categorias
é dito ser um \emph{$2$-functor }quando preserva não só os objetos
e os morfismos, mas também os $2$-morfismos. Particularmente, como
estes preservam homotopias, eles passam ao quociente e definem um
novo functor entre $\mathscr{H}\mathbf{C}$ e $\mathscr{H}\mathbf{D}$.
Em suma, a regra $\mathscr{H}:\mathrm{S}\mathbf{Cat}_{2}\rightarrow\mathrm{W}\mathbf{Cat}$,
onde $\mathrm{S}\mathbf{Cat}_{2}$ é formada de $2$-categorias e
$2$-functores, tem a característica functorial e, na verdade, define
um mergulho. 

Não se deve esperar que o functor $\mathscr{H}$ possua uma inversa
fraca. Isto é, não se deve esperar que qualquer categoria com equivalências
fracas seja uma $2$-categoria. Com efeito, para construir $\mathscr{H}$
viu-se necessário um classe de informação de ``nível categórico''
superior àquela presente nos elementos de $\mathrm{W}\mathbf{Cat}$:
tratam-se dos $2$-morfismos, das homotopias. Em contrapartida, no
próximo capítulo mostraremos haver uma subcategoria intermediária
\[
\mathrm{S}\mathbf{Cat}_{2}\subset\mathbf{Cat}_{\infty}^{1}\subset\mathrm{W}\mathbf{Cat},
\]
formada das chamadas \emph{$(\infty,1)$-categorias}, as quais também
possuem a noção de homotopia e que modelam as categorias com equivalências
fracas. Isto nos dará o slogan: \emph{as categorias próprias para
o estudo da homotopia são, precisamente, aquelas nas quais se têm
uma noção coerente de homotopia entre seus morfismos}.\emph{ }

Finalizamos esta subsecção apresentando alguns exemplos de $2$-categorias.
\begin{example}
Toda subcategoria $\mathbf{D}$ de uma $2$-categoria $\mathbf{C}$
também é uma $2$-categoria, cujos $2$-morfismos são os próprios
de $\mathbf{C}$. Mais precisamente, um $2$-morfismos $h:f\Rightarrow g$
é simplesmente um $2$-morfismos $h$ de $\mathbf{C}$ entre $\imath(f)$
e $\imath(g)$. Assim, outros exemplos de $2$-categorias incluem
as subcategorias de $\mathbf{Cat}$, com $2$-morfismos dados pelas
transformações naturais. Aqui se enquadram, por exemplo, $\mathbf{Mnd}$,
$\mathrm{Co}\mathbf{Mnd}$ e seus derivados $\mathrm{S}\mathbf{Mnd}$,
$\mathrm{SCo}\mathbf{Mnd}$ e $\mathbf{Hopf}$, bem como $\mathrm{W}\mathbf{Cat}$.

\begin{example}
Mesclando qualquer categoria $\mathbf{C}$ com $\mathscr{B}(\mathbf{C})$,
encontra-se uma $2$-categoria: seus objetos e seus $1$-morfismos
são os próprios objetos e morfismos de $\mathbf{C}$, ao passo que
os $2$-morfismos são os morfismos de $\mathscr{B}(\mathbf{C})$.
Em outras palavras, um $2$-morfismo entre $f,g:X\rightarrow Y$ é
formado por um par $\mathfrak{h}=(h,h')$ de morfismos em $\mathbf{C}$,
tal que $h'\circ f=g\circ h$. Cada conjunto $\mathrm{Mor}_{\mathbf{C}}(X;Y)$
pode ser feito uma categoria se consideramos a composição horizontal
de $2$-morfismos definida ``componente a componente'' $(h,h')\bullet(k,k')=(h\circ k,h'\circ k')$.
Por sua vez, relativamente a tais estruturas de categoria, tem-se
bifunctores 
\[
\circ:\mathrm{Mor}_{\mathbf{C}}(X;Y)\times\mathrm{Mor}_{\mathbf{C}}(Y;Z)\rightarrow\mathrm{Mor}_{\mathbf{C}}(X;Z),
\]
definidos por $(h,h')\circ(k,k')=(h,k')$, os quais representam a
composição vertical de $2$-morfismos. Estes são claramente associativos
e preservam identidades.

\begin{example}
A categoria $n\mathbf{Cob}$ tem variedades de dimensão $n-1$ como
objetos e variedades $n$-dimensionais como morfismos. Por sua vez,
$(n+1)\mathbf{Cob}$ é formada de $n$-variedades ligadas por entidades
de dimensão $n+1$. Isto nos leva a considerar, para cada $n$ fixo,
uma possível $2$-categoria $n\mathbf{Cob}_{2}$, mantendo $n\mathbf{Cob}$
e considerando como $2$-morfismos os morfismos de $(n+1)\mathbf{Cob}$.
Assim, seus objetos seriam variedades sem bordo, seus morfismos seriam
cobordismos, enquanto que os $2$-morfismos seriam cobordismos em
dimensão superior. No entanto, além de problemas envolvendo a diferenciabilidade
do bordo, haveriam problemas relacionados à associatividade da colagem.
Este último fato poderia ser sanado enfraquecendo as condições exigidas
sobre a composição vertical de uma $2$-categoria. Ao fazer isto,
chegaríamos ao conceito de $2$-\emph{categoria fraca}, o qual será
apresentado no sexto capítulo. Assim, em certo sentido, para todo
$n$, a correspondente $n\mathbf{Cob}_{2}$ é uma espécie $2$-categoria
fraca. 
\end{example}
\end{example}
\end{example}

\subsection*{\uline{Cilindros Naturais}}

$\quad\;\,$Na subsecção anterior, analisamos uma classe particular
de categorias próprias para o estudo homotopia: as $2$-categorias
estritas. Nesta, estudaremos outra classe que é fonte de exemplos
de categorias com equivalências fracas. Tratam-se, pois, das categorias
nas quais se tem uma noção natural de ``cilindro''.

Um \emph{cilindro natural} numa categoria $\mathbf{C}$ é um functor
$C:\mathbf{C}\rightarrow\mathbf{C}$, em conjunto com transformações
naturais $\imath_{0},\imath_{1}:id_{\mathbf{C}}\rightarrow C$, as
quais possuem uma inversa à esquerda $\pi:C\rightarrow id_{\mathbf{C}}$
em comum, ao mesmo tempo que estão vinculadas por outra transformação
$\tau:C\rightarrow C$. Isto significa que, para qualquer que seja
o $X$, os diagramas abaixo são comutativos.$$
\xymatrix{&X &&& X \ar[d]_-{\imath _1 (X)} \ar[r]^-{\imath _0 (X)} & CX\\
CX \ar[ru]^-{\pi (X)} & X \ar[u]^-{id} \ar[r]_-{\imath _1 (X)} \ar[l]^-{\imath _0 (X)} & CX \ar[lu]_-{\pi (X)} && CX \ar[ru]_-{\tau (X)}  }
$$ 

Vejamos, agora, que todo cilindro natural induz uma noção de equivalência
fraca. Com efeito, dados $f,g:X\rightarrow Y$, diz-se que eles são
\emph{$C$-homotópicos} quando existe um morfismo 
\[
H:CX\rightarrow Y,\quad\mbox{tal que}\quad H\circ\imath_{0}(X)=f\quad\mbox{e}\quad H\circ\imath_{1}(X)=g,
\]
ao qual se dá o nome de \emph{$C$-homotopia }entre $f$ e $g$. Observamos
que a relação que identifica morfismos $C$-homotópicos é de equivalência:
para a reflexividade, basta tomar $H=f\circ\pi(X)$. A transitividade
é evidente. Por fim, para a simetria, se $H$ é $C$-homotopia entre
$f$ e $g$, então $\tau\circ H$ é $C$-homotopia entre $g$ e $f$.
Além disso, ela é compatível com composições. Consequentemente, fica
definida uma nova categoria $\mathscr{H}_{C}(\mathbf{C})$, obtida
de $\mathbf{C}$ trocando morfismos por suas correspondentes classes
de $C$-homotopia. Posto isto, podemos tomar como equivalências fracas
em $\mathbf{C}$ todos os morfismos cujas classes são isomorfismos
em $\mathscr{H}_{C}(\mathbf{C})$.
\begin{example}
O protótipo de cilindro natural é o functor $-\times I$ em $\mathbf{Set}$,
dotado das transformações $\imath_{0}$ e $\imath_{1}$, tais que
$\imath_{0}(X)(x,t)=(x,0)$ e $\imath_{1}(X)(x,t)=(x,1)$, as quais
possuem inversa à esquerda dada pela projeção na primeira entrada.
Isto é, coloca-se $\pi(X)=\mathrm{pr}_{1}$ para todo $X$. Além disso,
estas transformações estão relacionadas por meio das correspondências
\[
\tau(X):X\times I\rightarrow X\times I,\quad\mbox{tais que}\quad\tau(X)(x,t)=(x,t-1).
\]
Neste caso, uma $C$-homotopia entre $f,g:X\rightarrow Y$ é simplesmente
uma função $H:X\times I\rightarrow Y$ satisfazendo $H(x,0)=f(x)$
e $H(x,1)=g(x)$ para todo $x$. Resta dizer que a análise continua
inteiramente válida se trocamos $\mathbf{Set}$ por $\mathbf{Top}$.

\begin{example}
Toda categoria monoidal $\otimes:\mathbf{C}\times\mathbf{C}\rightarrow\mathbf{C}$,
com objeto neutro $1$, admite um cilindro natural trivial, dado por
$-\otimes1$. As transformações $\imath_{0}$ e $\imath_{1}$ são
aquelas que para cada $X\in\mathbf{C}$ associam os isomorfismos naturais
$X\simeq X\otimes1$, enquanto que $\pi(X)$ nada mais é que o isomorfismo
inverso $X\otimes1\simeq X$. Com estas escolhas, dois morfismos são
$C$ homotópicos se, e somente se, são iguais. Assim, a categoria
$\mathscr{H}_{C}(\mathbf{C})$ coincide com $\mathbf{C}$, de modo
que as equivalências fracas são todos os isomorfismos de $\mathbf{C}$.
Em outras palavras, a estrutura de categoria própria para homotopia
fornecida à $\mathbf{C}$ por $\otimes$ é a trivial.
\end{example}
\end{example}
Observamos haver uma semelhança evidente entre a estratégia utilizada
para definir as categorias $\mathscr{H}_{C}(\mathbf{C})$ e $\mathscr{H}(\mathbf{C})$,
nos respectivos casos em que se tem cilindros naturais ou uma estrutura
de $2$-categoria. Com efeito, em ambos os casos utilizou-se de uma
noção prévia de ``homotopia'', a qual nos permitiu definir uma relação
de equivalência em cada conjunto de morfismos de $\mathbf{C}$. As
correspondentes categorias $\mathscr{H}_{C}(\mathbf{C})$ e $\mathscr{H}(\mathbf{C})$
foram então obtidas por passagem ao quociente com respeito a tais
relações. 

Existe, no entanto, uma diferença entre os dois procedimentos: se
por um lado sempre se sabe compôr $2$-morfismos de duas diferentes
maneiras, as quais são estritamente associativas, não se tem maneiras
canônicas de compôr $C$-homotopias. Além disso, ainda que se as tenha,
pode ser que as composições não sejam associativas. Assim, olhando
desta forma, $2$-categorias parecem possuir mais informação.

Em contrapartida, diferentemente dos $2$-morfismos, as $C$-homotopias
são morfismos da categoria subjacente. Faz sentido, portanto, considerar
$C$-homotopias entre $C$-homotopias, iterando o processo \emph{ad
infinitum}. Assim, os cilindros naturais induzem não só uma noção
de homotopia, mas também noções de ``homotopias de grau superior''. 

Em alguns casos (como no oitavo capítulo veremos acontecer em $\mathbf{Top}$),
todas as homotopias induzidas pelos cilindros naturais são, em certo
sentido, inversíveis. Assim, em tais situações tem-se entidades compostas
de objetos, morfismos, homotopias, homotopias entre homotopias, etc.,
cuja informação categórica de grau superior a um (das homotopias para
adiante) são sempre inversíveis. Tais entidades correspondem às $(\infty,1$)-categorias
que, como comentamos na subsecção anterior e provaremos no próximo
capítulo, modelam todas as categorias com equivalências fracas.

\subsection*{\uline{Caminhos Naturais}}

$\quad\;\,$A versão dual dos cilindros naturais são os \emph{espaços
de caminhos naturais}. Assim, estes são functores $P:\mathbf{C}\rightarrow\mathbf{C}$,
dotados de transformações naturais $\pi_{0},\pi_{1}:P\rightarrow id_{\mathbf{C}}$,
as quais possuem uma inversa à direita em comum $\imath:id_{\mathbf{C}}\rightarrow P$
e estão vinculadas por uma certa $\tau:P\rightarrow P$. Mais precisamente,
os diagramas abaixo são comutativos para todo $X$.$$
\xymatrix{PX \ar[r]^-{\pi _0 (X)} & X & PX \ar[l]_-{\pi _1 (X)} && X & PX \ar[l]_-{\pi _0 (X)} \ar[ld]^-{\tau (X)} \\
& X \ar[u]^-{id} \ar[lu]^-{\imath (X)} \ar[ru]_-{\imath (X)} &&& PX \ar[u]^-{\pi _1 (X)} }
$$

A presença de espaços de caminhos naturais também induz equivalências
fracas. Com efeito, diz-se que dois mapas $f,g:X\rightarrow Y$ são
\emph{$P$-homotópicos} quando existe uma \emph{$P$-homotopia} entra
eles. Isto é, quando existe um morfismo 
\[
H:X\rightarrow PY\quad\mbox{tal que}\quad\pi_{0}(Y)\circ H=f\quad\mbox{e}\quad\pi_{0}(Y)\circ H=g.
\]

De maneira análoga ao que vimos valer para cilindros naturais, a relação
de $P$-homotopia é de equivalência e compatível com composições.
Assim, para cada $P$ fica definida uma categoria $\mathscr{H}^{P}(\mathbf{C})$,
cujos objetos são os próprios objetos de $\mathbf{C}$ e cujos morfismos
são as classes de $P$-homotopia. Consequentemente, os morfismos de
$\mathbf{C}$ que são isomorfismos em $\mathscr{H}^{P}(\mathbf{C})$
podem ser tomados como equivalências fracas.
\begin{example}
Assim como $\mathbf{Set}$ admite o protótipo de cilindros naturais,
também admite o exemplo que motiva a definição (e a nomenclatura)
de espaços de caminhos. Com efeito, ali se tem o functor $P:\mathbf{Set}\rightarrow\mathbf{Set}$,
que a cada $X$ associa o conjunto $\mathrm{Mor}_{\mathbf{Set}}(I;X)$
dos caminhos em $X$ (isto é, das aplicações $\gamma:I\rightarrow X$).
Também se tem as transformações $\pi_{0}(X)$ e $\pi_{1}(X)$, que
a cada caminho $\gamma$ associam, respectivamente, seu ponto inicial
e seu ponto final. Elas estão ligadas pela transformação $\tau(X)$
que toma $\gamma$ e devolve outro caminho $\tau(X)(\tau)$, que em
cada instante vale $\gamma(1-t)$. Argumentos análogos se aplicam
em $\mathbf{Top}$.
\end{example}
Pode ocorrer que numa categoria admita tanto cilindros naturais $C$
quando espaços de caminhos naturais $P$. Neste caso, nela haverão
duas noções (a princípio distintas) de equivalências fracas. Gostar-se-ia
que tais noções fossem coincidentes. Ora, as equivalências fracas
são obtidas por meio das $C$-homotopias e das $P$-homotopias. Desta
forma, se estas forem coincidentes, então as respectivas noções de
equivalências fracas também o serão. Como logo se verifica, \emph{uma
condição necessária e suficiente para que se tenha correspondência
biunívoca entre $C$-homotopias e $P$-homotopias é a existência de
uma adjunção entre $C$ e $P$}. 
\begin{example}
Como vimos no últimos capítulo, $\mathbf{Set}$ é cartesianamente
fechada. Com efeito, para todo conjunto $Y$, o adjunto de $-\times Y$
é precisamente o functor $\mathrm{Mor}_{\mathbf{Set}}(Y;-)$. Por
sua vez, nos exemplos anteriores vimos que o functor $-\times I$
define cilindro natural, ao passo que $\mathrm{Mor}_{\mathbf{Set}}(I;-)$
define espaço de caminhos. Consequentemente, as respectivas noções
de homotopia e de equivalências fracas por eles proporcionadas são
idênticas.

\begin{example}
Por outro lado, nos últimos exemplos também vimos que $-\times I$
e $\mathrm{Mor}_{\mathbf{Top}}(I;-)$ servem de cilindros e de espaços
de caminhos naturais em $\mathbf{Top}$. Ressaltamos, no entanto,
que, \emph{diferentemente do que se passa em $\mathbf{Set}$, as noções
de homotopia por eles produzidas não coincidem}. Afinal, $\mathbf{Top}$
não é cartesianamente fechada. Assim, mesmo que uma função $H:X\times I\rightarrow Y$
seja contínua, a correspondente $H':X\rightarrow\mathrm{Mor}_{\mathbf{Top}}(I;Y)$,
obtida pela adjunção em $\mathbf{Set}$ e caracterizada por $H'(x)(t)=H(x,t)$,
pode não o ser.
\end{example}
\end{example}

\section{Categoria Homotópica}

$\quad\;\,$Como dito na introdução do capítulo, o contexto no qual
a teoria da homotopia se apresenta é o seguinte: quer-se substituir
uma categoria $\mathbf{C}$ por uma primeira aproximação, em que o
problema de classificação é ligeiramente mais brando que o de $\mathbf{C}$.
Nesta subsecção, mostraremos que, se $\mathbf{C}$ é própria para
o estudo da homotopia, então a subcategoria distinguida $\mathrm{Weak}$
nos permite construir a aproximação desejada. 

Observamos que, nos capítulos anteriores, sempre substituímos a relação
de igualdade entre functores surgia, esta era substituída por aquela
que identifica functores entre os quais há uma transformação natural.
Foi desta maneira que definimos os conceitos de \emph{equivalência
entre categorias} e de \emph{extensões de Kan}. No presente contexto,
esta atitude ressalta o fato de que $\mathbf{Cat}$ (ou, mais geralmente,
qualquer $2$-categoria $\mathbf{C}$) é aproximada por $\mathscr{H}\mathbf{Cat}$. 

Assim, temos um protótipo para o processo de aproximação que procuramos.
Para abstraí-lo, precisamos descrever $\mathscr{H}\mathbf{C}$ em
termos de conceitos presentes numa categoria arbitrária. Neste sentido,
observamos que $\mathscr{H}\mathbf{C}$ é universal com respeito ao
functor quociente $\pi:\mathbf{C}\rightarrow\mathscr{H}\mathbf{C}$.
Mais precisamente, toda categoria $\mathbf{D}$ para a qual existe
um functor $F:\mathbf{C}\rightarrow\mathbf{D}$ mandando equivalências
fracas em isomorfismos é fracamente isomorfa a $\mathscr{H}\mathbf{C}$.
Realmente, a hipótese assegura que $f\simeq g$ implica $F(f)=F(g)$,
de modo que (passando ao quociente) há um único functor $F_{*}:\mathscr{H}\mathbf{C}\rightarrow\mathbf{D}$
cumprindo $F_{*}\circ\pi=F$. Como logo se convence, ele define a
equivalência buscada.

Uma \emph{localização} de uma categoria $\mathbf{C}$ numa classe
de morfismos $\mathcal{C}$ é uma outra categoria $\mathbf{C}[\mathcal{C}]$
dotada de um functor $p:\mathbf{C}\rightarrow\mathbf{C}[\mathcal{C}]$
que manda morfismos de $\mathcal{C}$ em isomorfismos de $\mathbf{C}[\mathcal{C}]$,
sendo universal com respeito a tal propriedade. Assim, para uma $2$-categoria,
$\mathscr{H}\mathbf{C}$ é simplesmente a sua localização na classe
das equivalências fracas. 

Evidentemente, $\mathbf{C}[\mathcal{C}]$ é única a menos de um único
isomorfismo. No que diz respeito a sua existência, observamos que
ela pode ser descrita como sendo a categoria que possui os mesmos
objetos que $\mathbf{C}$ e cujos morfismos são os \emph{zig-zags}
de morfismos de $\mathbf{C}$ por morfismos de $\mathcal{C}$. Isto
é, as sequências$$
\xymatrix{X & \ar[l] A_1 \ar[r] & A_2 & \ar[l] A_3 \ar[r] & \cdots \ar[r] & Y}
$$ que não possuem identidades e nem morfismos consecutivos na mesma
direção, sendo tais que toda seta voltada para à esquerda está em
$\mathcal{C}$. Em suma, $\mathbf{C}[\mathcal{C}]$ é obtida de $\mathbf{C}$
acrescentando uma inversa formal a cada um de seus morfismos. Veja
\cite{calculus_fractions_GABRIEL}.

Voltando ao questionamento inicial, dada uma categoria $\mathbf{C}$
qualquer, o análogo de $\mathscr{H}\mathbf{C}$ há de ser a localização
de $\mathbf{C}$ nas equivalências fracas. Ela é chamada de \emph{categoria
homotópica }ou mesmo de \emph{categoria derivada} de $\mathbf{C}$,
sendo usualmente denotada por $\mathrm{Ho}(\mathbf{C})$. Evidentemente,
quando $\mathbf{C}$ é $2$-categoria, tem-se a equivalência $\mathrm{Ho}(\mathbf{C})\simeq\mathscr{H}\mathbf{C}$.

Se dois objetos $X$ e $Y$ são isomorfos na categoria homotópica,
escreve-se $X\simeq Y$ e diz-se que eles possuem o mesmo \emph{tipo
de homotopia}. No contexto das $2$-categorias, dois objetos tem o
mesmo tipo de homotopia se, e somente se, existe um morfismo $f:X\rightarrow Y$
(chamado de \emph{equivalência homotópica}) que admite inversa a menos
de $2$-morfismos. Isto é, tal que há um outro $g:Y\rightarrow X$,
bem como $2$-morfismos $g\circ f\simeq id_{X}$ e $f\circ g\simeq id_{Y}$.
\begin{example}
Na $2$-categoria $\mathbf{Cat}$, duas categorias tem o mesmo tipo
de homotopia precisamente quando são equivalentes. Isto justifica
uma máxima: \emph{numa $2$-categoria, em geral não se exige a igualdade
entre dois morfismos}. \emph{Exige-se, apenas, que eles coincidam
a menos de um $2$-morfismo}. Grosso modo, isto significa que, se
no contexto de trabalho tem-se em mãos a noção de homotopia, costuma-se
enfraquecer a exigência de comutatividade dos diagramas, passando
a exigir apenas comutatividade módulo homotopias. Em $\mathbf{Cat}$,
por exemplo, tal máxima recai em outra que já temos utilizado desde
o início do texto: \emph{não se exige a igualdade de functores, mas
apenas que sejam idênticos a menos de isomorfismos naturais}.
\end{example}

\subsection*{\uline{Completude}}

$\quad\;\,$Por definição, uma categoria própria para o estudo da
homotopia deve ser completa e cocompleta. Tenhamos em mente uma nova
questão: se existir, será $\mathrm{Ho}(\mathbf{C})$ também completa
e cocompleta? No intuito de analisá-la, observamos que $p:\mathbf{C}\rightarrow\mathrm{Ho}(\mathbf{C})$
induz, para cada categoria $\mathbf{J}$, um respectivo functor 
\[
p_{\mathbf{J}}:\mathrm{Func}(\mathbf{J};\mathbf{C})\rightarrow\mathrm{Func}(\mathbf{J};\mathrm{Ho}(\mathbf{C})),\quad\mbox{tal que}\quad p_{\mathbf{\mathbf{J}}}(F)=p\circ F\quad\mbox{e}\quad p_{\mathbf{\mathbf{J}}}(\xi)=\xi.
\]

Disto segue que, se $F:\mathbf{\mathbf{J}}\rightarrow\mathbf{C}$
possui limite e/ou colimite, então o correspondente $p_{\mathbf{J}}(F)$
também o possui. Portanto, uma vez que a categoria $\mathbf{C}$ é
ao mesmo tempo completa e cocompleta, se $p_{\mathbf{J}}$ for um
isomorfismo, então todo functor entre $\mathbf{J}$ e $\mathbf{C}$
se escreverá sob a forma $p_{\mathbf{J}}(F)$ para algum $F:\mathbf{\mathbf{J}}\rightarrow\mathbf{C}$
e $\mathrm{Ho}(\mathbf{C})$ possuirá todos os limites e colimites
em $\mathbf{J}$. 

Observamos que tal estratégia funciona quando $\mathbf{J}$ é discreta,
de modo que, quando existe, a categoria homotópica possui todos os
produtos e coprodutos. No caso particular em que $\mathbf{C}$ é $2$-categoria,
$\mathrm{Ho}(\mathbf{C})$ coincide com $\mathscr{H}\mathbf{C}$,
de tal modo que seus produtos e coprodutos são dados pelos próprios
produtos e coprodutos de $\mathbf{C}$, tendo como projeções e inclusões
as suas respectivas classes de homotopia. 

Assim, pelo teorema de existência de limites e colimites, para garantir
completude e co-completude de $\mathrm{Ho}(\mathbf{C})$, bastaríamos
mostrar que a estratégia anterior também se aplica quando $\mathbf{J}$
é constituída das categorias $\rightarrow\cdot\leftarrow$ ou $\leftarrow\cdot\rightarrow$.
Este, no entanto, não é o caso. Em geral, a categoria homotópica não
possui \emph{pullbacks} e nem \emph{pushouts}. Daremos exemplos deste
fato no oitavo capítulo, quando estudarmos teoria da homotopia em
$\mathbf{Top}$.

\section{Functores Derivados}

$\quad\;\,$Em princípio, a categoria homotópica atende como uma \emph{primeira
aproximação} da categoria da qual ela descende. Diante disso, dado
um functor $F:\mathbf{C}\rightarrow\mathbf{C}'$, torna-se interessante
qualquer processo que nos permita substituí-lo por algum functor entre
$\mathrm{Ho}(\mathbf{C})$ e $\mathrm{Ho}(\mathbf{C}')$, dito ser\emph{
derivado} de $F$ e que haverá de traduzir uma primeira aproximação
deste. 

A maneira imediata de realizar esta troca é por ``passagem ao quociente''.
Com efeito, para que haja um único $F':\mathrm{Ho}(\mathbf{C})\rightarrow\mathrm{Ho}(\mathbf{C}')$
tal que $F'\circ p=p'\circ F$, é necessário e suficiente que $F$
seja \emph{homotópico}. Assim, ao tomá-la, cair-se-ia numa igualdade
de functores, que é uma condição demasiadamente restrita. Como sempre,
a ideia é substituir a igualdade pela relação $\simeq$. Para fazer
isto, observamos que a passagem ao quociente acima nada mais é que
uma extensão do functor $p'\circ F$ em $\mathbf{Cat}$. Ora, mas
problemas de extensão já sabemos enfraquecer: ao invés de se tentar
estender estritamente, procura-se pelas extensões de Kan. 

Motivados por estes fatos, definamos: o \emph{functor} \emph{derivado}
\emph{à esquerda} (resp. \emph{à direita}) de $F$ nada mais é que
a extensão à esquerda (resp. direita) de $p'\circ F$ relativamente
a $p$.

No capítulo anterior, vimos que se um functor assume valores numa
categoria completa e cocompleta, então suas extensões de Kan sempre
existem e podem ser computadas por meio de \emph{ends} e \emph{coends}.
Observamos que isto não pode ser utilizado para garantir a existência
de functores derivados, uma vez que, em geral, categorias homotópicas
não são completas. 

Uma estratégia mais palpável para a obtenção de tais functores será
dada na última secção do capítulo seguinte, onde consideramos categorias
que estão dotadas, além das equivalências fracas, de outras classes
de ``morfismos auxiliares''. 

\subsection*{\uline{Limites Especiais}}

$\quad\;\,$Prosseguindo com o espírito da subsecção anterior, buscamos
pelas primeiras aproximações de limites de functores que assumem valores
em categoria próprias para o estudo da homotopia. Seja $\mathbf{C}$
uma categoria com equivalências fracas. Sendo ela completa, fixada
uma outra categoria $\mathbf{J}$ qualquer, todo functor $F:\mathbf{J}\rightarrow\mathbf{C}$
possui limite. Consequentemente, está bem definido um functor $\lim:\mathrm{Func}(\mathbf{J};\mathbf{C})\rightarrow\mathbf{C}$,
que cada $F$ associa o seu respectivo limite. O mesmo se passa com
os colimites, uma vez que a categoria $\mathbf{C}$ é também cocompleta.
Desta forma, limites e colimites se traduzem em functores, ao passo
que, pelo exemplo 5.1.3, a classe distinguida de $\mathbf{C}$ estabelece
uma outra em $\mathrm{Func}(\mathbf{J};\mathbf{C})$.

Pelo o que discutimos anteriormente, a maneira natural de transferir
functores à categoria homotópica é tomando suas versões derivadas.
Assim, define-se o \emph{limite homotópico} de $F$, denotado por
$h\lim F$, como sendo a sua imagem pelo functor derivado à esquerda
de $\lim$. Dualmente, o \emph{colimite homotópico} de $F$, ao qual
guardamos a notação $h\mathrm{colim}F$, é a sua imagem através do
functor derivado à direita de $\mathrm{colim}$. Em suma, 
\[
h\lim F=L\lim(F)\quad\mbox{e}\quad h\mathrm{colim}F=R\mathrm{colim}(F).
\]

Observamos que, como $\mathbf{C}$ é completa e cocompleta, todo functor
que nela assume valores sempre possui extensões de Kan, de tal modo
que, diferentemente dos limites e colimites usuais, limites e colimites
homotópicos sempre existem na categoria homotópica. 

Outra observação relevante é a seguinte: limites homotópicos podem
não ser passagens ao quociente de limites na categoria ordinária.
Mais precisamente, os functores $\lim$ e $\mathrm{colim}$ em geral
não preservam equivalências fracas (não são homotópicos). Daremos
exemplos deste fato no oitavo capítulo, quando estudarmos teoria da
homotopia em $\mathbf{Top}$.

Ainda não temos uma maneira prática de computar limites homotópicos.
No entanto, como mencionado na subsecção anterior, já no próximo capítulo,
quando contarmos com classes auxiliares de morfismos, aprenderemos
a calcular extensões de Kan e, particularmente, limites homotópicos.

\subsection*{\uline{Suspensões e Loops}}

$\quad\;\,$No intuito de ilustrar a diferença entre limites e limites
homotópicos, aqui apresentamos uma classe particular de limites, os
quais só não são triviais no âmbito homotópico. Tratam-se das \emph{suspensões}
e de sua versão dual, os \emph{loops}. 

Como abstraímos no terceiro capítulo, em qualquer categoria com objetos
nulos o \emph{kernel} (resp. \emph{co-kernel}) de um morfismo $f:X\rightarrow Y$
é simplesmente o \emph{pullback} (resp. \emph{pushout}) abaixo apresentado.
$$
\xymatrix{
\ker (f) \ar[d] \ar[r] & X \ar[d]^f & \mathrm{coker}(f) & Y \ar[l] \\
\mathrm{*} \ar[r] & Y & \mathrm{*} \ar[u] & \ar[l] X \ar[u]_f }
$$

Pela lei de colagem de \emph{pullbacks} e \emph{pushouts}, a composição
sucessiva de \emph{kernels} ou de \emph{co-kernels} pode ser computada
pelos limites dos quadrados externos abaixo expostos, os quais são
evidentemente triviais.$$
\xymatrix{
\ker (\ker (f)) \ar[d] \ar[r] & \ker (f) \ar[d] \ar[r] & \mathrm{*} \ar[d] & \mathrm{coker}(\mathrm{coker})(f) & \mathrm{coker}(f) \ar[l] & \mathrm{*} \ar[l] \\
\mathrm{*} \ar[r] & X \ar[r]_f & Y & \mathrm{*} \ar[u] & Y \ar[u] \ar[l] & X \ar[u] \ar[l]^f }
$$

Em particular, se estamos em uma categoria com equivalências fracas
$\mathbf{C}$, então sua pontuação possui objetos nulos, de modo que
os limites anteriores fazem sentido e são triviais. No entanto, também
podemos olhar tais diagramas enquanto limites homotópicos. Neste caso,
eles são devidamente chamados de \emph{mapping cocone} e \emph{mapping
cone} de $f$, sendo denotados por $\mathrm{ccone}(f)$ e $\mathrm{cone}(f)$.
Diferentemente do caso usual, quando iterados, estes produzem espaços
que podem não ser triviais. 

Uma vez que \emph{pullbacks} e \emph{pushouts} homotópicos também
satisfazem leis de colagem, segue-se que os espaços $\mathrm{ccone}(\mathrm{ccone}(f))$
e $\mathrm{cone}(\mathrm{cone}(f))$, respectivamente denominados
\emph{loop} e \emph{suspensão }de $X$ (denotados por $\Omega X$
e $\Sigma X$), podem ser computados, a menos de isomorfismos, pelos
limites homotópicos abaixo.$$
\xymatrix{
\Omega X \ar[d] \ar[r] & \mathrm{*} \ar[d] & \Sigma X & \mathrm{*} \ar[l] \\
\mathrm{*} \ar[r] & X & \mathrm{*} \ar[u] & \ar[l] X \ar[u] }
$$

\chapter{Modelos}

$\quad\;\,$Algumas categorias próprias para o estudo da homotopia
contam com duas classes auxiliares de morfismos, denominados \emph{fibrações
}e \emph{cofibrações}. Diz-se que tais classes determinam um \emph{modelo}
na categoria original. As categorias nas quais um modelo foi fixado
são os objetos de estudo do presente capítulo. Enfatizamos: ainda
que nelas se conte com classes auxiliares de morfismos, estes não
são mencionadas na construção da categoria derivada. Isto significa
que, \emph{para estudar homotopia, o conceito fundamental é realmente
o de equivalência fraca}.

A escolha de um modelo numa categoria é comparada à escolha de uma
base num espaço vetorial ou à escolha de um sistema de coordenadas
numa variedade. Tratam-se, pois, de ferramentas que facilitam a demonstração
de resultados: uma vez que se sabe que a estrutura de espaço vetorial
(resp. a categoria homotópica) independe da escolha de uma base (resp.
das fibrações e cofibrações fixadas), para provar um certo resultado,
pode-se escolher a base (resp. o modelo) que melhor lhe convém.

Iniciamos o capítulo apresentando a definição de categoria modelo
e apresentando algumas propriedades que dali surgem de maneira quase
que imediata. Na segunda secção, discutimos que toda categoria modelo
admite diversos cilindros e espaços de caminhos, os quais podem ser
utilizados para construir resoluções de objetos e de morfismos. Em
seguida, verificamos que tais entidades determinam todo o modelo.
Mais precisamente, mostramos que se uma categoria admite cilindros
e espaços de caminhos naturais adjuntos, então também admite um modelo.

Na terceira secção, discutimos diversas facilidades proporcionadas
pela introdução de modelos em categorias com equivalências fracas.
Ali mostramos, por exemplo, que as categorias homotópicas, assim como
os functores derivados e, consequentemente, os limites homotópicos,
podem ser determinados de maneira versátil. Por fim, apresentamos
uma classe de functores entre categorias modelo, os quais induzem
isomorfismos em homotopia sem a necessidade de preservarem equivalências
fracas.

As referências mais utilizadas para o estudo e para a escrita foram
\cite{Dwyer_model_categories,Hirschhorn_model_categories,Hovey_model_categories,model_categories_KAN-1}.

\section{Estrutura}

$\quad\;\,$Um \emph{modelo} numa categoria $\mathbf{C}$, previamente
dotada de equivalências fracas, é consistido de outras duas classes
de morfismos, respectivamente chamadas de \emph{fibrações }e \emph{cofibrações},
escolhidas de tal modo que as seguintes propriedades são satisfeitas
(quando uma fibração ou cofibração é ao mesmo tempo uma equivalência
fraca, diz-se que ela é \emph{acíclica} ou \emph{trivial}):
\begin{enumerate}
\item cada umas das classes é preservada por retrações. Isto significa que,
para quaisquer morfismos $f:X\rightarrow X'$ e $g:Y\rightarrow Y'$,
se $g$ pertence a uma classe e existem outros morfismos entre os
respectivos objetos que tornam comutativo o diagrama abaixo, então
$f$ faz parte desta mesma classe;$$
\xymatrix{\ar@/^{0.3cm}/[rr]^{id} \ar[d]_{f} X \ar[r] & \ar[d]^{g} Y \ar[r] & X \ar[d]^{f} \\
\ar@/_{0.3cm}/[rr]_{id} X' \ar[r] & Y' \ar[r] & X'}
$$
\item cofibrações têm a propriedade de levantamento relativamente a fibrações
acíclicas. Dualmente, fibrações possuem a propriedade de levantamento
com respeito a cofibrações acíclicas. Isto é, fornecidas uma fibração
$\jmath:X\rightarrow X'$ e uma cofibração $\imath:Y\rightarrow Y'$,
se alguma delas é também uma equivalência fraca, então, para quaisquer
morfismos $f:X\rightarrow Y$ e $g:X'\rightarrow Y'$ tais que $\imath\circ f=g\circ\jmath$,
existe $h$ que torna comutativo o seguinte diagrama:$$
\xymatrix{\ar[d]_{\imath} X \ar[r]^f & X' \ar[d]^{\jmath} \\
Y \ar@{-->}[ru]^h \ar[r]_{g} & Y'}
$$
\item todo morfismo $f:X\rightarrow Y$ de $\mathbf{C}$ pode ser functorialmente
decomposto sob a forma $f=\jmath\circ\imath$ para uma cofibração
$\imath$ e uma fibração $\jmath$, em que alguma das duas é acíclica.
\end{enumerate}
$\quad\;\,$Uma \emph{categoria modelo} é simplesmente uma categoria
com equivalências fracas na qual um modelo foi fixado. Quando se quer
evidenciar que uma seta é uma cofibração (resp. fibração) num certo
modelo, escreve-se ``$\rightarrowtail$'' (resp. $\twoheadrightarrow$)
para representá-la.
\begin{example}
Se $\mathbf{C}$ é uma categoria modelo, então $\mathbf{C}^{op}$
admite um modelo natural, cujas equivalências fracas são os morfismos
$f^{op}:Y\rightarrow X$ tais que $f:X\rightarrow Y$ é equivalência
fraca em $\mathbf{C}$, tendo fibrações dadas pelas cofibrações de
$\mathbf{C}$, e vice-versa. Por conta disso, tem-se um princípio
de dualização: \emph{todo resultado demonstrado numa categoria modelo
continua verdadeiro se em seu enunciado trocamos fibrações por cofibrações
e invertemos as setas}.

\begin{example}
Cada modelo $\mathbf{C}$ se estende a um modelo na pontuação $\mathbf{C}_{*}$:
é só considerar como fibrações e cofibrações as correspondentes fibrações
e cofibrações da categoria $\mathbf{C}$ que preservam ponto base.
\end{example}
\end{example}

\subsection*{\uline{Propriedades}}

$\quad\;\,$Diretamente da definição de modelo decorrem algumas propriedades
concernentes às classes distinguidas de morfismos. Por exemplo, vê-se
que a classe das equivalências fracas em conjunto com a classe das
fibrações determinam univocamente as cofibrações como sendo os morfismos
que possuem a propriedade de levantamento de morfismos relativamente
à cofibrações acíclicas (veja o lema 1.1.10 de \cite{Hovey_model_categories}
ou terceira secção de \cite{Dwyer_model_categories}). Por dualização,
conclui-se que equivalências fracas e cofibrações determinam as fibrações.

Como facilmente se convence, estes fatos garantem que, em qualquer
categoria modelo, as identidades tanto são fibrações quanto cofibrações.
Deles também se conclui que tais classes distinguidas de morfismos
são fechadas com respeito a composição e, consequentemente, definem
subcategorias $\mathrm{Fib}$ e $\mathrm{CFib}$.

Outra propriedade obtida diretamente dos axiomas é a invariância das
fibrações por \emph{pullbacks} e, por dualização, das cofibrações
por \emph{pushouts}. Com efeito, sejam $f:X\rightarrow Y$ uma fibração,
$g:X'\rightarrow Y$ um mapa qualquer e suponhamos haver o \emph{pullback}
do par $(f,g)$. Mostraremos que o morfismo $\jmath:\mathrm{Pb}\rightarrow X$,
apresentado no primeiro diagrama abaixo, também é uma fibração. Pelo
parágrafo anterior, basta verificar que $\jmath$ possui a propriedade
de levantamento com respeito a cofibrações acíclicas. Seja $\imath:A\rightarrow B$
uma cofibração acíclica, como exposto no diagrama do meio. Afirmamos
haver a seta $h$ indicada. Juntando ambos os dois diagramas, obtém-se
$h'$. Daí, como se vê no último diagrama, da universalidade do \emph{pullback}
segue a existência de $h$.$$
\xymatrix{ &&&&&& B \ar@/^/[drr]^{h'} \ar@/_/[ddr] \ar@{-->}[dr]^h \\
\mathrm{Pb} \ar[d]_{\jmath} \ar[r] & X \ar[d]^f & A \ar[d]_{\imath} \ar[r] & \mathrm{Pb} \ar[d]^{\jmath} & A \ar[d]_{\imath} \ar[r] & X \ar[d]^f && \mathrm{Pb} \ar[d]_{\jmath} \ar[r] & X \ar[d]^f \\
X' \ar[r]_g & Y & B \ar@{-->}[ru]^h \ar[r] & X' & B \ar[r] \ar[ru]^{h'} & Y && X' \ar[r]_g & Y }
$$
\begin{example}
Em virtude da invariância das fibrações por \emph{pullbacks}, segue-se
que, para quaisquer objetos $X_{1}$ e $X_{2}$, as projeções $\pi_{i}:X_{1}\times X_{2}\rightarrow X_{i}$
são fibrações. Da mesma forma, a invariância das cofibrações por \emph{pushouts}
garante que as inclusões $\imath_{j}:X_{j}\rightarrow X_{1}\oplus X_{2}$
são cofibrações.
\end{example}

\section{Resoluções}

$\quad\;\,$Diz-se que um objeto $X$ numa categoria modelo $\mathbf{C}$
é \emph{fibrante} ou \emph{cofibrante}\footnote{Na literatura inglesa, tais expressões correspondem, respectivamente,
a \emph{fibrant} e \emph{cofibrant}.} na medida que o morfismo $X\rightarrow*$ é fibração ou $\varnothing\rightarrow X$
é cofibração. Evidentemente, todo objeto admite ao menos uma resolução
fibrante e outra cofibrante. Com efeito, pelo axioma da decomposição,
os referidos morfismos podem sempre ser functorialmente decompostos
na forma 
\begin{equation}
\varnothing\twoheadrightarrow\mathcal{C}X\rightsquigarrow X\quad\mbox{e}\quad X\rightsquigarrow\mathcal{F}X\rightarrowtail*\quad\mbox{para certos}\quad\mathcal{C}X\;\mbox{e}\;\mathcal{F}X.\label{troca_fibrante}
\end{equation}

Vejamos que qualquer morfismo $f:X\rightarrow Y$ também admite resoluções
fibrantes e cofibrantes. Mais precisamente, mostremos a existência
de resoluções $X'\rightsquigarrow X$ e $X\rightsquigarrow X''$ que,
quando compostas com $f$, produzem fibrações e cofibrações. Observamos
que, diretamente da invariância das fibrações e cofibrações por \emph{pullbacks},
\emph{pushouts} e composições, segue-se que cada fibração $\mathrm{path}(Y)\twoheadrightarrow Y$
e cada cofibração $X\rightarrowtail\mathrm{cyl}(X)$ determinam as
resoluções $X'$ e $X''$ procuradas. $$
\xymatrix{X' \ar@{~>}[r] \ar@{->>}[d] & X \ar[d]^f && X'' & Y \ar@{~>}[l] \\
\mathrm{path}(Y) \ar@{->>}[r] & Y && \mathrm{cyl}(X) \ar[u] & X \ar[u]_f \ar[l] }
$$

Diante disso, nos resumimos a garantir a existência de fibrações $\mathrm{path}(Y)$
que assumem valores em $Y$, assim como de cofibrações $\mathrm{cyl}(X)$
partindo de $X$. Particularmente, como projeções são fibrações, ao
passo que inclusões são cofibrações, basta que encontremos mapas $X\oplus X\rightarrowtail\mathrm{cyl}(X)$
e $\mathrm{path}(Y)\twoheadrightarrow Y\times Y$. Estes existem,
pois $\nabla:X\oplus X\rightarrow X$ (resp.$\Delta:Y\rightarrow Y\times Y$)
pode sempre ser decomposto na forma $\jmath\circ\imath$, em que $\imath$
é cofibração (resp. cofibração acíclica) e $\jmath$ é fibração acíclica
(resp. fibração). Em cada um dos casos, $\mathrm{cyl}(X)$ e $\mathrm{path}(Y)$
são os objetos intermediários da decomposição, sendo chamados de \emph{cilindro}
de $X$ e de \emph{espaço de caminhos} de $Y$.

\subsection*{\uline{Cilindros Naturais}}

$\quad\;\,$Na subsecção anterior, vimos que uma categoria modelo
possui diversos cilindros e espaços de caminhos. Veremos, agora, uma
recíproca de tal condição. Mais precisamente, mostraremos que se uma
categoria admite cilindros e espaços de caminhos naturais (os quais
supomos adjuntos), então estes ali determinam um modelo.

A ideia é essencialmente a seguinte: cilindros naturais fornecem cofibrações
e equivalências fracas, ao passo que espaços de caminhos naturais
fornecem fibrações e equivalências fracas. Se estes são adjuntos,
então as correspondentes noções de equivalência fraca coincidem e
então podemos juntar as informações por eles fornecidas de modo a
produzir um modelo. Vejamos isto com um pouco mais de cuidado.

Pois bem, se $C:\mathbf{C}\rightarrow\mathbf{C}$ é cilindro natural
numa categoria $\mathbf{C}$, então ele ali induz uma correspondente
noção de homotopia, a qual nos serve como uma classe distinguida.
Por sua vez, em uma categoria modelo, as cofibrações são os morfismos
que possuem a propriedade de extensão com respeito a uma outra classe
de morfismos. 

A ideia, então, é definir $C$-cofibrações como sendo os morfismos
que têm a propriedade de extensão relativamente às $C$-homotopias.
De maneira mais precisa, diz-se que $\imath:X\rightarrow Y$ é uma
\emph{$C$-cofibração }quando, para todo morfismo $f:Y\rightarrow Z$
e toda $C$-homotopia $h$ partindo de $\imath$ (isto é, satisfazendo
$h\circ\imath_{0}=\imath$), é possível obter uma outra $C$-homotopia
$H:CY\rightarrow Z$, esta partindo de $f$, a qual pode ser escolhida
que modo que ao se restringir a $CX$ coincida com $h$. Em outras
palavras, vale a comutatividade do primeiro dos diagramas abaixo.$$
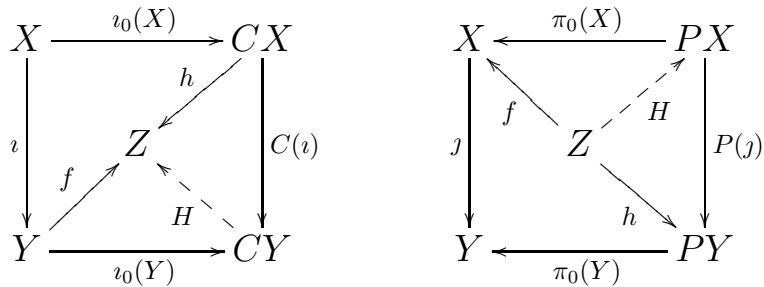
\xymatrix{X \ar[dd]_{\imath } \ar[rr]^{\imath _0 (X)} && CX \ar[ld]_h \ar[dd]^{C(\imath )} && X \ar[dd]_{\jmath} && PX \ar[ll]_{\pi _0 (X)} \ar[dd]^{P(\jmath )} \\
& Z &&&& Z \ar[ul]^f \ar[dr]_h \ar@{-->}[ur]_H \\
Y \ar[ru]^f \ar[rr]_{\imath _0 (Y)} && CY \ar@{-->}[lu]^H && Y && PY \ar[ll]^{\pi _0 (Y)} }
$$

De outro lado, se $P$ é espaço de caminhos natural em $\mathbf{C}$,
então ele também induz uma noção de homotopia, a qual utilizamos como
classe distinguida. Posto isto, definimos as $P$-fibrações como sendo
mapas que têm a propriedade de levantamento com respeito às $P$-homotopias
(segundo dos diagramas acima).

Cabem, aqui, duas observações:
\begin{enumerate}
\item para qualquer $X$, o correspondente $\imath_{0}(X)$ é tanto $C$-cofibração
quanto equivalência homotópica (isto é, um isomorfismo na categoria
homotópica $\mathscr{H}_{C}(\mathbf{C})$). Assim, eles funcionam
como as ``cofibrações acíclicas''. Em particular, disto segue que
todo $X\in\mathbf{C}$ serve como um ``objeto cofibrante'';
\item dualmente, os $\pi_{0}(Y)$ são ao mesmo tempo $P$-fibrações e isomorfismos
em $\mathscr{H}^{P}(\mathbf{C})$. Consequentemente, eles servem de
análogos às ``fibrações acíclicas'' e todo objeto admite comportamento
semelhante aos ``cofibrantes''.
\end{enumerate}
$\quad\;\,$Suponhamos, agora, a existência de uma adjunção entre
os functores $P$ e $C$. Neste caso, as respectivas noções coincidem
e as $C$-cofibrações e as $P$-fibrações admitem uma nova caracterização.
Com efeito, $\imath:X\rightarrow Y$ será uma $C$-cofibrações se,
e só se, para todo $f:X\rightarrow Y$ e toda $P$-homotopia $h$
partindo de $f$ existe uma outra $P$-homotopia $H$, agora partindo
de $\imath$, que ao ser restrita a $PX$ coincide com $h$ (veja
o primeiro dos diagramas abaixo). Em outras palavras, um mapa é uma
cofibração se, e só se, possui a propriedade de extensão com respeito
à $\pi_{0}(Y)$ e, portanto, pela observação anterior, com respeito
a uma ``fibração acíclica''.$$
\xymatrix{X \ar[d]_{\imath} \ar[r]^h & PZ \ar[d]^{\pi _0 (Z)} && Z \ar[d]_{\imath _0 (Z)} \ar[r]^f & X \ar[d]^{\jmath} \\
Y \ar@{-->}[ru]^H \ar[r]_f & Z && CZ \ar@{-->}[ru]^H \ar[r]_h & Y }
$$

De maneira análoga, a adjunção entre $P$ e $C$ nos permite caracterizar
as $P$-fibrações como sendo os morfismos que têm a propriedade de
levantamento com respeito à $\imath_{0}(X)$ e, portanto, com respeito
a uma ``cofibração acíclica''.

Diante disso, para mostrar que $C$-cofibrações, $P$-fibrações e
as correspondentes equivalências fracas (que coincidem sob a hipótese
de adjunção) definem um verdadeiro modelo em $\mathbf{C}$, resta
verificar que cada uma destas classes de morfismos é invariante por
retração, ao passo que todo morfismo em $\mathbf{C}$ admite uma decomposição
em termos de $C$-cofibrações seguidas de $P$-fibrações, sendo alguma
delas também uma equivalência fraca.

No que tange à invariância por retrações, observamos que a caracterização
que obtivemos para as $F$-cofibrações e para as $P$-fibrações nos
permite utilizar da mesma estratégia empregada na primeira secção
para demonstrar que cofibrações e fibrações são respectivamente invariantes
por \emph{pullbacks} e\emph{ }por \emph{pushouts}. No entanto, como
logo se convence, se uma classe de morfismos é invariante por \emph{pullbacks}
ou por \emph{pushouts}, então também é invariante por retrações.

Finalmente, no que diz respeito à existência de decomposições, elas
seguem do lema abaixo, usualmente chamado de \emph{lema de fatoração.
}Vamos enunciá-lo para fibrações, mas há uma versão dual para cofibrações
cuja demonstração é análoga.$$
\xymatrix{X \ar[rd]^{\imath} \ar@/^/[rrd]^{id} \ar@/_/[rdd]_{w\circ f} \\
& X' \ar@{~>}[r] \ar@{->>}[d] & X \ar[d]^f \\
& \mathrm{Path}(Y) \ar@{->>}[r]_-{\kappa} & Y }
$$
\begin{lem}
Seja $\mathbf{C}$ uma categoria dotada das noções de fibração, equivalência
fraca e de espaços de caminhos. Se as fibrações são invariantes por
pullbacks e todo objeto $Y$ é fracamente equivalente ao seu espaço
de caminhos $\mathrm{Path}(Y)$, então qualquer morfismo $f$ se decompõe
sob a forma $\jmath\circ\imath$, em que $\imath$ é equivalência
fraca e $\jmath$ é fibração.
\end{lem}
\begin{proof}
Iniciamos relembrando algo que já foi visto na subsecção anterior:
se $\mathrm{Path}(Y)$ é espaço de caminhos de $Y$, então toda fibração
$\kappa:\mathrm{Path}(Y)\twoheadrightarrow Y$ determina uma resolução
fibrante de qualquer $f:X\rightarrow Y$ (é só utilizar a invariância
das fibrações por \emph{pullbacks}). Agora, se $\omega:Y\rightarrow\mathrm{Path}(Y)$
é equivalência fraca, então, por universalidade, existe o mapa $\imath:X\rightarrow\mathrm{Pb}$
apresentado no diagrama acima. Em particular, a propriedade 2-de-3
faz dele uma equivalência fraca. Portanto, tomando $\jmath$ igual
à composição de $\kappa$ com o morfismo vertical do \emph{pullback},
vê-se que $f=\jmath\circ\imath$, fornecendo a decomposição procurada.
\end{proof}
$\qquad\qquad\qquad\qquad\qquad\qquad\quad$

\section{Utilidade}

$\quad\;\,$Ao final da secção anterior, mostramos que os cilindros
e os espaços de caminhos naturais determinam, em certo sentido, toda
a estrutura de um modelo. Assim, espera-se que muito do que se utiliza
no estudo das categorias próprias para o estudo da homotopia possa
ser determinado somente pela hipótese de existência de objetos fibrantes
e cofibrantes, assim como de cilindros e espaços de caminhos. 

Em contrapartida, se já supomos a presença de um modelo, então tais
classes de objetos existem aos montes. Como consequência, espera-se
que a escolha prévia de um modelo numa categoria com equivalências
fracas facilite diversos cálculos. Particularmente, espera-se que
estes sejam \emph{versáteis}. Afinal, eles dependem apenas da \emph{existência}
de objetos fibrantes, cilindros, etc, e não da maneia com estes foram
\emph{escolhidos} para realizar o referido cálculo.

Nesta secção, iniciamos confirmando a suspeita para as categorias
homotópicas. Em seguida, confirmamo-la também para functores derivados.
Por fim, tratamos de uma classe particular de adjunções, denominadas
\emph{adjunções de Quillen}, as quais induzem isomorfismos em homotopia
mesmo ser preservar equivalências fracas. 

\subsection*{\uline{Categoria Homotópica}}

$\quad\;\,$Diz-se que $f,g;X\rightarrow Y$, são \emph{formalmente
homotópicos à esquerda} quando existe $\mathrm{cyl}(X)$ no qual está
definido um morfismo $H:\mathrm{cyl}(X)\rightarrow Y$ que torna comutativa
a primeira parte do diagrama da próxima página. Nele, $n_{0}$ e $n_{1}$
denotam, respectivamente, a composição da cofibração $X\oplus X\twoheadrightarrow\mathrm{cyl}(X)$
com as inclusões na primeira segunda entrada de $X\oplus X$. Assim,
$f$ e $g$ são formalmente homotópicos à esquerda se $\nabla\circ f\oplus g$,
denotado por $f+g$, estende-se a algum cilindro.

De maneira semelhante, fala-se que $f$ e $g$ são \emph{formalmente
homotópicos à direita} quando é possível obter um espaço de caminhos
de $Y$ e um morfismo $h:X\rightarrow\mathrm{path}(Y)$ tal que a
segunda parte do diagrama acima é comutativa. Ali, $pr_{0}$ e $pr_{1}$
representam a composição da fibração $\mathrm{path}(X)\rightarrowtail X\times X$
com a projeção de $X\times X$ na primeira e segunda entrada. Assim,
os morfismos $f$ e $g$ são formalmente homotópicos à direita se
$f\times g\circ\Delta$, representado por $f\cdot g$, pode ser levantado
a algum espaço de caminhos.$$
\xymatrix{\ar[dd]_{n_{0} \oplus n_{1}} X\oplus X \ar[r]^{f \oplus g} & Y\oplus Y \ar[rd]^{\nabla} && Y \times Y \ar[r]^{f \times g} & X \times X \\
& & Y \ar[ru]^{\Delta} \ar[rd]_{h} \\
\mathrm{cyl}(X)\oplus \mathrm{cyl}(X) \ar[r]_-{\nabla} & \mathrm{cyl}(X) \ar[ru]_{H} && \mathrm{path}(X) \ar[r]_-{\Delta} & \mathrm{path}(X) \times \mathrm{path}(X) \ar[uu]_{pr_0 \times pr_1}  }
$$

Se $X$ é cofibrante, então a relação que identifica morfismos formalmente
homotópicos à esquerda é de equivalência no conjunto $\mathrm{Mor}(X;Y)$.
Dualmente, se $Y$ é fibrante, então a relação responsável por identificar
morfismos formalmente homotópicos à direita é de equivalência neste
mesmo conjunto. Além disso, se as duas condições são simultaneamente
satisfeitas, então ambas as equivalências coincidem. Isto é, se $X$
é cofibrante e $Y$ é fibrante, então $f,g:X\rightarrow Y$ são formalmente
homotópicos à esquerda se, e só se, o são à direita. Em tal caso,
diz-se simplesmente que eles são \emph{formalmente homotópicos} e
escreve-se $f\cong g$ (não confundir com a notação $f\simeq g$,
utilizada para denotar $1$-morfismos homotópicos em $2$-categorias).
Para detalhes, sugerimos a quarta secção de \cite{Dwyer_model_categories}
ou as secções 7.3-7.5 de \cite{Hirschhorn_model_categories}.

Em virtude do exposto, para quaisquer objetos $X$ e $Y$ numa categoria
modelo $\mathbf{C}$, o respectivo conjunto $\mathrm{Mor}_{\mathbf{C}}(\mathcal{F}\mathcal{C}X;\mathcal{F}\mathcal{C}Y)$
está dotado de uma relação de equivalência, a qual passa ao quociente
e define uma nova categoria $\mathcal{H}\mathbf{C}$, denominada \emph{categoria
homotópica formal}. Se dois objetos $X,Y\in\mathbf{C}$ são ali isomorfos,
escreve-se $X\cong Y$ e diz-se que possuem o mesmo \emph{tipo de
homotopia} \emph{formal}. Isto significa que é possível obter morfismos
$f:\mathcal{F}\mathcal{C}X\rightarrow\mathcal{F}\mathcal{C}Y$ e $g:\mathcal{F}\mathcal{C}Y\rightarrow\mathcal{F}\mathcal{C}X$,
denominados \emph{equivalência homotópica formal} e \emph{inversa
homotópica} \emph{formal}, cujas composições satisfazem $g\circ f\cong id_{X}$
e $f\circ g\cong id_{Y}$.

Tem-se um functor sobrejetivo $\pi:\mathbf{C}\rightarrow\mathcal{H}\mathbf{C}$,
que é identidade em objetos e que a cada morfismo $f$ entre $X$
e $Y$ faz corresponder a sua \emph{classe de homotopia formal} $[f]$
(isto é, o conjunto dos morfismos homotópicos a $\mathcal{FC}f$).
A coleção de todas estas classes é denotada por $[X;Y]_{\mathbf{C}}$.
Em outras palavras, costuma-se escrever $[X;Y]_{\mathbf{C}}$ ao invés
de $\mathrm{Mor}_{\mathcal{H}\mathbf{C}}(X;Y)$.

\subsection*{\uline{Equivalência}}

$\quad\;\,$Mostraremos que, para qualquer categoria modelo, sua categoria
homotópica e sua categoria homotópica formal são equivalentes. Para
tanto, necessitaremos do seguinte resultado, aqui denominado \emph{teorema
de Whitehead} por se tratar da generalização de um importante resultado
no estudo da homotopia clássica que leva este mesmo nome:
\begin{prop}
Um morfismo entre objetos que são simultaneamente fibrantes e cofibrantes
é equivalência fraca se, e somente se, é equivalência homotópica formal.
\end{prop}
\begin{proof}
Mostraremos que as fibrações acíclicas assumindo valores em objetos
cofibrantes são são equivalências homotópicas. Por dualização, o mesmo
poderá ser concluido sobre as cofibrações acíclicas definidas em objetos
fibrantes. Uma vez que qualquer equivalência fraca $f:X\rightarrow Y$
se decompõe em termos de uma cofibração acíclica $\imath:X\rightarrow X'$
seguida de uma fibração $\jmath:X'\rightarrow Y$, a propriedade 2-de-3
garantirá que $\jmath$ também é acíclica. Em particular, se $f$
estiver definida entre objetos que são ao mesmo tempo fibrantes e
cofibrantes, então $X'$ também será fibrante e cofibrante, de modo
que $\imath$ e $\jmath$ (e, consequentemente, $f=\jmath\circ\imath$)
serão equivalências homotópicas. Seja, então, $f:X\rightarrow Y$
uma fibração acíclica, com $Y$ cofibrante. Evidentemente, ela possui
a propriedade de levantamento com respeito a $\varnothing\rightarrow Y$.
Daí, existe um mapa $g:Y\rightarrow X$ que deixa comutativo o primeiro
dos diagramas abaixo.$$
\xymatrix{ \varnothing \ar[r] \ar[d] & X' \ar[d]^f && X\oplus X \ar[d]_{\imath} \ar[rr]^-{g\circ f + id} && X \ar[d]^{f} \\
Y \ar[r]_{id} \ar@{-->}[ru]^g & Y && \mathrm{cyl}(X) \ar@{-->}[rru]^H \ar[rr]_{f\circ \jmath} && Y}
$$De imediato, isto nos dá $f\circ g=id_{Y}$. Afirmamos que $g\circ f\simeq id_{X'}$,
o que nos mostrará que $f$ é equivalência homotópica, tendo $g$
como inversa. Dado um cilindro $\mathrm{cyl}(X)$ qualquer, digamos
obtido decompondo $\nabla$ sob a forma $\jmath\circ\imath$, observamos
que $f:X\rightarrow Y$ admite propriedade de levantamento relativamente
a $\imath:X\oplus X\twoheadrightarrow\mathrm{cyl}(X)$. Diante disso,
há homotopia $H:\mathrm{cyl}(X)\rightarrow X$ que torna comutativo
o segundo dos diagramas abaixo, garantindo o que havíamos afirmado.
A recíproca segue a mesma linha, mas utiliza de mais passos. Para
uma prova detalhada, veja o lema 4.24 à página 23 de \cite{Dwyer_model_categories}.
\end{proof}
De posse do teorema de Whitehead, vamos mostrar que $\pi:\mathbf{C}\rightarrow\mathcal{H}\mathbf{C}$
é localização de $\mathbf{C}$ com respeito a classe das equivalências
fracas. Por unicidade, isso implicará em $\mathcal{H}\mathbf{C}\simeq\mathrm{Ho}(\mathbf{C})$,
que é a equivalência que procuramos. Como consequência, isto nos mostrará
que \emph{dois objetos numa categoria modelo tem o mesmo tipo de homotopia
se, e somente se, possuem o mesmo tipo de homotopia formal}. Particularmente,
se $\mathbf{C}$ é $2$-categoria dotada de um modelo, então \emph{dois
de seus morfismos serão homotópicos quando, e só quando, o forem formalmente}.

Iniciamos observando que, se $f:X\rightarrow Y$ é equivalência fraca,
então $\mathcal{FC}f$ é equivalência homotópica e, portanto, $\pi(f)$
é isomorfismo. Seja $F:\mathbf{C}\rightarrow\mathbf{D}$ outro functor
satisfazendo esta mesma propriedade. Afirmamos haver $F_{*}:\mathcal{H}\mathbf{C}\rightarrow\mathbf{D}$
cumprindo $F_{*}\circ\pi=F$. Em objetos, $F_{*}$ é a identidade.
Para definí-lo em morfismos, observemos o segundo dos diagramas abaixo,
o qual foi obtido aplicando $F$ ao primeiro deles. Devidamente motivados,
colocando
\[
F_{*}([f])=F(\jmath_{y})\circ F(\imath_{y})^{-1}\circ F(f)\circ F(\imath_{x})\circ F(\jmath_{x})^{-1}
\]
vê-se que $F_{*}\circ\pi=F$ é trivialmente satisfeita. Por sua vez,
a comutatividade do segundo diagrama nos diz que todo mapa $F(f)$
na imagem do functor $F$ deve ter precisamente a forma lá apresentada,
o que garante a unicidade de $F_{*}$.$$
\xymatrix{X \ar[d]_{f} & \mathcal{C}X \ar[r]^{\jmath _x} \ar[l]_{\imath _x} & \mathcal{FC}X \ar[d]^{[f]} && F(X) \ar[d]_{F(f)} & F(\mathcal{C}X) \ar[r]^{F(\jmath _x)} \ar[l]_{F(\imath _x)} & F(\mathcal{FC}X) \ar[d]^{F([f])} \\
Y & \mathcal{C}Y \ar[r]^{\jmath _y} \ar[l]_{\imath _y} & \mathcal{FC}Y && F(Y) & F(\mathcal{C}Y) \ar[r]_{F(\jmath _y)} \ar[l]^{F(\imath _y)} & F(\mathcal{FC}Y) }
$$
\begin{rem}
Numa categoria com equivalências fracas, nem todo isomorfismo de $\mathrm{Ho}(\mathbf{C})$
provém de equivalências fracas. Quando este é o caso, diz-se que $\mathbf{C}$
é \emph{saturada} com respeito a $\mathrm{Weak}$. Por exemplo, toda
$2$-categoria é saturada com respeito às equivalências homotópicas.
O resultado aqui demonstrado nos garante que qualquer categoria modelo
é saturada relativamente às equivalências homotópicas formais.
\end{rem}

\subsection*{\uline{Functores Derivados}}

$\quad\;\,$Uma vez que categorias homotópicas independem de fibrações
e cofibrações, espera-se que a escolha de um determinado modelo não
produza condições \emph{necessárias} à existência de functores derivados.
Vejamos, no entanto, que cada modelo determina condições \emph{suficientes}
à existência de tais functores.
\begin{prop}
Seja $F:\mathbf{C}\rightarrow\mathbf{C}'$ um functor entre categorias
com equivalências fracas. Para que $LF$ exista, basta haver um modelo
em $\mathbf{C}$ com respeito ao qual $F$ mapeia cofibração acíclicas
entre objetos cofibrantes em equivalências fracas. Neste caso, $LF(X)\rightsquigarrow F(\mathcal{C}X)$
seja qual for a resolução cofibrante $\mathcal{C}$.
\end{prop}
\begin{proof}
Seja $\mathcal{C}F$ o functor que a cada $X$ associa a imagem $F(\mathcal{C}X)$
de $F$ por uma resolução cofibrante, e que a cada morfismo $f$ devolve
$F(\mathcal{C}f)$. Se $f$ é uma equivalência fraca, então $\mathcal{C}f$
é cofibração acíclica entre objetos cofibrantes, de modo que, pela
hipótese, $\mathcal{C}F(f)$ também é equivalência fraca. Daí, $p'\circ\mathcal{C}F$
mapeia equivalências fracas em isomorfismos, de modo que, por universalidade
da categoria homotópica de $\mathbf{C}$, existe um único functor
$LF:\mathrm{Ho}(\mathbf{C})\rightarrow\mathrm{Ho}(\mathbf{C}')$,
tal que $p\circ LF=p'\circ\mathcal{C}F$. Tem-se uma transformação
natural $\xi:\mathcal{C}F\rightarrow F$, que a cada $X$ faz corresponder
a imagem de $F$ por $\mathcal{C}X\rightsquigarrow X$. Portanto,
há também uma transformação natural entre $p'\circ\mathcal{C}F$ (que
nada mais é que $p\circ LF$) e $p'\circ F$. Com ela, $LF$ é universal
e, portanto, functor derivado à esquerda de $F$.
\end{proof}
Via dualização, obtém-se um resultado estritamente análogo ao anterior
para functores derivados à direita. Com efeito, para garantir a existência
$RF$ , é suficiente exibir um modelo em $\mathbf{C}$ no qual $F$
leva fibrações acíclicas entre objetos fibrantes em equivalências
fracas. Nele, tem-se $RF(X)\rightsquigarrow F(\mathcal{F}X)$ independentemente
da escolha da resolução fibrante $\mathcal{F}$.

\subsection*{\uline{Limites Homotópicos}}

$\quad\;\,$Os limites homotópicos são os functores derivados à esquerda
de $\lim:\mathrm{Func}(\mathbf{J};\mathbf{C})\rightarrow\mathbf{C}$.
Até o momento, apesar de sabermos que eles sempre existem, não temos
uma maneira prática de calculá-los. Com isto em mente, observemos
o seguinte: suponhamos que para uma certa categoria $\mathbf{J}$
haja um modelo em $\mathrm{Func}(\mathbf{J};\mathbf{C})$. Neste caso,
pela unicidade das extensões de Kan e pela última proposição, tem-se
\[
\mathrm{hlim}(F)\rightsquigarrow\lim(\mathcal{C}F)\quad\mbox{e}\quad\mathrm{hcolim}(F)\rightsquigarrow\mathrm{colim}(\mathcal{F}F)
\]
sejam quais forem o functor $F:\mathbf{J}\rightarrow\mathbf{C}$,
a resolução cofibrante $\mathcal{C}$ e a resolução fibrante $\mathcal{F}$.
Assim, sob tal hipótese, a menos de equivalências fracas, o limite
e o colimite homotópico de $F$ podem ser obtidos em termos de limites
e colimites usuais. 

Posto isto, procuremos por modelos em $\mathrm{Func}(\mathbf{J};\mathbf{C})$.
Vimos que, se $\mathbf{C}$ é própria para homotopia, então $\mathrm{Func}(\mathbf{J};\mathbf{C})$
também o é, qualquer que seja a categoria $\mathbf{J}$: suas equivalências
fracas são as transformações naturais $\xi$ tais que $\xi(X)$ é
equivalência fraca de $\mathbf{C}$ para todo $X$. Assim, a ideia
imediata é propor como fibrações ou cofibrações de $\mathrm{Func}(\mathbf{J};\mathbf{C})$
as transformações $\xi$ em que $\xi(X)$ é fibração ou cofibração
de $\mathbf{C}$. 

Quando a categoria $\mathbf{C}$ satisfaz condições favoráveis (veja
a secção A.3.3 de \cite{Higher_topos_LURIE}), tais suspeitas se concretizam.
Mais precisamente, existem dois modelos em $\mathrm{Func}(\mathbf{J};\mathbf{C})$,
chamados de \emph{modelo injetivo }e de \emph{modelo projetivo}, sendo
tais que: no primeiro (resp. segundo) deles as fibrações (resp. cofibrações)
são precisamente as transformações $\xi$, em que $\xi(X)$ é fibração
(resp. cofibração) de $\mathbf{C}$ para cada $X$.

\subsection*{\uline{Sequências Longas}}

$\quad\;\,$A menos de equivalências fracas, algumas classes de limites
homotópicos podem ser calculados sem a necessidade de haver um modelo
em $\mathrm{Func}(\mathbf{J};\mathbf{C})$. Este é o caso, por exemplo,
dos \emph{pullbacks} e dos \emph{pushouts} homotópicos. Com efeito,
para calcular o \emph{pullback} homotópico entre $f:X\rightarrow Y$
e $g:X'\rightarrow Y$, basta substituir cada um dos objetos e um
dos mapas ($f$ ou $g$) por uma resolução fibrante, e então computar
o limite do diagrama obtido. Numa categoria cujos objetos são todos
fibrantes, basta realizar a troca de um dos mapas. Condições duais
são válidas para os \emph{pushouts} homotópicos. Remetemos o leitor
à secção A.2.4 de \cite{Higher_topos_LURIE} e à secção 13.3 de \cite{Hirschhorn_model_categories}.$$
\xymatrix{ & X' \ar[d]^g && \mathcal{C}X' \ar[d] &&& \mathcal{C}X' \ar[d]  \\
X \ar[r]_f & Y & \mathcal{C} \ar[r] & \mathcal{C}Y & &\mathrm{path}(\mathcal{C}Y) \ar[d] \ar[r] & \mathcal{C}Y \\
&&&& \mathcal{C}X \ar[r] & \mathcal{C}Y}
$$

Acima ilustramos os passos concernentes à substituição dos objetos
e de um dos morfismos por suas versões fibrantes (não há diferença
em substituir $f$ ou $g$, como logo se convence). É possível realizar
o cálculo do diagrama final de duas maneiras, apresentadas abaixo.
Na primeira delas, efetua-se dois \emph{pullbacks}. Na segunda, efetua-se
três. O objeto $\mathrm{ccyl}(f)$ ali apresentado chama-se \emph{mapping
path space} de $f$. O objeto dual, que surge durante o cálculo de
\emph{pushouts} homotópicos, é chamado de \emph{mapping cylinder}
e denotado por $\mathrm{cyl}(f)$.
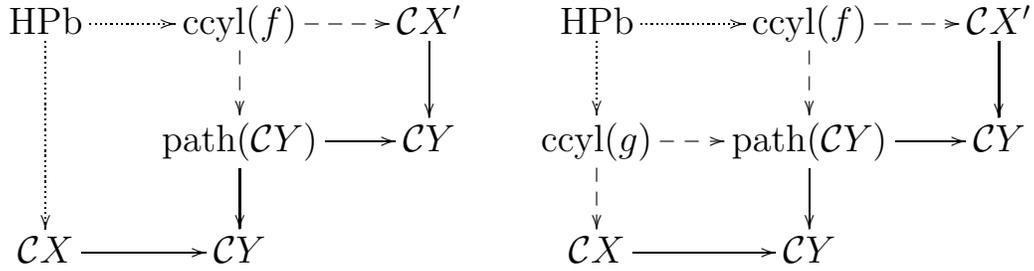
$$
\xymatrix{\mathrm{HPb} \ar@{.>}[r] \ar@{.>}[dd]& \mathrm{ccyl}(f) \ar@{-->}[r] \ar@{-->}[d] & \mathcal{C}X' \ar[d] & \mathrm{HPb} \ar@{.>}[r] \ar@{.>}[d]& \mathrm{ccyl}(f) \ar@{-->}[r] \ar@{-->}[d] & \mathcal{C}X' \ar[d]  \\
&\mathrm{path}(\mathcal{C}Y) \ar[d] \ar[r] & \mathcal{C}Y & \mathrm{ccyl}(g) \ar@{-->}[r] \ar@{-->}[d] & \mathrm{path}(\mathcal{C}Y) \ar[d] \ar[r] & \mathcal{C}Y  \\
\mathcal{C}X \ar[r] & \mathcal{C}Y && \mathcal{C}X \ar[r] & \mathcal{C}Y}
$$

Como consequência, \emph{mapping cones}, \emph{mapping cocones}, suspensões
e \emph{loops} podem todos ser comutados através de \emph{pullbacks}
e \emph{pushouts} usuais. Por universalidade, existem, então, os mapas
$-\Omega f$ e $-\Sigma f$ abaixo apresentados.$$
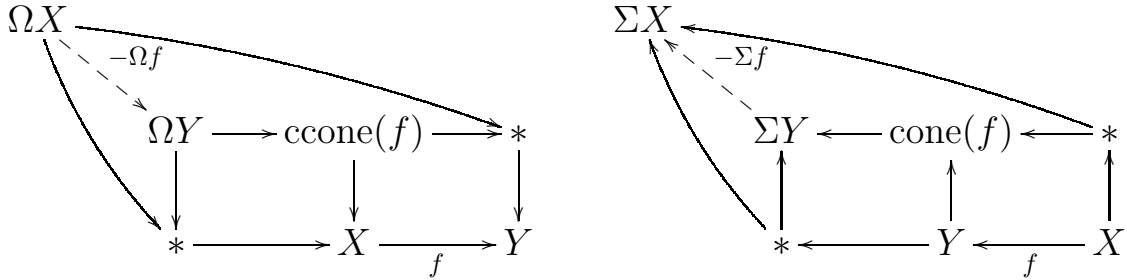
\xymatrix{\Omega X \ar@{-->}[rd]^{- \Omega f} \ar@/_/[rdd] \ar@/^/[drrr] &&&& \Sigma X   \\
& \Omega Y \ar[d] \ar[r] & \mathrm{ccone}(f) \ar[d] \ar[r] & \mathrm{*} \ar[d] && \ar@{-->}[lu]_{- \Sigma f} \Sigma Y   & \ar[l] \mathrm{cone}(f)   & \ar[l] \mathrm{*} \ar@/_/[ulll]  \\
& \mathrm{*} \ar[r] & X \ar[r]_f & Y && \ar@/^/[luu] \ar[u] \mathrm{*} & \ar[l] Y \ar[u] & X \ar[l]^f \ar[u] }
$$

Iterando o primeiro diagrama, conclui-se que qualquer morfismo $f:X\rightarrow Y$,
definido na pontuação $\mathbf{C}_{*}$ de uma categoria modelo $\mathbf{C}$,
a qual supomos ser formada somente por objetos fibrantes, origina
a seguinte sequência: $$
\xymatrix{\cdots \ar[r] & \Omega \mathrm{ccone}(f) \ar[r] & \Omega X \ar[r]^{-\Omega f} & \Omega Y \ar[r] & \mathrm{ccone}(f) \ar[r] & X \ar[r]^f & Y}
$$ 

A ela se dá o nome de \emph{sequência de fibrações} $f$. De maneira
análoga, agora supondo que a categoria em questão é formada de objetos
cofibrantes, ao se iterar o \emph{pushout} representado no segundo
dos diagramas acima, encontra-se a \emph{sequência de cofibrações
de} $f$: $$
\xymatrix{
X \ar[r]^f & Y \ar[r] & \mathrm{cone}(f) \ar[r] & \Sigma X  \ar[r]^{-\Sigma f} & \Sigma Y \ar[r] & \Sigma \mathrm{cone}(f) \ar[r] & \cdots}
$$

Observamos que, para qualquer que seja o objeto $Z\in\mathbf{C}_{*}$,
aplicando $[Z;-]_{\mathbf{C}_{*}}$ na sequência de fibrações de um
morfismo $f:X\rightarrow Y$, obtém-se uma sequência exata em $\mathbf{Set}$.
Dualmente, se aplicamos $[-;Z]_{\mathbf{C}_{*}}$ na sequência de
cofibrações, também obtemos uma sequência exata de conjuntos. Para
detalhes, veja a secção 6.5 de \cite{Hovey_model_categories}.

\subsection*{\uline{Adjunções de Quillen}}

$\quad\;\,$Procura-se por uma noção de isomorfismo entre categorias
modelo. Uma vez que a estrutura fundamental subjacente é a categoria
homotópica, a correspondente noção de isomorfismo deverá implicar
numa equivalência a nível das categorias homotópicas. Por sua vez,
não se deve esperar a validade recíproca, pois a categoria homotópica
descreve apenas aproximadamente a categoria ordinária. 

O \emph{insight} imediato é considerar os functores homotópicos. Esta,
no entanto, como já discutido anteriormente, é uma exigência muito
restritiva. A próxima ideia é então considerar functores $F:\mathbf{C}\rightarrow\mathbf{C}'$
tais que $LF$ existe e é equivalência entre as categorias homotópicas.
A inversa de $F$ deveria ser sua versão dual: um functor $G:\mathbf{C}'\rightarrow\mathbf{C}$
tal que $RG$ existe e é precisamente a inversa fraca de $LF$. Espera-se
que $G$ seja univocamente determinado por $F$. Isto se obtém exigindo
que ele seja adjunto de $F$.

Depois destas considerações, definamos: duas categorias modelo $\mathbf{C}$
e $\mathbf{C}'$ são ditas \emph{Quillen adjuntas} quando existem
functores $F:\mathbf{C}\rightarrow\mathbf{C}'$ e $G:\mathbf{C}'\rightarrow\mathbf{C}$,
os quais são adjuntos e preservam, respectivamente, cofibrações acíclicas
e fibrações acíclicas. Pela proposição 6.3.2, esta última condição
garante a existência de $LF$ e $RG$. Por sua vez, diz-se que $\mathbf{C}$
e $\mathbf{C}'$ são \emph{Quillen equivalentes} quando $LF$ e $RG$
são equivalências, sendo um a inversa fraca do outro.$\underset{\;}{\;}$
\begin{rem}
Será toda equivalência entre categorias homotópicas uma verdadeira
equivalência de Quillen? A resposta é negativa. Um contraexemplo é
dado em \cite{equivalence_not_model}.
\end{rem}

\chapter{$n$-Categorias}

$\quad\;\,$Na primeira secção deste capítulo, generalizamos o conceito
de categoria e consideramos entidades formadas não só de objetos e
morfismos entre eles, mas também de ``morfismos entre morfismos''
(chamados de \emph{homotopias}), de morfismos entre homotopias, e
assim sucessivamente. Isto nos dá a ideia intuitiva de $\omega$-\emph{categoria
fraca}. A secção termina com uma discussão acerca da dificuldade de
formalização de tal conceito e com a apresentação de uma família particular
de $\omega$-categorias fracas: as $(\infty,1)$-\emph{categorias},
cujos morfismos de grau superior são sempre isomorfismos. Nelas se
tem uma maneira natural de definir sua correspondente categoria homotópica. 

A segunda secção é marcada pela introdução de categorias cujos objetos
são ``combinatórios'', a partir da qual estabelecemos uma estratégica
para formalizar o conceito de $\omega$-categoria fraca, ao mesmo
tempo que apresentamos um modelo para as $(\infty,1)$-categorias. 

Finalmente, na terceira secção efetivamos uma construção que nos permite
associar a cada categoria com equivalências fracas uma $(\infty,1)$-categorias,
produzindo o seguinte slogan: \emph{as categorias próprias para o
estudo da homotopia são, precisamente, aquelas nas quais se tem a
noção coerente de homotopia} \emph{entre morfismos}.

Diferentemente dos capítulos anteriores, este não foi escrito seguindo
de perto uma ou outra referência: utilizou-se de um apanhado delas.
Ainda assim, pode-se dizer que os tópicos aqui abordados são cobertos
pelos textos \cite{higher_categories_CHENG,Higher_topos_LURIE,simplicial_homotopy_JARDINE}.

\section{Estrutura}

$\quad\;\,$Vimos que uma $2$-categoria estrita $\mathbf{C}$ é aquela
na qual há a noção ``morfismos entre morfismos'', chamados de homotopias,
que podem ser compostos de duas maneiras distintas e sendo tais que,
junto de cada uma delas, tem-se uma nova categoria. 

Assim, pode-se pensar numa $2$-categoria como sendo aquela em que
cada $\mathrm{Mor}_{\mathbf{C}}(X;Y)$ possui a estrutura de categoria,
aqui denotada por $\mathrm{Enr}_{\mathbf{C}}(X;Y)$. Seus morfismos
são precisamente as homotopias, ao passo que suas leis de composição
correspondem à composição horizontal. Finalmente, referindo à composição
vertical, é preciso que se tenha uma regra que toma homotopias e devolve
uma outra. Em outras palavras, dados $X,Y,Z\in\mathbf{C}$, é preciso
obter functores 
\begin{equation}
\circ:\mathrm{Enr}_{\mathbf{C}}(X;Y)\times\mathrm{Enr}_{\mathbf{C}}(X;Y)\rightarrow\mathrm{Enr}_{\mathbf{C}}(X;Y).\label{enrichment}
\end{equation}
que cumpram com associatividade e que preservam identidades. Assim,
em suma, uma $2$-categoria estrita é aquela que pode ser obtida mediante
o seguinte procedimento:
\begin{enumerate}
\item parte-se de uma categoria $\mathbf{C}$ e mantém-se seus objetos;
\item substitui-se cada conjunto $\mathrm{Mor}_{\mathbf{C}}(X;Y)$ por um
objeto $\mathrm{Enr}_{\mathbf{C}}(X;Y)$ de $\mathbf{Cat}$;
\item para cada terna de objetos de $\mathbf{C}$ associa-se um morfismo
$\circ$ de $\mathbf{Cat}$, como apresentado em \ref{enrichment},
o qual é associativo e preserva identidades.
\end{enumerate}
$\quad\;\,$Observamos que o procedimento acima possui sentido se
substituirmos $(\mathbf{Cat},\times)$ por qualquer outra categoria
monoidal $(\mathbf{V},\otimes)$. As entidades assim obtidas chamam-se
\emph{categorias} \emph{enriquecidas} \emph{sobre} $\mathbf{V}$. 
\begin{example}
As categorias usuais são aquelas enriquecidas sobre $\mathbf{Set}$,
ao passo que as $2$-categorias estritas são aquelas enriquecidas
sobre $\mathbf{Cat}$. 

\begin{example}
Toda categoria $\mathbf{C}$ pode ser trivialmente enriquecida sobre
qualquer $\mathbf{V}$: basta substituir cada $\mathrm{Mor}_{\mathbf{C}}(X;Y)$
pelo objeto neutro $1\in\mathbf{V}$ e considerar os morfismos $\circ$
todos dados pelo isomorfismo $1\otimes1\simeq1$. Para $\mathbf{Cat}$,
isto traduz o fato de que toda categoria é uma $2$-categoria estrita
trivial, tendo como $2$-morfismos somente as identidades. 
\end{example}
\end{example}
Observamos que um functor $F:\mathbf{C}\rightarrow\mathbf{D}$ entre
duas categorias quaisquer pode ser pensado como sendo composto de
uma regra que mapeia objetos em objetos, e, para quaisquer $X,Y$,
de uma outra função 
\[
F_{xy}:\mathrm{Mor}_{\mathbf{C}}(X;Y)\rightarrow\mathrm{Mor}_{\mathbf{D}}(F(X);F(Y)),
\]
que preserva composição e identidades. Por sua vez, quando as categorias
$\mathbf{C}$ e $\mathbf{D}$ estão enriquecidas sobre uma mesma estrutura
monoidal $\mathbf{V}$, seus ``conjuntos'' de morfismos são, na
verdade, \emph{objetos} de $\mathbf{V}$. Portanto, em tal situação,
é natural supor que as regras $F_{xy}$ sejam \emph{morfismos} em
$\mathbf{V}$, e não funções como no caso usual. Isto nos leva ao
conceito de functor enriquecido.

Com efeito, um \emph{functor} \emph{enriquecido} entre categorias
$\mathbf{C}$ e $\mathbf{D}$, ambas enriquecidas sobre uma mesma
estrutura monoidal $\mathbf{V}$, é consistido de uma regra $F$,
que manda objetos em objetos, e, para cada par $X,Y\in\mathbf{C}$,
de um correspondente morfismo
\[
F_{xy}:\mathrm{Enr}_{\mathbf{C}}(X;Y)\rightarrow\mathrm{Enr}_{\mathbf{D}}(F(X);F(Y)),
\]
em $\mathbf{V}$, o qual preserva composições e unidades. Obtém-se
uma categoria $\mathrm{Enr}(\mathbf{V})$, formada de categorias e
functores enriquecidos. Por exemplo, quando $\mathbf{V}=\mathbf{Set}$,
recaímos em $\mathbf{Cat}$.

Tal categoria pode ser feita monoidal de maneira natural. Desta forma,
é possível falar de categorias enriquecidas sobre $\mathrm{Enr}(\mathbf{V})$.
Em outras palavras, é possível iterar o processo de enriquecimento
sobre uma categoria monoidal fixa, de modo a obter uma sequência$$
\xymatrix{\cdots \ar[r] & n\mathrm{-Enr}(\mathbf{V}) \ar[r] & \cdots \ar[r] & 2\mathrm{-Enr}(\mathbf{V}) \ar[r] & 1\mathrm{-Enr}(\mathbf{V}) \ar[r] & \mathbf{V}.}
$$

Para $\mathbf{V}=\mathbf{Set}$, o $n$-ésimo termo da sequência é
denotado por $\mathrm{S}\mathbf{Cat}_{n}$. Seus objetos são chamados
de $n$-\emph{categorias} \emph{estritas}. Assim, uma $0$-categoria
estrita é simplesmente um conjunto, ao mesmo tempo que as $1$-categorias
estritas coincidem com as categoriais usuais, e as $2$-categorias
estritas são aquelas já apresentadas no último capítulo. Tem-se então
definido um sistema co-dirigido em $\mathbf{Cat}$, cujo limite é
a categoria $\mathrm{S}\mathbf{Cat}_{\omega}$ das $\omega$-\emph{categorias
estritas}.

Se uma $2$-categoria estrita é aquela na qual se tem objetos, morfismos
e $2$-morfismos, os quais podem ser compostos de duas maneiras distintas,
uma $3$-categoria estrita é aquela em que, além de toda essa informação
categórica, tem-se também $3$-morfismos, os quais admitem três distintas
formas de composição. Indutivamente, uma $\omega$-categoria estrita
é aquela formada de objetos, morfismos, $2$-morfismos, $3$-morfismos
e assim sucessivamente. Um $n$-morfismo pode ser composto de $n$
diferentes formas.

Relembramos que um $2$-functor é aquele que preserva não só objetos
e morfismos, mas também $2$-morfismos e suas correspondentes composições.
Analogamente, um mapeamento de $\mathrm{S}\mathbf{Cat}_{n}$ é um
\emph{$n$-functor}. Isto é, uma regra entre $n$-categorias que preserva
todas as estruturas envolvidas: objetos, cada uma das classes de morfismos
e cada uma das composições. 
\begin{example}
No sexto capítulo, vimos que qualquer categoria $\mathbf{C}$ pode
ser estendida de modo a se tornar uma $2$-categoria. Com efeito,
bastava-nos considerar como $2$-morfismos os quadrados comutativos
de $\mathbf{C}$ (isto é, os morfismos de $\mathscr{B}(\mathbf{C})$).
Vejamos agora que esta extensão pode ser continuada de modo a transformar
qualquer categoria numa $\omega$-categoria. Se $2$-morfismos são
quadrados comutativos, a ideia é considerar $3$-morfismos como ``cubos
comutativos'', e assim sucessivamente. De maneira mais precisa, para
cada $n$, a inclusão $\mathrm{S}\mathrm{\mathbf{Cat}}_{n}\rightarrow\mathbf{Cat}$
possui um adjunto à esquerda. Em outras palavras, \emph{seja qual
for o $n$, a teoria das $n$-categorias estritas é livremente gerada
pela teoria das categorias usual}.
\end{example}
Iniciamos o quarto capítulo tratando do procedimento de \emph{categorificação}.
Recordamos que ele possui um viés oposto ao problema de \emph{classificação}.
Com efeito, se por um lado classificar uma categoria significa encontrar
um conjunto que a traduz, por outro, categorificar um conjunto significa
determinar uma categoria que o engloba. Assim, classificar implica
em \emph{perder} informação categórica, ao passo que categorificar
faz com que a \emph{ganhemos}. Nesse sentido, pode-se pensar numa
$2$-categoria estrita como sendo simplesmente uma categorificação
das categorias usuais, e assim por diante. Em suma: \emph{enriquecer
sobre $\mathbf{Set}$ é categorificar}.

\subsection*{\uline{Enfraquecimento}}

$\quad\;\,$Na definição de $2$-categoria estrita (e, consequentemente,
na definição de $n$-categoria estrita), exigiu-se que os functores
$\circ$ cumprissem com associatividade e que preservassem identidades.
Isto se traduziu em igualdades entre functores, algo que já discutimos
ser uma restrição muito forte. Ao invés disso, se exigíssemos apenas
associatividade e preservação de identidades módulo isomorfismos naturais,
chegaríamos ao conceito de \emph{$2$-categoria fraca}.

A ideia seria então proceder de maneira indutiva, definindo $n$-categorias
fracas e, depois de uma tomada de limite, $\infty$-categorias fracas.
Uma tal entidade consistir-se-ia de objetos, morfismos entre objetos,
$2$-morfismos entre morfismos, e assim sucessivamente. Entretanto,
diferentemente do que acontece numa $\omega$-categoria estrita, morfismos
seriam associativos e preservariam identidades apenas módulo $2$-isomorfismos,
ao passo que $2$-morfismos o fariam a menos de $3$-isomorfismos,
\emph{ad infinitum}.

Sob este ponto de vista, seguindo o raciocínio construído no caso
estrito, a passagem indutiva de uma $(n-1)$-categoria fraca para
uma $n$-categoria fraca deveria ser dada num processo de ``enriquecimento
fraco''. Sua construção é delicada, sendo introduzida em \cite{weak_n_categories_TRIMBLE}
e recentemente generalizada (ano 2012) em \cite{weak_omega_categories_CHENG_TOM}.

Em contrapartida, com o objetivo de definir $\omega$-categorias fracas,
ao invés de se tentar enfraquecer o conceito de $n$-categoria estrita
e depois tomar o limite $n\rightarrow\infty$, poder-se-ia pensar
em tomar tal limite estrito e enfraquecer o resultado final. Seguindo
tal estratégia, diversas propostas de definições para $\omega$-categorias
fracas foram apresentadas. Um comparativo entre dez delas (publicado
no ano 2002) pode ser encontrado em \cite{survey_n_categories_TOM}.
Veja também os livros \cite{higher_categories_TOM,higher_categories_CHENG}.
Ao que parece, até o momento não se tem um consenso sobre qual definição
utilizar. Na próxima subsecção, discutiremos uma estratégia padrão
para obtê-las. 

A ideia é essencialmente a seguinte: vimos que as categorias usuais
geram todas as $\omega$-categorias estritas. Procura-se, então, por
categorias $\mathbf{C}$ que gerem $\mathbf{Cat}$ (e, portanto, $\mathrm{S}\mathbf{Cat}_{\omega}$)
e que tenham uma caracterização ``fácil de enfraquecer''. Define-se,
então, uma $\omega$-categoria fraca como sendo aquela que é gerada
por entes que satisfazem propriedades enfraquecidas de objetos de
$\mathbf{C}$. Por exemplo (que é justamente o que obteremos na próxima
secção), pode ser que os objetos da categoria $\mathbf{C}$ obtida
sejam caracterizados por um problema de extensão \emph{única}. Com
isso em mãos, considera-se como $\omega$-categorias fracas aquelas
geradas pelos objetos que satisfazem o mesmo problema de levantamento,
só que abonados da exigência de unicidade.

Uma vez com o conceito de $\omega$-categoria fraca em mãos, pode-se
falar de \emph{$(\infty,n)$-categorias}. Estas são simplesmente classes
particulares de $\omega$-categorias fracas tais que todo $m$-morfismo,
com $m>n$, é um $m$-isomorfismo. Por exemplo, as $(\infty,0)$-categorias,
chamadas de $\infty$\emph{-grupóides},\emph{ }são aquelas em que
qualquer $n$-morfismo é uma equivalência. Intuitivamente, o processo
de enriquecimento fraco sobre as \emph{$(\infty,n)$-}categorias deve
produzir uma $(\infty,n+1)$-categoria. 

Em uma $(\infty,1$)-categoria, a relação que identifica $1$-morfismos
que podem ser ligados por $2$-morfismos é de equivalência. Isto nos
permite generalizar aquilo que foi feito para as $2$-categorias:
a cada $(\infty,1$)-categoria $\mathbf{C}$ faz-se corresponder uma
nova categoria $\mathscr{H}\mathbf{C}$, cujos objetos são os mesmos
que os de $\mathbf{C}$ e cujos morfismos são as classes de $1$-morfismos. 

Como consequência, toda $(\infty,1)$-categoria torna-se própria para
o estuda da homotopia de maneira natural: basta considerar como equivalências
fracas os $1$-morfismos cujas classes são isomorfismos em $\mathscr{H}\mathbf{C}$.
Diante disso, tem-se uma cadeia de categorias 
\[
\mathrm{S}\mathbf{Cat}_{2}\subset\mathbf{Cat}_{\infty}^{1}\subset\mathrm{W}\mathbf{Cat}.
\]

No final do capítulo, associaremos a cada categoria com equivalências
fracas $\mathbf{C}$ uma $(\infty,1)$-categoria $\mathrm{Ext}^{\infty}L\mathbf{C}$,
de tal modo que $\mathscr{H}(\mathrm{Ext}^{\infty}\mathrm{L}\mathbf{C})\simeq\mathrm{Ho}(\mathbf{C})$.
Isto nos dará o seguinte slogan: \emph{as categorias próprias para
o estudo da homotopia são, precisamente, aquelas nas quais se tem
uma noção coerente de homotopia entre morfismos}.

\section{Simplexos}

$\quad\;\,$Como já comentado, se a categoria $\mathbf{C}$ é completa
e cocompleta, então todo $F:\mathbf{S}\rightarrow\mathbf{C}$ admite
extensão de Kan $\vert\cdot\vert:\mathbf{S}'\rightarrow\mathbf{C}$
relativamente a qualquer $p:\mathbf{S}\rightarrow\mathbf{S}'$ e,
em particular, com respeito ao mergulho de Yoneda $h_{-}:\mathbf{S}\rightarrow\mathrm{Func}(\mathbf{S}^{op};\mathbf{Set})$.
Esta última é chamada de \emph{realização geométrica} induzida por
$F$. Como pode ser conferido em \cite{simplicial_sets_RIEHL}, tal
realização possui um adjunto, ao qual denominaremos \emph{nervo} \emph{de
$F$}, sendo dado pelo functor $\mathrm{N}:\mathbf{C}\rightarrow\mathrm{Func}(\mathbf{S}^{op};\mathbf{Set})$,
que em objetos é caracterizado pelas composições $\mathrm{N}(X)=h_{X}\circ F^{op}$.

Observamos que a realização $\vert\cdot\vert$ é extensão de Kan ao
longo de $h_{-}$ e, portanto, pode ser descrita em termos de \emph{coends}.
Com efeito, para qualquer functor $\alpha:\mathbf{S}^{op}\rightarrow\mathbf{Set}$,
tem-se 
\begin{equation}
\vert\alpha\vert\simeq\int^{S}\mathrm{Nat}(h_{S};\alpha)\cdot F(S).\label{realizacao_geometrica}
\end{equation}

Diz-se que uma subcategoria $\mathbf{S}\subset\mathbf{C}$ é \emph{densa}
quando o nervo da inclusão $\imath:\mathbf{S}\rightarrow\mathbf{C}$
é um mergulho e, consequentemente, uma equivalência sobre sua imagem.
Isto significa que cada objeto de $\mathbf{C}$ pode ser visto como
um functor de $\mathbf{S}^{op}$ em $\mathbf{Set}$ cumprindo determinadas
condições. Em outras palavras, observando que $\imath$ é mergulho
se, e somente se, $\mathrm{N}(\vert\alpha\vert)\simeq\alpha$ para
cada $\alpha$, a subcategoria $\mathbf{S}\subset\mathbf{C}$ é densa
precisamente quando todo objeto de $\mathbf{C}$ pode ser \emph{construído},
por meio de (\ref{realizacao_geometrica}), através de objetos de
$\mathbf{S}$.
\begin{example}
Dado um corpo $\mathbb{K}$, a subcategoria $\mathrm{F}\mathbf{Vec}_{\mathbb{K}}\subset\mathbf{Vec}_{\mathbb{K}}$
dos $\mathbb{K}$-espaços de dimensão finita é densa, o que expressa
a possibilidade de se escrever qualquer espaço vetorial como o limite
de seus subespaços de dimensão finita.

\begin{example}
A subcategoria das variedades difeomorfas a algum $\mathbb{R}^{n}$
é densa em $\mathbf{Diff}$. Afinal, toda variedade é obtida colando
regiões difeomorfas a abertos de espaços euclidianos.
\end{example}
\end{example}
Como comentamos na subsecção anterior, existem diversas maneiras de
definir o que vem a ser uma $\omega$-categoria fraca. Há uma estratégia
comum que permeia várias destas definições. A ideia é buscar por subcategorias
densas $\mathbf{S}\subset\mathrm{S}\mathbf{Cat}_{\omega}$, permitindo-nos
identificar $\omega$-categorias estritas com functores $\mathbf{S}^{op}\rightarrow\mathbf{Set}$
cumprindo certas propriedades. Assim, pode-se definir uma $\omega$-categoria
fraca como sendo um functor $\mathbf{S}^{op}\rightarrow\mathbf{Set}$
que satisfaz propriedades ``enfraquecidas'' se comparadas com aquelas
que caracterizam as $\omega$-categorias estritas.

De um modo geral, escolhe-se categorias $\mathbf{S}$ cujos objetos
são combinatórios. Adotaremos esta postura e trabalharemos com a categoria
$s\mathbf{Set}$ dos \emph{conjuntos simpliciais}, introduzida na
próxima subsecção. Do ponto de vista histórico, ela foi a primeira
a ser adotada. Em contrapartida, ela é natural em diversos outros
aspectos, como ficará claro ao longo do texto. 

\subsection*{\uline{Objetos Simpliciais}}

$\quad\;\,$Seja $\Delta$ a categoria gerada pelos conjuntos de números
naturais $[n]=\{0,...,n\}$, parcialmente ordenados por inclusão.
Seus morfismos são as aplicações $f:[n]\rightarrow[m]$ tais que $f(i)\leq f(j)$
sempre que $i\leq j$. Existem duas subcategorias evidentes $\Delta_{+}$
e $\Delta_{-}$, respectivamente compostas dos morfismos injetivos
e sobrejetivos de $\Delta$. Estes satisfazem certas relações de compatibilidade,
as quais não utilizaremos aqui (veja a quinta secção do capítulo VII
de \cite{MACLANE_categories}, o primeiro capítulo \cite{MAY_3},
a secção 8.1 de \cite{Weibel_homological_algebra} ou qualquer livro
mais moderno de homotopia simplicial, como \cite{simplicial_homotopy_JARDINE}). 

Todo morfismo de $\Delta$ pode ser decomposto em termos de um morfismo
de $\Delta_{+}$ seguido de outro de $\Delta_{-}$. Consequentemente,
dar um functor $X:\Delta^{op}\rightarrow\mathbf{C}$ é o mesmo que
dar uma sequência de objetos $X_{n}\in\mathbf{C}$ e morfismos $\partial_{i}:X_{n}\rightarrow X_{n-1}$
e $\sigma_{i}:X_{n}\rightarrow X_{n+1}$, os quais satisfazem relações
de compatibilidade. Tais functores são chamados de \emph{objetos simpliciais}
de $\mathbf{C}$. O termo $X_{n}$ recebe o nome de \emph{$n$-símplice},
ao passo que os morfismos $\partial_{i}$ e $\sigma_{i}$ chamam-se,
respectivamente, \emph{operadores de face} e \emph{operadores de}
\emph{degenerescência}. Há uma categoria $s\mathbf{C}$, formada de
objetos simpliciais e de transformações naturais. Isto é, $s\mathbf{C}=\mathrm{Func}(\Delta^{op};\mathbf{C})$.

Há especial interesse nos objetos simplicias de $\mathbf{Set}$, aos
quais se dá o nome de \emph{conjuntos simpliciais}. Se $X$ é um deles,
fala-se que $X_{n}$ é seu conjunto de $n$-\emph{células}. A nomenclatura
se deve ao seguinte: pensa-se num conjunto simplicial como sendo uma
sequência de objetos de diversas dimensões (vértices, arestas, faces,
etc.) que podem ser ``realizados geometricamente'' e depois ``colados'',
formando um espaço específico. Assim, a estrutura do espaço como um
todo fica determinada pela estrutura de suas partes, transformando
um problema global numa questão combinatória. Em topologia, os espaços
que podem ser obtidos por tal procedimento chamam-se \emph{complexos
celulares}. Exemplos são, evidentemente, os poliedros.

Num conjunto simplicial $X$, se pensamos em $X_{0}$ como coleção
de vértices, em $X_{1}$ como família de arestas ligando vértices,
em $X_{2}$ como lista faces completando arestas, e assim sucessivamente,
conclui-se que os operadores $\partial_{i}$ e $s_{i}$ possuem um
papel geométrico bastante claro. Com efeito, $\partial_{i}x$ nada
mais é que a célula obtida retirando de $x$ o $i$-ésimo vértice,
ao passo que $s_{i}x$ é obtido repetindo o $i$-ésimo vértice. Por
exemplo, se $x$ é $2$-célula (isto é, face), então $\partial_{i}x$
é aresta oposta ao $i$-ésimo vértice. Tendo isto em mente, diz-se
que $x\in X$ é \emph{face }ou \emph{degenerescência} de uma outra
célula $x'\in X$ quando $x=s_{i}x'$ ou $x=\partial_{i}x'$. Se inexiste
$x'$ tal que $x=s_{i}x'$, fala-se que a célula $x$ é \emph{não-degenerada}. 

Quando $\mathbf{C}$ é completa ou cocompleta, então o mesmo se passa
com qualquer categoria de functores que nela assume valores e, em
particular, com $s\mathbf{C}=\mathrm{Func}(\Delta^{op};\mathbf{C})$.
Além disso, tal hipótese faz de $s\mathbf{C}$ cartesianamente fechada.
Com efeito, pelo discutido anteriormente, pela completude e co-completude
de $\mathbf{C}$, o functor $F:\Delta\rightarrow s\mathbf{C}$, tal
que $F([n])=h_{[n]}\times Y$, induz adjuntos $\vert\cdot\vert$ e
$\mathrm{N}$ em $s\mathbf{C}$. Como a realização geométrica $\vert\cdot\vert$
é extensão de Kan de $F$, tem-se $h_{[n]}\times Y=\vert h_{[n]}\vert$,
donde $\vert\cdot\vert=-\times Y$, garantindo o afirmado. Por sua
vez, se $\mathbf{C}$ possui uma estrutura monoidal definida por um
bifunctor $\otimes$, então esta induz uma estrutura monoidal em $s\mathbf{C}$,
tal que $n$-ésima símplice de $X\otimes X'$ nada mais é que $X_{n}\otimes X{}_{n}'$. 

Estes fatos ressaltam a importância de $s\mathbf{Set}$: ela é completa,
cocompleta, cartesianamente fechada e naturalmente monoidal.

\subsection*{\uline{Conjuntos Compliciais}}

$\quad\;\,$Tem-se um mergulho natural $\imath:\Delta\rightarrow\mathbf{Cat}$,
o qual vê cada $[n]$ como uma categoria formada de um só objeto.
A correspondente subcategoria $\Delta\subset\mathbf{Cat}$ é densa.
Daí, qualquer categoria pode ser construída em termos dos $\mathrm{N}_{s}([n])=h_{[n]}$,
respectivamente denotados por $\Delta^{n}$ e denominados \emph{simplexos
padrões} ($\mathrm{N}_{s}:\mathbf{Cat}\rightarrow s\mathbf{Set}$
é chamado de \emph{nervo simplicial}). 

É possível ``orientar'' cada $\Delta^{n}$ de tal maneira que, para
todo $n$, as possíveis orientações do simplexo $\Delta^{n}$ gerem
uma $\omega$-categoria estrita. Assim, há também um mergulho $\imath:\Delta\rightarrow\mathrm{S}\mathbf{Cat}_{\omega}$,
apresentado pela primeira vez em \cite{Street_oriented_complex},
o qual induz adjunções $\mathrm{N}_{\omega}$ e $\vert\cdot\vert_{\omega}$.
No entanto, a respectiva subcategoria $\Delta\subset\mathrm{S}\mathbf{Cat}_{\omega}$
não é densa. Isto significa que o $\omega$-nervo $\mathrm{N}_{\omega}:\omega\mbox{-}\mathrm{S}\mathbf{Cat}\rightarrow s\mathbf{Set}$
pode não ser um mergulho e, consequentemente, pode não ser equivalência
sobre sua imagem. 

Tal fato impede que utilizemos da estratégia previamente mencionada
para definir $\omega$-categorias fracas em termos de conjuntos simpliciais
satisfazendo alguma condição enfraquecida. Para contornar o problema,
ainda em \cite{Street_oriented_complex} as definições de $\omega$-nervo
e da respectiva realização geométrica foram ligeiramente estendidas
de modo a assumirem valores na categoria dos \emph{conjuntos simpliciais
estratificados} $ss\mathbf{Set}$. Estes nada mais são que conjuntos
simpliciais $X$ nos quais se fixou um conjunto \emph{$tX$} contendo
todas as células degeneradas de $X$ e nenhuma $0$-célula. 

Em \cite{Street_fillers_nerves}, por sua vez, demonstrou-se que todo
elemento na imagem de $\mathrm{N}_{\omega}:\omega\mbox{-}\mathrm{S}\mathbf{Cat}\rightarrow ss\mathbf{Set}$
é equivalente a um \emph{conjunto complicial}. Isto é, um conjunto
simplicial que satisfaz certas propriedades \emph{únicas} de extensão.
Diante disso, conjecturou-se que o $\omega$-nervo induz uma equivalência
entre $\omega\mbox{-}\mathrm{S}\mathbf{Cat}$ e a categoria dos conjuntos
compliciais. Observamos que, uma vez verificada esta conjectura, poder-se-ia
definir $\omega$-categorias fracas como aquelas cujo $\omega$-nervo
é um \emph{conjunto complicial fraco}. Isto é, um conjunto simplicial
que possui a referida propriedade de extensão, abonada da exigência
de unicidade. 

A conjectura em questão foi comprovada recentemente em \cite{Verity_complicial_sets_0}
e os conjuntos compliciais fracos foram estudados em \cite{Verity_complicial sets_1,Verity_complicial_sets_2}.
Uma prova alternativa pode ser encontrada \cite{Steiner_complicial_sets_1,Steiner_complicial_sets_2}.

Como havíamos dito na secção anterior, com a definição de $\omega$-categoria
fraca em mãos, pode-se falar $(\infty,n)$-categorias. Isto significa
que, dentro da nossa perspectiva, estas haverão de ser casos particulares
de conjuntos compliciais fracos. Temos especial interesse nas $(\infty,1)$-categoriais
e nos $\infty$-grupóides. Os conjuntos compliciais fracos que a eles
correspondem são, respectivamente, as \emph{quasi-categorias} e os
\emph{complexos de Kan}, os quais definimos na próxima subsecção.
As referênciais padrões para o estudo de tais entidades são \cite{Higher_topos_LURIE,Joyal_quasi_categories}.
Uma boa introdução é \cite{Groth_infinity_categories}.

Observamos ainda que, sob a perspectiva apresentada, o $\omega$-nervo
permite identificar em qual classe uma dada categoria $\mathbf{C}$
pertence: esta será uma $\omega$-categoria fraca precisamente quando
$\mathrm{N}_{\omega}\mathbf{C}$ for um conjunto simplicial fraco.
Por sua vez, ela será um $(\infty,1$)-categoria ou um $\infty$-grupóide
quando $\mathrm{N}_{\omega}\mathbf{C}$ for uma quasi-categoria ou
um complexo de Kan. Assim, se suspeitamos que $\mathbf{C}$ possui
estrutura em dimensão mais alta, para confirmá-la ou refutá-la, basta
aplicar o $\omega$-nervo e verificar em qual classe de conjuntos
simpliciais ele se enquadra. Na medida em que a estrutura de $\mathbf{C}$
for mais complicada, o correspondente $\mathrm{N}_{\omega}\mathbf{C}$
também o será.

Ressaltamos a diferença entre o nervo simplicial e o $\omega$-nervo:
enquanto que o $\omega$-nervo nos permite classificar toda a informação
categórica de uma categoria, o nervo simplicial traduz apenas informações
de ``primeira ordem'' (isto é, acerca dos objetos e dos $1$-morfismos).
Assim, se quisermos estudar apenas $(\infty,1)$-categorias e $\infty$-grupóides,
o nervo simplicial é suficiente. Em contrapartida, se quisermos analisar
informação de maior calibre categórico, precisamos utilizar de $N_{\omega}$.
Por exemplo, mesmo que $\mathbf{C}$ seja $\omega$-categoria fraca
(e não uma ($\infty,1$)-categoria), o nervo simplicial $N_{s}\mathbf{C}$
será uma quasi-categoria.

\subsection*{\uline{Complexos de Kan}}

$\quad\;\,$Iniciamos observando que o functor $\imath:s\mathbf{Set}\rightarrow\mathbf{Set}$,
que a cada conjunto simplicial $X$ associa a reunião disjunta de
suas símplicies, possui um adjunto $\jmath:\mathbf{Set}\rightarrow s\mathbf{Set}$.
Isto significa que faz sentido falar conjuntos simpliciais livremente
gerados. Define-se o $i$\emph{-ésimo bordo} de $\Delta^{n}$ como
sendo o conjunto simplicial gerado pela $i$-ésima função injetiva
de $\Delta_{n-1}^{n}$. Seja $\Lambda_{n}^{k}$ a reunião de todos
(menos o $k$-ésimo) bordo de $\Delta^{n}$. Isto é, seja $\Lambda_{n}^{k}$
o conjunto simplicial gerado por todas (a não ser a $k$-ésima) aplicação
injetiva de $\Delta_{n-1}^{n}$.

Um \emph{complexo de Kan} nada mais é que um conjunto simplicial $X$
no qual, para cada $0\leq k\leq n$, a respectiva inclusão $\imath:\Lambda_{n}^{k}\rightarrow\Delta^{n}$
possui a propriedade de extensão de morfismos de $X$. Vejamos que
tais entidades são, como mencionado, modelos para a definição de $\infty$-grupóides. 
\begin{prop}
Complexos de Kan modelam $\infty$-grupóides.
\end{prop}
\begin{proof}
A ideia é considerar $n$-células como $n$-morfismos. Mais precisamente,
os objetos são as $0$-células e os $1$-morfismos entre $x,y\in X_{0}$
são as $1$-células $f\in X_{1}$ tais que $\partial_{0}(f)=x$ e
$\partial_{1}(f)=y$. Por sua vez, toma-se como $2$-morfismos entre
$f,g:x\rightarrow y$ as $2$-células $H\in X_{2}$ satisfazendo as
respectivas condições 
\[
\partial_{0}(H)=id_{x},\quad\partial_{1}(H)=g\quad\mbox{e}\quad\partial_{2}(H)=f,
\]
em que $s_{0}(id_{x})=x$. Indutivamente, dados $(n-1)$-morfismos
$\alpha$ e $\alpha'$, define-se um $n$-morfismo entre eles como
sendo uma $n$-célula $\varphi$ tal que 
\[
\partial_{0}(\varphi)=id,\quad\partial_{1}(\varphi)=id,\quad\mbox{até}\quad\partial_{n-1}(\varphi)=\alpha'\quad\mbox{e}\quad\partial_{n}(\varphi)=\alpha.
\]
Para mostrar os Kan complexos são realmente $\infty$-grupóides, devemos
verificar que todo $n$-morfismo assim definido possui inversa. De
fato, a propriedade de extensão dos $\imath:\Lambda_{n}^{k}\rightarrow\Delta^{n}$,
com $k=0,n$ implica na invertibilidade dos $1$-morfismos. Por sua
vez, a propriedade de extensão das outras inclusões, com $0<k<n$,
garantem a invertibilidade dos morfismos superiores.
\end{proof}
Na proposição anterior, a propriedade de extensão das inclusões $\imath:\Lambda_{n}^{k}\rightarrow\Delta^{n}$,
com $k=0,n$ e $0<k<n$, garantiu a existência de inversas fracas
para $1$-morfismos e morfismos de ordem superior. Diante disso, espera-se
que os conjuntos simpliciais tais que as inclusões $\imath:\Lambda_{n}^{k}\rightarrow\Delta^{n}$,
somente com $0<k<n$, possuam a propriedade de extensão definam $(\infty,1)$-categorias.
De fato, estes (usualmente chamados de \emph{quasi-categorias}) são
precisamente os modelos de $(\infty,1)$-categorias\emph{ }que advém
na noção prévia de $\omega$-categorias fracas enquanto conjuntos
compliciais fracos.

\section{Slogan}

$\quad\;\,$Nesta subsecção, construímos um functor $\mathrm{Ext}^{\infty}:s\mathbf{Set}\rightarrow s\mathbf{Set}$
que nos permite substituir qualquer conjunto simplicial por um complexo
de Kan e, portanto, por um $\infty$-grupoide. Seguimos de perto a
quarta secção do terceiro capítulo de \cite{simplicial_homotopy_JARDINE}.
A construção original é de \cite{Kan_c.s.s_complexes}.

Seja $\mathrm{Pos}_{k}^{n}$ a coleção de todas $k$-símplices não-degeneradas
de $\Delta^{n}$. Isto é, seja $\mathrm{Pos}_{k}^{n}$ a família das
funções injetivas $[k]\rightarrow[n]$. Por exemplo, quando $k=n-1$,
tal coleção coincide com o bordo de $\Delta^{n}$. Na medida em que
se aumenta o $k$, o número de funções injetivas diminuem, de modo
que, para cada $n$, a respectiva coleção $\mathrm{Pos}^{n}$ dos
$\mathrm{Pos}_{k}^{n}$ se torna parcialmente ordenada por inclusão.
De maneira mais precisa, observando haver uma função injetiva $[k]\rightarrow[n]$
para cada subconjunto de cardinalidade $k+1$ de $[k]$, conclui-se
que $\mathrm{Pos}^{n}$ se identifica ao conjunto das partes de $[n]$,
ordenado por inclusão. 

Tem-se functor $\Delta\rightarrow\mathbf{Pos}$, que a cada $[n]$
associa $\mathrm{Pos}^{n}$. Compondo com o nervo simplicial, define-se
um novo functor $\mathrm{nd}:\Delta\rightarrow s\mathbf{Set}$. Este
induz adjuntos $\mathrm{N}:s\mathbf{Set}\rightarrow s\mathbf{Set}$
e $\vert\cdot\vert:s\mathbf{Set}\rightarrow s\mathbf{Set}$, que aqui
serão respectivamente denotados por $\mathrm{Ext}$ e $\mathrm{sd}$.
Diz-se que $\mathrm{sd}(X)$ é a \emph{subdivisão baricêntrica} de
$X$. 

Existe um morfismo natural $X\rightarrow\mathrm{Ext}(X)$, obtido
como segue. Como $\Delta\subset\mathbf{Cat}$ é densa, toda categoria
é descrita em termos dos simplexos padrões $\Delta^{n}$. Consequentemente,
tem-se $X\simeq\mathrm{colim}\Delta^{n}$ para todo $X\in s\mathbf{Set}$.
Sendo $\mathrm{\mathrm{sd}}$ adjunto à esquerda, ele preserva todas
as extensões de Kan à direita e, particularmente, os colimites. Daí,
$\mathrm{sd}(X)\simeq\mathrm{colim}\,\mathrm{sd}(\Delta^{n})$. Há
um mapa natural $f:\mathrm{sd}(\Delta^{n})\rightarrow\Delta^{n}$:
considera-se $v:\mathrm{Pos}^{n}\rightarrow[n]$, que a cada subconjunto
de $[n]$ associa seu maior elemento. Põe-se, então, $f=\mathrm{N}_{s}(v)$.
Tomando o colimite, encontra-se um morfismo entre $\mathrm{sd}(X)$
e $X$. Seu adjunto precisamente o morfismo $X\rightarrow\mathrm{Ext}(X)$
procurado. 

Utilizando de $\mathrm{Ext}^{k}(X)$ para denotar o conjunto simplicial
obtido por aplicações sucessivas do functor $\mathrm{Ext}$, cada
$X$ determina um sistema dirigido$$
\xymatrix{X \ar[r] & \mathrm{Ext}(X) \ar[r] & \mathrm{Ext}^2(X) \ar[r] & \mathrm{Ext}^3(X) \ar[r] & \cdots}
$$em $s\mathbf{Set}$, cujo colimite será denotado por $\mathrm{Ext}^{\infty}(X)$.
Observamos que, como as construções envolvidas são functoriais, está
bem definido um functor $\mathrm{Ext}^{\infty}:s\mathbf{Set}\rightarrow s\mathbf{Set}$.

Mostremos, por fim, o principal resultado da subsecção.
\begin{prop}
Para todo $X\in s\mathbf{Set}$, o conjunto simplicial $\mathrm{Ext}^{\infty}X$
é um complexo de Kan.
\end{prop}
\begin{proof}
Deve-se mostrar que as inclusões $\imath:\Lambda_{n}^{k}\rightarrow\Delta^{n}$
possuem a propriedade de levantamento de morfismos em $\mathrm{Ext}^{\infty}X$.
Mais precisamente, para cada $\Lambda_{n}^{k}\rightarrow\mathrm{Ext}^{\infty}X$
deve-se obter um $\Delta^{n}\rightarrow\mathrm{Ext}^{\infty}X$ que
deixa comutativo o segundo dos diagramas abaixo. A ideia da prova
consiste-se na observação de que tal diagrama é obtido tomando o colimite
daquele representado ao lado.$$
\xymatrix{\mathrm{Ext}^{i}X \ar[r] & \mathrm{Ext}^{i+1}X && \mathrm{Ext}^{\infty}X \ar[r]^{id} & \mathrm{Ext}^{\infty}X\\
\Lambda ^{k}_{n} \ar[u] \ar[r] & \Delta ^{n} \ar@{-->}[u] && \Lambda ^{k}_{n} \ar[u] \ar[r] & \Delta ^{n} \ar@{-->}[u] }
$$ Mais precisamente, a ideia é mostrar que todo morfismo $\Lambda_{n}^{k}\rightarrow\mathrm{Ext}^{i}X$
se estende a um mapa $\Delta^{n}\rightarrow\mathrm{Ext}^{i+1}X$.
Daí, ao se passar o limite $i\rightarrow\infty$, obtém-se o resultado
procurado. Como $\mathrm{Ext}^{i}$ é a composição sucessiva de $\mathrm{Ext}$,
basta mostrar a condição de extensão para $i=1$. Detalhes podem ser
encontrados nas páginas 187-188 de \cite{simplicial_homotopy_JARDINE}.
\end{proof}

\subsection*{\uline{Prova do }\emph{\uline{Slogan}}}

$\quad\;\,$Nesta subsecção, finalmente construiremos o slogan ``\emph{as
categorias próprias para o estudo da homotopia são, precisamente,
aquelas nas quais se tem uma noção coerente de homotopia entre morfismos}'',
comentado na secção anterior.

Iniciamos observando que o functor $\mathrm{Ext}^{\infty}$ nos permite
associar a cada categoria $\mathbf{C}$ enriquecida sobre $s\mathbf{Set}$
uma categoria $\mathrm{Ext}^{\infty}\mathbf{C}$ enriquecida sobre
os complexos de Kan, obtida substituindo os conjuntos simpliciais
$\mathrm{Mor}_{\mathbf{C}}(X;Y)$ pelos correspondentes $\mathrm{Ext}^{\infty}\mathrm{Mor}_{\mathbf{C}}(X;Y)$.
Como os complexos de Kan modelam $\infty$-grupoides, está definida
uma regra 
\[
\mathrm{Ext}^{\infty}:\mathrm{Enr}(s\mathbf{Set})\rightarrow(\infty,1)\mbox{-}\mathbf{Cat}.
\]

A estratégia para o \emph{slogan} gira entorno da construção de um
functor $\mathrm{L}:\mathrm{W}\mathbf{Cat}\rightarrow\mathrm{Enr}(s\mathbf{Set})$,
que a cada categoria com equivalências fracas $\mathbf{C}$ faz corresponder
uma categoria $\mathrm{L}\mathbf{C}$ enriquecida sobre $s\mathbf{Set}$,
de tal maneira que $\mathscr{H}(\mathrm{Ext}^{\infty}\mathrm{L}\mathbf{C})\simeq\mathrm{Ho}(\mathbf{C})$.
Como veremos em seguida, este functor existe, sendo usualmente chamado
de \emph{localização simplicial}. Remetemos o leitor aos trabalhos
\cite{simplicial_localization_1,simplicial_localization_2,simplicial_localization_3},
onde ele foi introduzido e estudado.

Sabe-se associar um conjunto simplicial a cada categoria: basta aplicar
o nervo $\mathrm{N}_{s}$. Diante disso, para definir $\mathrm{L}\mathbf{C}$,
primeiro construímos uma regra que substitui os conjuntos de morfismos
de $\mathbf{C}$ por categorias e depois aplicaremos $\mathrm{N}_{s}$
em cada uma delas. A categoria associada ao conjunto $\mathrm{Mor}_{\mathbf{C}}(X;Y)$
é $\mathrm{Ham}_{\mathbf{C}}(X;Y)$: seus objetos são os \emph{zig-zags} 

$$
\xymatrix{X & \ar@{~>}[l] A_1 \ar[r] & A_2 & \ar@{~>}[l] A_3 \ar[r] & \cdots \ar[r] & Y,}
$$ ao passo que seus morfismos são os seguintes diagramas comutativos
(ditos serem \emph{hammocks}): $$
\xymatrix{X \ar[d]_{id} & \ar@{~>}[l] A_1 \ar@{~>}[d] \ar[r] & A_2 \ar@{~>}[d] & \ar@{~>}[l] A_3 \ar@{~>}[d] \ar[r] & \cdots \ar@{~>}[d] \ar[r] & Y \ar[d]^{id} \\
X & \ar@{~>}[l] B_1 \ar[r] & B_2 & \ar@{~>}[l] B_3 \ar[r] & \cdots \ar[r] & Y}
$$

Terminamos o capítulo com a
\begin{prop}
Para toda categoria $\mathbf{C}$ própria para o estudo da homotopia,
há uma equivalência entre sua categoria homotópica $\mathrm{Ho}(\mathbf{C})$
e a categoria homotópica $\mathscr{H}(\mathrm{Ext}^{\infty}\mathrm{L}\mathbf{C})$
da $(\infty,1)$-categoria que lhe corresponde.
\end{prop}
\begin{proof}
Num conjunto simplicial $X$ qualquer, a relação que identifica $0$-células
entre as quais há uma $1$-célula é reflexiva e transitiva, mas pode
não ser simétrica. No entanto, se $X$ é complexo de Kan, então todo
$1$-morfismo (isto é, toda $1$-célula ligando $0$-células) possui
inverso, donde a validade da simetria. Denotamos com $\pi_{0}(X)$
o espaço quociente da $0$-símplice $X_{0}$ por tal relação. Esta
construção se estende a um functor $\pi_{0}:\mathbf{Kan}\rightarrow\mathbf{Set}$,
definido na subcategoria cheia $\mathbf{Kan}\subset s\mathbf{Set}$
dos complexos de Kan. Enriquecendo ambos os lados vê-se, por sua vez,
que tal functor se estende a um outro $\pi_{0}:(\infty,1)\mbox{-}\mathbf{Cat}\rightarrow\mathbf{Cat}$.
A cada categoria $\mathbf{C}$ enriquecida sobre complexos de Kan
ele faz corresponder a categoria com os mesmos objetos, sendo tal
que $\mathrm{Mor}_{\pi_{0}\mathbf{C}}(X;Y)=\pi_{0}\mathrm{Mor}_{\mathbf{C}}(X;Y)$.
Afirmamos que, $\pi_{0}(\mathrm{Ext}^{\infty}\mathrm{L}\mathbf{C})\simeq\mathscr{H}(\mathrm{Ext}^{\infty}\mathrm{L}\mathbf{C})$$.$
Como ambas possuem objetos iguais aos de $\mathbf{C}$, basta mostrarmos
que elas também têm os mesmos morfismos. Por um lado, os morfismos
em $\mathscr{H}(\mathrm{Ext}^{\infty}\mathrm{L}\mathbf{C})$ são classes
de $1$-morfismos de $\mathrm{Ext}^{\infty}\mathrm{L}\mathbf{C}$
ligados por $2$-morfismos. Por outro, os $n$-morfismos de $\mathrm{Ext}^{\infty}\mathrm{L}\mathbf{C}$
são precisamente as $(n-1)$-células da estrutura simplicial subjacente.
Daí, as classes de $1$-morfismos ligados por $2$-morfismos coincidem
com as classes de $0$-células ligadas por $1$-células, donde o afirmado.
Posto isto, a demonstração termina ao notarmos que $\pi_{0}(\mathrm{Ext}^{\infty}\mathrm{L}\mathbf{C})$
é localização de $\mathbf{C}$ com respeito às equivalências fracas,
donde $\pi_{0}(\mathrm{Ext}^{\infty}\mathrm{L}\mathbf{C})\simeq\mathrm{Ho}(\mathbf{C})$.
\end{proof}

\chapter{Homotopia Clássica}

$\quad\;\,$No presente capítulo, desenvolvemos a \emph{teoria da
homotopia clássica}. Em outras palavras, estudamos as propriedades
homotópicas de subcategorias convenientes $\mathscr{C}$ de $\mathbf{Top}$.
Tais subcategorias se caracterizam por serem completas e cocompletas
(propriedades herdadas de $\mathbf{Top}$) e por terem pontuação fechada
e simétrica com respeito a estrutura monoidal definida pelo correspondente
produto \emph{smash}. Esta última propriedade é extremamente importante
dentro da teoria e impede que tomemos $\mathscr{C}=\mathbf{Top}$.

Na primeira secção, mostramos que as categorias convenientes admitem
cilindros e espaços de caminhos naturais e adjuntos, permitindo-nos
ali introduzir um modelo. Verificamos, ainda, que para tais categorias
o \emph{slogan} ``\emph{categorias com cilindros naturais são protótipos
de $(\infty,1)$-categorias}'' se aplica fielmente. Finalmente, damos
um olhar intuitivo às homotopias deste modelo, o qual nos permite
analisar melhor a diferença entre limites e limites homotópicos.

Como temos afirmado desde o início, a Topologia Algébrica tem o objetivo
de construir invariantes topológicos poderosos por métodos puramente
algébricos. É na segunda secção que nos deparamos com os primeiros
exemplos de tais invariantes. Tratam-se, pois, dos \emph{grupos de
homotopia}. Eles estão associados à functores $\pi_{i}:\mathscr{HC}_{*}\rightarrow\mathbf{Grp}$
e são poderosos o suficiente para que finalmente consigamos discernir
o toro da esfera.

Pelo justo fato dos grupos de homotopia serem bastante poderosos,
eles também são tremendamente difíceis de serem calculados. Tecnicamente
isto se deve ao fato dos functores $\pi_{i}$ preservaremos poucos
limites/colimites homotópicos. Na terceira e última secção, discutimos
esta inconveniência prática e apresentamos algumas ferramentas das
existentes que nos ajudam a efetivar alguns cálculos.

Para o estudo e escrita deste capítulo, fizemos uso sistemático das
referências \cite{homotopia_strom,MAY_1,MAY_4}, as quais tratam o
assunto mais ou menos do ponto de vista que adotamos. Outras referências
clássicas sobre teoria da homotopia, as quais assumem uma postura
ligeiramente distinta daquela aqui empregada, incluem \cite{gray_homotopy,Hu homotopia,Spanier,whitehead_homotopy}.

\section{Estrutura}

$\quad\;\,$Uma categoria $\mathscr{C}$ é \emph{conveniente} para
se estudar topologia algébrica quando:
\begin{enumerate}
\item é completa e cocompleta;
\item a sua pontuação $\mathscr{C}_{*}$ é fechada e simétrica com respeito
ao produto \emph{smash} definido através de produtos binários.
\end{enumerate}
$\quad\;\,$A primeira condição é satisfeita por $\mathbf{Top}$.
No entanto, como vimos anteriormente, seu o produto \emph{smash} em
geral não é associativo, de modo que $\mathbf{Top}_{*}$ não cumpre
com a segunda das requisições. A ideia é, então, buscar por subcategorias
$\mathscr{C}\subset\mathbf{Top}$ que satisfaçam ambas. Um exemplo
foi introduzido por Steenrod em \cite{convenient_category_steenrod}.
Ele é construído a partir da subcategoria cheia $\mathscr{CG}\subset\mathbf{Top}$
dos espaços que são \emph{compactamente gerados}. Isto significa que
$U\subset X$ é aberto se, e somente se, $U\cap K$ é aberto de $K\subset X$,
sempre que este é compacto. Aqui se enquadram, por exemplo, os espaços
métricos e as variedades.

Claramente, todo conjunto pode ser dotado de uma topologia $\tau$
que o torna compactamente gerado. Portanto, a inclusão $\imath:\mathscr{CG}\rightarrow\mathbf{Top}$
possui um adjunto à esquerda $\kappa:\mathbf{Top}\rightarrow\mathscr{CG}$,
responsável por substituir a topologia de cada espaço $X$ por $\tau$.
Sua imagem pela subcategoria dos espaços Hausdorff constitui uma categoria
$\mathscr{C}$ satisfazendo as condições requisitadas, como justificamos
em seguida.

Já que $\kappa$ possui adjunto à esquerda, $\mathscr{C}$ é completa
e seus limites são as imagens por $K$ dos limites em $\mathbf{Top}$.
Por sua vez, como $\imath$ tem adjunto à direita, $\mathscr{C}$
também é cocompleta e seus colimites são os próprios colimites de
$\mathbf{Top}$. Assim, $\mathscr{C}$ cumpre com a primeira das requisições.

Fixados $X,Y\in\mathscr{C}$, consideremos o espaço $\mathrm{Map}(X;Y)\in\mathscr{C}$
obtido introduzindo a topologia compacto-aberta no conjunto das funções
contínuas $f:X\rightarrow Y$ e depois fazendo agir $\kappa$. Para
qualquer $Z\in\mathscr{C}$, tem-se bijeções
\[
\mathrm{Mor}_{\mathscr{C}}(X\times Y;Z)\simeq\mathrm{Mor}_{\mathscr{C}}(X;\mathrm{Map}(Y;Z)),
\]
mostrando-nos que $\mathscr{C}$ é cartesianamente fechada e, em particular,
enriquecida sobre si mesma. Como consequência, o produto \emph{smash}
introduz uma estrutura monoidal fechada e simétrica em $\mathscr{C}_{*}$,
de modo que $\mathscr{C}$ cumpre com a segunda das requisições e,
de fato, é conveniente ao estudo da topologia algébrica.

Sendo $\mathscr{C}$ cartesianamente fechada, $Y\wedge X$ possui
adjunto dado pelo o \emph{pullback} abaixo. Uma vez que limites em
$\mathscr{C}_{*}$ são computados como limites de $\mathscr{C}$,
e estes, por sua vez, são as imagens dos limites de $\mathbf{Top}$
por $K$, segue-se que, identificando elementos de $\mathbf{Top}_{*}$
com pares $(X,x_{o})$, tem-se a identificação entre $\mathrm{Map}_{*}(X;Y)$
e $(\mathrm{Map}_{\mathscr{C}_{*}}(X;Y),f_{o})$. Aqui, $\mathrm{Map}_{\mathscr{C}_{*}}(X;Y)$
é o subespaço de $\mathrm{Map}(X;Y)$, formado de toda $f:X\rightarrow Y$
contínua e cumprindo $f(x_{o})=y_{o}$, ao passo que $f_{o}$ é a
função constante e igual a $y_{o}$.$$
\xymatrix{\mathrm{Map_{*}}(X;Y) \ar[d] \ar[r] & \mathrm{*} \ar[d] \\
\mathrm{Map}(X;Y) \ar[r] & \mathrm{Map}(*;Y). }
$$

Maiores detalhes a respeito de $\kappa$ podem também ser encontrados
no capítulo oito de \cite{gray_homotopy} e entre as páginas 21-27
de \cite{whitehead_homotopy}. Veja também \cite{CGWH}.$\underset{\underset{\;}{\;}}{\;}$

\noindent \textbf{Advertência.} No que segue, a menos de explícita
menção em contrário, trabalharemos sempre na categoria conveniente
$\mathscr{C}$ dos espaços Hausdorff e compactamente gerados.

\subsection*{\uline{Modelo}}

$\quad\;\,$A categoria $\mathbf{\mathscr{C}}$ possui cilindros naturais,
os quais são dados pelo functor $C:\mathscr{C}\rightarrow\mathscr{C}$,
que a cada $X$ faz corresponder o espaço produto $X\times I$, e
pelas transformações naturais $n_{0}$ e $n_{1}$ entre $id_{\mathscr{C}}$
e $C$, definidas por $n_{0}(X)(x)=(x,0)$ e $n_{0}(X)(x)=(x,1)$.
Como logo se convence, todos os objetos de $\mathscr{C}$ são $C$-cofibrantes.
Daí, pelo lema de fatoração, qualquer $f:X\rightarrow Y$ pode ser
escrito na forma $f=\imath\circ\jmath$, em que $\imath$ é uma $C$-cofibração.

Em $\mathscr{C}$ também existem espaços de caminhos naturais, obtidos
do functor $P:\mathbf{\mathscr{C}}\rightarrow\mathbf{\mathscr{C}}$,
que a cada $X$ associa $\mathrm{Map}(I;X)$, e das transformações
$p_{0}$ e $p_{1}$, de $P$ em $id_{\mathscr{C}}$, caracterizadas
por $p_{0}(X)(f)=f(0)$ e $p_{1}(X)(f)=f(1)$, em que $f:I\rightarrow X$.
Os objetos de $\mathscr{C}$ são todos $P$-fibrantes, de modo que
qualquer morfismo é decomposto na forma $\imath\circ\jmath$, onde
$\jmath$ é $P$-fibração. 

Observamos que, sendo $\mathbf{\mathscr{C}}$ cartesianamente fechada,
$C$ e $P$ são adjuntos. Pelo discutido ao final do capítulo seis,
há, portanto, um modelo em $\mathbf{\mathscr{C}}$ cujas equivalências
fracas são as equivalências homotópicas, cujas cofibrações são as
$C$-cofibrações, e cujas fibrações são as $P$-fibrações. Ele é\emph{
}denominado \emph{modelo de Ström}. Nele, em particular, todos os
objetos são tanto fibrantes quanto cofibrantes.$\underset{\underset{\;}{\;}}{\;}$

\noindent \textbf{Advertência.} Quando não houver risco de confusão,
$C$-cofibrações e $P$-fibrações no modelo de Ström serão chamadas
simplesmente de \emph{cofibrações} e de \emph{fibrações}.$\underset{\underset{\;}{\;}}{\;}$

Vejamos alguns exemplos.
\begin{example}
Um fibrado $f:X\rightarrow Y$ em $\mathbf{Top}$ se diz \emph{localmente
trivial} quando é localmente conjugada ao fibrado trivial. Isto é,
quando cada ponto $x\in Y$ admite uma vizinhança aberta $U$, tal
que $f^{-1}(U)\rightarrow U$ é conjugado à projeção $\pi_{1}:U\times F_{x}\rightarrow U$,
para algum espaço $F_{x}$. Quando $F_{x}$ é o mesmo para cada $x$,
diz-se que ele é a \emph{fibra típica} de $f$.\emph{ }Observamos
que a projeção de qualquer fibrado localmente trivial é uma fibração
no modelo de Ström. Este resultado pode ser encontrado em diversos
livros sobre a teoria de fibrados. Uma referência clássica é o livro
\cite{Steenrod} de Steenrod, que contempla o referido resultado entre
as páginas 50 e 53.

\begin{example}
Em modelos gerados por cilindros naturais $C$, verifica-se facilmente
que um morfismo $\imath:A\rightarrow X$ é uma $C$-cofibração se,
e somente se, o mapa natural $\jmath:\mathrm{cyl}(\imath)\rightarrow CX$,
que aparece na definição de \emph{mapping cylinder} $\mathrm{cyl}(\imath)$
e que nada mais é que o mapa vertical do \emph{pushout} entre $\imath$
e $n_{0}(A)$, admite uma retração. No modelo de Ström, tem-se $\mathrm{cyl}(\imath)\simeq X\times0\cup A\times I$
sempre que $\imath$ é inclusão. Assim, suas cofibrações se caracterizam
por serem aquelas em que $X\times0\cup A\times I$ é retrato de $X\times I$.
Por exemplo, $\imath:\mathbb{S}^{n-1}\rightarrow\mathbb{D}^{n}$ é
cofibração: basta tomar como retração a correspondência 
\[
r:\mathbb{D}^{n}\times I\rightarrow\mathbb{D}^{n}\times0\cup\mathbb{S}^{n-1}\times I,\quad\mbox{tal que}\quad r(x,t)=(t\cdot\frac{x}{\Vert x\Vert}-t\cdot x+x,t).
\]
\end{example}
\end{example}

\subsection*{\uline{Alta Dimensão}}

$\quad\;\,$Anteriormente comentamos que, em geral, as categorias
com cilindros naturais são protótipos de $(\infty,1)$-categorias.
Vejamos que isto se passa com $\mathscr{C}$. Como sempre, $0$-morfismos
são os objetos (isto é, os espaços topológicos) e $1$-morfismos são
os morfismos (isto é, as funções contínuas). Os $2$-morfismos são
as homotopias. No modelo de Ström, estas são aplicações $H:X\times I\rightarrow Y$,
de modo que é possível falar de ``\emph{homotopias formais entre
}$2$\emph{-morfismos}'', as quais corresponderão aos $3$-morfismos.
Estas, por sua vez, haverão de ser funções $X\times I\times I\rightarrow Y$,
permitindo-nos falar de ``\emph{homotopias formais entre $3$-morfismos}'',
correspondendo aos $3$-morfismos. Indutivamente define-se os $n$-morfismos
para todo $n$.

Para concluirmos que a estrutura assim obtida é $(\infty,1)$-categoria,
devemos mostrar que, para qualquer $n>1$, os $n$-morfismos são todos
inversíveis e podem ser compostos de $n$ maneiras distintas, cada
uma das quais é associativa e preserva unidade a menos de morfismos
de ordem superior. Nos resumimos apresentar duas distintas maneiras
de compor homotopias formais e a mostrar que $2$-morfismos são inversíveis
a menos de $3$-morfismos. 

Dadas homotopias formais $H$ e $H'$, ambas entre $f,g:X\rightarrow Y$,
sua \emph{composição} \emph{vertical}, é a aplicação $H\bullet H'$,
também do produto $X\times I$ em $Y$, definida por concatenação:
\[
(H\bullet H')(x,t)=\begin{cases}
H(x,2t), & 0\leq t\leq1/2\\
H'(x,2t-1), & 1/2<t\leq1.
\end{cases}
\]

Por sua vez, se $H$ e $H'$ são respectivas homotopias formais entre
$f,g:X\rightarrow Y$ e $f',g':Y\rightarrow Z$, então a \emph{composição
horizontal} entre elas é a correspondência 
\[
H'\circ H:X\times I\rightarrow Z,\quad\mbox{tal que}\quad(H'\circ H)(x,t)=H'(H(x,t),t).
\]

Vejamos agora que, se existe uma homotopia formal $H:f\rightarrow g$,
então também existem homotopias formais $H':g\rightarrow f$, assim
como $3$-morfismos $\alpha$ entre $H'\bullet H$ e $id_{f}$, e
$\alpha'$ entre $H'\bullet H$ e $id_{g}$. Aqui, $id_{f}(x,t)=f(x)$
para todo $t$. A existência de $H'$ provém da simetria da relação
que identifica morfismos ligados por homotopias formais. Por sua vez,
a existência dos $\alpha$ e $\alpha'$ segue como caso particular
do seguinte fato mais geral: entre quaisquer homotopias formais $H$
e $H'$ entre $f,g:X\rightarrow Y$ há ao menos um $3$-morfismo $\alpha:X\times I\times I\rightarrow Y$,
definido por
\[
\alpha(x,t,s)=\begin{cases}
H(x,t-2ts) & 0\leq t\leq1/2\\
H'(x,2ts-t), & 1/2<t\leq1.
\end{cases}
\]

\subsection*{\uline{Intuição}}

$\quad\;\,$Intuitivamente, o papel das homotopias $H:f\Rightarrow g$
é deformar $f$ continuamente, até que se torne $g$. Assim, um espaço
$X$ terá o mesmo tipo de homotopia que outro espaço $Y$ quando nele
pode ser continuamente deformado. 
\begin{example}
Os espaços que podem ser continuamente deformados (via homotopia)
a pontos são ditos \emph{contráteis}. Este é o caso, por exemplo,
de qualquer subespaço convexo do $\mathbb{R}^{n}$. No estudo da homotopia
tem-se particular interesse em invariantes homotópicos. Neste sentido,
espaços contráteis possuem invariantes homotópicos triviais. Por exemplo,
como comentamos anteriormente, existem certos functores $\pi_{i}:\mathscr{C}\rightarrow\mathbf{Grp}$,
chamados de \emph{grupos de homotopia}. Estes são homotópicos e, portanto,
passam à categoria homotópica $\mathscr{H}\mathbf{\mathscr{C}}$.
Se um espaço $X$ é contrátil, então cada invariante $\pi_{i}(X)$,
com $i>0$, é trivial. A recíproca, no entanto, não é verdadeira:
no modelo de Ström, existem espaços com grupos de homotopia triviais
que não são contráteis. Isso significa que os functores $\pi_{i}$
não são suficientes para classificar $\mathscr{H}\mathbf{\mathscr{C}}$
como um todo. Realçamos que este fato \emph{depende do modelo fixado
em} $\mathscr{C}$. Com efeito, pode-se introduzir um novo modelo
numa subcategoria de $\mathscr{C}$, chamado de \emph{modelo de Quillen},
segundo o qual os functores $\pi_{i}$ também são homotópicos, mas
desta vez classificam a correspondente categoria homotópica (veja,
por exemplo, a secção 17.2 de \cite{MAY_4}).

\begin{example}
Sob quais condições um subespaço $A\subset X$ possui o mesmo tipo
de homotopia que o espaço inteiro? É suficiente que $A$ seja retrato
por deformação de $X$. Em outras palavras, é suficiente que a inclusão
admita inversa fraca $r$. Por sua vez, para que isto aconteça, basta
existir um \emph{fluxo} entre $A$ e $X$. Este se trata de uma família
a um parâmetro de funções $\varphi_{t}:X\rightarrow X$, tal que $\varphi_{0}=id_{X}$,
com $t\mapsto\varphi_{t}$ é contínua, e para o qual há uma outra
$\tau:X\rightarrow\mathbb{R}$, a qual se anula em $A$ e cumpre $\varphi_{\tau(x)}(x)\in A$
para todo $x\in X$. Com efeito, dado um tal fluxo, a regra $r(x)=\varphi_{\tau(x)}(x)$
evidentemente define uma retração de $X$ em $A$, sendo tal que $\imath_{A}\circ r\simeq id_{X}$,
com homotopia dada por $H(x,t)=\varphi_{t\tau(x)}(x)$. 
\end{example}
\end{example}
Observamos haver uma diferença entre a deformação efetivada por um
homeomorfismo e aquela proporcionada por uma equivalência homotópica:
enquanto a primeira ``preserva dimensões'', a segunda não se submete
a tal restrição. Em outras palavras, equivalências homotópicas têm
a permissão para ``colapsar dimensões''. Ilustremos este fato através
de alguns exemplos:
\begin{example}
Quando $n\neq m$, os espaços $\mathbb{R}^{n}$ e $\mathbb{R}^{m}$
nunca são homeomorfos. Este é um resultado conhecido como \emph{teorema
da invariância da dimensão}, cuja demonstração (assim como o teorema
da retração de Brouwer) segue dos cálculos $\pi_{n}(\mathbb{S}^{n})\simeq\mathbb{Z}$
e $\pi_{i}(\mathbb{S}^{n})=0$, com $1<i<n$, os quais obteremos um
pouco mais adiante. Com efeito, deles se retira $\pi_{n-1}(\mathbb{S}^{n})\neq\pi_{n-1}(\mathbb{S}^{n-1})$,
de modo que esferas de dimensões distintas não podem nem mesmo serem
homotópicas (daremos uma prova deste fato na próxima subsecção sem
fazer uso de cálculos envolvendo os grupos de homotopia). Observando
que compactificações preservam homeomorfismos, segue-se que $\mathbb{R}^{n}=\mathbb{S}^{n}\cup\infty$
e $\mathbb{R}^{m}=\mathbb{S}^{m}\cup\infty$ nunca são homeomorfos.
Em contrapartida, quaisquer espaços euclidianos são sempre homotópicos.
Afinal, ambos são contráteis.

\begin{example}
Os subespaços $\mathbb{R}^{n}-0$ e $\mathbb{S}^{n-1}$ do $\mathbb{R}^{n}$
não são homeomorfos: a esfera é compacta, algo não satisfeito se do
$\mathbb{R}^{n}$ retiramos um ponto. Em contrapartida, a aplicação
$x\mapsto x/\Vert x\Vert$, de $\mathbb{R}^{n}-0$ em $\mathbb{S}^{n-1}$,
é equivalência homotópica. Particularmente, ainda que o cilindro $\mathbb{S}^{1}\times\mathbb{R}$
não seja homeomorfo ao círculo $\mathbb{S}^{1}$, eles possuem a mesma
homotopia.
\end{example}
\end{example}
No modelo de Ström, homotopias admitem uma outra interpretação. Com
efeito, relembrando que a categoria $\mathscr{C}$ é simétrica e cartesianamente
completa, temos 
\[
\mathrm{Mor}_{\mathscr{C}}(X\times I;Y)\simeq\mathrm{Mor}_{\mathscr{C}}(I\times X;Y)\simeq\mathrm{Mor}_{\mathscr{C}}(I;\mathrm{Map}(X;Y)),
\]
de modo que uma homotopia entre $f,g;X\rightarrow Y$ se identifica
como um caminho em $\mathrm{Map}(X;Y)$ que parte de $f$ e chega
em $g$. Assim, duas aplicações são homotópicas se, e somente se,
estão na mesma componente conexa por caminhos do espaço $\mathrm{Map}(X;Y)$.

\subsection*{\uline{Limites Homotópicos}}

$\quad\;\,$No capítulo cinco comentamos que, para quaisquer categorias
$\mathbf{J}$ e $\mathbf{C}$, com $\mathbf{C}$ dotada de equivalências
fracas, o functor $\lim:\mathrm{Func}(\mathbf{J};\mathbf{C})\rightarrow\mathbf{C}$
geralmente não é homotópico. No presente contexto, isto significa
que limites podem não preservar tipo de homotopia. Vejamos dois exemplos:
\begin{example}
Se tomamos o disco $\mathbb{D}^{n}$ e identificamos seu bordo num
único ponto, obtemos a esfera $\mathbb{S}^{n}$, a qual não é contrátil,
pois tem o mesmo tipo de homotopia que $\mathbb{R}^{n}-0$ (primeiro
dos diagramas abaixo). Por sua vez, sabemos que $\mathbb{D}^{n}$
é contrátil e, portanto, homotópico a um ponto. No entanto, o segundo
dos \emph{pushouts} abaixo resulta no próprio $\mathbb{D}^{n}$. Assim,
ainda que troquemos objetos de um diagrama por outros com o mesmo
tipo de homotopia, os correspondentes limites podem não possuir a
mesma homotopia.$$
\xymatrix{\mathrm{Ps} & \ar[l] \mathbb{D}^n && \mathrm{Ps'} & \ar[l] \mathrm{*} \\
\mathrm{*} \ar[u] & \ar[l] \mathbb{S}^{n-1} \ar[u]_{\imath} && \mathrm{*} \ar[u] & \ar[l] \mathbb{S}^{n-1} \ar[u]_{\imath} }
$$

\begin{example}
E se ao invés de considerarmos diagramas com objetos homotópicos passarmos
a olhar para diagramas com morfismos homotópicos? Ainda assim, os
correspondem limites podem não ter a mesma homotopia. Com efeito,
o primeiro dos \emph{pullbacks} abaixo resulta no produto $\mathbb{S}^{n-1}\times\mathbb{S}^{n-1}$.
Como $\mathbb{D}^{n}$ é contrátil, sua identidade é homotópica a
uma (e, portanto, qualquer) aplicação constante. Seja $c$ tal aplicação.
Colocando-a no lugar da identidade, vê-se que o \emph{pullback} resultante
ou é formado de um só ponto (se $c(x)\in\mathbb{S}^{n-1}$), ou então
é vazio.$$
\xymatrix{\mathrm{Pb} \ar[r] \ar[d] & \mathbb{D}^n \ar[d]^{id} &&  \mathrm{Pb'} \ar[r] \ar[d] & \mathbb{D}^n \ar[d]^{c} \\
\mathbb{S}^{n-1} \ar[r]_{\imath} & \mathbb{D}^n && \mathbb{S}^{n-1} \ar[r]_{\imath} & \mathbb{D}^n }
$$
\end{example}
\end{example}
Alguns tipos de limites, no entanto, fogem à regra e preservam certas
propriedades homotópicas. É isto o que veremos no próximo exemplo.
Adiantamos que ele é o caminho para se classificar qualquer fibrado
em $\mathscr{C}$ que seja localmente trivial e cujas fibras estejam
dotadas da ação de um grupo topológico, algo que veremos na próxima
subsecção.
\begin{example}
Vejamos que, se $\pi:X\rightarrow X'$ é um fibrado localmente trivial
e $f,g:Y\rightarrow X'$ são funções homotópicas, então os \emph{pullbacks}
$f^{*}X$ entre $(f,\pi)$, e $g^{*}X$ entre $(g,\pi)$, não só tem
o mesmo tipo de homotopia, como também são fibrados isomorfos. Para
tanto, seja $H$ homotopia entre $f$ e $g$, e consideremos a aplicação
$h:f^{*}X\times I\rightarrow X'$ obtida compondo $pr\times id$ com
$H$. Como $\pi$ é fibração, existe o mapa $h'$ apresentado no segundo
diagrama. Daí, por universalidade obtém-se $u$, que é um isomorfismo.
Fazendo o mesmo para $g$ e utilizando de transitividade, chega-se
em $f^{*}X\times I\simeq g^{*}X\times I$, donde o fibrado $f^{*}X$
ser isomorfo a $g^{*}X$.
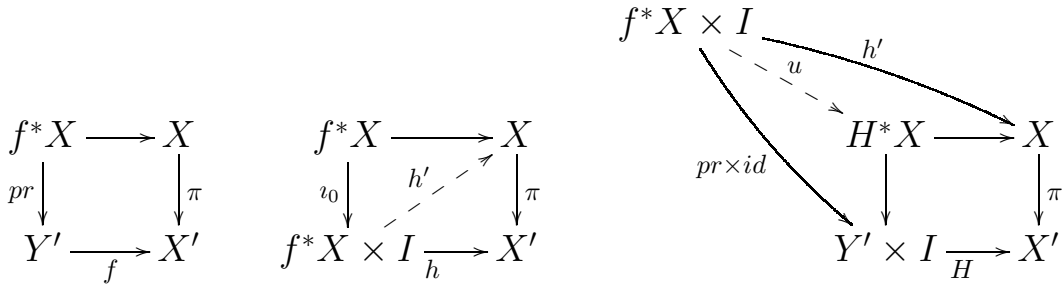
$$
\xymatrix{&&&& f^{*}X \times I \ar@/_/[rdd]_{pr \times id} \ar@/^/[rrd]^{h'} \ar@{-->}[rd]^u \\
f^{*}X \ar[r] \ar[d]_{pr} & X \ar[d]^{\pi} &  f^{*}X \ar[r] \ar[d]_{\imath _{0}} & X \ar[d]^{\pi} && H^{*}X \ar[r] \ar[d] & X \ar[d]^{\pi} \\
Y' \ar[r]_f & X' & f^{*}X \times I \ar[r]_h \ar@{-->}[ru]^{h'} & X' && Y' \times I \ar[r]_H & X'}
$$
\end{example}

\subsection*{\uline{Classificação dos Fibrados}}

$\quad\;\,$O último exemplo é particularmente interessante: ele nos
dá uma ideia clara de como classificar fibrados localmente triviais
(com base fixa) através da segunda estratégia apresentada no capítulo
dois. Com efeito, fixado $\pi:X\rightarrow X'$, seja $[-]:\mathscr{LB}\rightarrow\mathbf{Set}$
o functor que a cada fibrado localmente trivial associa a sua classe
de isomorfismos. Pelo referido exemplo, $\pi:X\rightarrow X'$ induz
uma transformação $\xi:[-,X']_{\mathscr{C}}\rightarrow[-]$, definida
por $\xi(Y)([f])=[f^{*}X]$. Busca-se, então, por fibrados $\pi:X\rightarrow X'$
tais que $\xi$ seja um isomorfismo natural. Diz-se que eles são \emph{universais}. 

Pode-se pensar no problema da seguinte maneira: o que o exemplo anterior
nos diz é que a parte ``$f:Y\rightarrow X$'' do \emph{pullback}
abaixo apresentado é homotópica. O fibrado $\pi:X\rightarrow X'$
será universal precisamente quando a parte que lhe corresponde (e,
portanto, o \emph{pullback} como um todo) for homotópico. Neste caso,
o fibrado classificado por $f$ pode ser computado através do \emph{mapping
cocylinder} de $\pi$:$$
\xymatrix{& X \ar[d]^{\pi} && \mathrm{HPb} \ar@{.>}[dd] \ar@{.>}[r] & \mathrm{ccyl}(\pi ) \ar@{-->}[r] \ar@{-->}[d] & X \ar[d]^{\pi} \\
Y \ar[r]_f & X' &&& \mathrm{path}(X') \ar[r] \ar[d] & X'\\
&&& Y \ar[r]_f &  X'}
$$

No contexto dos fibrados com grupo estrutural, a tarefa é mais simples:
todo $\pi:X\rightarrow X'$ com espaço total contrátil é trivial (veja,
por exemplo, \cite{Steenrod}). Como apresentaremos em seguida, há
uma construção canônica que permite associar a cada grupo topológico
$G$ um fibrado principal $EG\rightarrow BG$ sobre $G$, tal que
$EG$ é contrátil. Assim, $G$-fibrados principais (e, portanto, quaisquer
fibrados tendo $G$ como grupo estrutural) são classificados pelo
correspondente $BG$. Particularmente, uma aplicação $f:X\rightarrow Y$
classifica precisamente o seu \emph{mapping cocone}.

Iniciamos observando que, como os fibrados associados classificam
os fibrados com grupos estrutural, basta nos resumirmos aos fibrados
principais. Existem dois functores $\mathbf{Grp}\rightarrow\mathbf{Cat}$,
os quais nos permitem ver cada grupo $G$ como categorias $\mathbf{E}G$
e $\mathbf{B}G$, definidas como segue. Os objetos de $\mathbf{G}$
são precisamente os elementos de $G$, ao passo que entre $g,h\in\mathbf{E}G$
há um único morfismo $g\rightarrow h$, constituído do único elemento
$g'\in G$ tal que $g\circ s=h$. Por sua vez, $\mathbf{B}G$ possui
um único objeto $*$. Os morfismos $*\rightarrow*$ correspondem aos
elementos de $G$.

Sejam $N_{s}\mathbf{G}$ e $N_{s}\mathbf{G}'$ os nervos simpliciais
de $\mathbf{E}G$ e $\mathbf{B}G$. Suas $n$-símplices se identificam,
respectivamente, com produtos $G^{n+1}$ e $G^{n}$ de $n+1$ e de
$n$ cópias de $G$. Assim, há uma ação de $G$ em $G^{n+1}$ cujo
espaço de órbitas é isomorfo a $G^{n}$. Tal ação comuta operadores
de bordo, de modo que fica bem definido $p:N_{s}\mathbf{G}\rightarrow N_{s}\mathbf{G}'$
que em $G^{n}$ é $G^{n}\rightarrow G^{n}/G$.

Tem-se também um functor $\Delta\rightarrow\mathscr{C}$, que a cada
$[n]$ faz corresponder o subespaço $\Delta^{n}\subset\mathbb{R}^{n+1}$
formado de todas as listas de $n+1$ números reais não-negativos,
cuja soma é exatamente igual a um. Uma vez que $\mathscr{C}$ é completa
e cocompleta, segue-se a existência de adjuntos $\mathrm{Sing}:\mathscr{C}\rightarrow s\mathbf{Set}$
e $\vert\cdot\vert:s\mathbf{Set}\rightarrow\mathscr{C}$. Para cada
grupo topológico $G$, o espaço $\vert N_{s}\mathbf{G}\vert$ é contrátil,
de modo que o fibrado obtido aplicando $\vert\cdot\vert$ em $p:N_{s}\mathbf{G}\rightarrow N_{s}\mathbf{G}'$
é universal na categoria dos $G$-fibrados.

\section{Invariantes}

$\quad\;\,$Iniciamos esta secção mostrando que $\mathscr{C}_{*}$
pertence a uma classe de categorias com equivalências fracas na qual
suspensões e \emph{loops} podem ser computados através do produto
\emph{smash}. Por um lado, sendo um colimite homotópico, a suspensão
de qualquer espaço pontuado $X\in\mathscr{C}_{*}$ pode ser obtida
das duas formas abaixo apresentadas:$$
\xymatrix{\Sigma X & CX  \ar@{.>}[l]   & \ar@{-->}[l] \mathrm{*}  & \Sigma X \ar@{.>}[r] & CX \ar@{-->}[r]  & \mathrm{*}  \\
& \ar@{-->}[u] X\times I  & \ar[l] \ar[u] X & \ar@{.>}[u] CX  & \ar@{-->}[l] \ar@{-->}[u] X\times I   & \ar[l] X  \ar[u]  \\
\mathrm{*} \ar@{.>}[uu]  & \ar[l] \ar[u] X && \ar@{-->}[u] \mathrm{*}  & \ar[l] \ar[u] X}
$$

A primeira nos diz que a suspensão de $X$ é a colagem de dois cones
idênticos $CX$ que têm $X$ como base (estes se tratam, pois, do
\emph{mapping cylinder} de $X\rightarrow*$). O segundo diagrama mostra
que $\Sigma X$ é homotópico ao quociente $CX/X$. 

Como $I$ é contrátil, $X\vee I$ tem o mesmo tipo de homotopia que
$X$. Além disso, $X=X\wedge\mathbb{S}^{0}$, pois a esfera $\mathbb{S}^{0}$
é o objeto neutro da estrutura monoidal definida em $\mathscr{C}_{*}$
por $\wedge$. Substituindo tais colocações na primeira das maneiras
de calcular a suspensão, vemos que $\Sigma X$ tem o mesmo tipo de
homotopia que o quociente de $X\wedge I$ por $X\wedge\mathbb{S}^{0}$.
Por fim, sendo $\mathscr{C}$ cartesianamente fechada e simétrica,
$X\wedge-$ possui adjunto à direita e, portanto, preserva \emph{pushouts}.
Isto significa que 
\[
X\wedge I/X\wedge\mathbb{S}^{0}\simeq X\wedge(I/\mathbb{S}^{0}),\quad\mbox{donde}\quad\Sigma X\simeq X\wedge\mathbb{S}^{1}.
\]

De maneira dual, verifica-se que o \emph{loop} $\Omega X$ é equivalente
a $\mathrm{Map}_{*}(\mathbb{S}^{1};X)$. Particularmente, estes fatos
nos mostram que, além de adjuntos, os functores $-\wedge\mathbb{S}^{1}$
e $\mathrm{Map}_{*}$ são homotópicos, tendo passagens ao quociente
coincidentes com $\Sigma$ e $\Omega$. Em virtude destas coincidências,
não faremos distinção entre um e outro.
\begin{example}
Tem-se homeomorfismos $\mathbb{S}^{n}\wedge\mathbb{S}^{m}\simeq\mathbb{S}^{n+m}$,
donde $\Sigma\mathbb{S}^{n}\simeq\mathbb{S}^{n+1}$. Uma forma de
visualizar o resultado é a seguinte: por meio das projeções estereográficas,
a esfera admite a decomposição $\mathbb{S}^{n}\simeq e^{0}\cup e^{n}$,
em que $e^{i}\simeq\mathrm{int}\mathbb{D}^{i}$ é chamado de \emph{célula
de dimensão $i$}. Da mesma forma, 
\[
\mathbb{S}^{n}\times\mathbb{S}^{m}\simeq e^{0}\cup e^{n}\cup e^{m}\cup e^{n+m}\quad\mbox{e}\quad\mathbb{S}^{n}\vee\mathbb{S}^{m}\simeq e^{0}\cup e^{n}\cup e^{m},
\]
Passando ao quociente, obtém-se o resultado procurado: 
\[
\mathbb{S}^{n}\wedge\mathbb{S}^{m}\simeq\mathbb{S}^{n}\times\mathbb{S}^{m}/\mathbb{S}^{n}\vee\mathbb{S}^{m}\simeq e^{0}\cup e^{n+m}\simeq\mathbb{S}^{n+m}.
\]
Uma construção explícita do homeomorfismo $\Sigma\mathbb{S}^{n}\simeq\mathbb{S}^{n+1}$
pode ser encontrada na sexta secção do primeiro capítulo de \cite{Spanier}. 
\end{example}

\subsection*{\uline{\mbox{$H$}-Espaços}}

$\quad\;\,$Além da estrutura proveniente do produto \emph{smash},
$\mathscr{H}\mathscr{C}_{*}$ também admite estruturas monoidais definidas
por seus produtos e coprodutos binários (lembre-se de que, apesar
da categoria homotópica $\mathscr{H}\mathscr{C}_{*}$ não ser completa
e cocompleta, ela possui produtos e coprodutos). Os grupos e cogrupos
inerentes a cada uma delas são respectivamente chamados de $H$-\emph{espaços}
e de $H$\emph{-coespaços}, tendo categorias denotadas por $\mathrm{HSp}(\mathscr{C}_{*})$
e $\mathrm{H^{op}Sp}(\mathscr{C}_{*})$.
\begin{example}
Evidentemente, todo grupo em $\mathscr{C}_{*}$ é um $H$-espaço.
Portanto, todo grupo topológico Hausdorff e compactamente gerado,
estando pontuado pelo seu elemento neutro, é um $H$-espaço. Este
é o caso de $(\mathbb{S}^{1},1)$. Da mesma forma, todo cogrupo em
$\mathscr{C}_{*}$ é um $H$-coespaço, de modo que o círculo também
admite uma estrutura natural de $H$-coespaço.

\begin{example}
Provaremos agora que, se $X$ é $H$-coespaço, então o produto \emph{smash}
$X\wedge Y$ também o é, seja qual for o $Y\in\mathscr{C}_{*}$. Com
efeito, basta mostrar que $[X\wedge Y;-]_{\mathscr{C}_{*}}$ é grupo
para todo $Y$. Como $[X\wedge Y;Z]_{\mathscr{C}_{*}}\simeq[X;\mathrm{Map}_{*}(Y;Z)]_{\mathscr{C}_{*}}$,
se $X$ é $H$-co-espaço, então $[X;-]$ é grupo e, portanto, cada
um dos conjuntos da direita também são grupos. Há uma única estrutura
de grupo em $[X\wedge Y;Z]_{\mathscr{C}_{*}}$ que faz das referidas
bijeções isomorfismos. Assim, $[X\wedge Y;-]_{\mathscr{C}_{*}}$ é
grupo, garantindo o que havíamos afirmado.
\end{example}
\end{example}
Como consequência imediata dos exemplos anteriores, segue-se que cada
$\Sigma^{n}$ assume valores na categoria $\mathrm{H^{op}Sp}(\mathscr{C}_{*})$.
Isto significa que cada $X\in\mathscr{C}_{*}$ define uma sequência
de functores $\pi_{n}^{X}:\mathscr{H}\mathscr{C}_{*}\rightarrow\mathbf{Grp}$,
que tomam $Y\in\mathscr{C}_{*}$ e associam o grupo $\pi_{n}^{X}(Y,y_{o})=[\Sigma^{n}X;Y]$.
Em suma, estes são os\emph{ primeiros exemplos de invariantes topológicos
construídos por métodos algébricos com os quais nos deparamos}. Observamos,
em particular, que estes são invariantes \emph{homotópicos}. Isto
significa que, se $Y\simeq Y'$, então os correspondentes grupos a
eles associados são obrigatoriamente isomorfos.

De maneira dual, poderíamos considerar os functores $\pi_{X}^{n}$,
que a cada espaço pontuado $Y$ fazem corresponder $[X;\Omega^{n}Y]_{\mathscr{C}_{*}}$.
As adjunções $[\Sigma^{n}X;Y]_{\mathscr{C}_{*}}\simeq[X;\Omega^{n}Y]_{\mathscr{C}_{*}}$,
no entanto, garantiriam que os respectivos invariantes associados
por $\pi_{X}^{n}$ e $\pi_{n}^{X}$ são sempre isomorfos. Pelo argumento
argumento de Eckmann-Hilton, isto é, pelos isomorfismos 
\[
\mathrm{Mon}(\mathrm{Mon}(\mathbf{Set};\times),\times_{\mathrm{M}})\simeq\mathrm{_{\mbox{\ensuremath{c}}}Mon}(\mathbf{Set};\times)=\mathbf{AbGrp},
\]
para quaisquer que sejam $n,m>0$ e $X,Y\in\mathscr{C}_{*}$, os respectivos
grupos $[\Sigma^{n}X;\Omega^{m}Y]_{\mathscr{C}_{*}}$ são abelianos.
Daí, para $n>1$, os functores $\pi_{n}^{X}$ e $\pi_{X}^{n}$ chegam
em $\mathbf{AbGrp}$. Em outras palavras, na sequência de invariantes
definidos por $X$, somente o primeiro não é abeliano.

\subsection*{\uline{Grupos de Homotopia}}

$\quad\;\,$Procuramos por invariantes que, em certo sentido, não
carregam arbitrariedades. Isto nos leva a questionar os papéis do
espaço $X$ e do ponto base $y_{o}$ em $\pi_{n}^{X}(Y,y_{o})$. 

No que toca o primeiro questionamento, observemos que, como os functores
$-\wedge\mathbb{S}^{1}$ são homotópicos, espaços $X$ e $X'$ que
tenham o mesmo tipo de homotopia produzem functores $\pi_{n}^{X}$
e $\pi_{n}^{X'}$ naturalmente isomorfos. Daí, os invariantes $\pi_{n}^{X}(Y,y_{o})$
não são sensíveis à troca de espaços que estejam na mesma classe de
$X$. Em particular, isto retira nosso interesse dos invariantes definidos
por espaços contráteis. Pois, se $X\simeq*$, então o grupo $\pi_{n}^{X}(Y,y_{o})$
é sempre trivial.

O tipo homotópico mais simples depois de um ponto é um ``ponto pontuado''.
Em outras palavras, tratam-se dos espaços que estão na classe de $\mathbb{S}^{0}$.
Vejamos que, em grande parte dos casos, estes produzem invariantes
$\pi_{n}^{\mathbb{S}^{0}}(X,x_{o})$, denotados simplesmente por $\pi_{n}(X,x_{o})$
e chamados de \emph{grupos de homotopia} de $X$, que não dependem
do ponto base $x_{o}$ escolhido e, portanto, não possuem arbitrariedades.
Isto os torna particularmente interessantes.

Iniciamos observando haver uma $\omega$-categoria $\Pi(X)$, cujos
objetos são os elementos de $X$, cujos $1$-morfismos $x_{o}\rightarrow x_{o}'$
são os caminhos ligando tais pontos, cujos $2$-morfismos são as homotopias
livres (isto é, que preservam os pontos inicial e final) entre caminhos,
e assim sucessivamente. Tal categoria é, em verdade, um $\infty$-grupoide.
Consequentemente, também podemos vê-la enquanto grupoide: basta manter
os objetos e considerar como morfismos as classes de homotopia de
caminhos $x_{o}\rightarrow x_{o}'$. Daí, todo functor $F^{n}:\Pi(X)\rightarrow\mathbf{Grp}$
cumprindo $F^{n}(x_{o})=\pi_{n}(X,x_{o})$ manda classes $[\gamma]:x_{o}\rightarrow x_{o}'$
em isomorfismos. Como consequência, $\pi_{n}(X,x_{o})$ depende apenas
da componente conexa por caminhos de $X$ a qual $x_{o}$ pertence.

Um exemplo destes functores é aquele que toma $[\gamma]$ e associa
$F_{\gamma}^{n}:\pi_{n}(X,x_{o})\rightarrow\pi_{n}(X,x_{o}')$, que
a cada classe $[f]\in[\mathbb{S}^{n};X,x_{o}]_{\mathscr{C}_{*}}$
devolve a correspondente classe das aplicações contínuas $g:\mathbb{S}^{n}\rightarrow(X,x_{o}')$
que podem ser livremente deformadas em $f$ através de $\gamma$.
Isto é, tais que existe homotopia livre $H:\mathbb{S}^{n}\times I\rightarrow X$
entre $g$ e $f$, com $H(*,t)=\gamma(t)$, em que $*$ base de $\mathbb{S}^{n}$.

\subsection*{\uline{Grupo Fundamental}}

$\quad\;\,$O primeiro dos grupos de homotopia é chamado de \emph{grupo
fundamental}. Ele possui uma interpretação bastante clara, como passamos
a discutir.

Seja $\mathscr{C}_{2}$ a categoria dos pares de $\mathscr{C}$. Diz-se
que dois mapas em $\mathscr{C}_{2}$ são homotópicos quando existe
uma homotopia usual entre eles, a qual preserva os subespaços distinguidos.
Isto define relações de equivalência em cada conjunto de morfismos
de $\mathscr{C}_{2}$, a qual é compatível com composições. Daí, fica
definida a categoria $\mathscr{H}\mathscr{C}_{2}$, cujos morfismos
são as classes de mapas entre pares. Como pode ser conferido em \cite{convenient_category_steenrod,gray_homotopy},
quando $A\subset X$ é fechado e $Y\in\mathscr{C}_{*}$, tem-se bijeções
\begin{equation}
\mathrm{Mor}_{\mathscr{HC}_{2}}(X,A;Y,y_{o})\simeq[X/A;Y]_{\mathscr{C}_{*}}.\label{prop_espacos_cmpct_gerados}
\end{equation}

Assim, $\pi_{1}(X,x_{o})$ nada mais é que o conjunto das classes
de homotopia de caminhos $\gamma:I\rightarrow X$ cumprindo $\gamma(0)=x_{o}$
e $\gamma(1)=x_{o}$. Desta forma o grupo fundamental de $(X,x_{o})$
será trivial se, e somente se, todo \emph{loop} em $x_{o}$ puder
ser continuamente deformado até tornar o caminho constante em $x_{o}$.
Isto é, se, e só se, $X$ não possuir buracos no entorno de $x_{o}$.
Assim, o invariante $\pi_{1}$ mensura, exatamente, a existência de
buracos.
\begin{example}
Como discutimos no primeiro capítulo, os invariante proporcionados
pela topologia geral não são muito fortes. Por exemplo, através deles
não se consegue provar que a esfera $\mathbb{S}^{2}$ não é homeomorfa
ao toro $\mathbb{S}^{1}\times\mathbb{S}^{1}$. A topologia algébrica,
como também discutimos, vêm para construir invariantes mais poderosos.
De fato, acabamos de obter um invariante $\pi_{1}$ que mensura a
existência de buracos. Ora, a diferença entre a esfera e o toro está
justamente neste ponto: $\mathbb{S}^{1}\times\mathbb{S}^{1}$ tem
buracos, enquanto que $\mathbb{S}^{2}$ não os possui. Daí, deve-se
ter $\pi_{1}(\mathbb{S}^{1}\times\mathbb{S}^{1})\neq\pi_{1}(\mathbb{S}^{2})$,
de modo que tais espaços não são homeomorfos. Em particular, como
$\pi_{1}$ é invariante homotópico, isto mostra que $\mathbb{S}^{1}\times\mathbb{S}^{1}$
e $\mathbb{S}^{2}$ não possuem nem mesmo igual tipo de homotopia.

\begin{example}
No quarto capítulo, introduzimos uma estrutura monoidal na categoria
$\mathbf{C}^{n}$ das subvariedades compactas e orientáveis de $\mathbb{R}^{n+1}$,
cujo objeto neutro é $\mathbb{S}^{n}$. Quando $n=2$, comentamos
que $\mathrm{Eqv}(\mathbf{C}^{2})$ é isomorfo aos naturais $\mathbb{N}$,
pelo morfismo que a cada $g$ associa a classe de $\mathbb{S}_{g}^{1}$.
O número $g$ é chamado de \emph{genus} da superfície e mede o número
de buracos que ela possui. Assim, pode-se dizer que a classificação
de $\mathbf{C}^{2}$ se dá exatamente por $\pi_{1}$.
\end{example}
\end{example}
Para terminar, mostremos que a independência do ponto base dos grupos
de homotopia se traduz numa ação de $\pi_{1}(X,x_{o})$ em cada $\pi_{n}(X,x_{o})$,
cujo quociente é precisamente o conjunto $[\mathbb{S}^{n};X]_{\mathscr{C}}$
das classes de homotopia livre. Com efeito, pela identificação (\ref{prop_espacos_cmpct_gerados}),
segue-se que $\mathbf{B}\pi_{1}(X,x_{o})$ é a subcategoria de $\Pi(X)$
que possui somente $x_{o}$ como objeto e $x_{o}\rightarrow x_{o}$
como morfismos. Enquanto restrito a ela, o functor $F^{n}:\Pi(X)\rightarrow\mathbf{Grp}$
se traduz numa ação 
\[
*:\pi_{1}(X,x_{o})\times\pi_{n}(X,x_{o})\rightarrow\pi_{n}(X,x_{o}),\quad\mbox{definida por}\quad[\gamma]*[f]=F_{\gamma}^{n}([f]),
\]
cuja órbita está em bijeção com $[\mathbb{S}^{n};X]_{\mathscr{C}}$.
Consequentemente, \emph{se um espaço conexo por caminhos possui $\pi_{1}(X,x_{o})$
trivial para algum $x_{o}$, situação em que é dito ser simplesmente
conexo, então tratar de mapas pontuados é o mesmo que tratar de mapas
sem pontuação. }
\begin{example}
Pode-se ter bijeções $\pi_{n}(X,x_{o})\simeq[\mathbb{S}^{n};X]_{\mathscr{C}}$
mesmo que $X$ não seja simplesmente conexo. Com efeito, basta que
a ação de $\pi_{1}$ em cada $\pi_{n}$ seja trivial. Isto é, basta
que $[\gamma]*[f]=[f]$ para cada $\gamma$ e cada $f$. Isto ocorre,
por exemplo, quando $X$ é um $H$-espaço: dados $\gamma$ e $f$,
a aplicação $H(t,x)=\gamma(t)\cdot f(x)$ é homotopia livre de $f$
em $f$ ao longo de $\gamma$. Veja o teorema 4.18 do terceiro capítulo
da referência \cite{whitehead_homotopy}.
\end{example}

\section{Cálculo}

$\quad\;\,$Uma vez obtidos invariantes $F:\mathbf{C}\rightarrow\mathbf{D}$,
de uma determinada categoria $\mathbf{C}$ preocupa-se em calculá-los.
Isto é, procura-se determinar $F(X)$ módulo isomorfismos. Neste sentido,
as ferramentas mais importantes são as sequências exatas envolvendo
$F$. No contexto da homotopia clássica, estas serão estudadas nas
duas próximas subsecções. 

Observamos, por outro lado, que a tarefa de calcular um invariante
se torna mais simples na medida em que este preserva maior número
de limites e colimites. Com efeito, se um objeto $X$ for obtido por
limites/colimites preservados por $F$, então poderemos calcular $F(X)$
através dos invariantes associados aos objetos cujo limite/colimite
originam $X$. Assim, particularmente, se estamos interessados em
homotopia, haveremos de buscar por functores que preservem o maior
número de limites/colimites homotópicos.

Na subsecção anterior, obtivemos os functores $\pi_{i}:\mathscr{HC}_{*}\rightarrow\mathbf{Grp}$,
cujos invariantes são os grupos de homotopia. Tais invariantes são
tremendamente difíceis de serem calculados, especialmente quando $n>1$.
Por exemplo, ainda não se sabe calcular exatamente os grupos $\pi_{i}(\mathbb{S}^{n})$,
com $i>n$. Uma das razões disso reside justamente no fato de que
os $\pi_{i}$ preservam poucos limites/colimites homotópicos. Com
efeito, na próxima subsecção mostraremos que, em geral, os $\pi_{i}$
não preservam coprodutos. Por sua vez, logo mais veremos que tais
functores costumeiramente falham na preservação de \emph{pushouts}
homotópicos.

Uma exceção à regra são os limites de $\mathscr{HC}_{*}$ (que existem
em pequena quantidade, como já discutimos). Isto se deve precisamente
ao fato dos $\pi_{i}$ serem todos representáveis. Como exemplo, segue-se
que eles preservam produtos. Em outras palavras, tem-se os seguintes
isomorfismos: 
\[
\pi_{i}(X\times Y,x_{o}\times y_{o})\simeq\pi_{i}(X,x_{o})\times\pi_{i}(Y,y_{o}).
\]

Também é fácil mostrar que os grupos de homotopia preservam certos
limites indutivos. Mais precisamente, se $X$ é limite indutivo de
um sistema dirigido de inclusões $X_{j}\rightarrow X_{j+1}$, então
existem os isomorfismos abaixo (a ideia da prova reside na observação
de que qualquer mapa $K\rightarrow X$, com $K$ compacto, assume
valores em algum $X_{i}$): 
\[
\pi_{i}(X)\simeq\mathrm{colim}\pi_{i}(X_{j}).
\]

\subsection*{\uline{Seifert-Van Kampen}}

$\quad\;\,$Uma exceção especial à falta de preservação de colimites
homotópicos por partes dos grupos de homotopia $\pi_{i}$ acontece
para $i=1$. Nela, conta-se com um resultado conhecido como \emph{Teorema
de Seifert-Van Kampen}, o qual afirma que $\pi_{1}$ preserva certos
\emph{pushouts }homotópicos. Mais precisamente, ele garante que, se
um quadrado é comutativo em $\mathscr{C}_{*}$ e um \emph{pushout}
homotópico em $\mathscr{C}$, com $P$ conexo por caminhos, então
sua imagem por $\pi_{i}$ é um \emph{pushout} em $\mathbf{Grp}$.
$$
\xymatrix{P & X \ar@{-->}[l] &&  \pi_1(P) & \pi_1(X) \ar@{-->}[l] \\
X' \ar@{-->}[u] & Y \ar[l] \ar[u]  \ar@<0.7cm>@{=>}[rr]^{\pi _1} && \pi_1(X') \ar@{-->}[u] & \pi_1(Y) \ar[l] \ar[u] }
$$

Ressaltamos que este fato não é válido num caso geral: o problema
está na exigência de que o \emph{pushout }homotópico parta de um diagrama
comutativo em $\mathscr{C}_{*}$.
\begin{example}
Se $X$ é espaço no qual age um grupo topológico $G$, o \emph{pushout}
homotópico associado à órbita da ação em geral não se estende a um
diagrama comutativo em $\mathscr{C}_{*}$. Isto porque nem sempre
a ação de $G$ preserva um ponto base previamente fixado em $X$. 
\end{example}
Assim, o problema está justamente no fato de trabalharmos com espaços
pontuados. Desta forma, a ideia é substituir o grupo fundamental pelo
grupoide fundamental. Isto pode realmente ser feito e, neste caso,
mostra-se que $\Pi$ preserva qualquer colimite homotópico (veja \cite{fundamental_group_homotopy_colimits}). 

Devido ao poderio do teorema de Seifer-Van Kampen, uma questão natural
diz respeito a possibilidade de se obter um análogo para os outros
grupos de homotopia. Tendo-se isto em mente, observamos que tal teorema
foi bastante melhorado quando considerado sob a perspectiva do grupoide
fundamental. Assim a ideia seria buscar outros grupoides, os quais
estejam relacionados com os grupos de homotopia de grau superior.
Ocorre que homotopias de grau superior indicam maior informação categórica.
Desta forma forma, ao grupo $\pi_{n}$ espera-se ter associado não
só um grupoide, mas um $n$-grupoide $\Pi_{n}$. Com o intuito de
analisar esta relação, respondamos algumas questões:
\begin{itemize}
\item \emph{Quais propriedades $\Pi_{n}$ deve possuir}? No caso $n=1$,
tem-se $\Pi(X)\simeq\mathbf{B}\pi_{1}(X,x)$. De maneira análoga,
espera-se poder olhar $\pi_{n}(X,x)$ como uma $n$-categoria com
um único objeto e obter $\Pi_{n}\simeq\pi_{n}(X,x)$. No entanto,
há um problema: o grupo fundamental é o únicos dos grupos de homotopia
que não é necessariamente abeliano. Assim, ao tentar mimetizar o que
é feito para $\pi_{1}$, ao invés de trabalhar com os grupos de homotopia
de grau superior, deve-se trabalhar com functores que a eles estejam
relacionados, mas que ao mesmo tempo assumam valores em certas $n$-categorias
$\mathbf{C}_{n}$ ``não-abelianas''. Diz-se, neste caso, que se
está em busca de uma \emph{topologia algébrica não-abeliana}.
\item \emph{Quem são as} $\mathbf{C}_{n}$? Por um lado, sabemos que $\pi_{1}$
assume valores em $\mathbf{Grp}$, que é uma categoria usual. Por
outro, queremos functores que chegam em $n$-categorias e que estejam
relacionados com os grupos de homotopia. Assim, queremos sair de $\pi_{1}$
e aumentar o nível de informação categórica. A maneira natural de
se fazer isso é através de \emph{categorificação}. Portanto, de maneira
indutiva, espera-se que $\mathbf{C}_{2}$ seja a categoria $\mathrm{Cat}(\mathbf{Grp})$
das categorificações de $\mathbf{Grp}$, e que $\mathbf{C}_{n}$ seja
formada das categorificações de $\mathbf{C}_{n-1}$. Ora, uma maneira
de categorificar é enriquecer sobre $\mathbf{Cat}$. Desta forma,
$\mathrm{Cat}(\mathbf{Grp})$ há de ser identificado com enriquecimentos
da categoria $\mathbf{Grp}$ na categoria dos grupoides, e assim sucessivamente.
\item \emph{Em que sentido as categorias $\mathbf{C}_{n}$ assim definidas
são não-abelianas}? Os functores $\pi_{i\geq2}$ assumem valores em
$\mathbf{AbGrp}\simeq\mathbf{Mod}_{\mathbb{Z}}$. Por sua vez, $R$-módulos
são ações de $R$ em grupos abelianos. Assim, os $\pi_{i\geq2}$ assumem
valores na categoria das ações de grupos abelianos em grupos abelianos.
Olhando por este lado, poder-se-ia pensar em definir $\mathbf{C}_{n}$
como sendo alguma categoria de ações de grupos em grupos (não necessariamente
abelianos). Ocorre que $\mathrm{Cat}(\mathbf{Grp})$ se identifica
com a categoria dos \emph{módulos cruzados}, que são ações deste tipo.
\end{itemize}
$\quad\;\,$Depois desta digressão, conclui-se que para obter uma
generalização do teorema de Seifert-Van Kampen deve-se olhar para
$n$-grupoides correspondentes à functores assumindo valores nas categorias
$\mathbf{C}_{n}$. Tais functores podem realmente ser construídos,
forncendo generalização procurada. Explicitar esta construção e apresentar
importantes consequências do teorema do Seifert-Van Kampen generalizado
constituem os principais objetivos da referência \cite{non_abelian_algebraic_topology}.

Um último comentário: no desenvolvimento anterior, partimos do fato
de que, diferentemente do grupo fundamental, os outros grupos de homotopia
são sempre abelianos. Posto isto, buscamos por functores que substituíssem
os $\pi_{i\geq2}$ e assumissem valores em categorias ``não-abelianas''.
Assim, os passos fundamentais nesta estratégia de generalização consistiram
em \emph{manter o foco nos grupos de homotopia e realizar a passagem
do abeliano para o não-abeliano}. 

Ao invés disso, poderíamos ter-nos desapegado dos grupos de homotopia
e procurado por novos invariantes abelianos que cumpram condição semelhante
ao teorema de Seifert-Van Kampen. Neste caso, cairíamos nas chamadas
\emph{teorias de homologia/cohomologia}, as quais fazem parte de uma
outra área da Topologia Algébrica Clássica. Boas referências sobre
o assunto incluem \cite{Dold_topologia_algebrica,Hu_homologia,Greenberg_topologia_algebrica},
assim como o clássico \cite{axiomas_homologia}. 

A vantagem desta abordagem é que ela conta com as ferramentas da Álgebra
Homológica, que lida essencialmente com a construção e manipulação
de sequências exatas. A respeito disso, referenciamos o leitor aos
textos \cite{Hilton_homological_algebra,homological_algebra_CARTAN,MACLANE_homology},
que tratam do assunto na categoria dos módulos, e também às obras
\cite{Heller_homological_algebra,Rotman_homological_algebra,Weibel_homological_algebra},
que desenvolvem a teoria no contexto mais abstrato das categorias
abelianas.

\subsection*{\uline{Homotopia Relativa}}

$\quad\;\,$Em geral, na tentativa de comparar os invariantes de um
subespaço $A\subset X$ com os invariantes que do espaço todo, introduz-se
invariantes ``relativos'' associados ao par $(X,A)$. Nesta subsecção
buscaremos uma boa definição para os \emph{grupos de homotopia relativa},
usualmente denotados por $\pi_{i}(X,A,x_{o})$, que contemple esta
ideia. Em boas situações, eles serão responsáveis por mensurar efetivamente
a diferença entre $\pi_{i}(A,x_{o})$ e $\pi_{i}(X,x_{o})$.

Como já comentamos, as ferramentas básicas para o cálculo de invariantes
são as sequências exatas. No âmbito das categorias modelo, vimos que
todo morfismo gera sequências exatas longas de fibrações e de cofibrações.
No presente contexto, o morfismo fundamental é a inclusão $\imath:A\rightarrow X$,
que produz a sequência de fibrações$$
\xymatrix{\cdots \ar[r] & \Omega \mathrm{ccone}(\imath) \ar[r] & \Omega A \ar[r]^{-\Omega \imath} & \Omega X \ar[r] & \mathrm{ccone}(\imath) \ar[r] & A \ar[r]^f & Y}
$$ 

Buscamos por uma sequência exata que relacione os grupos $\pi_{i}(A)$
e $\pi_{i}(X)$. Ora, cada $\pi_{i}$ é representável e isomorfo a
$[\mathbb{S}^{n};-]_{\mathscr{C}_{*}}$. Mas functores representáveis
preservam sequências de fibrações. Assim, ao aplicar $\pi_{i}$ na
sequência anterior, como $\pi_{i}(\Omega X)\simeq\pi_{i+1}(X)$, obteremos
aquela que procuramos (veja o diagrama abaixo). Ressaltamos que, da
maneira como foi obtida, tal sequência é exata em $\mathbf{Set}$.
No entanto, pode-se mostrar que, se $i\geq1$ ela é exata em $\mathbf{Grp}$,
ao passo que, se $i\geq2$, então ela é exata em $\mathbf{AbGrp}$
(confira a página 62 de \cite{MAY_1}). $$
\xymatrix{\cdots \ar[r] \ar[r] & \pi_{i+1} (A) \ar[r] & \pi_{i+1}(X) \ar[r] & \pi_{i}(\mathrm{ccone}(\imath)) \ar[r] & \pi_{i}(A) \ar[r] & \pi_{i}(X) \ar[r] & \cdots}
$$ 

Agora observamos que o termo que faz a conexão entre os grupos de
$A$ e os grupos de $X$ é precisamente o \emph{mapping cocone $\mathrm{ccone}(\imath)$
}da inclusão $\imath:A\rightarrow X$. Por definição, este se trata
do \emph{pullback }homotópico abaixo, o qual pode ser calculado substituindo
$\imath$ por uma resolução fibrante, como apresentado ao seu lado.$$
\xymatrix{\mathrm{ccone}(\imath) \ar@{-->}[r] \ar@{-->}[d] & A \ar[d]^{\imath} && \mathrm{ccone}(\imath) \ar@{..>}[dd] \ar@{..>}[r] & \mathrm{ccyl}(\imath) \ar@{-->}[d] \ar@{-->}[r] & A \ar[d]^{\imath} \\
\mathrm{*} \ar[r] & X &&& \mathrm{Path}(X) \ar[d] \ar[r] & X \\
&&& \mathrm{*} \ar[r] & X}
$$

Se tomamos a resolução fibrante canônica, então $\mathrm{Path}(X)\simeq\mathrm{Map}(I;X)$
e o \emph{mapping path space} $\mathrm{ccyl(\imath)}$ se identifica
com o espaço de todos os caminhos em $X$ que chegam em $A$. Portanto,
vemos que o \emph{mapping cocone} é simplesmente o espaço dos caminhos
em $X$ que partem do ponto base $x_{o}\in A\subset X$ e chegam em
$A$, o qual depende da terna par $(X,A,x_{o})$.

Concluímos, assim, que o grupo $\pi_{i}(\mathrm{ccone}(\imath))$
depende não só de $X$, mas também de $A$ e do ponto base $x_{o}$.
Além disso, é responsável por conectar invariantes de $X$ e de $A$.
Estes fatos servem de motivação para defini-lo como sendo o \emph{$i$-ésimo
grupo de homotopia relativa}. 

Há, no entanto, uma requisição bastante natural a ser feita: \emph{espera-se}
\emph{que um invariante relativo do par $(X,A)$ recaia no invariante
absoluto de $X$ quando $A=X$}. Observamos que isto não é satisfeito
se definimos $\pi_{i}(X,A)$ como sendo $\pi_{i}(\mathrm{ccone}(\imath))$.
De fato, se $A=X$, então $\imath=id_{X}$, cujo \emph{mapping cocone}
é $\Omega X$. Portanto, tomando a definição anterior, teríamos 
\[
\pi_{i}(X,X)=\pi_{i}(\mathrm{ccone}(id_{X}))\simeq\pi_{i}(\Omega X)\simeq\pi_{i+1}(X),
\]
contrariando o desejado. O problema pode ser facilmente contornado
definindo $\pi_{i}(X,A)$ como sendo $\pi_{i-1}(\mathrm{ccone}(\imath))$.
É esta a definição costumeiramente adotada na literatura, da qual
compartilharemos neste tomo. Com ela, as sequência exata de fibração
de $\imath:A\rightarrow X$, agora chamada de \emph{sequência exata
em homotopia do par} $(X,A)$, se escreve$$
\xymatrix{\cdots \ar[r] \ar[r] & \pi_{i+1} (A) \ar[r] & \pi_{i+1}(X) \ar[r] & \pi_{i+1}(X,A) \ar[r] & \pi_{i}(A) \ar[r] & \pi_{i}(X) \ar[r] & \cdots}
$$ 

Vejamos um exemplo de situação na qual a homotopia relativa mostra
sua utilidade:
\begin{example}
Mostraremos que, \emph{se $\imath:A\rightarrow X$ admite uma secção
$s$, então para $i\geq2$ a homotopia relativa mede, efetivamente,
quão a homotopia de $A$ difere da de $X$}. De fato, se $s$ é secção
para $\imath$, então $\pi_{i}(s)$ é secção para cada $\pi_{i}(\imath)$,
seja qual for o $i$. Isto faz com que a sequência longa de fibração
de $\imath$ possa ser obtida juntando sequências exatas curtas, referentes
a cada $i$. Mais precisamente, o fato de existirem secções na sequência
longa faz com que $\pi_{i}(\imath)$ sejam sempre sobrejetivas e,
consequentemente, que as $\partial$ sejam injetiva, de modo que as
sequências curtas$$
\xymatrix{0 \ar[r] & \pi_{i+1}(X,A) \ar[r]^-{\partial} & \pi_i (A) \ar[r] &  \pi_i (X) \ar[r] & 0}
$$são todas exatas. No âmbito abeliano, uma sequência exata curta admite
uma secção na segunda aplicação se, e somente se, o objeto do meio
se decompõe como o coproduto dos termos exteriores. Este é o conteúdo
do \emph{sppliting lemma} (veja \cite{Hilton_homological_algebra,MACLANE_homology,Rotman_homological_algebra,Weibel_homological_algebra}),
o qual pode ser entendido como uma generalização do teorema do núcleo
e da imagem utilizado na álgebra linear. Observamos que, se $i\geq2$,
então os grupos de homotopia são todos abelianos e então estamos no
domínio de validade do \emph{sppliting lemma}. Portanto, em tal situação,
tem-se 
\[
\pi_{i}(A)\simeq\pi_{i}(X)\oplus\pi_{i+1}(X,A).
\]
Situação estritamente análoga vale quando $\imath:A\rightarrow X$
possui uma retração ao invés de uma secção. Neste caso, no entanto,
as $\pi_{i}(\imath)$ ficam injetivas (e não mais sobrejetivas), de
modo que são as sequências curtas$$
\xymatrix{0 \ar[r] & \pi_{i}(A) \ar[r] & \pi_i (X) \ar[r] &  \pi_{i} (X,A) \ar[r] & 0}
$$que se mostram exatas. Por uma versão do \emph{sppliting lemma} segue-se
que, para $i\geq2$,
\[
\pi_{i}(X)\simeq\pi_{i}(A)\oplus\pi_{i}(X,A).
\]

\begin{example}
O exemplo anterior nos permite mostrar que, em geral, os grupos de
homotopia não preservam coprodutos. Com efeito, dados espaços pontuados
$X,Y\in\mathscr{H}\mathscr{C}_{*}$, a inclusão natural $X\vee Y\rightarrow X\times Y$
admite uma secção, de modo que, para $n\geq2$, 
\begin{eqnarray*}
\pi_{i}(X\vee Y) & \simeq & \pi_{i}(X\times Y)\oplus\pi_{i+1}(X\times Y,X\vee Y)\\
 & \simeq & \pi_{i}(X)\oplus\pi_{i}(Y)\oplus\pi_{i+1}(X\times Y,X\vee Y).
\end{eqnarray*}
\end{example}
\end{example}

\subsection*{\uline{Sequências de Fibrações}}

$\quad\;\,$Existem diversos livros na literatura que fazem uma introdução
à Topologia Algébrica via grupo fundamental e espaços de recobrimento
(exemplos são \cite{fundamental_group_FULTON,fundamental_group_MASSEY,grupo_fundamental_ELON}
e a segunda parte de \cite{munkres}). A razão é a seguinte: como
provaremos em seguida, toda fibração (como aquelas oriundas dos fibrados
localmente triviais) define uma sequência exata, a qual nos permite
relacionar o grupo de homotopia das fibras, com os grupos de homotopia
da base e do espaço total. 

Ora, quando se quer introduzir um assunto, procura-se motivar seu
interesse mostrando que a teoria desenvolvida contempla alguns teoremas.
Para demonstrá-los, necessita-se de poderosas ferramentas, das quais
sequências exatas constituem exímio exemplo. Por outro lado, o grupo
fundamental é o mais simples dos grupos de homotopia, ao passo que
os espaços de recobrimento são os mais simples fibrados localmente
triviais. Assim, ao se juntar os dois, tem-se uma teoria simples,
mas dotada de boas ferramentas.

Vamos à construção da sequência comentada. Iniciamos observando que,
uma vez fixado um ponto base $x_{o}\in X$, todo mapa $f:(X,x_{o})\rightarrow(Y,y_{o})$,
com $y_{o}=f(x_{o})$, induz uma inclusão natural $\imath$ da fibra
$X_{o}=f^{-1}(y_{o})$ em $X$. Tem-se as sequências exatas longas
de fibrações associadas a $f$ e à $\imath$. A primeira é responsável
por relacionar os grupos da base $Y$ com os grupos do espaço total
$X$, ao passo que a segunda tem o papel $\imath$ relacionar os grupos
da fibra $X_{o}$ com os de $X$. Portanto, ao se ligar tais sequências
obtém-se um vínculo entre todos estes grupos. 

Tal ligação é obtida no caso especial em que $f$ é uma fibração.
Com efeito, esta hipótese faz com que o \emph{pushout} homotópico
associado ao \emph{mapping cocone} de $f$ produza espaços com mesmo
tipo de homotopia se calculado com (ou sem) o uso de uma resolução
fibrante para $f$. Se não a usamos (isto é, se calculamos tal \emph{pushout}
de maneira estrita), obtemos $\mathrm{ccone}(f)\simeq X_{o}$, de
modo que as sequências exatas de $f$ e de $\imath$ se conectam:$$
\xymatrix{\cdots \ar[r] & \Omega \mathrm{ccone}(f) \ar@{~>}[d] \ar[r] & \Omega X  \ar@{~>}[d]^{id} \ar[r] & \Omega Y \ar[d] \ar[r] & \mathrm{ccone}(f)  \ar@{~>}[d]  \ar[r] & X \ar@{~>}[d]^{id} \ar[r]^f & Y \\
\cdots \ar[r] & \Omega X_{o} \ar[r] & \Omega X \ar[r] & \mathrm{ccone}(\imath) \ar[r] & X_{o} \ar[r]_{\imath} & X}
$$ 

Agora, tomando $\pi_{i}$ no diagrama anterior, obtemos o diagrama
abaixo, o qual é formado de cinco setas, sendo as das extremidades
isomorfismos. Consequentemente, por um resultado conhecido como \emph{lema
dos cinco} (confira \cite{Hilton_homological_algebra,MACLANE_homology,Rotman_homological_algebra,Weibel_homological_algebra}),
a seta do meio também é um isomorfismo.$$
\xymatrix{\pi_{i+1}(\mathrm{ccone}(f)) \ar[d] \ar[r] & \pi _{i+1}(X)  \ar[d] \ar[r] & \pi _{i+1}(Y) \ar[d] \ar[r] & \pi_i(\mathrm{ccone}(f))  \ar[d]  \ar[r] & \pi _i(X) \ar[d]\\
\pi_{i+1}(X_{o}) \ar[r] & \pi _{i+1}(X) \ar[r] & \pi_{i+1}(X,X_o) \ar[r] & \pi _i(X_o) \ar[r] & \pi _i(X)}
$$ 

Utilizando de tal identificação na sequência exata longa de $\imath$,
chegamos, finalmente, naquela que cumpre o papel que procurávamos:$$
\xymatrix{\cdots \ar[r] & \pi_{i+1}(X_{o}) \ar[r] & \pi _{i+1}(X) \ar[r] & \pi_{i+1}(Y) \ar[r] & \pi _i(X_o) \ar[r] & \pi _i(X) \ar[r] & \cdots}
$$

Ilustremos a aplicabilidade de tal sequência por meio de dois exemplos.
No segundo deles, esclarecemos a relação entre espaços de recobrimento
e grupo fundamental que a pouco realçamos.
\begin{example}
Dos exemplos da última subsecção segue-se que, para todo fibrado localmente
trivial $f:X\rightarrow Y$ que possui uma secção $s:Y\rightarrow X$,
a correspondente sequência curta$$
\xymatrix{0 \ar[r] & \pi_{i}(X_{o}) \ar[r] & \pi _{i}(X) \ar[r] & \pi_{i}(Y) \ar[r] & 0}
$$é exata, seja qual for o $i$. Se $i\geq2$, então os grupos de homotopia
são abelianos, de modo que podemos utilizar do \emph{splliting lemma
}para concluirmos $\pi_{i}(X)\simeq\pi_{i}(X_{o})\oplus\pi_{i}(Y)$.
Particularmente, se as fibras são contráteis, então $\pi_{i}(X)\simeq\pi_{i}(Y)$.
Este é o caso dos fibrados vetoriais (que englobam os fibrados tangentes
das variedades), cujas fibras são espaços lineares.

\begin{example}
De maneira dual, se $f:X\rightarrow Y$ possui uma retração, então
a sequência curta abaixo apresentada é exata e, em particular, quando
$i\geq2$ podemos aplicar o \emph{splliting lemma}.$$
\xymatrix{0 \ar[r] & \pi_{i}(X) \ar[r] & \pi _{i}(Y) \ar[r] & \pi_{i-1}(X_o) \ar[r] & 0}
$$Suponhamos que $f$ é um \emph{espaço de recobrimento}. Em outras
palavras, suponhamos ser $f$ um fibrado localmente trivial com fibras
discretas. Isto não garante nem mesmo que ele possui uma retração,
mas faz com que cada $\pi_{i}(f)$ a tenha, o que é suficiente para
que as sequências curtas acima continuem exata. Em particular, para
$i\geq2$, faz com que $\pi_{i-1}(X_{o})=0$, de modo que $\pi_{i}(X)\simeq\pi_{i}(Y)$.
Num âmbito geral não se tem muita informação quando $i=1$. Ocorre
que, sob condições bastante razoáveis, se um espaço $Y$ admite um
recobrimento $X$, então também admite um recobrimento simplesmente
conexo $X'$. Daí, a exatidão da sequência anterior para fornece uma
$\pi_{1}(Y)\simeq\pi_{0}(X_{o}')$. Em suma, para boa parte dos espaços
que admitem um recobrimentos, calcular grupo fundamental é contar
o número de elementos de alguma fibra. Por exemplo, a aplicação $\exp:\mathbb{R}\rightarrow\mathbb{S}^{1}$,
definida por $\exp(t)=(\cos t,\sin t)$ é recobrimento para o círculo.
Como a reta é simplesmente conexa, $\pi_{1}(\mathbb{S}^{1})$ está
em bijeção com $\exp^{-1}(0,1)\simeq\mathbb{Z}$. De outro lado, como
$\pi_{i}(\mathbb{R})\simeq\pi_{i}(\mathbb{S}^{1})$, vê-se que todos
os grupos de homotopia de ordem superior do círculo são triviais. 
\end{example}
\end{example}

\end{document}